\documentclass[12pt,reqno]{amsart}

\allowdisplaybreaks[1]

\usepackage{amssymb,amsmath,amsthm,amscd,ifthen,xr,array}
\usepackage{color,cancel,todonotes}
\usepackage[normalem]{ulem}
\usepackage{comment}
\usepackage{graphicx}
\usepackage{xcolor}
\usepackage{scalerel}
\usepackage{tikz}
\usetikzlibrary{decorations.shapes}
\usetikzlibrary{decorations.pathreplacing}
\usetikzlibrary{matrix}
\author{M. McKee}
\address{338 Hammond Lane, Providence, UT 84332, USA}
\email{mark.mckee.zoso@gmail.com}
\author{A. Pasquale}
\address{Université de Lorraine, CNRS, IECL, F-57000 Metz, France}
\email{angela.pasquale@univ-lorraine.fr}

\author{T. Przebinda}
\address{Department of Mathematics, University of Oklahoma, Norman, OK 73019, USA}
\email{przebinda@gmail.com}

\title[Symmetry breaking operators for dual pairs with one member compact]{Symmetry breaking operators\\ for dual pairs with one member compact}

\def\ch{\mathop{\hbox{\rm ch}}\nolimits}

\def\g{\mathfrak g}
\def\z{\mathfrak z}

\def\h{\mathfrak h}
\def\sp{\mathfrak {sp}}

\def\o{\mathfrak o}

\def\u{\mathfrak u}
\def\k{\mathfrak k}

\def\R{\mathbb{R}}
\def\C{\mathbb{C}}
\def\Ze{\mathbb{Z}}
\def\Ha{\mathbb{H}}

\def\z{\mathfrak z}
\def\c{\mathfrak c}

\def\so{\mathfrak s_{\overline 0}}
\def\ss1{\mathfrak s_{\overline 1}}

\def\hs1{\mathfrak h_{\overline 1}}

\def\supp{\mathrm{supp}}
\def\Op{\mathrm{Op}}
\def\Ker{\mathop{\mathrm{Ker}}}

\def\G{\mathrm{G}}
\def\Lg{\mathrm{L}}
\def\N{\mathrm{N}}
\def\Zg{\mathrm{Z}}

\def\SOg{\mathrm{S}\mathrm{O}}
\def\K{\mathrm{K}}
\def\H{\mathrm{H}}
\def\M{\mathrm{M}}
\def\Zg{\mathrm{Z}}
\def\Sg{\mathrm{S}}
\def\Spin{\mathrm{Spin}}

\def\L{\mathrm{L}}
\def\Bbb{\mathbb}

\def\N{\mathrm{N}}

\def\H{\mathrm{H}}

\def\GL{\mathrm{GL}}
\def\SL{\mathrm{SL}}
\def\SO{\mathrm{SO}}
\def\Sp{\mathrm{Sp}}

\def\Og{\mathrm{O}}
\def\Ug{\mathrm{U}}
\def\Vg{\mathrm{V}}
\def\Se{\mathcal{S}}

\def\Mg{\mathrm{M}}
\def\fo{\mathcal{F}}

\newcommand{\anticomm}[2]{\null^{#1}#2}
\newcommand{\danticomm}[2]{\null^{\anticomm{#1}{#2}}#2}

\def \t{\tilde}
\def \wt{\widetilde}
\newcommand{\reg}[1]{ {#1}^{reg}}

\def\W{\mathrm{W}}
\def\Wv{\mathrm{W}}
\def\W+{\mathrm{W}_{\BB C}}
\def\Vv{\mathrm{V}}
\def\Uv{\mathrm{U}}
\def\V{\mathsf{V}}

\def\X{\mathrm{X}}
\def\Y{\mathrm{Y}}
\def\Xv{\mathrm{X}}
\def\Yv{\mathrm{Y}}
\def\Sy{\mathsf{S}}

\def\Dc{\mathbb {D}}
\def\Zb{\mathbb {Z}}
\def\Hc{\mathcal {H}}

\def\End{\mathop{\hbox{\rm End}}\nolimits}
\def\diag{\mathop{\hbox{\rm diag}}\nolimits}
\def\ad{\mathop{\hbox{\rm ad}}\nolimits}
\def\Ad{\mathop{\hbox{\rm Ad}}\nolimits}
\def\Hom{\mathop{\rm Hom}\nolimits}
\def\Re{\mathop{\rm Re}\nolimits}
\def\Im{\mathop{\rm Im}\nolimits}
\def\triv{\mathop{\rm triv}}

\def\tr{\mathop{\rm tr}\nolimits}

\def\sgn{\mathop{\hbox{\rm sgn}}\nolimits}

\def\vol{\mathop{\hbox{\rm vol}}\nolimits}

\def\lim{\mathop{\hbox{\rm lim}}\nolimits}

\newcommand\inner[2]{\langle #1,#2\rangle}

\def\Ind{\mathop{\hbox{\rm Ind}}\nolimits}
\def\supp{\mathop{\hbox{\rm supp}}\nolimits}

\def\Oo{\mathcal{O}}

\def\Ss{\mathcal{S}}

\newcommand{\Obigmat}[3]{
{\footnotesize
\left(
\begin{array}{c|c}
#1 & 0 \\ 
\hline
0 & 
\begin{matrix}
#2 & 0\\
0 & #3
\end{matrix}\!\!
\end{array}
\right)}}

\newcommand{\Omat}[1]{
\left(
\begin{array}{c|c}
h_\bullet & 0   \\ 
\hline
0 &   #1
\end{array}
\right)}

\newcommand{\Omatinv}[1]{
\left(
\begin{array}{c|c}
h_\bullet^{-1} & 0   \\ 
\hline
0 &   #1
\end{array}
\right)}

\newcommand{\Omatplus}[2]{
\left(
\begin{array}{c|c}
#1 & 0   \\ 
\hline
0 &   #2
\end{array}
\right)}

\newcommand{\Obigdiag}[3]{
\left(
\begin{array}{c|c}
\!\begin{matrix}
#1 & &0   \\
     & \ddots & \\
0 & &   #2
\end{matrix} 
\!\!
& 0 \\
\hline
0 &   #3
\end{array}
\right)}

\def\fontindex{\arabic}

\def\fonttitre{\textsf}
\newcounter{thh}
\newtheorem{thm}[thh]{\fonttitre{Theorem}}

\newtheorem{pro}[thh]{\fonttitre{Proposition}}
\newtheorem*{pro*}{\fonttitre{Proposition}}
\newtheorem{cor}[thh]{\fonttitre{Corollary}}
\newtheorem*{coro*}{\fonttitre{Corollary}}
\newtheorem{lem}[thh]{\fonttitre{Lemma}}
\theoremstyle{definition} \newtheorem{rem}{Remark}

\newtheorem*{defi*}{\fonttitre{Définition}}

\newtheorem*{nota*}{\fonttitre{Notation}}
\newenvironment{prf}{\begin{proof}}{\end{proof}}
\def\muet{ \ifthenelse{\equal{a}{b}}}
\def\nn{\nonumber}

\definecolor{darkgreen}{rgb}{0.09, 0.45, 0.27}
\definecolor{brown}{RGB}{139, 0, 0}

\newcommand{\tC}{\scaleto{\text{C}}{4pt}}

\def\biblio{\sloppy
\bibliographystyle{alpha}
\bibliography{article}}

\begin{document}

\thanks{The second author is grateful to the University of Oklahoma for hospitality and financial support. She is partially supported by the ANR Project OpART (ANR-23-CE40-0016).
The third author acknowledges hospitality and financial support from the 
Universit\'e de Lorraine and partial support from the NSA under Grant H98230-13-1-0205 and the 
NSF under Grant DMS-2225892. Part of this research took place within the online Research Community on Representation Theory and Noncommutative Geometry sponsored by the American Institute of Mathematics. Some work was done while the last two authors were at the Institute of Mathematical Sciences of the National University of Singapore. They gratefully thank the Institute for hospitality and financial support.}

\date{}

\maketitle
\begin{abstract}
We consider a dual pair $(\G,\G')$, in the sense of Howe, with $\G$ compact acting on $\L^2(\R^n)$, for an appropriate $n$, via the Weil representation $\omega$. Let $\wt\G$ be the preimage of $\G$ in the metaplectic group. 
Given a genuine irreducible unitary representation $\Pi$ of $\wt\G$, let $\Pi'$ be the corresponding irreducible unitary representation of $\wt\G'$ in Howe's correspondence. 
The orthogonal projection onto the $\Pi$-isotypic component $\L^2(\R^n)_\Pi$ is, up to a constant multiple, the unique symmetry breaking operator in $\Hom_{\wt\G\wt{\G'}}(\Hc_\omega^\infty, \Hc_\Pi^\infty\otimes \Hc_{\Pi'}^\infty)$. We study this operator by computing its Weyl symbol. 
Our results allow us to recover the known list of highest weights of irreducible representations of $\wt\G$ occurring in Howe's correspondence when the rank of $\G$ is strictly bigger than the rank of $\G'$. They also allow us to compute the wavefront set of $\Pi'$ by elementary means.  
\end{abstract}

\tableofcontents

%\linenumbers
%%

\section*{\bf Introduction}

Let $\Wv$ be a finite dimensional  vector space over $\Bbb R$ equipped with a non-degenerate symplectic form $\langle\cdot ,\cdot \rangle$ and let $\Sp(\Wv)$ denote the corresponding symplectic group. Write $\wt\Sp(\Wv)$ for the metaplectic group. Let us fix the character $\chi$ of $\R$ given by $\chi(r)=e^{2\pi i r}$, $r\in\R$. Then the Weil representation of $\wt\Sp(\Wv)$ associated to $\chi$ is denoted by $(\omega, \Hc_\omega)$.

For $\G, \G'\subseteq \Sp(\Wv)$ forming a reductive dual pair in the sense of Howe, let 
$\wt\G$, $\wt\G'$ denote their preimages  in $\wt\Sp(\Wv)$.
Howe's correspondence (or local $\theta$-correspondence) for $\wt\G$, $\wt\G'$ is a
bijection $\Pi\leftrightarrow \Pi'$ between the irreducible admissible representations of 
$\wt\G$ and $\wt{\G'}$ which occur as smooth quotients of $\omega$, \cite{HoweTrans}. 
It can be formulated as follows.
Assume that $\Hom_{\wt\G}(\Hc_\omega^\infty, \Hc_\Pi^\infty)\ne 0$.  Then $\Hom_{\wt\G}(\Hc_\omega^\infty, \Hc_\Pi^\infty)$ is a $\wt{\G'}$-module under the action via $\omega$.
Howe proved that it has a unique irreducible quotient, which is an irreducible admissible representation $(\Pi', \Hc_{\Pi'})$ of $\wt{\G'}$. Conversely, $\Hom_{\wt{\G'}}(\Hc_\omega^\infty, \Hc_{\Pi'}^\infty)$ is a $\wt\G$-module which has a unique irreducible admissible quotient, infinitesimally equivalent to  $(\Pi, \Hc_\Pi)$. 
Furthermore,  $\Pi\otimes\Pi'$ occurs as a quotient of $\omega^\infty$ in a unique way, i.e. 
\begin{equation}
\label{dim-symm-breaking}
\dim\Hom_{\wt\G\wt{\G'}}(\Hc_\omega^\infty, \Hc_\Pi^\infty\otimes \Hc_{\Pi'}^\infty)=1\,.
\end{equation}
In \cite{TKobayashiProgram}, the elements of 
\[
\Hom_{\wt\G}(\Hc_\omega^\infty, \Hc_\Pi^\infty)\,,\quad  \Hom_{\wt{\G'}}(\Hc_\omega^\infty, \Hc_{\Pi'}^\infty)\quad\text{and} \quad \Hom_{\wt\G\wt{\G'}}(\Hc_\omega^\infty, \Hc_\Pi^\infty\otimes \Hc_{\Pi'}^\infty)
\]
are called symmetry breaking operators. Their construction is part of Stage C of Koba\-yashi's program for branching problems in the representation theory of real reductive groups.

Since the last space is one dimensional, it deserves a closer look. The explicit contruction of the (essentially unique) symmetry breaking operator in 
$\Hom_{\wt\G\wt{\G'}}(\Hc_\omega^\infty, \Hc_\Pi^\infty\otimes \Hc_{\Pi'}^\infty)$
provides an alternative and direct approach to Howe's correspondence. 
To do this is the aim of the present paper. 
 
Our basic assumption is that $(\G, \G')$ is an irreducible dual pair with $\G$ compact.  
As shown by  Howe \cite{howetheta}, up to an isomorphism, $(\G, \G')$ is one of the pairs
\begin{equation} \label{classificationGG'}
(\Og_d, \Sp_{2m}(\R))\,,  \qquad  (\Ug_{d}, \Ug_{p,q})\,, \qquad        
(\Sp_d, \Og^*_{2m})\,.
\end{equation}
Then the representations $\Pi$ and  $\Pi'$ together with their contragredients are arbitrary irreducible unitary highest weight representations. They have been defined by Harish-Chandra in \cite{HC1955c}, were classified in \cite{EnrightHoweWallach} and have been studied in terms of Zuckerman functors in \cite{Wallachholomorphic}, \cite{Adamsdiscrete} and \cite{Adamshighestweight}. The 1-1 correspondence of representations in terms of their highest weights was first determined by Kashiwara and Vergne in \cite{KashiwaraVergne}. 

The crucial fact for constructing the symmetry breaking operator in $\Hom_{\wt\G\wt{\G'}}(\Hc_\omega^\infty, \Hc_\Pi^\infty\otimes \Hc_{\Pi'}^\infty)$ is that, up to a non-zero constant multiple, there is a unique $\G\G'$-invariant tempered distribution $f_{\Pi\otimes\Pi'}$ on 
$\Wv$ such that 
\begin{equation}
\label{intertwining}
\Hom_{\wt\G\wt{\G'}}(\Hc_\omega^\infty, \Hc_\Pi^\infty\otimes \Hc_{\Pi'}^\infty)=\C(\Op \circ \mathcal{K})(f_{\Pi\otimes\Pi'})\,,
\end{equation}
where $\Op$ and $\mathcal{K}$ are classical transformations which we shall review in section \ref{section:preliminaries}.
In \cite{PrzebindaUnitary}, $f_{\Pi\otimes\Pi'}$ is called the intertwining distribution associated to 
$\Pi\otimes\Pi'$.
In fact, if we work in a Schr\"odinger model of $\omega$, then $f_{\Pi\otimes\Pi'}$ happens to be the Weyl symbol, \cite{Hormander}, of the operator $(\Op \circ \mathcal{K})(f_{\Pi\otimes\Pi'})$.

The previous paragraph does not require $\G$ to be compact.
Suppose that the group $\G$ is compact. Let $\Theta_\Pi$ and $d_\Pi$ respectively denote the character and the degree of $\Pi$. Then the projection onto the $\Pi$-isotypic component of 
$\omega$ is equal to $d_\Pi/2$ times
\begin{equation}\label{projectiononomegapi}
\int_{\wt\G}\omega(\t g) \check\Theta_\Pi(\t g)\,d\t g=\omega(\check\Theta_\Pi)\,,
\end{equation}
where $\check\Theta_\Pi(\t g)=\Theta_\Pi(\t g^{-1})$ and we normalize the Haar measure $d\t g$ of $\t\G$ to have the total mass $2$. 
(This explains the constant multiple $\frac{1}{2}$ needed for the projection. 
In this way, the mass of $\G$ is equal to $1$.)
By Howe's correspondence with $\G$ compact, { the projection onto} the $\Pi$-isotypic component of $\omega$ is a {symmetry} breaking operator for $\Pi\otimes\Pi'$. 
The intertwining distribution for $\Pi\otimes\Pi'$ 
is therefore determined by the equation
\begin{equation}\label{projectiononomegapi1}
(\Op \circ \mathcal{K})(f_{\Pi\otimes\Pi'})=
\frac{1}{2}
\omega(\check\Theta_\Pi)\,.
\end{equation}
There are more cases when $f_{\Pi\otimes\Pi'}$ may be computed via the formula \eqref{projectiononomegapi1}, see \cite{PrzebindaUnitary}. However, if the group $\G$ 
is compact then the 
distribution character $\Theta_{\Pi'}$ may also be recovered from $f_{\Pi\otimes\Pi'}$ via an explicit formula, see \cite{PrzebindaUnipotent}. Thus, in this case, we have a diagram
\begin{equation}\label{first diagram}
\Theta_\Pi\longrightarrow f_{\Pi\otimes\Pi'} \longrightarrow \Theta_{\Pi'}\,.
\end{equation}
In general, the asymptotic properties of $f_{\Pi\otimes\Pi'}$ { relate} the associated varieties of the primitive ideals of $\Pi$ and $\Pi'$ and, under some more assumptions, the wave front sets of these representations, see \cite{PrzebindaUnitary},  \cite{PrzebindaUnipotent} and \cite{McKeePasqualePrzebindaWC_WF}.

The usual, often very successful, approach to Howe's correspondence avoids any work with distributions on the symplectic space. Instead, one finds Langlands parameters (see \cite{Moeglinarchi}, \cite{AdamsBarbaschcomplex}, \cite{AnnegretunitaryI}, \cite{AnnegretunitaryII}, 
\cite{Annegretorthosymplectic}, \cite{Jian-ShuLi-Cheng-boZhu}), character formulas (see \cite{Adamslift}, \cite{RenardLift}, \cite{DaszPrzebindaInv}, 
\cite{PrzebindaStableUnitary}, \cite{Merino2019characters}, \cite{LokePrzebinda_stable}), or candidates for character formulas (as in \cite{BerPrzeCHC_inv_eig}, \cite{PrzebindaCauchy}, \cite{LokePrzebinda_chc_padic_def}), or one establishes preservation of unitarity (as in \cite{Jian-ShuLiSingular}, \cite{HeHongyu}, \cite{PrzebindaUnitary}, \cite{AdamsBarbaschPaulTrapaVogan},  \cite{HarrisLiSun}, \cite{MaSunZhu_2017}). 
However, in the background (explicit or not), there is the orbit correspondence induced by the unnormalized moment maps
\[
\g^*\longleftarrow \Wv \longrightarrow \g'{}^*\,,
\]
where $\g$ and $\g'$ denote the Lie algebras of $\G$ and $\G'$, respectively, and 
$\g^*$ and $\g'^*$ are their duals. This correspondence of orbits has been studied in \cite{DaszKrasPrzebindaComplex}, \cite{DaszKrasPrzebindaK-S2} and \cite{PanReal}. 
Furthermore, in their recent work, \cite{LockMaassocvar}, Loke and Ma computed the associated variety of the representations for the dual pairs in the stable range in terms of the orbit correspondence. The $p$-adic case was studied in detail in \cite{Moeglinarchiwave}.

Working with the $\G\G'$-invariant distributions on $\Wv$ is a more direct approach than relying on the orbit correspondence and provides different insights and results. As a complementary contribution to all work mentioned above, we compute the intertwining distributions $f_{\Pi\otimes\Pi'}$ explicitly, see section \ref{main results}. As an application, we obtain the wave front set of $\Pi'$ by elementary means. 
The computation will be sketched in section \ref{section:wavefront}, and the detailed proof appeared in \cite{McKeePasqualePrzebindaWC_WF}. 
Another application of the methods presented in this paper leads to the explicit formula for the character of the corresponding irreducible unitary representation $\Pi'$ of $\wt\G'$. This can be found in \cite{AllanMerino-Thesis, Merino2019characters}.

The explicit formulas for the intertwining distribution provide important information on the nature of the symmetry breaking operators. Namely, they show that none of the 
symmetry breaking operators of the form $(\Op \circ \mathcal{K})(f_{\Pi\otimes\Pi'})$ is a differential operator. For the present situation, this answers in the negative the question on the existence of differential  symmetry breaking operators, addressed in different contexts by several authors (see for instance \cite{Kobayashi-Pevzner-DSBO-I, Kobayashi-Pevzner-DSBO-II, Kobayashi_Speh_AMS} and the references given there). This property is the content of Corollary \ref{cor:SBO-no-diff}.

Finally, observe that our computations leading to the intertwining distributions apply to any genuine irreducible representation $\Pi$ of the compact member $\wt{\G}$ of a dual pair. They provide an explicit formula for the Weyl symbol of the projection of $\omega|_{\wt{\G}}$
onto the $\Pi$-isotypic component. According to Howe's duality theorem, this projection is non-zero if and only if there is a unitary highest weight representation $\Pi'$ of $\wt{\G}'$ such that $\Pi\otimes \Pi'$ occurs in $\omega|_{\wt{\G}\wt{\G}'}$, i.e. $\Pi$ occurs in Howe's correspondence. 
When the rank of $\G$ is strictly bigger than that of $\G'$, we recover the known necessary and sufficient conditions on the highest weights of $\Pi$ so that it occurs in Howe's correspondence. 
See Corollary \ref{HW-l>l'}. 

The paper is organized as follows. In section \ref{section:preliminaries}, we introduce some notation and review the construction of the intertwining distributions. Section \ref{The center of the metaplectic group} computes the intertwining distribution for the dual pair 
$(\Zg, \Sp(\Wv))$, where $\Zg=\Og_1$ is the center of the symplectic group $\Sp(\Wv)$, and introduces some properties needed in the sequel. Section \ref{section:dualpairs-supergroups} recalls how to realize the dual pairs with one member compact as Lie supergroups, 
and section \ref{section:orbital integrals} collects some definitions and properties of the Weyl--Harish-Chandra integration formulas on $\Wv$ that we will need to compute the intertwining distributions. 
Section \ref{main results} states the main results of this paper. 
The dual pairs $(\Og_2,\Sp_{2l'}(\R))$ are particular because the group $\SOg_2$ is abelian. The intertwining distributions corresponing to these pairs are computed in section \ref{section:O2}. The smallest example of $(\Og_2,\Sp_{2}(\R)=\SL_2(\R))$ is presented with more details. An additional example is given in section \ref{examples}, where we illustrate the main two theorems when $(\G,\G')=(\Ug_l,\Ug_{p,p})$ and $\Pi$ is the trivial representation of $\Ug_l$. 
The proofs of the main results are in sections \ref{Intertwining distributions as an integral over g}, \ref{Intertwining distributions as an integral over h} and \ref{An intertwining distribution in terms of orbital integrals on the symplectic space}. We treat the special cases concerning the non-identity connected components of the orthogonal groups in sections \ref{The special case even}, \ref{The special case, odd 1}, \ref{section: proof of thm det} and \ref{The special case, odd 2}. Here we need the Weyl's integral and character formulas found by Wendt in \cite{Wendt}. 
Section \ref{section: proof corollary on UU l<=p+q} contains the proof of a necessary condition of a representation of $\wt{\Ug}_l$ to occur in Howe's correspondence for $(\Ug_l,\Ug_{p,q})$
when 
$
p=\min(p,q) <l \leq l'=p+q$. In section \ref{section: proof corollary HW-H-l <=l'}, we consider the dual pair $(\Sp_l,\Og^*_{2l'})$. 
Using intertwining distribution, we recover the known fact that  certain representations of $\Sp_l$ 
occur in Howe's correspondence. 
Finally, in section \ref{section:wavefront}, we outline how the results of this paper lead, for each representation $\Pi$ of $\wt{\G}$ occurring in Howe's duality, to the computation of the wave front set of the representation $\Pi'$ dual to $\Pi$. 
The details are in \cite{McKeePasqualePrzebindaWC_WF}.
The nine appendices collect and prove some auxiliary results. 
\medskip

\textit{Acknowledgement:} 
We are indebted to the anonymous referee whose extremely careful reading and valuable comments made us aware of errors and omissions in the original manuscript. 
The questions raised by the referee have lead us to make significant additions, 
which have greatly improved our paper. 
%%%%%%%%%%%%%
\section{\bf Notation and preliminaries}
\label{section:preliminaries}
Let us first recall the construction of the metaplectic group $\wt{\Sp}(\Wv)$ and the Weil representation $\omega$. We are using the approach of \cite[Section 4]{AubertPrzebinda_omega}, to which we refer the reader for more details.

Let $\sp$ denote the Lie algebra of $\Sp(\Wv)$, both contained in ${\rm End}(\Wv)$. 
Fix a positive definite compatible complex structure $J$ on $\Wv$, that is an element $J\in \sp$ such that $J^2=-1$ (minus the identity on $\Wv$) and the symmetric bilinear form $\langle J\cdot ,\cdot \rangle$ is positive definite on $\Wv$. 
For an element $g\in \Sp(\Wv)$, let $J_g=J^{-1}(g-1)$. The adjoint of $J_g$ with respect to the form $\langle J \cdot,\cdot\rangle$ is $J_g^*=Jg^{-1}(1-g)$. In particular, $J_g$ and $J_g^*$ have the same kernel. Hence the image of $J_g$ is
\[
J_g\Wv=(\Ker J_g^*)^\perp=(\Ker J_g)^\perp\,,
\]
where $\perp$ denotes the orthogonal 
complement 
with respect to $\langle J \cdot,\cdot\rangle$.
Therefore, the restriction of $J_g$ to $J_g\Wv$ defines an invertible element. Thus 
for every $g\neq 1$,
it makes sense to talk about $\det(J_g)_{J_g\Wv}^{-1}$, the reciprocal of the determinant of the restriction of $J_g$ to $J_g\Wv$. 
With this notation, we have
\begin{equation}\label{metaplectic group}
\wt{\Sp}(\Wv)=\{\t g=(g;\xi)\in \Sp(\Wv)\times \C,\ \ \xi^2=i^{\dim (g-1)\Wv}\det(J_g)_{J_g\Wv}^{-1}\}\,,
\end{equation}
with the convention that $\det(J_g)_{J_g\Wv}^{-1}=1$ if $g=1$.
There exists a $2$-cocycle $C:\Sp(\Wv)\times \Sp(\Wv) \to \C$, 
explicitly described in 
 \cite[Proposition 4.13]{AubertPrzebinda_omega}, such that $\wt{\Sp}(\Wv)$ is a group with respect to the multiplication 
\begin{equation}\label{multiplicationSp}
(g_1;\xi_1)(g_2;\xi_2)=(g_1g_2;\xi_1\xi_2 C(g_1,g_2))
\end{equation} 
and the homomorphism 
\begin{equation}\label{coveringhomomorphism}
\wt{\Sp}(\Wv)\ni (g;\xi)\to g\in \Sp(\Wv)
\end{equation} 
does not split.

Let $\mu_\Wv$ (or simply $dw$) be the Lebesgue measure on $\Wv$ normalized by the condition that the volume of the unit cube with respect to the form $\langle J\cdot ,\cdot \rangle$ is $1$. (Since all positive complex structures are conjugate by elements of $\Sp$, this normalization does not depend on the particular choice of $J$.) Let $\Wv=\Xv\oplus \Yv$ be a complete polarization. 
We suppose that $\Xv$, $\Yv$ and $J$ are chosen so that $J(\Xv)=\Yv$.
Similar normalizations are fixed for the Lebesgue measures on every vector subspace of $\Wv$, for instance on $\Xv$ and on $\Yv$. Furthermore, for every finite dimensional real vector space $\Vv$, we write 
$\Ss(\Vv)$ for the Schwartz space on $\Vv$ and 
$\Ss'(\Vv)$
for the space of tempered distributions on $\Vv$. 
We use the notation $\G'$ for the second member of a dual pair because it is the centralizer of $\G$ in $\Sp(\Wv)$. We also use the notation $\cdot{\,}'$ for all the objects associated with $\G'$, such as $\g'$, $\Pi'$, ... Unfortunately, this collides with the usual notation for the dual of a linear topological space in functional analysis, also used in this paper, such as $\mathcal{D}'(\R^n)$, $\mathcal{S}'(\R^n)$, ... We hope the reader will guess from the context the correct meaning of the notation.

Each element $K\in 
\Ss'(\Xv\times \Xv)$ defines an operator
$\Op(K)\in \Hom(\Ss(\Xv),\Ss'(\Xv))$ by
\begin{equation}\label{Op}
\Op(K)v(x)=\int_\Xv K(x,x')v(x')\,dx'.
\end{equation}
The map 
\begin{equation}
\label{Op}
\Op: \Ss'(\Xv\times \Xv) \to \Hom(\Ss(\Xv),\Ss'(\Xv))
\end{equation}
is an isomorphism of linear topological spaces. This is known as the Schwartz Kernel Theorem, 
\cite[Theorem 51.7]{Treves}. 
The Weyl transform is the linear isomorphism $\mathcal K:
\Ss'(\Wv)\to \Ss'(\Xv\times \Xv)$ defined for $f \in \Ss(\Wv)$ by 
\begin{equation}\label{K}
\mathcal K(f)(x,x')=\int_\Yv f(x-x'+y)\chi\big(\frac{1}{2}\langle y, x+x'\rangle\big)\,dy\,,
\end{equation}
(Recall that $\chi$ is the character of $\R$ we fixed at the beginning of the introduction.) 

For $g\in \Sp(\Wv)$, let
\begin{equation} \label{eq:chicg}
\chi_{c(g)}(u)=\chi\big(\tfrac{1}{4}\langle (g+1)(g-1)^{-1}u, u\rangle\big) \qquad (u=(g-1)w,\ w\in\Wv)\,.
\end{equation}
Notice that, if $g-1$ is invertible on $\Wv$, then 
$$
\chi_{c(g)}(u)=\chi\big(\tfrac{1}{4}\langle c(g)u, u\rangle\big)\,,
$$ 
where $c(g)=(g+1)(g-1)^{-1}$ is the usual Cayley transform.

Following \cite[Definition 4.23 and (114)]{AubertPrzebinda_omega}, we define
\begin{equation}\label{Tt}
T:\wt\Sp(\Wv)\ni \t g =(g;\xi)\longrightarrow  \xi\, \chi_{c(g)}\mu_{(g-1)\Wv}\in 
\Ss'(\Wv)\,,
\end{equation}
where $\mu_{(g-1)\Wv}$ is the Lebesgue measure on the subspace $(g-1)\Wv$ normalized as above, i.e. the volume of the unit cube with respect to the form $\langle J\cdot ,\cdot \rangle$ is $1$. 
Set
\begin{equation}
\label{omega}
\omega=\Op\circ \mathcal K \circ T\,.
\end{equation}
As proved in \cite[Theorem 4.27]{AubertPrzebinda_omega}, $\omega$ is a unitary representation of $\wt\Sp$ on $\L^2(\Xv)$. In fact, $(\omega, \L^2(\Xv))$ is the 
Schr\"odinger model of Weil representation of $\wt\Sp$ attached to the character $\chi$ and the 
polarization $\Wv=\Xv\oplus\Yv$. In this realization, $\mathcal{H}_\omega=\L^2(\Xv)$ and 
$\mathcal{H}_\omega^\infty=\mathcal{S}(\Xv)$.

The distribution character of the Weil representation turns out to be the function
\begin{equation}\label{THETA}
\Theta: \wt{\Sp}(\Wv)\ni (g;\xi)\to\xi\in\C^\times\,,
\end{equation}
\cite[Proposition 4.27]{AubertPrzebinda_omega}. Hence for $\t g\in \wt{\Sp}(\Wv)$ in the preimage of $ g\in \Sp(\Wv)$ under the double covering map \eqref{coveringhomomorphism}, we have
\begin{equation}\label{THETAT}
T(\t g)=\Theta(\t g) \chi_{c(g)}\mu_{(g-1)\Wv}\qquad (\t g\in \wt{\Sp}(\Wv))\,.
\end{equation}

Suppose now that $\G, \G'\subseteq \Sp(\Wv)$ is a dual pair. Every irreducible admissible representation $\Pi\otimes\Pi'$ of $\wt\G\times\wt\G'$ occurring in Howe's correspondence 
may be realized, up to infinitesimal equivalence, as a subspace of $\mathcal H_\omega^\infty{}'=
\Ss'(\Xv)$. Hence
$$
\Hom_{\wt\G\wt{\G'}}(\Hc_\omega^\infty, \Hc_\Pi^\infty\otimes \Hc_{\Pi'}^\infty)
\subseteq \Hom(\mathcal{S}(\Xv),\Ss'(\Xv))\,.
$$
The existence of the interwining distribution $f_{\Pi\otimes\Pi'}\in \Ss'(\Wv)$ defined (up to a multiplicative constant) by \eqref{intertwining} is thus a
consequence of \eqref{dim-symm-breaking}, \eqref{Op} and \eqref{K}.

Finally, because of \eqref{omega}, equation \eqref{projectiononomegapi} and \eqref{projectiononomegapi1} lead to the equality
\begin{equation}
\label{to-compute}
f_{\Pi\otimes\Pi'}=
\frac{1}{2}T(\check\Theta_\Pi)
=\int_\G \check\Theta_\Pi(\t g)T(\t g) \, dg\,.
\end{equation}
The problem of finding an explicit expression for $f_{\Pi\otimes\Pi'}$ is hence transformed into the task of computing the right-hand side of \eqref{to-compute}. 

%%%%%%%%%%%%%%
\section{\bf The center of the metaplectic group}
\label{The center of the metaplectic group}
Let $\Zg=\{1, -1\}$ be the center of the symplectic group $\Sp(\Wv)$. 
Then $(\Zg, \Sp(\Wv))$ is a dual pair in $\Sp(\Wv)$ with compact member $\Zg$.
Let $(\wt\Zg, \wt{\Sp}(\Wv))$ be the corresponding dual pair in the metaplectic group 
$\wt{\Sp}(\Wv)$.
Then  $\wt\Zg$  coincides with the center of $\wt{\Sp}(\Wv)$ and is equal to
\begin{equation}\label{center of metaplectic group}
\wt\Zg=\{(1;1), (1; -1), (-1; \zeta), (-1; -\zeta)\}\,,
\end{equation}
where $\zeta=\left(\frac{i}{2}\right)^{\frac{1}{2}\dim \Wv}$.

In this section we illustrate how to compute the intertwining distributions for the pair $(\Zg, \Sp(\Wv))$. At the same time, we introduce some facts that will be needed in the rest of the paper. 

The formula for the cocycle in \eqref{multiplicationSp}
is particularly simple over $\Zg$:
\[
C(1,\pm1)=C(-1,1)=1\ \ \text{and}\ \ C(-1,-1)=2^{\dim\Wv}\,.
\]
Also, $C(g,1)=C(1,g)=1$ for all $g\in \Sp(\Wv)$ by \cite[Proposition 4.13]{AubertPrzebinda_omega}.
Notice that 
\begin{equation}
\label{product in Ztilde}
(-1; \pm\zeta)^2=(1;\zeta^2 C(-1,-1))=(1; (-1)^{\frac{1}{2}\dim \Wv})
\,.
\end{equation}
Hence the covering \eqref{coveringhomomorphism} restricted to $\wt\Zg$,
\begin{equation}\label{center of metaplectic group covering}
\wt\Zg\ni \t z\to z\in \Zg
\end{equation}
splits if and only if $\frac{1}{2}\dim \Wv$ is even.

By \eqref{Tt} and \eqref{metaplectic group}, we have
\begin{equation*}
\begin{aligned}
&T(1;1)=\delta\,,\quad &&T(1;-1)=-\delta\,,\\ 
&T(-1;\zeta)=\zeta\,\mu_\Wv\,,\quad &&T(-1;-\zeta)=-\zeta\,\mu_\Wv\,.
\end{aligned}
\end{equation*}
Moreover, \cite[Proposition 4.28]{AubertPrzebinda_omega} shows that for $v\in \L^2(\Xv)$ and $x\in \Xv$,
\begin{equation*}
\begin{aligned}
&\omega(1;1)v(x)=v(x)\,,\quad  &&\omega(1;-1)v(x)=-v(x)\,,\\
&\omega(-1;\zeta)v(x)=\frac{\zeta}{|\zeta|}v(-x)\,,\quad &&\omega(-1;-\zeta)v(x)=-\frac{\zeta}{|\zeta|}v(-x)\,.
\end{aligned}
\end{equation*}
Since $T(\t z)=\Theta(\t z) \chi_{c(z)}\mu_{(z-1)\Wv}$ for $\t z\in \t Z$, it follows that
\begin{equation}\label{ThetaT-center}
\omega(\t z)v(x)=\frac{\Theta(\t z)}{|\Theta(\t z)|} v(zx)\qquad (\t z\in \wt\Zg)\,.
\end{equation}
The fraction
\begin{equation}\label{chi+}
\chi_+(\t z)=\frac{\Theta(\t z)}{|\Theta(\t z)|} \qquad (\t z\in \wt\Zg)
\end{equation}
defines an irreducible character $\chi_+$ of the group $\wt\Zg$. Let $\varepsilon$ be the unique non-trivial irreducible character of the two element group $\Zg$. Then
\begin{equation}\label{chi-}
\chi_-(\t z)=\varepsilon(z)\frac{\Theta(\t z)}{|\Theta(\t z)|} \qquad (\t z\in \wt\Zg)
\end{equation}
is also an irreducible character of $\wt\Zg$. 

Let $\L^2(\Xv)_+\subseteq \L^2(\Xv)$ denote the subspace of the even functions and let $\L^2(\Xv)_-\subseteq \L^2(\Xv)$ denote the subspace of the odd functions. Then, as is well known, \cite[(6.9)]{KashiwaraVergne}, the restriction $\omega_{\pm}$ of $\omega$ to $\L^2(\Xv)_\pm$ is irreducible. As we have seen above, the center $\wt\Zg$ acts on $\L^2(\Xv)_\pm$ via the character $\chi_\pm$.
Thus $\chi_\pm$ is the central character of $\omega_\pm$.

Hence, in the case of the dual pair $(\Zg, \Sp(\Wv))$, Howe's correspondence looks as follows
\begin{equation}\label{HowesCorrespondenceO1Sp}
(\chi_+, \C)\leftrightarrow (\omega_+,\L^2(\Xv)_+)\ \ \ \text{and}\ \ \ (\chi_-, \C)\leftrightarrow (\omega_-,\L^2(\Xv)_-)\,.
\end{equation}
The projections 
\[
\L^2(\Xv)\to \L^2(\Xv)_+ \quad \text{and} \quad \L^2(\Xv)\to \L^2(\Xv)_-
\]
are respectively given by 
\[
\frac{1}{2}\omega(\check\chi_+)=\frac{1}{4}\sum_{\t z\in \wt\Zg} \check\chi_+(\t z) \omega(\t z) 
\quad \text{and} \quad  
\frac{1}{2}\omega(\check\chi_-)= \frac{1}{4}\sum_{\t z\in \wt\Zg} \check\chi_-(\t z) \omega(\t z)\,.
\]
The corresponding intertwining distributions are
\begin{equation}\label{intertwiningdistributionsforO1Sp}
\begin{aligned}
f_{\chi_+\otimes\omega_+}&=
\frac{1}{4}\sum_{\t z\in \wt\Zg} \check\chi_+(\t z) T(\t z)=
\frac{1}{2}
\big(\delta+2^{-\frac{1}{2}\dim\Wv}\mu_\Wv\big)\,,\\ 
f_{\chi_-\otimes\omega_-}&=
\frac{1}{4}\sum_{\t z\in \wt\Zg} \check\chi_-(\t z) T(\t z)=
\frac{1}{2}\big(\delta-2^{-\frac{1}{2}\dim\Wv}\mu_\Wv\big)\,,
\end{aligned}
\end{equation}
where we normalize the total mass of $\Zg$ to be $1$, as we did for a general dual pair $(\G, \G')$ with $\G$ compact. 

The right-hand side of \eqref{intertwiningdistributionsforO1Sp} is a sum of two homogenous distributions of different { homogenity} degrees. So, asymptotically, they can be isolated. This allows us to recover $\mu_\Wv$, { and hence $\tau_{\sp(\Wv)}(\Wv)$}, the wave front of $\omega_\pm$, out of the intertwining distribution.

%%%%%%%%%
\section{\bf Dual pairs as Lie supergroups}
\label{section:dualpairs-supergroups}

To present the main results of this paper, we need the realization of dual pairs with one member compact as Lie supergroups. The content of this section is taken from \cite{PrzebindaLocal} and \cite{McKeePasqualePrzebindaSuper}. We recall the relevant material for making our exposition self-contained.

For a dual pair $(\G, \G')$ as in \eqref{classificationGG'}, there is a division algebra $\Bbb D=\R$, $\C$, $\Ha$ with an involution  over $\R$, a finite dimensional 
right $\Bbb D$-vector space $\Vv$ with a positive definite hermitian form $(\cdot,\cdot)$ and a finite dimensional 
right $\Bbb D$-vector space $\Vv'$ with a non-degenerate skew-hermitian form $(\cdot,\cdot)'$ such that $\G$ coincides with the isometry group of $(\cdot,\cdot)$ and $\G'$ coincides with the isometry group of $(\cdot,\cdot)'$.
We assume that $\G$ centralizes the complex structure $J$ and that $J$ normalizes $\G'$. Then the conjugation by $J$ is a Cartan involution on $\G'$, which we denote by $\theta$. 

Let $\V_{\overline 0}=\Vv$, $d=\dim_\Dc\V_{\overline 0}$, $\V_{\overline 1}=\Vv'$ and $d'=\dim_\Dc\V_{\overline 1}$.  We assume that both $\V_{\overline 0}$ and $\V_{\overline 1}$ are 
right vector spaces over $\Dc$. Set $\V=\V_{\overline 0}\oplus \V_{\overline 1}$ and define an element $\Sy \in \End(\V)$ by
\[
\Sy (v_0+v_1)=v_0-v_1 \qquad (v_0\in \V_{\overline 0}, v_1\in \V_{\overline 1})\,.
\]
Let
\begin{eqnarray*}
&&\End(\V)_{\overline 0}=\{x\in \End(\V);\ \Sy x=x\Sy \}\,,\\
&&\End(\V)_{\overline 1}=\{x\in \End(\V);\ \Sy x=-x\Sy \}\,,\\
&&\GL(\V)_{\overline 0}=\End(\V)_{\overline 0}\cap \GL(\V)\,.
\end{eqnarray*}
Denote by $(\cdot ,\cdot )''$ the direct sum of the two forms $(\cdot ,\cdot )$ and $(\cdot ,\cdot )'$. Let
\begin{eqnarray}\label{super liealgebra}
&&\so=\{x\in \End(\V)_{\overline 0};\ (xu,v)''=-(u,xv)'',\ u,v\in\V\}\,,\\
&&\ss1=\{x\in \End(\V)_{\overline 1};\ (xu,v)''=(u,\Sy xv)'',\ u,v\in\V\}\,,\nn\\
&&\mathfrak s=\so\oplus \ss1\,,\nn\\
&&\Sg=\{s\in \GL(\V)_{\overline 0};\  (su,sv)''=(u,v)'',\ u,v\in\V\}\,,\nn\\
&&\langle x, y\rangle=\tr_{\Dc/\R}(\Sy xy)\,.
\label{symplectic-structure}
\end{eqnarray}
(Here $\tr_{\Dc/\R}(x)$ denotes the trace of $x$ considered as a real endomorphism of $\V$.)
Then $(\Sg, \mathfrak s)$ is a real Lie supergroup, i.e. a real Lie group $\Sg$ together with a real Lie superalgebra $\mathfrak s=\so\oplus \ss1$, whose even component $\so$ is the Lie algebra of $\Sg$. (In terms of \cite[\S 3.8]{DeligneMorgan99}, $(\Sg, \mathfrak s)$ is a Harish-Chandra pair.)
We shall write $\mathfrak s(\V)$ instead of $\mathfrak s$ whenever we want to specify the Lie superalgebra $\mathfrak s$ constructed as above from $\V$ and $(\cdot ,\cdot)''$. 

The group $\Sg$ acts on $\mathfrak s$ by conjugation
and 
$\langle\cdot ,\cdot\rangle$ is a non-degenerate $\Sg$-invariant form on the real vector space $\mathfrak s$, whose restriction to $\so$ is symmetric and restriction to $\ss1$ is skew-symmetric. We shall employ the notation 
\begin{eqnarray}
\label{S adjoint action on s}
&&s.x=\Ad(s)x=sxs^{-1}  \qquad (s\in \Sg\,,\ x\in \mathfrak{s})\,,\\
\label{s0 adjoint action on s}
&&x(w)=\ad(x)(w)=xw-wx \qquad (x\in \so\,,\ w\in \ss1)\,.
\end{eqnarray}
In terms of the notation introduced at the beginning of this section,
\[
\g=\so|_{\V_{\overline 0}}\,,\quad \g'=\so|_{\V_{\overline 1}}\,,\quad \G=\Sg|_{\V_{\overline 0}}\,,\quad \G'=\Sg|_{\V_{\overline 1}}\,.
\]
Define $\Wv=\Hom_\Dc(\V_{\overline{1}},\V_{\overline{0}})$. 
Then, by restriction, we have the identification
\begin{equation}
\label{eq:WasHom}
 \Wv=\ss1\,.
\end{equation}
Under this identification, the adjoint action of $\G$ on $\ss1$ becomes the action on $\Wv$ by the left  (postmultiplication). Similarly, the adjoint action of $\G'$ on $\ss1$ becomes the action 
of $\G'$ on $\Wv$ via the right (premultiplication) by the inverse.
Also, we have the unnormalized moment maps
\begin{equation}\label{unnormalized moment maps}
\tau:\Wv\ni w\to w^2|_{\V_{\overline 0}}\in \g\,,\qquad \tau':\Wv\ni w\to w^2|_{\V_{\overline 1}}\in \g'\,.
\end{equation}

An element $x \in \mathfrak s$ is called semisimple (resp., nilpotent) if $x$ is semisimple (resp., nilpotent) as an endomorphism of $\V$. We say that a semisimple element $x \in \ss1$ is regular if it is nonzero and $\dim(S.x) \geq \dim(S.y)$ for all semisimple $y \in \ss1$. Let $x \in \ss1$ be fixed. 
For $x,y\in \ss1$ let $\{x,y\}=xy+yx \in \so$ be their anticommutator.
\label{semisimple-elements-W}

The anticommutant and the double anticommutant
of $x$ in $\ss1$ are  
\begin{eqnarray*}
\anticomm{x}{\ss1}&=&\{y \in \ss1:\{x,y\}=0\}\,,\\
\danticomm{x}{\ss1}&=&\bigcap_{y \in \anticomm{x}{\ss1}} \anticomm{y}{\ss1}\,,
\end{eqnarray*}
respectively. A Cartan subspace $\hs1$ of $\ss1$ is defined as the double anticommutant
of a regular semisimple element $x \in \ss1$.
We denote by $\reg{\hs1}$ the set of regular elements in $\hs1$.

Next we describe the Cartan subspaces $\hs1\subseteq\ss1$. We refer to \cite[\S 6]{PrzebindaLocal} and \cite[\S 4]{McKeePasqualePrzebindaSuper} for the proofs omitted here. 
Let $l$ be the rank of $\g$, $l'$ the rank of $\g'$, and set 
\begin{equation}
\label{l''}
l''=\min(l,l')\,.
\end{equation}
Given a Cartan subspace $\hs1$, there are $\Ze/2\Ze$-graded subspaces $\V^j\subseteq\V$ such that the restriction of the form $(\cdot ,\cdot )''$ to each $\V^j$ is non-degenerate, $\V^j$ is orthogonal to $\V^k$ for $j\ne k$ and
\begin{equation}\label{decomposition of space for a cartan subspace}
\V=\V^0\oplus \V^1\oplus\V^2\oplus\dots \oplus \V^{l''}\,.
\end{equation}
The subspace $\V^0$ coincides with the intersection of the kernels of the elements of $\hs1$ (equivalently, $\V^0=\Ker(x)$ if $\hs1=\danticomm{x}{\ss1}$).
For $1\leq j\leq l''$, the subspaces $\V^j=\V^j_{\overline 0}\oplus \V^j_{\overline 1}$ are described as follows.

Suppose $\Dc=\R$. Then there is a basis $v_0$, $v_0'$ of $\V_{\overline 0}^j$ and a  
basis $v_1$, $v_1'$ of $\V_{\overline 1}^j$ such that
\begin{eqnarray}\label{individual basis}
&&(v_0,v_0)''=(v_0',v_0')''=1\,,\qquad (v_0,v_0')''=0\,,\\
&&(v_1,v_1)''=(v_1',v_1')''=0\,,\qquad (v_1,v_1')''=1\,.\nn
\end{eqnarray}
The following formulas define an element $u_j\in \ss1(\V^j)$,
\begin{eqnarray*}
&&u_j(v_0)=\frac{1}{\sqrt{2}}(v_1-v_1')\,,\qquad u_j(v_1)=\frac{1}{\sqrt{2}}(v_0-v_0')\,,\\
&&u_j(v_0')=\frac{1}{\sqrt{2}}(v_1+v_1')\,,\qquad u_j(v_1')=\frac{1}{\sqrt{2}}(v_0+v_0')\,.
\end{eqnarray*}
Suppose $\Dc=\C$. Then there are vectors $v_0$ and $v_1$ such that $\V_{\overline 0}^j=\C v_0$, $\V_{\overline 1}^j=\C v_1$, $(v_0,v_0)''=1$ and $(v_1,v_1)''=\delta_j i$, with  $\delta_j=\pm 1$ fixed by the form $(\cdot,\cdot)'$. 
The following formulas define an element $u_j\in \ss1(\V^j)$,
\begin{equation}\label{deltaj}
u_j(v_0)=e^{-i\delta_j \frac{\pi}{4}} v_1\,,\qquad \ u_j(v_1)=e^{-i\delta_j \frac{\pi}{4}} v_0\,.
\end{equation}
Suppose $\Dc=\Ha$. Then $\V_{\overline 0}^j=\Ha v_0$, $\V_{\overline 1}^j=\Ha v_1$, where $(v_0,v_0)''=1$ and $(v_1,v_1)''= i$. 
The following formulas define an element $u_j\in \ss1(\V^j)$,
\[
u_j(v_0)=e^{-i \frac{\pi}{4}} v_1\,, \qquad  u_j(v_1)=e^{-i\frac{\pi}{4}} v_0\,.
\]
In any case, by extending each $u_j$ by zero outside $\V^j$, we have
\begin{equation}\label{a cartan subspace}
\hs1=\sum_{j=1}^{l''} \R u_j\,.
\end{equation}

The formula (\ref{a cartan subspace}) describes a maximal family of mutually non-conjugate Cartan subspaces of $\ss1$. By classification, see \cite[\S 6]{PrzebindaLocal}, 
there is only one such subspace unless the dual pair $(\G, \G')$ is isomorphic to $(\Ug_l, \Ug_{p,q})$ with $l''=l< p+q$. In the last case there are $\min(l,p) - \max(l-q,0) +1$ such subspaces, assuming $p\leq q$. For each $m$ such that $\max(l-q,0)\leq m\leq \min(p,l)$ there is a Cartan subspace $\hs1{}_{,m}$ determined by the condition that $m$ is the number of positive $\delta_j$'s in (\ref{deltaj}). We may assume that $\delta_1=\dots=\delta_m=1$ and $\delta_{m+1}=\dots=\delta_l=-1$.
If $(\G, \G')$ is isomorphic to $(\Ug_l, \Ug_{p,q})$ with $l\geq l''=p+q$, then there is a unique Cartan subspace of $\ss1$ up to conjugation. It is determined by the condition that in \eqref{deltaj}
there are $p$ positive and $q$ negative $\delta_j$'s. We may assume that the first $p$
$\delta_j$'s are positive.

The Weyl group $W(\Sg,\hs1)$ is the quotient of the stabilizer of $\hs1$ in $\Sg$  by the subgroup $\Sg^{\hs1}$ fixing each element of $\hs1$. If $\Dc\neq \C$, then 
$W(\Sg,\hs1)$
acts by all sign changes and all permutations of the $u_j$'s. 
If $\Dc=\C$, the Weyl group acts by all sign changes and all permutations of the $u_j$'s which preserve $(\delta_1,\dots,\delta_{l''})$, see \cite[(6.3)]{PrzebindaLocal}.

Set $\delta_j=1$ for all $1\leq j\leq l''$, if $\Dc\ne \C$, and in any case, i.e. $\Dc\ne \C$ or $\Dc= \C$, define 
\begin{equation}\label{complex structures for l leq l' 2}
J_j=\delta_j\tau(u_j)\,,\qquad  J_j'=\delta_j\tau'(u_j) \qquad (1\leq j\leq l'')\,.
\end{equation}
Then $J_j$,  $J_j'$  are complex structures on $\V_{\overline 0}^j$ and $\V_{\overline 1}^j$ respectively. Explicitly,
\begin{eqnarray}\label{explicit J_j}
&&J_j(v_0)=-v_0',\quad J_j(v_0')=v_0\,,\quad  J_j'(v_1)=-v_1'\,,\quad J_j'(v_1')=v_1\,,\quad  \text{if}\ \ \Dc=\R\,,\\
&&J_j(v_0)=- i v_0\,,\quad J_j'(v_1)=- i v_1\,,\quad  \text{if}\ \ \Dc=\C\ \text{or}\ \Dc=\Ha\,.\nn
\end{eqnarray}
(The point of the multiplication by the $\delta_j$ in (\ref{complex structures for l leq l' 2}) is that the complex structures $J_j$, $J_j'$ do not depend on the Cartan subspace $\hs1$.) 
In particular, if $w=\sum_{j=1}^{l''} w_j u_j\in \hs1$, then
\begin{eqnarray}\label{explicit tau and tau' on cartan subspace}
\tau(w)=\sum_{j=1}^{l''} w_j^2\delta_j J_j\ \ \text{and}\ \ \tau'(w)=\sum_{j=1}^{l''} w_j^2\delta_j J_j'\,.
\end{eqnarray}
(Notice that $w_j^2\geq 0$.) Let  $\hs1^2\subseteq \so$ be the subspace spanned by all the squares $w^2$, $w\in\hs1$. (This is a linear space, not a collection of squares. We hope that the notation $\hs1^2$ will not cause any confusion.) Then
\begin{equation}\label{hs1squared}
\hs1^2=\sum_{j=1}^{l''} \R(J_j+J_j')\,.
\end{equation}
We shall use the following identification
\begin{equation}\label{the identification}
\hs1^2|_{\V_{\overline 0}}\ni \sum_{j=1}^{l''} y_jJ_j = \sum_{j=1}^{l''}y_jJ_j'\in \hs1^2|_{\V_{\overline 1}}\,.
\end{equation}

Recall from \eqref{l''} that $l''=\min(l,l')$. 
If $l''=l$, then $\hs1^2|_{\V_{\overline 0}}$ is a Cartan subalgebra of $\g$
which we denote by $\h$. The identification \eqref{the identification} embeds $\h$ 
diagonally in $\g$ and in $\g'$. It is contained in an elliptic Cartan subalgebra of $\g'$, say $\h'$. 
(``Elliptic'' means that all the roots of $\h$ in $\g'_\C$ are purely imaginary.) 
Similarly, if $l''=l'$, then $\hs1^2|_{\V_{\overline 1}}$ is an elliptic Cartan subalgebra of $\g'$ which we denote by $\h'$.
If $l\leq l'$ we denote by $\z'\subseteq \g'$ the centralizer of $\h$. Similarly, if $l'\leq l$ we denote by  $\z\subseteq \g$ the centralizer of $\h'$. In particular, if $l'=l$, then $\z'=\h'=\h=\z$, where the first equality is in $\g$, the second is $\eqref{the identification}$ and the 
last is in $\g'$.

Let ${\so}_\C=\g_\C\oplus \g'_\C$ be the complexification of $\so$. Fix a system of positive roots for the adjoint action of $\hs1^2$ on ${\so}_\C$. Suppose first that $l\leq l'$. By the identification \eqref{the identification}, $\h$ preserves both $\g_\C$ and $\g'_\C$. So our choice of positive roots for $(\hs1^2{}_\C,{\so}_\C)$ fixes a positive root system of $(\h_\C,\g_\C)$ and extends to a compatible positive root system for $(\h'_\C,\g'_\C)$. Let $\pi_{\g/\h}$ be the product of positive roots of $(\h_\C,\g_\C)$ and let $\pi_{\g'/\z'}$ be the product of positive roots of $(\h'_\C,\g'_\C)$ such that the corresponding root spaces do not occur in $\z'_\C$. If 
$l>l'$, then $\pi_{\g'/\h'}$ and $\pi_{\g/\z}$ can be similarly defined. 
See Appendix \ref{section:positive root products} for the explicit expressions of these root products restricted to the elements in \eqref{the identification}.

Suppose $l'<l$. Then $\V_{\overline 1}^0=0$, $\V_{\overline 0}^0\ne 0$ and 
\begin{equation}\label{extendedcsa0}
\V_{\overline 0}=\V_{\overline 0}^0\oplus \V_{\overline 0}^1\oplus\V_{\overline 0}^2\oplus\cdots \oplus \V_{\overline 0}^{l''}
\end{equation}
is a direct sum orthogonal decomposition with respect to the positive definite hermitian form $(\cdot,\cdot)$. We extend $\h\subseteq \g$ to a Cartan subalgebra $\h(\g)\subseteq \g$ as follows. The restriction of $\h(\g)$ to $\V_{\overline 0}^1\oplus\V_{\overline 0}^2\oplus\dots \oplus \V_{\overline 0}^{l''}$ coincides with $\h$. Pick an orthogonal direct sum decomposition 
\begin{equation}\label{extendedcsa1}
\V_{\overline 0}^0=\V_{\overline 0}^0{}^{,0}\oplus \V_{\overline 0}^0{}^{,l''+1}\oplus \V_{\overline 0}^0{}^{,l''+2}\oplus \cdots\oplus \V_{\overline 0}^0{}^{,l}\,,
\end{equation}
where for $j>l''$, 
$\dim_{\mathbb{D}} \V_{\overline 0}^0{}^{,j}=2$ if $\Bbb D=\R$ and 
$\dim_{\mathbb{D}} \V_{\overline 0}^0{}^{,j}
=1$ if $\Bbb D\ne\R$. 
Also $\V_{\overline 0}^0{}^{,0}=0$ unless $\G=\Og_{2l+1}$, in which case 
$\dim_{\mathbb{D}} \V_{\overline 0}^0{}^{,0}=1$. 
In each space $\V_{\overline 0}^0{}^{,j}$, with  $j>l''$, we pick an orthonormal basis and define $J_j$ as in \eqref{explicit J_j}. Then
\begin{equation}\label{h(g)}
\h(\g)=\sum_{j=1}^l \R J_j\,.
\end{equation}
If $l\leq l'$, then we set $\h(\g)=\h$.

Let $J_j^*$, $1\leq j\leq l$, be the basis of the space $\h(\g)^*$ which is dual to 
$J_1,\dots, J_l$, and set
\begin{eqnarray}\label{eq:ej}
e_j=-iJ_j^*\,, \qquad 1\leq j\leq l\,.
\end{eqnarray}
If $\mu \in i\h(\g)^*$, then $\mu=\sum_{j=1}^l \mu_j e_j$ with $\mu_j \in \R$. 
We say that $\mu$ is strictly dominant if $\mu_1>\mu_2> \dots >\mu_l$.
%%%%

\section{\bf Orbital integrals on $\Wv$}
\label{section:orbital integrals}

In this section we recall from \cite{McKeePasqualePrzebindaSuper} and \cite{McKeePasqualePrzebindaWCestimates} some definitions and results concerning the orbital integrals on $\Wv$ that we will need in the following sections.
%Moreover, using the surjectivity properties of Harish-Chandra's orbital integrals, we prove that the image of the so-called Harish-Chandra regular almost-elliptic orbital integral on $\Wv$ is ``large enough'' in the sense of Corollary \ref{cor:sufficiently many Fphi} { below}  when $l\geq l'$. 

Let $\Ss'(\Wv)^\Sg$ denote the space of $\Sg$-invariant tempered distributions on $\Wv$, where the $\Sg$-action is induced by \eqref{S adjoint action on s}. Let $\hs1$ be a Cartan subspace of $\Wv$.
Suppose first that $\G$ is different from $\Og_{2l+1}$ with $l<l'$.
For $w\in \reg{\hs1}$, the orbital integral attached to the orbit $\mathcal{O}(w)=\Sg.w$ is the element $\mu_{\mathcal{O}(w),\hs1}$ of $\Ss'(\Wv)^\Sg$ defined for $\phi\in \Ss(\Wv)$ by 
\begin{equation}
\label{muwforall othergroups}
\mu_{\mathcal{O}(w),\hs1}(\phi)=
\int_{\Sg/\Sg^{\hs1}} \phi(s.w)\, d(s\Sg^{\hs1})\,.
\end{equation}
Suppose now that $\G=\Og_{2l+1}$ with $l<l'$. Then one needs to modify \eqref{muwforall othergroups} because the union of the orbits $\Sg.w$ over all $w\in\reg{\hs1}$ would not be dense in $\Wv$; see \cite[Theorem 20]{McKeePasqualePrzebindaSuper}. Let $w_0\in \ss1(\V^0)$ be a non-zero element, $w\in \reg{\hs1}$ and $\Sg^{\hs1+w_0}$ the centralizer of $w+w_0$ in $\Sg$. Set $\Oo(w)=\Sg.(w+w_0)$ and define
\begin{equation}\label{muwforoodd}
\mu_{\Oo(w),\hs1}(\phi)
=\int_{\Sg/\Sg^{\hs1+w_0}}\phi(s.(w+w_0))\,d(s\Sg^{\hs1+w_0})\,.
\end{equation}
(Since $\ss1(\V^0)\setminus \{0\}$ is a single $\Sg(\V^0)$-orbit, the $\Sg$-orbit of $w+w_0$, and hence the right-hand side of \eqref{muwforoodd}, does not depend on the choice of $w_0\in \ss1(\V^0)$.) 
The orbital integrals \eqref{muwforall othergroups} and \eqref{muwforoodd} are well-defined, tempered distribution on $\Wv$, which depend only on $\tau(w)$, 
or equivalently $\tau'(w)$, via the identification (\ref{the identification}).

For $w\in \hs1$, set
\begin{equation}
\label{products of roots in so}
\pi_{\so/\hs1^2}(w^2)=
\begin{cases}
\pi_{\g/\h}(\tau(w))\pi_{\g'/\z'}(\tau'(w)) &\text{if $l\leq l'$},\\[.1em]
\pi_{\g/\z}(\tau(w))\pi_{\g'/\h'}(\tau'(w)) &\text{if $l\geq l'$}\,.
\end{cases}
\end{equation}
As shown in \cite[Lemma 1.2]{McKeePasqualePrzebindaWCestimates}, there is a constant 
$C(\hs1)$, depending on $\hs1$ and with $|C(\hs1)|=1$, such that
\begin{equation}
\label{Ch1}
\pi_{\so/\hs1^2}(w^2)=C(\hs1) |\pi_{\so/\hs1^2}(w^2)|\,.
\end{equation}
The set $\reg{\hs1}$ of regular elements of $\hs1$ is explicitly given by 
\begin{equation} 
\label{h1reg}
\reg{\hs1}=\{w \in \hs1; \;\pi_{\so/\hs1^2}(w^2)\neq 0\}\,.
\end{equation}
Choose a positive Weyl chamber $\hs1^+\subseteq \reg{\hs1}$, i.e. an open fundamental domain for the action of the Weyl group, $W(\Sg,\hs1)$.
There is a normalization $d\tau(w)$ of the Lebsegue measure on $\h$, respectively a normalization $d\tau'(w)$ of the Lebsegue measure on $\h'$, such that the following equalities hold for all $\phi\in\mathcal{S}(\Wv)$:
\begin{alignat}{2}
\label{weyl int on w 1}
\mu_\Wv(\phi)&=\sum_{\hs1}\int_{\tau(\hs1^+)}|\pi_{\so/\hs1^2}(w^2)|\mu_{\Oo(w),\hs1}(\phi)\,d\tau(w) \qquad &&\text{if $l\leq l'$}\,,\\
\label{weyl int on w 2}
\mu_\Wv(\phi)&=\int_{\tau'(\hs1^+)}|\pi_{\so/\hs1^2}(w^2)|\mu_{\Oo(w),\hs1}(\phi)\,d\tau'(w)
\qquad &&\text{if $l\geq l'$}\,.
\end{alignat}
Formulas \eqref{weyl int on w 1} and \eqref{weyl int on w 2} are the Weyl--Harish-Chandra integration formulas on $\Wv$, \cite[Theorem 21]{McKeePasqualePrzebindaSuper}. 
The sum in \eqref{weyl int on w 1} is over the family of mutually non-conjugate Cartan subspaces $\hs1\subseteq \Wv$. (It therefore reduces to a single term 
for $(\G,\G')$ different from $(\Ug_l, \Ug_{p,q})$ with $l< l'=p+q$.) The formulas agree for $l=l'$ once we identify $\tau(w)$ and $\tau'(w)$ via \eqref{the identification}.

Let $C_{\hs1}=C(\hs1)\cdot i^{\dim \g/\h}$, where $C(\hs1)$ is as in \eqref{Ch1}. 
If $(\G,\G')=(\Ug_l, \Ug_{p,q})$ with $l< l'=p+q$, let
\[
\bigcup_{\hs1}\tau(\reg{\hs1})=\bigcup_{m=\max(l-q,0)}^{\min(p,l)}\tau(\hs1{}_{,m})\,.
\]
In all other cases, $\bigcup_{\hs1}\tau(\reg{\hs1})$ will denote $\tau(\reg{\hs1})$, where $\hs1$ is the fixed Cartan subspace. The Harish-Chandra regular almost-elliptic orbital integral on $\Wv$ is the function 
\[
F:\bigcup_{\hs1}\tau(\reg{\hs1})\to \Ss'(\Wv)^\Sg
\]
defined 
for every $y\in \bigcup_{\hs1}\tau(\reg{\hs1}),\ y=\tau(w)=\tau'(w)$ as follows:
\begin{equation}
\label{eq:originalf-def}
F(y)=\begin{cases}
\sum_{\hs1} C_{\hs1}\pi_{\g'/\z'}(y)\mu_{\Oo(w),\hs1} &\text{if $l\leq l'$}\,,\\[.3em]
C_{\hs1}\pi_{\g'/\h'}(y)\mu_{\Oo(w),\hs1} &\text{if $l>l'$}\,.
\end{cases}
\end{equation}
Following Harish-Chandra's notation, we shall write $F_\phi(y)$ for $F(y)(\phi)$.

Suppose first that $l\leq l'$.
According to \cite[Theorem 3.6]{McKeePasqualePrzebindaWCestimates}, $F$ uniquely extends 
to a function $F:\h\to \Ss'(\Wv)^\Sg$ satisfying 
\begin{equation}\label{extension by the symmetry condition1}
F(sy)=\sgn_{\g/\h}(s) F(y) \qquad (s\in W(\G,\h),\ y\in \h)\,.
\end{equation}
This extension is supported in $\h\cap \tau(\Wv)$. 
The extended map $F$ is smooth on the subset of $y=\sum_{j=1}^l y_j J_j$ where each $y_j\ne 0$ and, for any multi-index $\alpha=(\alpha_1,\dots,\alpha_l)$ with
\begin{equation}
\label{jump-conditions}
\max(\alpha_1,\dots,\alpha_l)\leq \left\{
\begin{array}{lll}
d'-r-1\ &\text{if}\ \Dc=\R\ \text{or}\ \C\,,\\
2(d'-r)\ &\text{if}\ \Dc=\Ha\,,\\
\end{array}\right.
\end{equation}
%%%
the function $\partial(J_1^{\alpha_1}J_2^{\alpha_2}\dots J_l^{\alpha_l})F(y)$ extends to a continuous function on $\h\cap\tau(\Wv)$ vanishing on the boundary of $\h\cap\tau(\Wv)$.

For any values of $l$ and $l'$, there is 
the pullback via the unnormalized
moment map $\tau':\Wv\to \g'$, namely 
$$
{\tau'}^*: \Ss(\g')
\ni \psi
 \to \psi\circ \tau'\in \Ss(\Wv)^\G\,.
$$ 
According to \cite[(25)]{McKeePasqualePrzebindaWCestimates} (a special case of a theorem 
of Astengo, Di Blasio and Ricci \cite[Theorem 6.1]{Ricci_2009}), there is a continuous map $\tau'_*:\Ss(\Wv)^\G \to \Ss(\g')$ such that 
\begin{equation}
\label{tau'star-surjective}
{\tau'}^*\circ \tau'_*(\phi)=\phi \qquad (\phi\in \Ss(\Wv)^\G)\,.
\end{equation}
In particular, the map ${\tau'}^*$ is surjective. 
We will denote by $\phi^\G$ the projection of $\phi\in \Ss(\Wv)$ onto the space of the $\G$-invariants in $\Ss(\Wv)$,
	\begin{equation}
		\label{phiG}
		\phi^\G(w)=\int_\G\phi(g. w)\,dg \qquad (w\in\Wv)\,.
	\end{equation}
	(Recall that we have normalized the Haar measure on $\G$ so that its mass is $1$.)

Suppose now that $l>l'$. Then by 
\cite[(39)]{McKeePasqualePrzebindaWCestimates}, 
\begin{equation}
\label{HCintegral}
F_\phi(y)=C'_{\hs1} 
\pi_{\g'/\h'}(y)
\int_{\G'/\H'} \psi(g'.y)\; d(g'\H') \qquad (\phi\in \mathcal{S}(\Wv), \, y \in \reg{\hs1})\,,
\end{equation}
where $\H'\subseteq \G'$ is the Cartan subgroup corresponding to $\h'$,
\begin{equation}
\label{HCintegral-bis}
\psi=\tau'_*(\phi^\G) \in \mathcal{S}(\g')\,,
\end{equation}
and $C'_{\hs1}$ is a suitable non-zero constant. The right-hand side of \eqref{HCintegral} is 
Harish-Chandra's orbital integral of $\psi$. It provides a $W(\G',\h')$-skew-invariant extension of 
$F_\phi$ to $\h'{}^{In-reg}$, where $\h'{}^{In-reg}\subseteq \h'$ is the subset where no 
non-compact roots vanish. Furthermore, as a function of $\phi$, $F_\phi(y)$ is $\Sg$-invariant; see  \cite[Theorem 3.3]{McKeePasqualePrzebindaWCestimates}.

Notice that, by \cite[(69)--(72)]{McKeePasqualePrzebindaWCestimates}, formulas \eqref{HCintegral}
and \eqref{HCintegral-bis}  also hold when $l=l'$ because $\Zg'=\H'$ in this case. 

\begin{lem}
\label{lem:supportF}
Suppose that $l\leq l'$ and $\G\neq \Og_{2l+1}$. Let $\Ug\subseteq \reg{\h}$ be a nonempty 
$W(\G,\h)$-invariant open subset. Then there is a nonzero function $\phi\in C^\infty_c(\Wv)^\G$ such that $\phi\geq 0$ and $\supp F_\phi \subseteq \Ug$. (Here $\supp$ denotes the support.)
\end{lem}
\begin{proof}
Let $\Vg$ be a nonempty open set in $\reg{\h}$ with closure $\overline{\Vg}\subseteq \Ug$. By \cite[p. 19, especially (9)]{VaradarajanHarmonic}, the set $\G.\Vg$ is open in $\g$. Hence $\tau^{-1}(\G.\Vg)$ is open and $S$-invairant in $\Wv$. Let $\phi\in C_c^\infty(\Wv)^\G$ be a nonzero function such that $\phi\geq 0$ and $\supp \phi \subseteq \tau^{-1}(\G.\Vg)$. We want to prove that $\supp F_\phi \subseteq \Ug$.

Suppose first that $\G\neq \Ug_l$. Hence $F_\phi(y)=C_{\hs1}\pi_{\g'/\z'}(y)\mu_{\Oo(w),\hs1}(\phi)$
for all $y\in \tau(\reg{\hs1})$.
(Here $\pi_{\g'/\z'}(y)=\pi_{\g'/\z'}(\tau'(w))$ where $y=\tau(w)=\tau'(w)$.)
Since the zero set of $\pi_{\g'/\z'}$ is a finite union of root hyperplanes, $\supp F_\phi$ is the closure in $\h$ of the set of the $y=\tau(w)$ with $w\in \reg{\hs1}$ such that $\mu_{\Oo(w),\hs1}(\phi)\neq 0$.
If $\mu_{\Oo(w),\hs1}(\phi)\neq 0$, then $\Oo(w)\cap \supp\phi\neq \emptyset$, where 
$\Oo(w)=\Sg.w$. Hence $(\Sg.w)\cap \tau^{-1}(\G.\Vg)\neq \emptyset$. This means that there are
$g,g_1\in \G$, $g'\in \G'$ and $v\in \Vg$ such that $gg'.w=\tau^{-1}(g_1.v)$. Therefore
$$
g.y=g.\tau(w)=\tau(gg'.w)=g_1.v \quad\text{and hence} \quad g_1^{-1}g. y=v\in \G.y\cap \h\,.
$$
By \cite[Corollary 23]{VaradarajanHarmonic}, $y\in W(\G,\h)v$. Thus $y\in \Vg$ because $\Vg$ is 
$W(\G,\h)$-invariant. This proves that $\supp F_\phi \subseteq \overline{\Vg}\subseteq \Ug$.

The same argument extends to the case of $\G=\Ug_l$ because all Cartan subspaces $\h_{\overline{1},m}$ satisfy $\tau(\h_{\overline{1},m})\subseteq \h$.
\end{proof}

\begin{rem}
\label{rem:Weyl-invariant l>l'}
The Cartan subalgebra $\h'$ is $\theta$-stable, where $\theta$ is the fixed Cartan involution 
of $\g'$. Let $\H' \subseteq \G'$ be the Cartan subgroup which is the centralizer of $\h'$ in $\G'$, and let $\K'$ be the maximal compact subgroup of $\G'$ which is fixed by $\theta$. Then, by \cite[Lemma 10]{HC-56}, the Weyl group $W(\G',\h')$ coincides with $W(\K',\h')$, i.e. the normalizer of $\H'$ in $\K'$ modulo the centralizer of $\H'$ in $\K'$. Explicitly, $\K'$ is $\Ug_{l'}$ if $\Dc=\R$ or $\Ha$, and $\Ug_p \times \Ug_q$ if $\Dc=\C$. Hence $W(\G',\h')$ acts on $\h'$ by permuting the $J'_j$, \eqref{complex structures for l leq l' 2}, if $\Dc=\R$ or $\Ha$, and by separately permuting the first $p$ and the last $q$ elements $J'_j$ if $\Dc=\C$. Since $\delta_j=1$ for all $j=1,\dots,l'$ if $\Dc=\R$ or $\Ha$, and  $\delta_j=1$ for $j=1,\dots,p$ and $\delta_j=-1$ for $j=p+1,\dots,p+q$ if $\Dc=\C$, it follows from \eqref{explicit tau and tau' on cartan subspace} that the domain of integration $\tau'(\reg{\hs1})$ appearing in \eqref{weyl int on w 2} is $W(\G',\h')$-invariant. This property will be relevant in Proposition
\ref{cor:sufficiently many Fphi} below.
\end{rem}

%%%
Recall from page \pageref{semisimple-elements-W} the notions of semisimple and regular elements in $\Wv=\ss1$.
By \cite[Theorem 20]{McKeePasqualePrzebindaSuper}, the set of semisimple elements is dense in 
$\Wv$ for every dual pair with one member compact unless $(\G,\G')=(\Og_{2l+1},\Sp_{2l'}(\R))$ with $2l+1<2l'$. %(In the latter case, one should remplace the semisimple element the almost semisimple elements.) 
As noticed in section \ref{section:dualpairs-supergroups}, $\Wv$ has a unique class of Cartan subalgebras unless $(\G,\G')=(\Ug_l,\Ug_{p,q})$ with $l<l'=p+q$. 
Suppose these two families of dual pairs are excluded. Let $\hs1$ denote the Cartan subalgebra 
in $\Wv$ fixed in \eqref{a cartan subspace}. Then $\reg{\Wv}=\Sg.\reg{\hs1}$ is the set of regular semisimple elements of $\Wv$. It is open and dense in $\Wv$. 

\begin{pro}
\label{cor:sufficiently many Fphi}
Suppose that $l\geq l'$. Let $\Phi$ be a $W(\G',\h')$-invariant function
on $\tau'(\reg{\hs1})$. Then there is a unique $\Sg$-invariant function $\Phi^\sharp$
on $\reg{\Wv}$ such that 
$$
\Phi^\sharp(y)=(\Phi\circ \tau')(y) \qquad (y\in \reg{\hs1})\,.
$$
Moreover, 
\begin{equation}
\label{variation-on-WHC-integration-l>=l'}
\frac{1}{|W(\G',\h')|} \int_{\tau'(\reg{\hs1})} \Phi(y) \pi_{\g/\z}(y) F_\phi(y)\, dy=\int_{\Wv} \Phi^\sharp(w)\phi(w)\, dw
\qquad (\phi \in C_c^\infty(\Wv))
\end{equation}
provided the integrals are absolutely convergent. 
%%
%
%Suppose now that $l\leq l'$and $\Dc=\Ha$. 
%Let $W_0(\G,\h)$ be the subgroup of $W(\G,\h)$ acting on $\h=\sum_{j=1}^l \R J_j$ by permutation of the coordinates and let $\Phi$ be a $W_0(\G,\h)$-invariant function
%on $\tau(\reg{\hs1})$. Then there is a unique $\Sg$-invariant function $\Phi^\sharp$
%on $\reg{\Wv}$ such that 
%$$
%\Phi^\sharp|_{\reg{\hs1}}=(\Phi\circ \tau)|_{\reg{\hs1}}.
%$$
%Moreover,
%%%
%\begin{equation}
%\label{variation-on-WHC-integration-l<=l'}
%\frac{1}{|W_0(\G,\h)|} \int_{\tau(\reg{\hs1})} \Phi(y) \pi_{\g/\h}(y) F_\phi(y)\, dy=\int_{\reg{\Wv}} \Phi^\sharp(w)\phi(w)\, dw
%\qquad (\phi \in C_c^\infty(\reg{\Wv}))\,.
%\end{equation}
%%%
\end{pro}
\begin{proof}
%The existence of $\Phi^\sharp$ is due to the fact that $\Phi\circ \tau'$ and 
%$\Phi\circ \tau$ are $W(\Sg,\hs1)$-invariant functions on $\reg{\hs1}$. 
The existence of $\Phi^\sharp$ is due to the fact that $\Phi\circ \tau'$ is a $W(\Sg,\hs1)$-invariant function on $\reg{\hs1}$. 
The Weyl group $W(\G',\h')$ acts on $\tau'(\reg{\hs1})$ by permuting the coordinates with respect to the basis $\{J_1',\dots, J'_{l'}\}$.
The action is simple and transitive and $\tau'(\hs1^+)$ is a fundamental domain. Since the function $\Phi(y) \pi_{\g/\z}(y) F_\phi(y)$ is $W(\G',\h')$-invariant on 
$\tau'(\reg{\hs1})$,  the formula \eqref{variation-on-WHC-integration-l>=l'} is a restatement of the Weyl--Harish-Chandra integration formulas on $\Wv$ for $l\geq l'$, see \eqref{weyl int on w 1}.

%Suppose now that $l\leq l'$. The set $\tau(\reg{\hs1})$ is $W_0(\G,\h)$-invariant and the function 
%$\Phi(y) \pi_{\g/\h}(y) F_\phi(y)$ is $W_0(\G,\h)$-invariant on $\tau(\reg{\hs1})$.
%Moreover,  $\tau(\hs1^+)$ is a fundamental domain for the action of $W_0(\G,\h)$ on 
%$\tau(\reg{\hs1})$. Hence \eqref{variation-on-WHC-integration-l<=l'} is a restatement of the
%Weyl--Harish-Chandra integration formulas on $\Wv$ for $l\leq l'$, see \eqref{weyl int on w 2}.
\end{proof}

%%%%%%

%%
\section{\bf Main results}
\label{main results}
Suppose an irreducible representation $\Pi$ of $\wt\G$ occurs in Howe's correspondence. This means that there is a subspace $\Hc_\Pi\subseteq \L^2(\Xv)$ on which the restriction of $\omega$ coincides with $\Pi$. Since $\wt\Zg\subseteq \wt\G\cap\wt\G'$, then either $\Hc_\Pi\subseteq \L^2(\Xv)_+$ or $\Hc_\Pi\subseteq \L^2(\Xv)_-$. In the first case the restriction of the central character $\chi_\Pi$ of $\Pi$ to $\wt\Zg$ is equal to $\chi_+$ and in the second case to $\chi_-$. Thus for $\t z\in\wt\Zg$ and $\t g\in\wt\G$,
\begin{eqnarray}\label{centralcharacterofpi}
&&\Theta_\Pi(\t z\t g)=\chi_+(\t z)\Theta_\Pi(\t g)\ \ \ \text{if}\ \ \ \Hc_\Pi\subseteq \L^2(\Xv)_+\,,\\
&&\Theta_\Pi(\t z\t g)=\chi_-(\t z)\Theta_\Pi(\t g)\ \ \ \text{if}\ \ \ \Hc_\Pi\subseteq \L^2(\Xv)_-\,.\nn
\end{eqnarray}
We see from equations \eqref{THETAT}, \eqref{chi+},  \eqref{chi-} and \eqref{centralcharacterofpi} that the function
\[
\wt\G\ni \t g\to T(\t g)\check\Theta_\Pi(\t g)\in\Ss'(\Wv)
\]
is constant on the fibers of the covering map \eqref{coveringhomomorphism}.
The following lemma is a restatement of \eqref{to-compute}. Our main results will be the explicit expressions of the various integrals appearing on the right-hand sides of the equations below.
\begin{lem}\label{main theorem summary}
Let $\G^0\subseteq\G$ denote the connected identity component.
Suppose $(\G,\G')=(\Ug_{d}, \Ug_{p,q})$ or $(\Sp_d, \Og^*_{2m})$. Then
$\G=\G^0=-\G^0$ and
\begin{equation} \label{main theorem summary1}
f_{\Pi\otimes\Pi'}=\int_{\G}\check\Theta_\Pi(\t g) T(\t g)\,dg=\int_{-\G^0}\check\Theta_\Pi(\t g) T(\t g)\,dg\,.
\end{equation}
Formula \eqref{main theorem summary1} holds also  if $(\G,\G')=(\Og_d, \Sp_{2m}(\R))$ with $d$ even and $\Theta_\Pi$ supported in  $\wt{\G^0}$, {because $\G^0=\SO_d=-\SO_d=-\G^0$}. In the remaining cases
\begin{equation} \label{main theorem summary2}
f_{\Pi\otimes\Pi'}=\int_{\G}\check\Theta_\Pi(\t g) T(\t g)\,dg=\int_{-\G^0}\check\Theta_\Pi(\t g) T(\t g)\,dg+\int_{\G\setminus(-\G^0)}\check\Theta_\Pi(\t g) T(\t g)\,dg\,.
\end{equation}
\end{lem}
The integrals over $-\G^0$ in \eqref{main theorem summary1} and \eqref{main theorem summary2} are given in Theorems \ref{main thm for l<l'} and \ref{main thm for l>=l'} below, proved in section 
\ref{An intertwining distribution in terms of orbital integrals on the symplectic space}.
The integrals over the other connected component in \eqref{main theorem summary2}  are computed in Theorems \ref{main thm for l<l', special}, \ref{main thm for l<l', special odd 1} and \ref{main thm for l>=l', special odd 2}, respectively, and proved in sections \ref{The special case even}, \ref{The special case, odd 1}, and \ref{The special case, odd 2}.
Theorem \ref{main theorem O2l for l>l'}, proved in this section, will furthermore show that the second integral on the right-hand side of \eqref{main theorem summary2} coincides with the first integral when $(\G,\G')=(\Og_{d},\Sp_{2l'}(\R))$, where $d=2l$ or $d=2l+1$, provided $l>l'$.

\begin{rem}\label{projection on invariants}
Notice that, since the character $\Theta_\Pi$ is conjugation invariant,
\[
\int_\G \check\Theta_\Pi(\t g) T(\t g)(\phi)\,dg=\int_\G \check\Theta_\Pi(\t g) T(\t g)(\phi^\G)\,dg\,,
\]
where $\phi^\G$ is defined as in \eqref{phiG}.
\end{rem}

Let
\begin{equation}\label{eq:iota}
\iota=
\begin{cases} 
1 &\text{if $\Dc=\R$ or $\C$}\,,\\
\frac{1}{2} &\text{if $\Dc=\Ha$}\,,
\end{cases}
\end{equation}
and let 
\begin{equation}\label{number r 10}
r=\frac{2\dim \g}{\dim \Vv_\R}\,,
\end{equation}
where the subscript $\R$ indicates that we are viewing $\Vv$ as a vector space over $\R$.
Explicitly,
\begin{equation}\label{number r 1}
r=\left\{
\begin{array}{lll}
2l-1\ &\text{if}\ \G=\Og_{2l}\,,\\
2l\ &\text{if}\ \G=\Og_{2l+1}\,,\\
l\ &\text{if}\ \G=\Ug_{l}\,,\\
l+\frac{1}{2}\ &\text{if}\ \G=\Sp_{l}\,.
\end{array}\right.
\end{equation}

Let 
\begin{equation}
\label{eq:delta,beta}
\delta=\frac{1}{2\iota}(d'-r+\iota) \quad \text{and} \quad \beta=
\frac{2\pi}{\iota}\,.
\end{equation} 
Fix an irreducible representation $\Pi$ of $\wt\G$ that occurs in the restriction of the Weil representation $\omega$ to $\wt\G$. Let 
$\mu\in i\h(\g)^*$ be the Harish-Chandra parameter of $\Pi$ with $\mu_1>\mu_2>\cdots$. 
This means that $\mu=\lambda+\rho$, where $\lambda$ is the highest weight of $\Pi$ and $\rho$ is one half times the sum of the positive roots of $(\g_\C,\h_\C)$. 
If $\G=\Ug_1$ then $\rho=0$ and $\mu=\lambda$ is the weight of $\Pi$. 
 If $\G=\Og_2$ then $\rho=0$. In this case, if $\Pi$ is trivial or $\det$, then $\mu=0$.
Otherwise $\Pi|_{\SOg_2}$ has two weights and we pick any one of them.

Let $P_{a,b}$ and $Q_{a,b}$ be the 
piecewise polynomial functions
defined in (\ref{D0'}) and (\ref{D0''}). Define
\begin{eqnarray}
\label{eq:ajbj}
&&
a_j=-\mu_j-\delta+1\,, \qquad  
b_j=\mu_j-\delta+1\,,\\
\label{pjqj}
&&
p_j(\xi)=P_{a_j,b_j}(\beta \xi)e^{-\beta |\xi|}\,,\qquad 
q_j(\xi)=\beta ^{-1}
Q_{a_j,b_j}(\beta^{-1}\xi) 
\qquad (1\leq j\leq l,\; \xi\in \R)\,, 
\end{eqnarray}
where $\delta$ and $\beta$ are as in (\ref{eq:delta,beta}). 
Notice that $a_j$ and $b_j$ are integers (see Lemma \ref{ximuchexplicit}).
Furthermore, set 
\begin{equation}
\label{kappa0}
\kappa_0=\begin{cases}
1/2 &\text{if $\G=\Og_{2l}$ and $\lambda_l=\mu_l=0$}\,,\\
1 &\text{otherwise}\,.
\end{cases}
\end{equation}

\begin{thm}\label{main thm for l<l'}
Let $l\leq l'$.
Then there is a non-zero constant $C$ which depends only on the dual pair $(\G,\G')$
such that for all $\phi\in\Ss(\Wv)$
\begin{equation}\label{main thm for l<l' a}
\int_{-\G^0}\check\Theta_\Pi(\t g) T(\t g)(\phi)\,dg=C \,
\kappa_0
\check{\chi}_\Pi(\t{c}(0)) 
\int_{\h\cap\tau(\Wv)}
\left(\prod_{j=1}^l  \left(p_j(y_j) +q_j(-\partial_{y_j})\delta_0(y_j)\right)\right)\cdot F_\phi(y)\,dy\,,
\end{equation}
where $\chi_\Pi$
is the central character of $\Pi$ \textup{(}see \eqref{centralcharacterofpi}\textup{)}, $\t{c}$ is a real analytic lift of the Cayley transform \textup{(}see \eqref{eq:tildec}\textup{)}, $\delta_0$ is the Dirac delta at $0$, 
and $F_\phi(y)$ is the Harish-Chandra regular almost-elliptic orbital integral on $\Wv$
of $\phi$ at $y$ \textup{(}see 
\cite[Definition 3.2]{McKeePasqualePrzebindaWCestimates}
and \eqref{eq:originalf-def}\textup{)}.

The term 
\begin{equation}
\label{product-theorem2}
\prod_{j=1}^l  \left(p_j(y_j) +q_j(-\partial_{y_j})\delta_0(y_j)\right)
\end{equation}
is:
\begin{enumerate}
\item a function of $y$ if and only if all the $q_j$'s are zero, and this happens if and only if $l=l'$ 
and $(\G,\G')\neq (\Og_{2l},\Sp_{2l'}(\R))$;
\item a linear combination of products of functions and Dirac delta's at $0$ in some coordinates $y_j$ if and only if all the $q_j$'s are of degree zero. This happens if and only if 
either $(\G,\G')= (\Og_{2l},\Sp_{2l}(\R))$, or $l'=l+1$ and $\Dc=\C$ or 
$\Ha$.
\end{enumerate}
In the remaining cases, \eqref{product-theorem2} is a distribution, but not a measure. 
\end{thm}
\begin{rem}
\label{rem:integration-domain-main thm for l<l'}
The integration domain $\h\cap \tau(\Wv)$ appearing in Theorem \ref{main thm for l<l'}
was explicitly determined in 
\cite[Lemma 3.4]{McKeePasqualePrzebindaWCestimates}. 
It is equal to $\h$ if $\Dc\neq \C$ or if $\Dc=\C$ and $l\leq \min(p,q)$.
By \eqref{2deltaminus2}, \eqref{2deltaminus2-bis} and Appendix \ref{appenE}, we see that 
$a_j\leq 0$ for all $1\leq j\leq l$ when $l\leq l'$. Hence each $P_{a_j,b_j}(\beta y_j)$ vanishes for $y_j<0$. In cases (1) and (2) of Theorem \ref{main thm for l<l'} with $\Dc=\R$ or $\Ha$, we can therefore replace the domain of integration $\h$ with the smaller domain $\tau(\hs1)$.
\end{rem}

In the case $l>l'$, up to conjugation, there is a unique Cartan subspace $\hs1$ in $\Wv$. 
Recall that for $\Dc=\C$ we are supposing that $p\leq q$.

Define $s_0\in W(\G,\h(\g))$ by 
\begin{alignat}{2}
\label{s0-R-H}
s_0(J_j)&=J_j \qquad\text{($1\leq j\leq l)$} &&\quad\text{if $\Dc=\R$ or $\Ha$\,,}\\
\label{s0-C}
s_0(J_j)&=\begin{cases} 
J_{j} &\text{($1\leq j\leq p)$}\\
J_{q+j} &\text{($p+1\leq j\leq l-q)$}\\
J_{j-l+l'} &\text{$(l-q+1\leq j\leq l)$}\, 
\end{cases} &&\qquad\text{if $\Dc=\C$}\,.
\end{alignat}
\begin{thm}\label{main thm for l>=l'}
Let $l>l'$. Consider a genuine irreducible representation $\Pi$ of $\wt\G$. (Its highest weight is among the weights listed in Appendix \ref{appenE}).
Then
\begin{equation}
\label{integral-Thm3}
\int_{-\G^0}\check\Theta_\Pi(\t g) T(\t g)\,dg \neq 0
\end{equation}
if and only if the highest weight $\lambda=\sum_{j=1}^l\lambda_j e_j$ of $\Pi$ is of the form  
\begin{eqnarray*}
\textup{(a)} \; &&
\text{$\lambda_1\geq \lambda_2 \geq \cdots \geq \lambda_{l'}\geq 0$
and $\lambda_j=0$ for $l' +1\leq j\leq l$\,, \quad if $\Dc=\R$ or $\Ha$\,,}\\
\textup{(b)} \; &&
\text{$\lambda_j=\frac{p-q}{2}+\nu_j$\,, \quad where}\\ 
&&\nu_1\geq \cdots \geq\nu_{p}\geq 0\,, \;
\nu_j=0 \; \text{for $p+1\leq j\leq l-q$}\,, \; 0\geq \nu_{l-q+1}\geq \cdots \geq \nu_l\,,
\quad \text{if $\Dc=\C$\,.} 
\end{eqnarray*}
Suppose that (a) and (b) are satisfied.
Then there is a 
non-zero
constant $C$ which depends only on the dual pair $(\G,\G')$ 
such that for all $\phi\in\Ss(\Wv)$
\begin{equation}\label{main thm for l>=l' b}
\int_{-\G^0}\check\Theta_\Pi(\t g) T(\t g)(\phi)\,dg
=C\, \kappa_0
\check{\chi}_\Pi(\t{c}(0)) \int_{\tau'(\reg{\hs1})}\Big(\prod_{j=1}^{l'}  
p_{s_0^{-1}(j)}(y_j)
\Big)\cdot F_\phi(y)\,dy\,,
\end{equation}
where 
$\kappa_0$ is as in \eqref{kappa0}
and, explicitly,
$$
\prod_{j=1}^{l'}  
p_{s_0^{-1}(j)}(y_j)
=\begin{cases}
\prod_{j=1}^{l'}  p_j(y_j) &\text{if $\Dc=\R$ or $\Ha$}\,,\\
\big(\prod_{j=1}^{p}  p_j(y_j)\big)\big(
\prod_{j=p+1}^{l'}  p_{j+l-l'}(y_j)
\big) &\text{if $\Dc=\C$}\,.
\end{cases}
$$
The right-hand side of \eqref{main thm for l>=l' b} can be written as a non-zero constant multiple of
\begin{equation}
\label{integral on W for l>=l'}
\kappa_0
\check{\chi}_\Pi(\t{c}(0)) \int_{\tau'(\reg{\hs1})} \Phi(y)\pi_{\g/\z}(y) F_\phi(y) \, dy
=\int_{\Wv} \Phi^\sharp(w)\phi(w) \,dw\,,
\end{equation}
where 
\begin{multline}
\label{main thm for l>=l' c} 
\Phi(y)=
\frac{
\sum_{s'\in W(\G',\mathfrak{h}')} \sgn_{\g'/\h'}(s') 
\prod_{j=1}^{l'} P_{a_{s_0,j},b_{s_0,j},2\delta_j} (\beta (s'y)_j) }
{\pi_{\g/\z}(y)} e^{-\beta \sum_{j=1}^{l'} |y_j|}\,,\\
\qquad (y_j={J_j'}^*y,  \; y=\tau(w)=\tau'(w), w\in \reg{\hs1})
\end{multline}
is a non-zero $W(\G',\h')$-invariant real-valued continuous function on $\tau'(\reg{\hs1})$, and 
$\Phi^\sharp$ is an $\Sg$-invariant function such that
$\Phi^\sharp(w)=\Phi(\tau'(w))$ for all $w\in \reg{\hs1}$. 
In \eqref{main thm for l>=l' c},  $\mu$ is the Harish-Chandra parameter of $\Pi$,
\begin{equation}
\label{asj-bsj}
a_{s,j}=-(s\mu)_j-\delta+1\,,\qquad  
b_{s,j}=(s\mu)_j-\delta+1 \qquad 
(s\in W(\G,\h),\ 1\leq j\leq l)\,,
\end{equation}
$P_{a,b,\pm 2}$ is the polynomial defined in \eqref{eq:Pab2} or \eqref{eq:Pabminus2}, and the $\delta_j$'s are as in \eqref{deltaj}. (See \eqref{the identification} for the identifications $y=\tau(w)=\tau'(w)$ in \eqref{main thm for l>=l' c} .)
\end{thm}

\begin{rem}
Recall from Remark \ref{rem:Weyl-invariant l>l'} that the domain of integration $\tau'(\reg{\hs1})$ appearing in Theorem \ref{main thm for l>=l'} is $W(\G',\h')$-invariant. 
Formula \eqref{main thm for l>=l' c} will prove, by Proposition \ref{cor:sufficiently many Fphi}, that the intertwining distribution is not zero when the conditions (a) or  (b) are satisfied. 
\end{rem}

\begin{rem}
Conditions (a) and (b) in Theorem \ref{main thm for l>=l'} are precisely those ensuring that $\Pi$ occurs in Howe's correspondence. See Corollary \ref{HW-l>l'} below.  (They are contragredient to those listed in \cite[Appendix]{PrzebindaInfinitesimal}, because the Weil representation used there is contragredient to the one used here.)
\end{rem}

Before considering the integrals over $\G\setminus (-\G^0)$ in \eqref{main theorem summary2},  let us introduce some notation concerning the irreducible representations of the orthogonal groups. \label{notation-Od-irreps}
Since $\Dc\neq \C$, we can choose a polarization $\Wv=\Xv\oplus \Yv$ so that $\G$ preserves $\Xv$ and $\Yv$. We shall suppose in what follows that we have made such a choice. 

Suppose that $\G=\Og_d$. Then, for each highest weight $\lambda$ of an irreducible representation of $\G^0$ there are one or two unitary genuine representations 
of $\wt{\G}$ having highest weight $\lambda$. 
There are two if and only if either $d=2l$ and $\lambda_l=0$, or $d=2l+1$. See e.g. \cite[\S 5.5.5]{GoodmanWallach2009}. 
Let $\Pi_{\lambda,+}$ and $\Pi_{\lambda,-}$ be these representations. 
Set 
\begin{equation}
\label{chi+-on-O}
\chi_+(\wt{g})=\frac{\Theta(\wt{g})}{|\Theta(\wt{g})|} \qquad (g\in \Og_d)\,,
\end{equation}
where $\Theta$ is defined in \eqref{THETA}. Then $\chi_+$ is a character of $\wt{\G}$.
Notice that \eqref{chi+-on-O} is an extension of \eqref{chi+} from $\wt{\Zg}$ to $\wt{\G}$. 
In fact, Proposition 4.28 in \cite{AubertPrzebinda_omega} implies that $(\chi_+(\wt{g}))^2=(\det g)_\X^{-1}$, where $(\det g)_\X$ indicates the determinant of $g$ as endomorphism of $\X$.

\label{O_2l+1, l>l', notation}
Then, in the Schr\"odinger model for the Weil representation $\omega$, for which the space of smooth vectors is $\mathcal{S}(\X)$, 
\begin{equation}
\label{omega-on-G}
\left(\omega\otimes\chi_+^{-1}\right)(\wt g)f(x)=f(g^{-1}x) \qquad (g\in \G, \, f\in \mathcal{S}(\X), \, x\in \X)\,.
\end{equation}
Hence $\omega\otimes\chi_+^{-1}$ descends to a representation $\omega_0$ of $\G$ given by 
\begin{equation}
\label{omega0}
\omega_0(g)f(x)=f(g^{-1}x) \qquad (g\in \G, \, f\in \mathcal{S}(\X), \, x\in \X)\,.
\end{equation}
\begin{thm}
\label{main theorem O2l for l>l'}
Suppose that $l>l'$. Let $\Pi$ be an irreducible representation of $\wt{\Og}_d$ occurring in the restriction of the Weil representation to $\wt{\Og}_d$. 
If $d=2l$, then $\lambda_l=0$. In both cases $d=2l$ or $d=2l+1$, the second irreducible genuine representation of $\wt{\Og}_d$ having the same highest weight as $\Pi$ does not occur in the restriction of the Weil representation to $\wt{\Og}_d$. Moreover,
\begin{equation}
\label{O2l-l>l'}
\int_{\G}\check\Theta_\Pi(\t g) T(\t g)\,dg=2\int_{\G^0}\check\Theta_\Pi(\t g) T(\t g)\,dg=2\int_{-\G^0}\check\Theta_\Pi(\t g) T(\t g)\,dg
\,.
\end{equation}
In particular, 
\begin{equation}
\label{equal-integrals-Od-l>l'}
\int_{\G\setminus(-\G^0)}\check\Theta_\Pi(\t g) T(\t g)\,dg\,=\int_{-\G^0}\check\Theta_\Pi(\t g) T(\t g)\,dg\,.
\end{equation}

The integral on the very right-hand side of \eqref{O2l-l>l'} was computed in Theorem \ref{main thm for l>=l'}. 
\end{thm}
\begin{proof}
Let $\lambda$ be the highest weight of $\Pi$, and let $d=2l$ or $2l+1$. 
Recall the notation introduced before \eqref{chi+-on-O}.

Suppose that both $\Pi_{\lambda,+}$ and $\Pi_{\lambda,-}$ occur. 
Then $\Pi_{\lambda,\pm}\otimes \chi_+^{-1}$ descends to a representation $(\Pi_{\lambda,\pm}\otimes\chi_+^{-1})|_\G$ of $\G$ occurring in $\omega_0$. 
Let $\mathcal{S}(\X)_{\Pi_{\lambda,\pm}}\subseteq\mathcal{S}(\X)$ denote the $\Pi_{\lambda,\pm}$-isotypic component in $\mathcal{S}(\X)$.
By \eqref{omega0}, 
\begin{equation}
\label{omega0-for-Pil}
(\Pi_{\lambda,\pm}\otimes\chi_+^{-1})|_{\G}(g)f(x)=f(g^{-1}x) \qquad (g\in \G, \, 
 f\in \mathcal{S}(\X)_{\Pi_{\lambda,\pm}}, 
\, x\in \X)\,.
\end{equation}
Let $\Pi_{\lambda,0}$ denote an irreducible representation of $\G$ whose restriction to the identity component has highest weight $\lambda$.
As one can see from \cite[\S 5.5.5]{GoodmanWallach2009},
\begin{equation}
\label{Pi-tensor-chi+}
\text{
if $(\Pi_{\lambda,+}\otimes\chi_+^{-1})|_\G=\Pi_{\lambda,0}$, then $(\Pi_{\lambda,-}\otimes\chi_+^{-1})|_\G=\Pi_{\lambda,0}\otimes\det$\,.
}
\end{equation}
Hence $\Pi_{\lambda,0}\otimes\Pi_{\lambda,0}\otimes \det$ occurs in $\omega_0\otimes\omega_0$,
acting on $\mathcal{S}(\X\oplus\X)$. 
Recall that $\Pi_{\lambda,0}=\Pi_{\lambda,0}^c$ is self-contragredient. Since $\Pi_{\lambda,0}^c\otimes\Pi_{\lambda,0}$ contains the trivial representation, we conclude that
$\det$ occurs in $\omega_0\otimes\omega_0$.
Observe that $\omega_0\otimes\omega_0$ acts on $\mathcal{S}(\X\oplus\X)$ by 
$$
\omega_0\otimes\omega_0(g)f(x)=f(g^{-1}x) \qquad (g\in \G, \, f\in \mathcal{S}(\X\oplus\X), 
\, x\in \X)\,.
$$
It is therefore the ``representation $\omega_0$'' corresponding to a dual pair $(\Og_d,\Sp_{4l'}(\R))$.
By Proposition \ref{non-occurence-det}, it follows that $d\leq 2l'$, contrary to our assumption.  

Suppose first that $\Pi_{\lambda,+}$ is not isomorphic to $\Pi_{\lambda,-}$, which by the description of the irreducible representations of orthogonal groups \cite[\S 5.5.5]{GoodmanWallach2009} can occur only when $\lambda_l=0$ if $d=2l$. Then the above argument shows that only one of $\Pi_{\lambda,+}$ and $\Pi_{\lambda,-}$ (i.e. $\Pi$) occurs in the restriction of the Weil representation. 

On the other hand, if $\Pi_{\lambda,+}$ is isomorphic to $\Pi_{\lambda,-}$, then $d=2l$ (because 
$\det(-\mathrm{I}_{2l+1})=-1$) and, again by \cite[\S 5.5.5]{GoodmanWallach2009}, $\lambda_l\ne 0$. In this case, $\Pi_{\lambda,0}=\Pi_{\lambda,0}\otimes \det$ and 
the above argument shows that the representation does not occur in $\omega$.

%If $\Pi_{\lambda,+}$ is not isomorphic to $\Pi_{\lambda,-}$ then the above argument shows that only one of them occurs, and if $d=2l$ then $\lambda_l=0$.
%
%If $\Pi_{\lambda,+}$ is isomorphic to $\Pi_{\lambda,-}$  then $d=2l$ (because 
%$\det(-\mathrm{I}_{2l+1})=-1$) and, by the well known description of the representations of the compact orthogonal groups \cite{GoodmanWallach2009}, $\lambda_l\ne 0$. In this case, the above argument shows that the representation does not occur in $\omega$. 
 
Thus the second representation of $\wt{\Og}_d$ which has the same restriction as $\Pi$ to $\G^0=\SO_d$, does not occur. Hence the 
$\Pi|_{\wt{\SO}_{d}}$-isotypic component of $\omega$ coincides with the $\Pi$-isotypic component of $\omega$. Therefore
\[
\int_{\G}\check\Theta_\Pi(\t g) T(\t g)\,dg=2\int_{\G^0}\check\Theta_\Pi(\t g) T(\t g)\,dg\,.
\]
(The factor 2 is a consequence of the normalization of the measures.)
In particular, $\int_{\G\setminus\G^0}\check\Theta_\Pi(\t g) T(\t g)\,dg=
\int_{\G^0}\check\Theta_\Pi(\t g) T(\t g)\,dg$\,. If $\G=\Og_{2l}$, then $\G^0=-\G^0$ and if $\G=\Og_{2l+1}$, then $\G\setminus\G^0=-\G^0$.  This explains the second equality in \eqref{O2l-l>l'}.
\end{proof}

\begin{rem}
It should be pointed out that the proof of Theorem \ref{main theorem O2l for l>l'} does not use the known classification of the highest weights of the genuine irreducible representations occurring in Howe's correspondence. 
\end{rem}

Consider now the case $(\G, \G')=(\Og_{2l},\Sp_{2l'}(\R))$ and the character $\Theta_\Pi$ not supported in the preimage $\wt{\G^0}$ of the connected identity component 
$\G^0\subseteq \G$. 

Suppose that 
$l\leq l'$ and 
$l\neq 1$. 
Then the graded  vector space \eqref{decomposition of space for a cartan subspace} is equal to
\[
\V=\V^0_{\overline{1}} \oplus\V^1\oplus\V^2\oplus\dots \oplus \V^{l}\,.
\]
Recall from \eqref{individual basis} that in each $\V_{\overline 0}^j$ we selected an orthonormal basis $v_0$, $v_0'$. For convenience, we introduce the index $j$ in the notation and  we write $v_{2j-1}=v_0$ and $v_{2j}=v_0'$, for $1\leq j\leq l$. Then $v_1$, $v_2$,\dots, $v_{2l}$ is an orthonormal basis of 
$\V_{\overline 0}$ and 
\[
J_jv_{2j-1}=-v_{2j}\,,\ \ \ J_jv_{2j}=v_{2j-1} \qquad (1\leq j\leq l)\,.
\]
In terms of the dual basis \eqref{eq:ej} of $\h_\C^*$, the positive roots are
\[
e_j\pm e_k \qquad (1\leq j<k\leq l)\,.
\]
Define an element $s\in\G$ by
\begin{equation}\label{s}
sv_1=v_1,\ sv_2=v_2,\; \dots,\; sv_{2l-1}=v_{2l-1},\ sv_{2l}=-v_{2l}\,.
\end{equation}
Then $\G=\G^0\cup\G^0s$ is the disjoint union of two connected components. 
Set
\[
\V_{\overline 0, s}=\V_{\overline 0}^1\oplus \V_{\overline 0}^2\oplus  \cdots\oplus \V_{\overline 0}^{l-1}\oplus \R v_{2l}\,,
\ \ \text{and}\ \   \V_s=\V_{\overline 0, s}\oplus \V_{\overline 1}\,.
\]
The dual pair corresponding to $(\V_{\overline 0, s},\V_{\overline 1})$ is $(\G_s, \G'_s)=(\Og_{2l-1}, \Sp_{2l'}(\R))$ acting on the symplectic space 
$\Wv_s=\Hom(\V_{\overline 1},\V_{\overline 0, s})$. 
The objects corresponding to $\Wv_s$ will be distinguished by the subscript $s$.

Let $\h_s=\sum_{j=1}^{l-1}\R J_j$. This is the centralizer of $s$ in $\h=\sum_{j=1}^{l} \R J_j$. 
Set
\begin{equation}
\label{rhos_O2l}
\rho^{\tC}_s=(l-1)e_1+(l-2)e_2 + \dots + e_{l-1}\,.
\end{equation}
Let 
\[
\lambda=\sum_{j=1}^{l-1}\lambda_j e_j
\]
be the highest weight of $\Pi$. (Here $\lambda_l=0$ because we assume that $\Theta_\Pi$ is not supported in $\wt{\G^0}$.) 
Define
\[
\mu^{\tC}=\lambda+\rho_s^{\tC}\,.
\]
The number  $r$, \eqref{number r 1}, for the group $\G$ is equal to
\[
r=2l-1
\]
and the number $\delta$, \eqref{eq:delta,beta}, for the dual pair $(\G, \G')$ is equal to
\[
\delta=\frac{1}{2}(2l'-r+1) =l'-l+1\,.
\]
Set
\[
a_j^{\tC}=-\mu_j^{\tC}-\delta+1=-\mu_j^{\tC}-l'+l\,,\quad b_j^{\tC}=\mu_j^{\tC}-\delta+1=\mu_j^{\tC}-l'+l\,, \qquad (1\leq j\leq l-1)\,.
\]
Notice that $a_j^{\tC}=a_j$ and $b_j^{\tC}=b_j$ for $1\leq j\leq l-1$ because 
$\rho_s^{\tC}$ coincides with the restriction of $\rho$ to $\h_s$.
Using these numbers in place of $a_j$ and $b_j$ in \eqref{pjqj}, define the functions $p_j^{\tC}$ and $q_j^{\tC}$. 
\begin{thm}\label{main thm for l<l', special}
Let $(\G, \G')=(\Og_{2l}, \Sp_{2l'}(\R))$ with $1<l\leq l'$. Assume that the character $\Theta_\Pi$ is not supported in $\wt{\G^0}$. 
Then there is a constant $C$ which depends only on the dual pair $(\G,\G')$  such that 
for any $\phi\in\Ss(\Wv)$
\begin{equation}\label{main thm for l<l' a, special}
\int_{\G^0s}\check\Theta_\Pi(\t g) T(\t g)(\phi)\,dg
=C 
D_\Pi 
\check{\chi}_\Pi(\t{c}(0)) 
\int_{\h_s}\prod_{j=1}^{l-1}
\left(p_j^{\textup{\tC}}(y_j)+q_j^{\textup{\tC}}(-\partial_{y_j})
\delta_0(y_j)\right)
 \cdot F_{\phi^\G|_{\Wv_s}}(y)\,dy\,, 
\end{equation}
where $\check{\chi}_\Pi(\t{c}(0))$ 
and $D_\Pi$ are equal to 
$\pm 1$, 
and $D_\Pi$ distinguishes $\Pi$ and $\Pi\otimes \det$.
\end{thm}

Theorem \ref{main thm for l<l', special} excludes the dual pairs $(\G, \G')=(\Og_2, \Sp_{2l'}(\R))$ because its proof relies on an analogue of the Weyl's character formula for $\G\setminus \G^0$ proved by \cite{Wendt} for nonconnected compact semisimple Lie groups. These excluded cases will  be treated in subsection \ref{example:O2Sp2}. 
\medskip

\label{2l+1 smaller than 2l' special-the page}
Now we consider the case $(\G, \G')=(\Og_{2l+1},\Sp_{2l'}(\R))$ with $1\leq l\leq l'$. Recall from \eqref{decomposition of space for a cartan subspace} the graded vector space $\V$.
In the case we consider, $\dim \V_{\overline 0}^{0}=1$, $\dim \V_{\overline 1}^{0}=2(l'-l)$ and for $1\leq j\leq l$, $\dim \V_{\overline 0}^{j}=\dim \V_{\overline 1}^{j}=2$.
Let 
\[
\Wv_s=\Hom(\V_{\overline 1},\V_{\overline 0}^{1}\oplus \cdots\oplus \V_{\overline 0}^{l})
\quad \text{and} \quad 
\Wv_s^\perp=\Hom(\V_{\overline 1},\V_{\overline 0}^{0})\,.
\]
Then
\begin{equation}\label{2l+1 smaller than 2l' special}
\Wv=\Wv_s\oplus \Wv_s^\perp
\end{equation}
is a direct sum orthogonal decomposition. Let $\G_s\subseteq \G$ be the subgroup acting trivially on the space $\V_{\overline 0}^{0}$. 
The Lie algebra $\g_s$ of $\g$ embeds as those elements 
acting as zero on $\V_{\overline 0}^{0}$.
Let $\G_s'=\G'$.
Then the dual pair corresponding to $\Wv_s$ is 
$(\G_s,\G'_s) \simeq 
(\Og_{2l}, \Sp_{2l'}(\R))$ and dual pair corresponding to $\Wv_s^\perp$ is $(\Og_{1}, \Sp_{2l'}(\R))$. 
If $\H$ is a Cartan subgroup of $\G$, then $\H^0=\H_s^0$ is a Cartan subgroup of $\G_s^{0}$,
 and the Lie algebras $\g$ and $\g_s$ share the same Cartan subalgebra $\h=\h_s$.
The following theorem will be proved in section \ref{The special case, odd 2}.

\begin{thm}
\label{main thm for l<l', special odd 1}
Let $(\G, \G')=(\Og_{2l+1}, \Sp_{2l'}(\R))$ with $1\leq l\leq l'$.  
Then there is a nonzero constant $C$ such that for all $\phi\in \mathcal{S}(\Wv)$
\begin{equation}
\label{main thm for l<l' a, special odd 1bis} 
\int_{\G^0}\check\Theta_\Pi(\t g) T(\t g)(\phi)\,dg
= C
(-1)^{|\lambda|}
\int_{\h}
\prod_{j=1}^l  
\left(p_j(y_j) +q_j(-\partial_{y_j})
\delta_0(y_j)\right)
F_{\phi^\G|_{\Wv_s}}(y)\,dy\,,
\end{equation}
where $p_j, q_j$ are defined as in \eqref{pjqj}, $\lambda$ is the highest weight of $\Pi$ and $|\lambda|=\sum_{j=1}^l \lambda_j$ is a nonnegative integer. (See Appendix \ref{appenE}.)

If $l=l'$, then $F_{\phi^\G|_{\Wv_s}}$ is proportional to $F_{\phi}$ (independently of $\phi$).
\end{thm}
\begin{rem}
\label{rem:integration-domain-main thm for l<l'-O2l+1}
As in Theorem \ref{main thm for l<l'}, the term
$$
\prod_{j=1}^l  \left(p_j(y_j) +q_j(-\partial_{y_j})\delta_0(y_j)\right)
$$
is a function of $y$ (i.e. all the $q_j$'s are zero) if and only if $l=l'$. 
In the other cases, it is a distribution, but not a measure. Furthermore, if $l=l'$, 
we can replace the domain of integration $\h$ with the smaller domain $\tau(\hs1)$. 
\end{rem}

\begin{rem}
It is known from the classification of the representations occurring in Howe's correspondence (see e.g. 
\cite[Appendix]{PrzebindaInfinitesimal}) that for the pair $(\G,\G')=(\Og_{2l+1},
\Sp_{2l'}(\R))$ with $l\leq l'$ there are two representations of $\wt{\G}$ with the same highest weight $\lambda$ that occur in the correspondence, namely $\Pi(\t g)$ and $\Pi(\t g)\otimes\det(g)$. They agree on $\G^0$, so the integral 
on the left-hand side of \eqref{main thm for l<l' a, special odd 1bis} cannot distinguish them. In particular, we cannot replace the factor $(-1)^{|\lambda|}$ with $\check{\chi}_\Pi(\t{c}(0))$, which appears in Theorems \ref{main thm for l<l'} and \ref{main thm for l>=l'}.
\end{rem}

\begin{rem}
The pair $(\Og_1, \Sp_{2l'}(\R))$ was studied in detail in section \ref{The center of the metaplectic group}.
\end{rem}

Suppose $(\G,\G')=(\Og_d, \Sp_{2l'}(\R))$, where $d=2l$ or $2l+1$ and $d>2$.
In Theorem \ref{thm:det} below, the integral over $\G\setminus (-\G^0)$ of the distribution-valued map $g \to \check\Theta_\Pi(\t g) T(\t g)$ is reduced to an integral over $-\G^0_s$. 
The resulting equality, which holds independently of the mutual relation between the ranks $l$ and $l'$, will be needed in \cite{McKeePasqualePrzebindaWC_WF}. 
Recall that 
%$\G\setminus -\G^0$ is equal to $\G^0s$ if $\G=\Og_{2l}$ and to $\G^0$ if $\G=\Og_{2l+1}$. 
$$
\G\setminus (-\G^0)= \begin{cases}\G^0s & \text{if $\G=\Og_{2l}$}\,,\\
                                              \G^0 & \text{if $\G=\Og_{2l+1}$}\,.\end{cases}
$$
Moreover, 
$-\G^0_s=\G_s^0$ if $\G=\Og_{2l+1}$.

\begin{thm}
\label{thm:det}
Let $\G=\Og_d$ with $d>2$. If  $d=2l$, suppose that the character $\Theta_\Pi$ is not supported in $\wt{\G^0}$. 
Then for all $\phi\in\Ss(\Wv)$
\begin{equation}
\label{main thm for l<l' a, special odd 1}
\int_{\G\setminus (-\G^0)}\check\Theta_\Pi(\t g) T(\t g)(\phi)\,dg=
\frac{1}{2}
 \int_{-\G_s^0}\check\Theta_\Pi(\t g)\det(1-g) T_s(\t g)(\phi^\G|_{\Wv_s})\,dg\,,
\end{equation}
where $T_s$ is the operator $T$, see \eqref{Tt}, corresponding to the symplectic space $\Wv_s$.
\end{thm}

We prove Theorem \ref{thm:det}  in section \ref{section: proof of thm det}.

\begin{rem}
The term $\det(1-g)$ appearing in \eqref{main thm for l<l' a, special odd 1}
admits a representation theoretical interpretation. Indeed, let
$\sigma$ be the spin representation 
of the spin cover 
of $\G_s^0$. Then 
the tensor product $\sigma\otimes\sigma^c$ is a representation of $\G_s^0$ and, 
by \cite[Ch. XI, III., p. 254]{Littlewood}
\begin{equation}\label{character of spin}
\Theta_{\sigma\otimes\sigma^c}(g)=|\Theta_\sigma(g)|^2=\det(1+g)\qquad (g\in \G_s^0)\,.
\end{equation}
So $\det(1-g)=\Theta_{\sigma\otimes\sigma^c}(-g)$.
\end{rem}

Suppose $l>l'$. Theorem \ref{main theorem O2l for l>l'}  reduces the computation of $\int_{\G} \check\Theta_\Pi(\t g) T(\t g)\,dg$ to that of $\int_{-\G^0} \check\Theta_\Pi(\t g) T(\t g)\,dg$, done in Theorem \ref{main thm for l>=l'}. 
One could still try to compute the integral on $\G\setminus (-\G^0)$ directly, without relying on
Theorem \ref{main theorem O2l for l>l'}. 
As an example, we do it for $\Og_{2l+1}$ in Theorem \ref{main thm for l>=l', special odd 2} below. Nevertheless, the result is less precise than that from Theorem \ref{main theorem O2l for l>l'} since we are only able to prove that the integral over $\G\setminus (-\G^0)$ is a nonzero constant multiple of the one over $-\G^0$. Determining the constant is a serious issue even in the much easier situation of $(\Ug_l,\Ug_{l'})$; see \cite{McKeePasqualePrzebindaWCSymmetryBreakingUU}. 

To consider the case $(\G, \G')=(\Og_{2l+1},\Sp_{2l'}(\R))$ with $l>l'$, recall the graded vector space $\V$, \eqref{decomposition of space for a cartan subspace} and the formula \eqref{extendedcsa0}, 
\[
\V=\V^0\oplus \V^{1}\oplus \cdots\oplus \V^{l'}\,,
\]
where, as in \eqref{extendedcsa1}, 
\begin{eqnarray*}
\V_{\overline 0}^0&=&\V_{\overline 0}^{0,0}\oplus \left(\V_{\overline 0}^{0,0}\right)^\perp\,,\\
\V_{\overline 1}^0&=&0\,,
\end{eqnarray*}
with $\dim \V_{\overline 0}^{0,0}=1$ and $\dim \left(\V_{\overline 0}^{0,0}\right)^\perp=2(l-l')$.
Let 
\[
\Wv_s=\Hom(\V_{\overline 1}^{1}\oplus \cdots\oplus \V_{\overline 1}^{l'}, \left(\V_{\overline 0}^{0,0}\right)^\perp
\oplus \V_{\overline 0}^{1}\oplus \cdots\oplus \V_{\overline 0}^{l'})\,,\ \ \ 
\Wv_s^\perp=\Hom(\V_{\overline 1},\V_{\overline 0}^{0,0})\,.
\]
(Notice that $\V_{\overline 1}^{1}\oplus \cdots\oplus \V_{\overline 1}^{l'}=\V_{\overline 1}$ and 
$\left(\V_{\overline 0}^{0,0}\right)^\perp
\oplus \V_{\overline 0}^{1}\oplus \cdots\oplus \V_{\overline 0}^{l'}$ is the orthogonal complement of the one dimensional space $\V_{\overline 0}^{0,0}$ in $\V_{\overline 0}$.)
Then
\begin{equation}\label{main thm for l>=l', special odd 2 -1}
\Wv=\Wv_s\oplus \Wv_s^\perp
\end{equation}
is a direct sum orthogonal decomposition. Let $\G_s\subseteq \G$ be the subgroup acting trivially on the space $\V_{\overline 0}^{0,0}$ and let $\G_s'=\G'$.
The dual pair corresponding to $\Wv_s$ is 
$(\G_s,\G'_s)
\simeq
(\Og_{2l}, \Sp_{2l'}(\R))$ and dual pair corresponding to $\Wv_s^\perp$ is $(\Og_{1}, \Sp_{2l'}(\R))$.

\begin{thm}\label{main thm for l>=l', special odd 2}
Let $(\G, \G')=(\Og_{2l+1}, \Sp_{2l'}(\R))$ with $l> l'$. 
Then
\begin{equation}
\label{integral-Thm7}
\int_{\G^0}\check\Theta_\Pi(\t g) T(\t g)\,dg \neq 0
\end{equation}
if and only if the highest weight $\lambda=\sum_{j=1}^l\lambda_j e_j$ of $\Pi$ satisfies
condition (a) of Theorem \ref{main thm for l>=l'} for $\Dc=\R$.
%%%
%\begin{equation}
%\label{condition-on-lambda-O2l+1}
%\text{$\lambda_1\geq \lambda_2 \geq \cdots \geq \lambda_{l'}\geq 0$
%and $\lambda_j=0$ for $l' +1\leq j\leq l$}\,.
%\end{equation}
%%%
Suppose that this condition is satisfied. 
Then there is a 
non-zero
constant  $C$ which depends only on the dual pair $(\G,\G')$ such that for all $\phi\in\Ss(\Wv)$
\begin{equation}
\label{main thm for l<l' a, special odd 1-bis}
\int_{\G^0}\check\Theta_\Pi(\t g)T(\t g)(\phi)\,dg
%\label{main thm for l>=l' b, special odd 2}
=C (-1)^{|\lambda|}  \int_{\tau'(\reg{\hs1})}\Big(\prod_{j=1}^{l'}  p_j(y_j)\Big) F_{\phi}(y)\,dy\,.
\end{equation}
\end{thm}

As a byproduct of our calculations of the intertwining distributions, we obtain the list of highest weights of the genuine irreducible representations $\Pi$ of $\wt{\G}$
that occur in Howe's correspondence when $l>l'$. This list was first determined (without any restrictions on the ranks $l$ and $l'$) in \cite{KashiwaraVergne}. 

\begin{cor}
\label{HW-l>l'}
Suppose that $l>l'$. 
A genuine representation $\Pi\in \wt{\G}^\wedge$ occurs in Howe's correspondence if and only if its highest weight satisfies conditions (a) or (b) of Theorem \ref{main thm for l>=l'}. 
\end{cor}
\begin{proof}
Our computations of the intertwining distribution $\int_{\G} \wt{\Theta}_\Pi(\wt{g}) T(\wt{g}) \, dg$ can be applied to any genuine irreducible representation $\Pi\in \wt{\G}^\wedge$ 
(not necessarily occurring in Howe's correspondence).
 This distribution is 
nonzero if and only if $\omega|_{\wt{\G}}$ has a nonzero $\Pi$-isotypic component. This is equivalent to the fact that there is a unitary highest weight representation $\Pi'$ of $\wt{\G'}$ such that $\Pi\otimes\Pi'$ occurs in $\omega|_{\wt{\G}\wt{\G'}}$. 
The nonvanishing of the intertwining distributions leads to conditions (a) or (b) of Theorem \ref{main thm for l>=l'} when $\G=\Ug_l$ or $\Sp_l$. 
In the case of orthogonal groups, we can further use 
Theorem \ref{main theorem O2l for l>l'} and conclude that the nonvanishing of the intertwining distributions is equivalent to the nonvanishing of the integral of $\wt{\Theta}_\Pi(\wt{g}) T(\wt{g})$ over $-\G^0$. The claim then follows again from Theorem \ref{main thm for l>=l'}.
\end{proof}

As we shall see in the proofs in section \ref{An intertwining distribution in terms of orbital integrals on the symplectic space}, the list of highest weights in Theorem \ref{main thm for l>=l'} is 
obtained by comparing the support of the function $\prod_{j=1}^{l'}  p_j(y_j)$ with the domain of
integration, $\tau'(\reg{\hs1})$. Unfortunately, this method is not refined enough to 
provide necessary and sufficient conditions when $l\leq l'$. 

Let us now consider the dual pair $(\Ug_l, \Ug_{p,q})$. Recall that in this case $l'=p+q$ and that we assume that $p\leq q$.
If $l\leq p$ all irreducible genuine representations of $\wt{\Ug_l}$ occur because the pair is in the stable range with $\Ug_l$ the smaller member; see \cite{Jian-ShuLiSingular} or \cite{ProtsakPrzebindaStable}. 
The absence of conditions on the highest weight in Theorem \ref{main thm for l<l'} is consistent with this fact (despite the fact that we cannot see that our intertwining operator is not $0$).

If $p<l\leq p+q$ then the next corollary gives precise necessary conditions on the highest weight of $\Pi$ to occur in the correspondence. The proof is independent of the classification and is based on a refined analysis of the intertwining distribution; see section \ref{section: proof corollary on UU l<=p+q}.

\begin{cor}
\label{HW-l <=l'}
Suppose that $\Dc=\C$ and $p<l\leq p+q$. Let $\Pi\in \wt{\G}^\wedge$ be a genuine irreducible representation of highest weight $\lambda$.
If either $\lambda_{p+1}>\frac{p-q}{2}$ or 
(when $q<l$ holds) $\lambda_{l-q}<\frac{p-q}{2}$,
 then $\Pi$ does not occur in Howe's correspondence.
\end{cor}

For the dual pair $(\Sp_l, \Og^*_{2l'})$, by the known classification of highest weights of representations of $\Sp_l$ occurring in Howe's correspondence, all irreducible genuine representations of $\wt{\Sp_l}$ occur if $l\leq l'$. We can recover this fact out of the formula for the intertwining distribution determined in Theorem \ref{main thm for l<l'} (and hence without using the classification) 
only when $l'\in \{l, l+1\}$. 
This is the content of the following corollary, proved in section \ref{section: proof corollary HW-H-l <=l'}. 

\begin{cor}
\label{HW-H-l <=l'}
Suppose that $\Dc=\Ha$ 
and $l\leq l'$. 
Let $\Pi$ be an irreducible genuine representation of $\G=\Sp_l$ with highest weights $\lambda_1\geq\dots\geq \lambda_l$. If $\lambda_l\geq l'-l-1$ then 
$\Pi$ occurs in Howe's correspondence. In particular, if $l'=l$ or $l'=l+1$, then every genuine irreducible representation $\Pi\in \wt{\G}^\wedge$ occurs in Howe's correspondence.
\end{cor}

We terminate our discussion on the highest weights of the genuine irreducible representations of 
$\wt{\G}$ occurring in Howe's correspondence with the pair $(\Og_2, \Sp_{2l'}(\R))$. For this dual pair, we compute the intertwining distributions in section
\ref{section:O2}. We will recover the (well-known) list of representations of 
$\wt{\Og}_2$ occurring in Howe's correspondence by their explicit formulas. See also Remark 
\ref{rem:classificationO2-Sp2l'-l'>1-triv-det}.

\begin{rem}
\label{reverse-p-q}
In this article we have considered the group $\Ug_{p,q}$ with $p\leq q$. Suppose now that $q\geq p$. This is equivalent to replacing the form $(\cdot,\cdot)'$ into its opposite. Correspondingly, the symplectic form $\langle\cdot,\cdot\rangle$ becomes its opposite. The inner product $-\langle J \cdot,\cdot\rangle$ is now positive definite provided we select $-J$ instead of $J$. In the notation at the beginning of section \ref{section:preliminaries},  the equation defining the 
preimages of $g\in\Sp(\Wv)$ in $\wt{\Sp}(\Wv)$ becomes
$$
\xi^2=i^{\dim(g-1)\Wv} \det(-J_g)^{-1}_{J_g\Wv}=(-i)^{\dim(g-1)\Wv} \det(J_g)^{-1}_{J_g\Wv}\,,
$$
because $(-1)^{\dim(J_g\Wv)}=(-1)^{\dim(g-1)\Wv}$. This means that $\xi$ is transformed into 
$\bar{\xi}$. Since $\Theta((g;\xi))=\xi$, we conclude that $\Theta$ needs to be changed into $
\overline{\Theta}$, i.e. the metaplectic representation $\omega$ is replaced by its contragredient 
$\omega^\vee$. Therefore 
$$\omega|_{\wt{\G}\times\wt{\G'}}=\bigoplus (\Pi\otimes \Pi')
\quad\text{is replaced by} \quad \omega^\vee|_{\wt{\G}\times\wt{\G'}}=\bigoplus (\Pi^\vee\otimes (\Pi')^\vee)\,.
$$
The highest weights of the representations of 
$\Ug_{l}$ occurring in $\omega^\vee$ are obtained from those listed for far in this paper by changing their sign and permuting them so that they are in decreasing order. Those written in 
\eqref{genuine-hw-Ul}, are replaced for $\Ug_{p,q}$, where $q\geq p$, with 
$$
\lambda_j=\frac{q-p}{2}+\nu_j, \quad \nu_j\in \Zb\,, \quad \nu_1 \geq \nu_2 \geq \cdots \geq \nu_l\,.
$$
\end{rem}

We conclude this section with a result on the non-differential operator nature of the symmetry breaking operators in $\Hom_{\wt\G\wt{\G'}}(\Hc_\omega^\infty, \Hc_\Pi^\infty\otimes \Hc_{\Pi'}^\infty)$.

\begin{cor}
\label{cor:SBO-no-diff}
Let $(\G,\G')$ be a real reductive dual pair with one member compact. Then the essentially unique non-zero symmetry breaking operator in 
$$\Hom_{\wt\G\wt{\G'}}(\Hc_\omega^\infty, \Hc_\Pi^\infty\otimes \Hc_{\Pi'}^\infty)$$ 
is not a differential operator. 
\end{cor}
\begin{proof}
We are going to show that $(\Op \circ \mathcal{K})(f_{\Pi\otimes\Pi'})$ is not a differential operator.
 
Let $f\in \Ss'(\Wv)$ and recall the definition of $\mathcal K(f)$ in \eqref{K}.
According to \cite[Theorems 5.2.1 (the Schwartz kernel theorem) and 5.2.3]{Hormander}, the continuous linear map $\Op\circ \mathcal K(f)$ is a distribution-valued differential operator if and only if $\mathcal K(f)\in \Ss'(\Xv\times\Xv)$ is supported by
the diagonal $\Delta=\{(x,x);\ x\in\Xv\}$. This implies that $f$ is supported in $\Yv$. Indeed, given $\varphi\in \Ss(\Xv\times \Xv)$, let $\psi\in \Ss(\Xv\times\Xv)$ be defined by $\varphi(x,x')=\psi(x-x',x+x')$ for all $x,x'\in \Xv$. Furthermore, 
let $\psi(\cdot,\widehat{\cdot})\in \Ss(\Xv\times \Yv)$ denote the partial Fourier transform of $\psi$ with respect to its second variable, defined by
$$
\psi(a,\widehat{y})=\int_\Xv \chi\big(\frac{1}{2}\langle y,b\rangle\big) \psi(a,b) \, db \qquad ((a,y)\in \Xv\times \Yv)\,.
$$
Then 
$$
\supp\varphi \cap \Delta=\emptyset 
\quad\text{if and only if}\quad
\supp\psi(\cdot,\widehat{\cdot})\cap (\{0\}\times \Yv)=\emptyset\,.
$$
Since $\mathcal{K}(f)(\varphi)=f(\psi(\cdot,\widehat{\cdot}))$ 
by \eqref{K}, we obtain the claim. 

Notice that this cannot happen in our case. Indeed, the support of $f_{\Pi\otimes\Pi'}$ is $\G\G'$-invariant. 
Since the complex structure $J\in \G'$ permutes $\Xv$ and $\Yv$, the only 
$\G\G'$-orbit in $\Yv$ is the zero orbit. Hence the inclusion $\supp f_{\Pi\otimes\Pi'}\subseteq \Yv$ would imply $\supp\,f_{\Pi\otimes\Pi'}=\{0\}$.  
This would mean that the wavefront set of $\Pi'$ is $0$, i.e. $\Pi'$ is finite dimensional. By classification, see Appendix \ref{appenE} all highest weight representations occurring in Howe's correspondence are infinite dimensional unless $\G'=\Ug_{l'}$, which is compact. 
In this case, the intertwining distribution is a smooth function; see \cite{McKeePasqualePrzebindaWCSymmetryBreakingUU}.
In particular, its support is not $0$. 
Hence the intertwining operator is not a differential operator.  
\end{proof}

%%%%%%%%%%%%%
\section{\bf The pair $(\Og_2,\Sp_{2l'}(\R))$}
\label{section:O2}

We consider here the case $(\G,\G')=(\Og_2,\Sp_{2l'}(\R))$. 
By \eqref{E4} and Proposition \ref{pro:det-covering}, we can identify 
$$
\wt\Og_2=\{(g;\zeta)\in \Og_2\times \C^\times; \zeta^2=
(\det g)^{l'} \}\,.
$$
and the $\det^{1/2}$-covering $\wt\Og_2\ni (g;\zeta) \to g\in\Og_2$ 
splits if and only if $l'$ is even.  
Let $\Pi\in \wt\Og_2$ occur in Howe's correspondence and 
let 
$\chi_+:\wt\Og_2\to \C^\times$ be the character of $\wt\Og_2$ defined by \eqref{chi+-on-O}. 

Since $\Pi$ is genuine, there is $\Pi_0\in \widehat{\Og_2}$ such that 
$\Pi_0(g)=(\Pi\otimes \chi_+^{-1})(\wt g)$.
Accordingly,
$$
\int_{\Og_2} \check{\Theta}_\Pi(\wt{g})\omega(\wt{g})\; dg=
\int_{\Og_2} \check{\Theta}_{\Pi_0}(g)\omega_0(g)\; dg\,,
$$
where $\omega_0$ is as in \eqref{omega0}.

Observe 
that  the image under the metaplectic cover of $\supp(\Theta_\Pi)$ is equal to $\supp(\Theta_{\Pi_0})$.
Since $\wt\SOg_2 \to \SOg_2$ splits by \eqref{E6}, we conclude that $\Theta_\Pi$ is supported in $\wt{\G^0}=\wt{\SOg_2}$ if and only if $\Theta_{\Pi_0}$ is supported in $\SO_2$.
In the sequel, $\triv$ denotes the trivial representation.

\begin{pro}
\label{O2Sp2l'R}
Let $(\G,\G')=(\Og_2,\Sp_{2l'}(\R))$ and let $\Pi$ be a genuine irreducible representation of 
$\wt{\G}$ with character $\Theta_\Pi$ not supported in $\wt{\G^0}$. 
Then either $\Pi=\wt{\triv}=\chi_+$, or 
$\Pi=\wt{\det}$ is the character of $\wt{\G}$ such that 
$(\wt{\det}\otimes \chi_+^{-1})(\t g)=\det(g)$ for all $\t g\in \wt{\G}$.

%Let $l'>1$.
Decompose $\Wv=\M_{2,2l'}(\R)$ as $\Wv=\Wv_1\oplus\Wv_2$, where 
$\Wv_1$ is subspace of the $w\in\Wv$ for which all entries of the second row are $0$ and 
$\Wv_2$ is subspace of the $w\in\Wv$ for which all entries  of the first row are $0$.
Then
\begin{equation}
\label{det-l'>1 outside}
\int_{(\SO_2)s} \chi_+^{-1}(\wt{g})T(\wt{g})(\phi) \, dg=\mu_\mathcal{O}(\phi)\,,
\end{equation}
where $s$ is as in \eqref{s},
$\mathcal{O}$ is the $\Og_2\times \Sp_{2l'}(\R)$-orbit of $n_0=\begin{pmatrix}
0 & 0 & \dots & 0\\ 1 & 0 & \dots & 0
\end{pmatrix}\in \Wv$ and 
$\mu_\mathcal{O}\in \mathcal{S}'(\Wv)$ is the invariant measure on $\mathcal{O}$
defined by 
\begin{equation}
\label{muO}
\mu_\mathcal{O}(\phi)=2^{l'-1} \int_{\Wv_2}\int_{\Og_2} \phi(gw) \; dg \,d\mu_{\Wv_2}(w)
\qquad  (\phi\in \mathcal{S}(\Wv))\,.
\end{equation}
Therefore
\begin{equation}
\label{det-l'>1}
\int_{\Og_2} \check\Theta_{\wt{\det}}(\wt{g}) T(\wt{g})(\phi) \,dg=\int_{\SOg_2} \chi_+^{-1}(\wt{g})T(\wt{g})(\phi) \, dg-
\mu_\mathcal{O}(\phi)  \qquad  (\phi\in \mathcal{S}(\Wv))\,
\end{equation}
and 
\begin{equation}
	\label{det-l'>1, triv}
\int_{\Og_2} \check\Theta_{\wt{\triv}}(\wt{g}) T(\wt{g})(\phi) \,dg=\int_{\SOg_2} \chi_+^{-1}(\wt{g})T(\wt{g})(\phi) \, dg+
	\mu_\mathcal{O}(\phi)  \qquad  (\phi\in \mathcal{S}(\Wv))\,.
\end{equation}
The integral over $\SOg_2$ is computed by Theorem \ref{main thm for l<l'}.

If  $l'=1$,  then $\wt{\det}$ does not occur in Howe correspondence and
hence
$$
\int_{\Og_2} \check\Theta_{\wt{\det}}(\t g)T(\t g)\, dg=0\,.
$$
Moreover,
\begin{equation}
\label{1 for O2}
\int_{\Og_2} \check{\Theta}_{\wt{\triv}}(\wt{g})T(\wt{g})\; dg= 2 \int_{\SOg_2} \chi_+^{-1}(\wt{g})T(\wt{g}) \; dg=
2\mu_{\mathcal{O}}\,.
\end{equation}
%%
%is computed by Theorem \ref{main thm for l<l'}; see also subsection \ref{example:O2Sp2}.  
\end{pro}
%%%
\begin{proof}
For $n\in \Zb$, let $\rho_n$ be the character of $\SOg_2$ defined by 

\[
\rho_n(
\left(
\begin{array}{lll}
\cos \theta & \sin\theta\\
-\sin\theta & \cos\theta
\end{array}
\right))=e^{in\theta}\,.
\]

Up to equivalence, the irreducible representations of $\Og_2$ are of the form $\Pi_{0,n}=\Ind_{\SOg_2}^{\Og_2}(\rho_n)$ with $n>0$, together with the trivial representation $\triv$ and $\det$. (Moreover, 
$\Pi_{0,n}\simeq \Pi_{0,-n}$ and $\Pi_{0,0}=1\oplus \det$.) Hence $\Theta_{\Pi_0}$ does not have support contained in $\wt{\SO_2}$ if and only if ${\Pi_0}|_{\SOg_2}=1$. 
Hence the only possible cases are $\triv$ and $\det$.

%%%%%%%%%%%
Since
\begin{align*}
\check{\Theta}_{\wt{\triv}}(\t g)T(\t g)&=\chi_+^{-1}(\t g)T(\t g)\,,\\
\check{\Theta}_{\wt{\det}}(\t g)T(\t g)&=
\check{\Theta}_{\wt{\det}}(\t g)\chi_+(\t g)\chi_+^{-1}(\t g)T(\t g)=
\det(g)\chi_+^{-1}(\t g)T(\t g)\,,
\end{align*}
we see that
\begin{align*}
\int_{\Og_2} \check{\Theta}_{\wt{\triv}}(\t g)T(\t g)\, dg&
=\int_{\SO_2} \chi_+^{-1}(\t g)T(\t g) \, dg
+\int_{(\SO_2)s} \chi_+^{-1}(\t g)T(\t g) \, dg\,,\\
\int_{\Og_2} \check{\Theta}_{\wt{\det}}(\t g)T(\t g)\, dg&=\int_{\SO_2} \chi_+^{-1}(\t g)T(\t g) \, dg
-\int_{(\SO_2)s} \chi_+^{-1}(\t g)T(\t g) \, dg\,.
\end{align*}

We now compute the integral over $(\SO_2)s$. Let $g_t=\begin{pmatrix}
\cos(t) & \sin(t) \\ -\sin(t) & \cos(t)
\end{pmatrix}\in \SO_2$ and recall from \eqref{s} that  
  $s=\begin{pmatrix}
	1 & 0 \\ 0 &-1
\end{pmatrix}\in \Og_2\setminus \SO_2$. 
Then $g_{t}s=g_{t/2}sg_{-t/2}$. If $f$ is any function on $(\SO_2)s$, then 
\begin{align*}
\int_{(\SO_2)s} f(g)\, dg&=
\int_{\SO_2} f(g_ts) dg_t
=\frac{1}{2\pi}\int_0^{2\pi} f(g_{t}s)\, dt=\frac{1}{2\pi}\int_0^{2\pi} f(g_{t/2}sg_{-t/2})\, dt\\
&=\frac{1}{2\pi}\int_0^{\pi} f(g_{t}sg_{-t})\cdot 2dt=\frac{1}{2\pi}\int_0^{\pi} f(g_{t}sg_{-t})\, dt+\frac{1}{2\pi}\int_\pi^{2\pi} f(g_{t}sg_{-t})\, dt\\
&=\frac{1}{2\pi}\int_0^{2\pi} f(g_{t}sg_{-t})\, dt
=\int_{\SO_2} f(g_{-t}sg_{t})\, dg_t\,.
\end{align*}
Applying this to $\SO_2\ni g\to \chi_+^{-1}(\wt{g})T(\wt{g})\in \mathcal{S}'(\R)$, we get
\begin{equation}
\label{intSO2s-1}
\int_{(\SO_2)s} \chi_+^{-1}(\wt{g})T(\wt{g}) \, dg
=
\int_{\SO_2} \chi_+^{-1}(\wt{g^{-1}sg})T(\wt{g^{-1}sg}) \, dg
\,.
\end{equation}
Decompose $\Wv=\M_{2,2l'}(\R)$ as in the statement of the theorem and let $g\in \Og_2$. 
Then $\Wv=g^{-1}\Wv_1\oplus g^{-1}\Wv_2$ is an orthogonal decomposition such that
$g^{-1}sg$ preserves both $g^{-1}\Wv_1$ and $g^{-1}\Wv_2$. Notice that 
\begin{align*}
g^{-1}sg|_{g^{-1}\Wv_1}&=1_{g^{-1}\Wv_1} \quad \text{because $s|_{\Wv_1}=1$}\,,\\
g^{-1}sg|_{g^{-1}\Wv_2}&=-1_{g^{-1}\Wv_2} \quad \text{because $s|_{\Wv_2}=-1$}\,.
\end{align*}
By Lemma \ref{lemma:WandT}, 
\begin{equation}
\label{decompositionTWforg1}
\chi_+^{-1}(\wt{g^{-1}sg})T_\Wv(\wt{g^{-1}sg})=
\chi_+^{-1}(\wt{1_{g^{-1}\Wv_1}})T_\Wv(\wt{1_{g^{-1}\Wv_1}})\otimes 
\chi_+^{-1}(\wt{-1_{g^{-1}\Wv_2}})T_\Wv(\wt{-1_{g^{-1}\Wv_2}})\,,
\end{equation}
independently of the choices of the preimages of $g^{-1}sg$, $1_{g^{-1}\Wv_1}$ and $-1_{g^{-1}\Wv_2}$ in $\wt{\Sp}(\Wv)$, $\wt{\Sp}(g^{-1}\Wv_1)$ and $\wt{\Sp}(g^{-1}\Wv_2)$, respectively. 
We can therefore fix $\wt{1_{g^{-1}\Wv_1}}$ to be the identity element of $\wt{\Sp}(g^{-1}\Wv_1)$, which 
gives $\chi_+^{-1}(\wt{1_{g^{-1}\Wv_1}})=1$. Hence 
$$
\chi_+^{-1}(\wt{1_{g^{-1}\Wv_1}})T_\Wv(\wt{1_{g^{-1}\Wv_1}})=\delta_{0,g^{-1}\Wv_1}\,,
$$
where $\delta_{0,g^{-1}\Wv_1}$ indicates Dirac's delta at $0$ in the space $g^{-1}\Wv_1$.

By \cite[Definition 4.16 and Remark 4.5]{AubertPrzebinda_omega}, $\Theta_\Wv^2(-1)=(-2i)^{\dim \Wv}$. Hence $|\Theta_{\Wv}(\wt{-1})|=2^{\dim \Wv/2}$ only depends on the dimension of $\Wv$. 
In particular,
$$
|\Theta_{g^{-1}\Wv_2}(\wt{-1})|=|\Theta_{\Wv_2}(\wt{-1})|=2^{\dim \Wv_2/2}\,.
$$
So 
$$
\chi_+^{-1}(\wt{-1_{g^{-1}\Wv_2}})T_\Wv(\wt{-1_{g^{-1}\Wv_2}})= |\Theta_{g^{-1}\Wv_2}(\wt{-1})|
\mu_{g^{-1}\Wv_2}=2^{\dim \Wv_2/2}\mu_{g^{-1}\Wv_2}\,.
$$
Thus \eqref{decompositionTWforg1} becomes
\begin{equation}
\label{decompositionTWforg2}
\chi_+^{-1}(\wt{g^{-1}sg})T_\Wv(\wt{g^{-1}sg})=2^{\dim \Wv_2/2} \delta_{0,g^{-1}\Wv_1} \otimes
\mu_{g^{-1}\Wv_2}\,.
\end{equation}
By \eqref{intSO2s-1}, for all $\phi\in \mathcal{S}(\Wv)$,  
\begin{align*}
\int_{(\SO_2)s} \chi_+^{-1}(\wt{g})T(\wt{g})(\phi) \, dg
&=2^{\dim \Wv_2/2} \int_{\SO_2} (\delta_{0,g^{-1}\Wv_1} \otimes
\mu_{g^{-1}\Wv_2})(\phi) \,dg\\
&= 2^{\dim \Wv_2/2} \int_{\SO_2} \int_{g^{-1}\Wv_2}\phi(w) \; 
d\mu_{g^{-1}\Wv_2}(w) \,dg\\
&=2^{\dim \Wv_2/2} \int_{\Wv_2}\int_{\SO_2} \phi(gw) \; dg \,d\mu_{\Wv_2}(w) \,.
\end{align*}
Notice that, since $sw=-w$ for $w\in \Wv_2$,
\begin{align*}
\int_{\Wv_2}\int_{\SO_2} \phi(gw) \; dg \,d\mu_{\Wv_2}(w)&=
\int_{\Wv_2}\int_{\SO_2} \phi(-gw) \; dg \,d\mu_{\Wv_2}(w)\\
&=\int_{\Wv_2}\int_{\SO_2} \phi(gsw) \; dg \,d\mu_{\Wv_2}(w)\\
&=\int_{\Wv_2}\int_{(\SO_2)s} \phi(gw) \; dg \,d\mu_{\Wv_2}(w)\,.
\end{align*}
Hence,
\begin{align*}
&\int_{(\SO_2)s} \chi_+^{-1}(\wt{g})T(\wt{g})(\phi) \, dg\\
&=2^{\dim \Wv_2/2} 
\Big( \frac{1}{2} \int_{\Wv_2}\int_{\SO_2} \phi(gw) \; dg \,d\mu_{\Wv_2}(w) +  \frac{1}{2} \int_{\Wv_2}\int_{(\SO_2)s} \phi(gw) \; dg \,d\mu_{\Wv_2}(w)  \Big)\\
&=2^{\dim \Wv_2/2 -1}
 \int_{\Wv_2}\int_{\Og_2} \phi(gw) \; dg \,d\mu_{\Wv_2}(w) \,.
\end{align*}
In conclusion, 
$$
\int_{(\SO_2)s} \chi_+^{-1}(\wt{g})T(\wt{g})(\phi) \, dg=  \mu_\mathcal{O}(\phi) \qquad (\phi\in \mathcal{S}(\Wv))\,,
$$
where $\mu_\mathcal{O}$ is as in \eqref{muO}. 

We now show that $\mu_\mathcal{O}$ is a $\Og_2\times\Sp_{2l'}(\R)$-invariant measure on the orbit 
$\mathcal{O}$. 
Notice first that $\Wv_2\setminus \{0\}=\Sp_{2l'}(\R).n_0$. Indeed, $n_0\in \Wv_2$
and $\Sp_{2l'}(\R)$ preserves $\Wv_2$. Conversely, let $w_2=\begin{pmatrix}
0 & 0 \\ u & v
\end{pmatrix} \in \Wv_2\setminus\{0\}$, where $u,v\in \M_{1,l'}(\R)$. Since $J=\begin{pmatrix}
0 & I_{l'} \\ -I_{l'} & 0
\end{pmatrix}\in \Sp_{2l'}(\R)$ and $w_2J=\begin{pmatrix}
0 & 0 \\ -v & u
\end{pmatrix}$, we can suppose that $u\neq 0$. If $a\in \GL_{l'}(\R)$ has $u$ as its first row and $b$ is a symmetric matrix having $v$ as its first row, then 
$\begin{pmatrix}
a & b \\ 0 & (a^t)^{-1}
\end{pmatrix}\in \Sp_{2l'}(\R)$ and $n_0 \begin{pmatrix}
a & b \\ 0 & (a^t)^{-1}
\end{pmatrix} =w_2$.
It follows from this that $\{gw_2; g\in \Og_2, w_2\in \Wv_2\}=\mathcal{O}\cup \{0\}$.
The right-hand side of $\eqref{muO}$ is clearly $\Og_2$-invariant, and we see that it is $\Sp_{2l'}(\R)$-invariant by linear changes of variables in the integral over $\Wv_2$ because the elements of 
$\Sp_{2l'}(\R)$ have determinant 1.

Let  $l'=1$. By Proposition \ref{non-occurence-det}, $\wt{\det}$ does not occur in Howe correspondence. Let $\Pi={\wt{\triv}}$. Since $\wt{\det}$ does not occur, the projection onto the $\Og_2$-isotypic component is equal to the projection onto the-$\SOg_2$ isotypic component. Therefore, \eqref{1 for O2} follows,
because the volume of $\SOg_2$ is $\frac{1}{2}$.

Since \eqref{det-l'>1} vanishes when $l'=1$, we have 
$$
\int_{\Og_2} \check{\Theta}_{\wt{\triv}}(\t g)T(\t g)\, dg= 2\int_{\SO_2} \chi_+^{-1}(\t g)T(\t g) \, dg=2\mu_{\mathcal{O}}\,.
$$
\end{proof}

\begin{rem} \label{rem:classificationO2-Sp2l'-l'>1-triv-det}
Formulas \eqref{det-l'>1} and \eqref{det-l'>1, triv} show that $\wt{\det}$ and $\wt{\triv}$ occur in the Howe correspondence when $l'>1$. 
This is compatible with the classification, as for $l'>1$ the 
the dual pair $(\Og_2,\Sp_{2l'}(\R))$ is in the stable range, so all genuine representations occur.
\end{rem}

\subsection{\bf The special case $(\G,\G')=(\Og_2,\Sp_2(\R)=\SL_2(\R))$}
\label{example:O2Sp2}

In this case, $\H=\SO_2$ and 
$\g=\h=\R J_1$, where $J_1=\R\begin{pmatrix}
0 & 1\\ -1 & 0
\end{pmatrix}$. Moreover, $\tau(\hs1)=\R^+ J_1$ and $\h\cap \tau(\Wv)=\h$. The Harish-Chandra parameter of 
$\Pi\in \wt{\Og}_2^\wedge$ (which coincides with its highest weight since $\rho=0$) 
is of the form $\mu e_1$, where $\mu\geq 0$ is an integer. Hence, in the notation \eqref{eq:ajbj}, $a=-b=-\mu$ and $\beta=2\pi$. 

If $\mu=0$, then $P_{-\mu,\mu}=0$.
If $\mu>0$, then the function $P_{-\mu,\mu}$ is supported in $[0,+\infty)$ and, 
by \eqref{D0'} and Remark \ref{Pab-Laguerre}, 
\begin{equation}
\label{P-O2}
P_{-\mu,\mu,2}(2\pi y_1)=2 (-1)^{\mu-1} L^1_{\mu-1}(4\pi y_1)=2 (-1)^{\mu-1} 
\sum_{h=0}^{\mu-1}
{\mu\choose\mu-1-h} \frac{(-4\pi y_1)^h}{h!}\,,
\end{equation}
where $L^1_{\mu-1}$ is a Laguerre polynomial. Moreover, by \eqref{D0''},  
$Q_{-\mu,\mu}(y)=2\pi (-1)^\mu$ for all $\mu\geq 0$.

Suppose first $\mu>0$. Then $\Pi$ is supported in $\wt{\SO_2}$ and,
by Lemma \ref{main theorem summary} and Theorem \ref{main thm for l<l'},
for every $\phi\in \Ss(\Wv)$,
\begin{align}
f_{\Pi\otimes \Pi'}(\phi)&=\int_{\SO_2} \check{\Theta}_{\Pi}(\wt{g})T(\wt{g})(\phi)\; dg \notag\\
\label{fPIPi'-O2-general}
&=2\pi C (-1)^\mu \int_0^{+\infty} P_{-\mu,\mu,2}(2\pi y_1) e^{-2\pi y_1} F_\phi(y_1 J_1)\, dy_1
+  C  \int_\h \delta_0(y) F_\phi(y) \; dy\,,
\end{align}
where $C$ is the constant appearing in Theorem \ref{main thm for l<l'}.
To make formula \eqref{fPIPi'-O2-general} explicit, we need to calculate the terms involving $F(y)$,
the Harish-Chandra regular almost-elliptic orbital integral on $\Wv$.

By \cite[Definition 3.1, (39) and (27)]{McKeePasqualePrzebindaWCestimates} 
and \eqref{integralonS/Sh1-2} with $\Zg'=\H'$, there are constants $C_{\hs1}$ and $C'_{\hs1}$
such that,
for all $y=y_1J_1=
\tau(w)\in \tau(\hs1)$,
\begin{equation}
\label{orbital-integrals-SL2R}
F_\phi(y)=C_{\hs1} \pi_{\g'/\h'}(y') \int_{\Sg/\Sg^{\hs1}} \phi(s.w) \, d(s\Sg^\hs1)
=C'_{\hs1} \pi_{\g'/\h'}(y') \int_{\G'/\H'} \psi(g'. y') \, d(g'\H')\,, 
\end{equation}
where $y'=y_1J_1'=y_1 \begin{pmatrix}
0 & 1\\ -1 & 0
\end{pmatrix}=\tau'(w)$, and 
$\psi=\tau'_*(\phi^\G)\in \Ss(\g')$. 
The right-hand side of \eqref{orbital-integrals-SL2R} is Harish-Chandra's orbital integral for the
orbit $\G'.y'$.

Notice that, for $\G=\Og_2$ and $l=1\leq l'$, the extension of $F(y)$ from $y\in\h^+=\tau(\hs1)$ to $-\tau(\hs1)$ is even in $y$; see \cite[Theorem 3.6]{McKeePasqualePrzebindaWCestimates}. 
Hence, 
\begin{equation*}
\int_\h \delta_0(y) F_\phi(y) \; dy=
\lim_{y_1\rightarrow 0+} 
%\lim_{\substack{y\in\tau(\hs1)\\y\rightarrow 0}} 
F_\phi(y_1 J_1) \qquad (\phi\in \Ss(\Wv))\,.
\end{equation*}
Write $x\in \g'$ as 
$$
x=x_1 \begin{pmatrix}
1 & 0 \\ 0 & -1
\end{pmatrix} + x_2 \begin{pmatrix}
0 & 1 \\ 1 & 0\end{pmatrix}  + x_3 J'_1=
\begin{pmatrix}
x_1 & x_2 +x_3 \\ x_2-x_3 & -x_1 
\end{pmatrix}=A(x_1,x_2,x_3)\,,
$$
where $(x_1,x_2,x_3) \in \R^3$. Then the map $A:\R^3\to \g'$ is a linear isomorphism. It transfers the 
adjoint action of $\G'$ on $\g'$ to the natural action on $\R^3$ by $\SO(2,1)^0$, the identity component of $\SO(2,1)$, i.e. the group of isometries of $x_1^2+x_2^2-x_3^2=-\det(A(x_1,x_2,x_3))$ preserving the 
positive light cone 
$$
X^{0+}=\{(x_1,x_2,x_3)\in \R^3; x_1^2+x_2^2=x_3^2, x_3>0\}\,.
$$
See \cite[Chapter IV, \S 5.1]{HoweTan}.
Under the map $A$, the orbit $\G'.y'$ with $y'=y_1 J'_1$ and $y_1>0$ is the image of the hyperboloid's upper sheet  
$$
O_{y_1}^-=\{(x_1,x_2,x_3)\in \R^3; x_1^2+x_2^2-x_3^2=-y_1^2, x_3>0\}\,.
$$
Under $A$, the positive light cone $X^{0+}$ corresponds to the $\G'$-orbit of 
$x_0=\begin{pmatrix}
0 & 1\\ 0 & 0
\end{pmatrix}$. Moreover $\G'. x_0\simeq \G'/\M\N$, where $\M=\{\pm 1\}$ and $\N=\exp(\R x_0)=
\Big\{ \begin{pmatrix}
1 & t\\ 0 & 1
\end{pmatrix}; t\in \R \Big\}$.
As the geometry suggests, for suitable normalizations of the $\SO(2,1)^0$-invariant orbital measures, 
$$
\lim_{y_1\to 0^+} \int_{O_{y_1}^-} f \,d\mu_{O_{y_1}^-}= \int_{X^{0+}} f \,d\mu_{X^{0+}} \qquad 
(f\in \Ss(\R^3))\,.
$$
Thus, for a suitable positive constant $C''_{\hs1}$
\begin{equation}
\label{F0-O2}
\int_\h \delta_0(y) F_\phi(y) \; dy=C''_{\hs1} \int_{\G'/\M\N} \psi(g'.x_0) \; d(g'\M\N) 
\qquad (\phi\in \Ss(\Wv), \psi\in \Ss(\g')^\G, \psi \circ \tau'= \phi^\G)\,.
\end{equation}

Suppose now that $\mu=0$.
Then, by Proposition \ref{O2Sp2l'R}, $\Pi=\widetilde{\triv}=\chi_+$ and 
for $\phi\in\Ss(\Wv)$,
$$
f_{\wt{\triv}\otimes \wt{\triv}'}(\phi)=2\int_{\SO_2} \chi_+^{-1}(\wt{g})T(\wt{g})(\phi)\, dg
=2C \int_\h \delta_0(y) F_\phi(y) \; dy\,,
$$
where $\wt{\triv}'$ denotes the representation of $\wt{\Sp}_2(\R)$ in Howe correspondence with
$\wt{\triv}$ and the last equality follows from Theorem \ref{main thm for l<l'}.

\smallskip
%%%%%%%%%%%%%%%
%%%%%%%%%%%%%%%

%%
\section{\bf Another example: $(\G,\G')=(\Ug_l,\Ug_{p,p})$ and $\Pi=\wt{\triv}$}
\label{examples}

Let $(\G,\G')=(\Ug_l,\Ug_{p,p})$. Hence $l'=2p$. Consider the trivial representation $\triv$ of $\Ug_l$. In the Schr\"odinger model, with a polarization $\Wv=\Xv\oplus \Yv$ preserved by $\G$, we have
\begin{equation}
\label{omegachiplus}
\omega(\wt g)v(x)=\chi_+(\wt g)v(g^{-1}x) \qquad (\wt g\in \wt\G,\, v\in \Se(\Xv),\, x\in \Xv) \,,
\end{equation}
where $\chi_+:\wt{\Sp}(\Wv)\to \Ug_1$ is a function whose restriction to $\wt \G$ is a character. See 
\cite[Proposition 4.28]{AubertPrzebinda_omega}. Let $\wt{\triv}$ denote this restriction. Then 
$\wt{\triv}$ is the lift to $\wt{\Ug_l}$ of $\triv$, which occurs in Howe's correspondence. Moreover, 
\eqref{omegachiplus} implies that 
$$
\omega(\check{\Theta}_{\wt \triv})v(x)=\int_\G v(g^{-1}x) \, dg \qquad (v\in \Se(\Xv), \, x\in \Xv)\,.
$$
Let $\wt{\triv}'$ be the representation of $\wt{\Ug}_{p,p}$ which corresponds to $\wt{\triv}$.
If $l=1$, then $\wt{\triv}'$ is a minimal representation of $\Ug_{p,p}$, called the Wallach representation. 

In this section we are computing $f_{\wt{\triv}\otimes\wt{\triv}'}$, which is the Weyl symbol of the operator $\omega(\check{\Theta}_{\wt \triv})$. As in our main theorems, we distinguish the cases 
$l\leq l'$ and $l>l'$. 
Notice first that the parameters appearing in \eqref{eq:delta,beta} are 
$$
\beta=2\pi 
\qquad \text{and} \qquad \delta=p+\frac{1-l}{2}=\frac{1+l'-l}{2}\,.
$$
Moreover, 
$
\rho=\sum_{j=1}^l \big( \frac{l+1}{2} -j\big) e_j
$
for $\G=\Ug_l$. 

\subsubsection{\bf The case $l\leq l'$}
The parameters \eqref{eq:ajbj} corresponding to $\Pi=\wt \triv$ are 
\begin{equation} 
\label{ajbjUlUpp}
a_j=-\frac{l'}{2}+j \qquad \text{and} \qquad b_j=-\frac{l'}{2}+l+1-j\,,
\end{equation}
where $1\leq j\leq l$. Observe that the $a_j$'s and the $b_j$'s describe the same set 
$$\{-l'/2+1,\dots, -l'/2+l-1,-l'/2+l\}$$
and 
$b_{l+1-j}=a_j$ for all $1\leq j\leq l$.
Hence, by \eqref{absymmetryPabQab},
\begin{eqnarray}
\label{symmPab1}
P_{a_{l+1-j},b_{l+1-j}}(\xi)=P_{b_j,a_j}(\xi)=P_{a_j,b_j}(-\xi)\,,\\
\label{symmQab1}
Q_{a_{l+1-j},b_{l+1-j}}(\xi)=Q_{b_j,a_j}(\xi)=Q_{a_j,b_j}(-\xi)\,.
\end{eqnarray}
Since $a_j=b_{l+1-j}\leq 0$ for all $1\leq j\leq\min(l, l'/2)$, by \eqref{eq:Pabminus2},
\begin{equation}
\label{Pab2=0}
P_{a_j,b_j,-2}(\xi)=%%P_{b_{l+1-j},a_{l+1-j},-2}(\xi)
P_{a_{l+1-j},b_{l+1-j},2}(\xi)=0  \qquad (1\leq j\leq \min(l, l'/2))\,.
\end{equation}
Also, $a_j\leq 0$ for all $j$ (and hence $b_j\leq 0$ for all $j$) if and only if $l\leq l'/2$. 
Furthermore, $a_j+b_j=l-l'+1$, which is independent of $j$, is $\geq 1$ if and only if $l=l'$.
As a consequence (see \eqref{D0''}), 
\begin{eqnarray*}
&&\text{$P_{a_j,b_j}=0$ for all $1\leq j\leq l$ \quad if and only if \quad $l\leq \frac{l'}{2}$}\,,\\
&&\text{$Q_{a_j,b_j}\neq 0$ for all $1\leq j\leq l$ \quad if $l<l'$}\,,\\
&&\text{$Q_{a_j,b_j}=0$ for all $1\leq j\leq l$ \quad if $l=l'$}\,.
\end{eqnarray*}

We now examine more precisely the formula for $f_{\wt{\triv}\otimes\wt{\triv}'}$ when
$l\leq l'/2$. This is the stable range case. As remarked above, 
$P_{a_j,b_j}=0$ for all $1\leq j\leq l$, whereas (see \eqref{D0''})
$$
Q_{a_j,b_j}(y_j)=2\pi (1+y_j)^{-a_j}(1-y_j)^{-b_j}\,.
$$
Hence $p_j=0$ for all $1\leq j\leq l$, whereas
$$
q_j(-\partial_{y_j})^*=q_j(\partial_{y_j})=\Big(1+\frac{1}{2\pi}\partial_{y_j}\Big)^{\frac{l'}{2}-j}
\Big(1-\frac{1}{2\pi}\partial_{y_j}\Big)^{\frac{l'}{2}-(l-j+1)}\,,
$$
where $\null^*$ denotes the formal adjoint.
Theorem \ref{main thm for l<l'} yields for $\phi\in \Se(\Wv)$ 
\begin{eqnarray}
f_{\wt{\triv}\otimes\wt{\triv}'}(\phi)&=&\int_{\Ug_l} \check{\Theta}_{\wt \triv}(\wt g)T(\wt g)(\phi) \, dg \nn\\
&=& C \int_\h \Big[\prod_{l=1}^l q_j(-\partial_{y_j}) \delta_0(y_j) \Big] F_\phi(y)\, dy \nn\\ 
&=& C \Big[\Big(\prod_{l=1}^l q_j(\partial_{y_j})\Big) F_\phi\Big] (0)\,,
\end{eqnarray}
where $C$ is a nonzero constant. Hence
$f_{\wt{\triv}\otimes\wt{\triv}'}$ has support inside the nilpotent cone in $\Wv$.

Another case where the formula for $f_{\wt{\triv}\otimes\wt{\triv}'}$ simplifies is when $l=l'=2p$ because $Q_{a_j,b_j}=0$ for all $j$. Since $a_j=b_{2p+1-j}\leq 0$ for $1\leq j \leq p$, we have 
$$
P_{a_j,b_j}(\xi)=
\begin{cases}
2\pi  P_{a_j,b_j,2}(\xi) \mathbb{I}_{\R^+}(\xi) &\text{if $1\leq j\leq p$}\,,\\
2\pi  P_{a_j,b_j,-2}(\xi) \mathbb{I}_{\R^-}(\xi) &\text{if $p+1\leq j\leq 2p$}\,.
\end{cases}
$$
In particular, in this case, we can replace in \eqref{main thm for l<l' a} 
the domain of integration $\h\cap \tau(\Wv)$ with $\tau(\hs1)$, where $\hs1$ is the unique 
Cartan subspace of $\Wv$ and $\tau(\hs1)$ is determined by the condition that the first $p$ values $\delta_j$ in \eqref{deltaj} are equal to 1 and the last $p$ are equal to $-1$. The explicit expression for 
$f_{\wt{\triv}\otimes\wt{\triv}'}$ can be easily computed using \eqref{main thm for l<l' a}, 
\eqref{eq:Pab2} and \eqref{eq:Pabminus2}. For instance, if $p=1$, i.e. $(\G,\G')=(\Ug_2,\Ug_{1,1})$, then 
$$
f_{\wt{\triv}\otimes\wt{\triv}'}(\phi)=C \int_0^\infty\!\! \int_{-\infty}^0 e^{2\pi(y_2-y_1)} F_\phi(y_1,y_2) \; dy_2 dy_1 \qquad (\phi\in \Ss(\Wv))\,,
$$
where $C$ is a nonzero constant.

\subsubsection{\bf The case $l>l'$} In this case, $Q_{a_j,b_j}=0$.
The Weyl group $W(\Ug_{p,p},\h')$ acts on $\h'$ by permuting the first $p$ coordinates and the last $p$ coordinates (see Remark \ref{rem:Weyl-invariant l>l'}). The parameter $a_{s,j}$ and $b_{s,j}$
appearing in \eqref{main thm for l>=l' c} are therefore
obtained by separately permuting the first $p=l'/2$ and the last $p$ terms appearing in  
\eqref{ajbjUlUpp}. 
Notice that 
\begin{eqnarray*}
&&\text{$a_j\leq 0$  \quad if and only if \quad $1\leq j \leq \tfrac{l'}{2}$}\,,\\
&&
\text{$b_j\leq 0$  \quad if and only if $l+1-\tfrac{l'}{2} \leq j \leq l$}\,.
\end{eqnarray*}
In particular, since $l>l'$, for each $j$, at most one between $a_j$ and $b_j$ can be $\leq 0$. Moreover, there is at least one index $j$ for which both $a_j$ and $b_j$ are positive, namely
$j=\tfrac{l'}{2}+1$.

When $\G'=\Ug_{1,1}$ (and hence $l'=2$), then $W(\Ug_{1,1},\h')$ is trivial
and $s_0$ maps $J_1$ to itself and $J_l$ to $J_2$,
 and 
\eqref{main thm for l>=l' c} simplifies to a nonzero constant multiple of 
\begin{eqnarray*}
\label{main thm for l>=l',U11} 
\frac{P_{a_1,b_1,2} (2\pi y_1) P_{a_l,b_l,-2} (2\pi y_2)}{(y_2-y_1)(y_1y_2)^{l-2}} e^{-2\pi (y_1-y_2)} \qquad (y=\tau'(w), w\in \reg{\hs1})\,,
\end{eqnarray*}
where $a_j, b_j$ are as in \eqref{ajbjUlUpp} and the denominator is the root product \eqref{product of positive roots for g'/z'}.

\section{\bf The integral over $-\G^0$ as an integral over $\g$\rm}
\label{Intertwining distributions as an integral over g}

Let $\sp(\Wv)$ be the Lie algebra of $\Sp(\Wv)$. Set
\begin{eqnarray}
\label{spc}
&&\sp(\Wv)^c=\{x\in\sp(\Wv);\ x-1\ \text{is invertible in $\End(\Wv)$}\}\,,\\
\label{Spc}
&&\Sp(\Wv)^c=\{g\in\Sp(\Wv);\ g-1\ \text{is invertible in $\End(\Wv)$}\}\,.
\end{eqnarray}
The Cayley transform $c:\sp(\Wv)^c  \to \Sp(\Wv)^c$ is the bijective rational map defined by $c(x)=(x+1)(x-1)^{-1}$. Its inverse $c^{-1}:\Sp(\Wv)^c\to \sp(\Wv)^c$ is given by the same formula, 
$c^{-1}(g)=(g+1)(g-1)^{-1}$.

Since  all eigenvalues of $x \in \g\subseteq \End(\Wv)$ are purely imaginary, 
$x-1$ is invertible. Therefore $\g\subseteq \sp(\Wv)^c$. 
Moreover, $c(\g)\subseteq \G$.
Since the map $c$ is continuous, the range $c(\g)$ is connected. Also, $-1=c(0)$ is in $c(\g)$. Furthermore, for $x\in\g$,
\[
c(x)-1=(x+1)(x-1)^{-1}-(x-1)(x-1)^{-1}=2(x-1)^{-1}
\]
is invertible. Hence $c(\g)\subseteq \G\cap \Sp(\Wv)^c$. 
This is an equality because $c(c(y))=y$ and $c(\G)\subseteq\g$. Thus
\[
c(\g)=\{g\in\G;\ \det(g-1)\ne 0\}\,.
\]
This is a connected open dense subset of $-\G^0$.
Hence 
\begin{equation}\label{T1cg}
\int_{-\G^0}T(\t g)\check\Theta_\Pi(\t g)\, dg=\int_{c(\g)}T(\t g)\check\Theta_\Pi(\t g)\, dg\,.
\end{equation}
If $\G\ne\Og_{2 l+1}$, then $\G^0=-\G^0$. 
If $\G=\Og_{2 l+1}$, then $\G$ is the disjoint union of $\G^0$ and $-\G^0$. 
Let 
\begin{equation}\label{eq:tildec}
\t c:\g\to\wt\G
\end{equation}
be a real analytic lift  of $c$. Set $\t c_-(x)=\t c(x) \t c(0)^{-1}$. 
Then $\t c_-(0)$ is the identity of the group $\wt\Sp(\Wv)$.
By \eqref{Tt}, we have 
\begin{equation}\label{the omega}
T(\t c(x))=\Theta(\t c(x))\,\chi_{x}\,\mu_\Wv\,.
\end{equation}
Therefore, for a suitable normalization of the Lebesgue measure on $\g$,
\begin{equation}\label{T3}
\int_{-\G^0}\check\Theta_\Pi(\t g) T(\t g)\,dg=\int_{\g}\check\Theta_\Pi(\t c(x))\, \Theta(\t c(x))\, j_\g(x)\,\chi_x\,\mu_\Wv\,dx\,,
\end{equation}
where $j_\g(x)$ is the Jacobian of the map $c:\g\to c(\g)$ (see Appendix \ref{appenA} for its explicit expression). Also, since $\t c(0)$ is in the center of the metaplectic group, 
\begin{equation}\label{T4}
\int_{-\G^0}\check\Theta_\Pi(\t g) T(\t g)\,dg=\check\chi_\Pi(\t c(0))\int_{\g}\check\Theta_\Pi(\t c_-(x))\,\Theta(\t c(x))\,j_\g(x)\,\chi_x\,\mu_\Wv\,dx\,,
\end{equation}
where $\chi_\Pi$ is the central character of $\Pi$; see \eqref{centralcharacterofpi}.
In the rest of this paper we shall write $dw=d\mu_\Wv(w)$, when convenient.

\section{\bf The invariant integral over $\g$ as an integral over $\h$}
\label{Intertwining distributions as an integral over h}
We now apply the Weyl integration formula to reduce the integral on $\g$ in \eqref{T4}
to an integral on a Cartan subalgebra of $\g$. In section \ref{section:dualpairs-supergroups}, this Cartan subalgebra was denoted by $\h(\g)$, see \eqref{h(g)}. To make our notation lighter, in this section we will write $\h$ instead of $\h(\g)$. 
Let $\H\subseteq \G$ be the corresponding Cartan subgroup. 
Fix a system of positive roots of $(\g_\C, \h_\C)$. For any positive root $\alpha$ let $\g_{\C,\alpha}\subseteq\g_\C$ be the corresponding $\ad(\h_\C)$-eigenspace and let $X_\alpha\in \g_{\C,\alpha}$ be a non-zero vector. 
Let $\H^0\subseteq \H$ denote the connected component of the identity.
There is a character (continuous group homomorphism) $\xi_\alpha: \H^0\to\C^\times$ such that
\[
\Ad(h)X_\alpha= \xi_\alpha(h)X_\alpha \qquad (h\in \H^0)\,.
\]
The derivative of $\xi_\alpha$ at the identity coincides with $\alpha$. Let $\rho\in \h_\C^*$ denote one half times the sum of all the positive roots. Then in all cases except when $\G=\Og_{2l+1}$ or $\G=\Ug_{l}$ with $l$ even, there is a character 
$\xi_\rho: \H^0\to\C^\times$ whose derivative at the identity is equal to $\rho$, see \cite[(2.21) and p. 145]{GoodmanWallach2009}. When $\G=\Og_{2l+1}$ or $\G=\Ug_{l}$ with $l$ even, the character $\xi_\rho$ exists as a map defined on a non-trivial double cover 
\begin{equation}\label{doublecoverofH}
\widehat{\H^0}\ni \widehat h\to h\in \H^0
\end{equation}
of $\H^0$.
In particular the Weyl denominator
\begin{equation} \label{eq:Weyl-den}
\Delta(h)=\xi_\rho(h)\prod_{\alpha>0}(1-\xi_{-\alpha}(h)) 
\end{equation}
is defined for $h\in\H^0$ or $h\in\widehat{\H^0}$ according to the cases described above. 
We will see below how the Weyl group $W(\G,\h)$ acts on $\widehat{\H^0}$. 
The sign representation $\sgn_{\g/\h}$ of the Weyl group $W(\G,\h)$ is defined by
\begin{equation} \label{eq:Weyl-den-sign}
\Delta(sh)=\sgn_{\g/\h}(s) \Delta(h)\qquad (s\in W(\G,\h))\,,
\end{equation}
where either $h\in\H^0$ or $h\in \widehat{\H^0}$.

Suppose first that $\G=\Og_{2l+1}$. Then $\H=\H^0\cdot\Zg=\H^0\times \Zg$ is the direct product of $\H^0$ and the center $\Zg$ of $\Sp(\Wv)$. The group $\widehat{\H^0}$ and the action of the Weyl group on it are described in Appendix \ref{appenB}. The double cover of $\H$ is $\wt{\H}=\H^0\times\wt{\Zg}$. Set $\widehat{\wt{\H}}=\widehat{\H^0}\times\wt\Zg$.
We have a chain of double covering homomorphisms

\begin{equation}\label{chainofcoverings}
\end{equation}
\vskip -1.7cm
\null

\begin{center}
\begin{tikzpicture}
  \matrix (m) [matrix of math nodes,row sep=0mm,column sep=5mm,minimum width=2mm]
  {
    \widehat{\wt{\H}}=\widehat{\H^0}\times\wt\Zg  & \H^0\times\wt\Zg & \H^0\times\Zg & \H^0\,, \\
     (\widehat h,\t z) & (h,\t z) &  (h,z) & h\,. \\};
  \path[-stealth]
    (m-1-1) edge  (m-1-2)
    (m-1-2) edge  (m-1-3)
    (m-1-3) edge  (m-1-4)
    (m-2-1) edge  (m-2-2)
    (m-2-2) edge  (m-2-3)
    (m-2-3) edge  (m-2-4);
\end{tikzpicture}
\end{center}
We extend $\Delta$, $\xi_\mu$ and $\check{\Theta}_\Pi$ to $\widehat{\wt{\H}}$ by defining 
$\Delta(\hat h,\t z)=\Delta(\hat h)$ and $\xi_\mu(\hat h,\t z)=\xi_\mu(\hat h)$ or $\xi_\mu(h)$ if it exists, and $\check{\Theta}_\Pi(\hat h,\t z)=\check{\Theta}_\Pi(h,\t z)$. 
Recall from \eqref{hatc-} the section 
\[
\widehat c_-:\h\ni x\to \widehat{\H^0}
\]
and define
\begin{equation}\label{widehatc}
\widehat c_-:\h\ni x\to (\widehat c_-(x),1)\in \widehat{\wt{\H}}\,.
\end{equation}
This is a real analytic lift  of the modified Cayley transform defined on $\h$ by
\begin{equation}
\label{c_-}
c_-(x)=(1+x)(1-x)^{-1}=-c(x)\,.
\end{equation}
Suppose now that $\G=\Ug_{l}$. Then $\H^0=\H$. Consider the case when $l$ is even. If $\G'=\Ug_{p,q}$ with $p+q$ odd, then 
the covering $\widetilde \H\to \H$ does not split (see Proposition \ref{pro:det-covering}). 
Hence $\Delta$, $\xi_\mu$ and $\check\Theta_\Pi$ are defined on $\widehat \H=\widetilde \H$ and the Weyl group of $\H$ acts on $\widehat \H$ in a way compatible with the cover 
$\widetilde \H \to \H$. We have the modified Cayley transform $c_-:\h\to\H$,  an analytic section
$
\sigma:c_-(\h)\to \widehat \H
$
and the map
\begin{equation}\label{widehatcUodd}
\widehat c_-:\h\ni x\to \sigma (c_-(x))\in \widehat \H\,.
\end{equation}
If $\G'=\Ug_{p,q}$ with $p+q$ even, then define $\widehat \H$ to be the Cartan subgroup of the group $\sqrt{\G}$ defined in Proposition \ref{pro:det-covering} covering $\H$. (In particular, we have the action of the Weyl group $W(\G, \h)$ on $\widehat \H$ because $W(\G, \h)=W(\sqrt{\G}, \h)$.) Then $\Delta$, $\xi_\mu$ and $\check\Theta_\Pi$ are defined on $\widehat \H$. 
By Proposition \ref{pro:det-covering},
%By \cite[Proposition 4.28]{AubertPrzebinda_omega}, 
the metaplectic cover $\wt \H=\H\times\{1,\t 1\}$ splits and we have maps
\begin{eqnarray}\label{chainofcoveringsU}
&&\widehat \H \longrightarrow \H \longrightarrow \wt\H \longrightarrow \H\,,\\
&&\widehat h\; \to \; h \;\to  (h;1) \to  h\,.\nn
\end{eqnarray}
Again $\Delta$, $\xi_\mu$ and $\check\Theta_\Pi$ are defined on $\widehat \H$ and \eqref{widehatcUodd} defines the lift of the Cayley transform we shall use. 
In this case, we set $\widehat{\wt\H}=\widehat \H$.

For the remaining dual pairs, $\widehat\H=\H$ and we lift $\Delta$ and $\xi_\mu$ to functions on $\wt\H$ constant on the fibers of the covering map $\wt\H\to\H$ and write $\widehat c_-$ for $\t c_-$, which was defined under the equation \eqref{eq:tildec}.

\begin{lem}
Let $\mu \in i\h^*$. Then 
\begin{equation}
\label{ximuc}
\xi_{-\mu}(\widehat{c}_-(x))
= \prod_{j=1}^l \left(\frac{1+ix_j}
{1-ix_j}\right)^{\mu_j}
=\prod_{j=1}^l (1+ix_j)^{\mu_j}
(1-ix_j)^{-\mu_j} \qquad (x\in \h)\,.
\end{equation}
\end{lem}
\begin{proof}
By \eqref{decomposition of space for a cartan subspace}, it is enough to verify this formula when $l=1$. In this case, $x=x_1J_1$ and $\mu=\mu_1 e_1=-i\mu_1  J_1^*$. 
Let $\log$ denote the local inverse of the exponential map near $1$. Then, for $x$ sufficiently close to $0$,
$$
\log(c_-(x))=\log\big((1+x)(1-x)^{-1}\big)=\log(1+x)-\log(1-x)
$$
is a real analytic odd function of $x$. Hence it admits a Taylor series expansion
$$
\sum_{n\geq 0} a_n x^{2n+1}=\sum_{n\geq 0} a_n (-1)^n x_1^{2n+1} J_1\,.
$$ 
Thus
$$
\mu(\log(c_-(x)))=-\sum_{n\geq 0} a_n (-1)^n x_1^{2n+1} i \mu_1=
-\sum_{n\geq 0} a_n (ix_1)^{2n+1}  \mu_1=\ln\left(\frac{1-ix_1}
{1+ix_1}\right) \mu_1\,.
$$
By taking exponentials, we obtain
$$
\xi_{-\mu}(\widehat{c}_-(x))=e^{-\mu(\log(c_-(x)))}=\left(\frac{1+ix_1}
{1-ix_1}\right)^{\mu_1}\,,
$$
and the result extends to all $x\in \h$ by real analyticity.
\end{proof}

Let $\Pi$ be an irreducible representation of $\wt\G$, and let $\mu \in i\h^*$ represent the infinitesimal character of $\Pi$. When $\mu$ is dominant, then we will refer to it as the Harish-Chandra parameter of $\Pi$. This is consistent with the usual terminology; see e.g. \cite[Theorem 9.20]{knappLie2}. Then the corresponding character $\xi_\mu$ is defined as
$\xi_\mu=\xi_\rho\,\xi_{\mu-\rho}$, where $\xi_{\mu-\rho}$ is one of the extremal $\H^{0}$-weights of $\Pi$. 
In these terms, Weyl's character formula 
looks as follows,
\begin{equation}\label{Weyl character formula}
\Theta_\Pi(h)\Delta(h)=\kappa_0\sum_{s\in W(\G,\h)} \sgn_{\g/\h}(s)\xi_{s\mu}(h)\,,
\end{equation}
where $h\in \wt{\H^0}$ or $h\in \widehat{\wt{\H^0}}$, according to the cases above, and
$\kappa_0$ is as in \eqref{kappa0}.

\begin{lem}\label{general formula for the int distr}
Let 
$\pi_{\g/\h}$ be the product of the positive roots of $(\g_\C,\h_\C)$ and let
\[
\kappa(x)=
\kappa_0 \frac{\pi_{\g/\h}(x)}{\Delta(\widehat c_-(x))}\,\Theta(\t c(x))\, j_\g(x)
 \qquad (x\in \h)\,.
\] 
%where $m$ is the number of positive roots.
Then, for a suitable normalization of the Lebesgue measure on $\h$ and any $\phi\in \Ss(\Wv)$,
\begin{eqnarray*}
&&\int_{-\G^0}\check\Theta_\Pi(\t g) T(\t g)(\phi)\,dg\\
&&\qquad =\frac{\check\chi_\Pi(\t{c}(0))}{|W(\G,\h)|}
\int_\h
(\Theta_\Pi\Delta)(\widehat c_-(x)^{-1})
\,
\frac{\kappa(x)}{\kappa_0}
\,\pi_{\g/\h}(x)\int_\Wv \chi_x(w)\phi^\G(w)\,dw\,dx\\
&&\qquad=\check\chi_\Pi(\t{c}(0))\int_\h \xi_{-\mu}(\widehat c_-(x))\kappa(x)
\pi_{\g/\h}(x)\int_\Wv \chi_x(w)\phi^\G(w)\,dw\,dx\,,
\end{eqnarray*}
where $\phi^\G$ is as in $\eqref{phiG}$ and each consecutive integral is absolutely convergent.
\end{lem}
\begin{proof}
Applied to a test function $\phi\in\Ss(\Wv)$, the first integral over $-\G^0$ and hence over  $c(\g)$,
is absolutely convergent because both, the character and the function $T(\wt g)(\phi)$ are continuous and bounded (see for example \cite[Proposition 1.13]{PrzebindaUnitary}) and the group $\G$ is compact. Hence, each consecutive integral in the formula \eqref{T4} applied to $\phi$, 
\begin{equation}\label{first sterp0}
\int_{-\G^0}\check\Theta_\Pi(\t g) T(\t g)(\phi)\,dg=\check\chi_\Pi(\t c(0))\int_{\g}\check\Theta_\Pi(\t c_-(x))\,\Theta(\t c(x))\, j_\g(x)\int_\Wv\chi_x(w)\phi(w)\,dw\,dx\,,
\end{equation}
is absolutely convergent. Since 
$$
\chi_{g.x}(w)=\chi_{x}(g^{-1}.w)
$$
and the Lebesgue measure $dw$ is $\G$-invariant, 
\[
\int_\G\int_\Wv\chi_{g.x}(w)\phi(w)\,dw\,dg=\int_\Wv\chi_x(w)\phi^\G(w)\,dw\,.
\]
Observe also that $\wt{\Ad}(\t g)=\Ad(g)$ and the characters $\check\Theta_\Pi$ and $\Theta$ are $\wt{\G}$-invariant.
Moreover, by \eqref{eq:Weyl-den} and \eqref{ximuc},
$$
\overline{\Delta(\widehat c_-(x))}=\Delta(\widehat c_-(x)^{-1})=(-1)^m \Delta(\widehat c_-(x)) \qquad (x\in \h)\,,
$$ 
where $m$ is the number of positive roots, and
$$
\overline{\pi_{\g/\h}(x)}=(-1)^m \pi_{\g/\h}(x) \qquad (x\in \h)\,.
$$
Therefore the Weyl integration formula on $\g$ shows that \eqref{first sterp0} is equal to $\frac{\check\chi_\Pi(\t c(0))}{|W(\G,\h)|}$ times
\begin{eqnarray*}
&&\int_{\h}|\pi_{\g/\h}(x)|^2\check\Theta_\Pi(\t c_-(x))\,\Theta(\t c(x))\, j_\g(x) \int_\Wv\chi_x(w)\phi^\G(w)\,dw\,dx\\
&&\qquad =\int_{\h} 
\check\Theta_\Pi(\widehat c_-(x))
\overline{\Delta(\widehat c_-(x))}
\left(\frac{\pi_{\g/\h}(x)}{\Delta(\widehat c_-(x))}\,\Theta(\t c(x))\, j_\g(x)\right)
 \pi_{\g/\h}(x)\int_\Wv\chi_x(w)\phi^\G(w)\,dw\,dx\\
 &&
 \qquad =
 \int_{\h} 
\Theta_\Pi(\widehat c_-(x)^{-1})
\Delta(\widehat c_-(x)^{-1})
\frac{\kappa(x)}{\kappa_0}
 \pi_{\g/\h}(x)\int_\Wv\chi_x(w)\phi^\G(w)\,dw\,dx\,.
\end{eqnarray*}
(Here, we suppose that the Haar measure on $\H$ is normalized to have total mass 1.) 
 This verifies the first equality and the absolute convergence. 
By \eqref{Weyl character formula} and 
\eqref{ximuchexplicit1} below, 
\begin{align*}
\Theta_\Pi(\widehat c_-(x)^{-1})\Delta(\widehat c_-(x)^{-1})&=\kappa_0
\sum_{s\in W(\G,\h)}\sgn_{\g/\h}(s)\xi_{s\mu}(\widehat c_-(x)^{-1})\\
&=\kappa_0
\sum_{s\in W(\G,\h)}\sgn_{\g/\h}(s)\xi_{-s\mu}(\widehat c_-(x))\\
&=\kappa_0
\sum_{s\in W(\G,\h)}\sgn_{\g/\h}(s)\xi_{-\mu}(\widehat c_-(s^{-1}x))\,.
\end{align*}
Since $\chi_{sx}(w)=\chi_{x}(s^{-1}w)$ and $\phi^\G$ and the Lebesgue measure $dw$ are 
$W(\G,\h)$-invariant, we see that the integral 
$\int_\Wv\chi_x(w)\phi^\G(w)\,dw$ is $W(\G,\h)$-invariant as a function of $x$, too.
The second equality in the statement of the lemma then follows from the skew-symmetry of $\pi_{\g/\h}$ and the $W(\G,\h)$-invariance of $\kappa$, which is a consequence of Lemma \ref{kappa} below.
\end{proof}
Since any element $x\in \g$, viewed as an endomorphism of $\Vv$ over $\R$, has imaginary eigenvalues which come in complex conjugate pairs, we have $\det(1-x)_{\Vv_\R}\geq 1$.   
%Here the subscript $\Vv_\R$ indicates that $x$ is viewed as an endomorphism of $\Vv$ over $\R$.
Define
\begin{equation}\label{ch}
\ch(x)=\det(1-x)_{\Vv_\R}^{1/2} \qquad (x\in\g)\,.
\end{equation}
Recall the symbols $r$ and $\iota$ from \eqref{number r 10} and \eqref{eq:iota}.

\begin{lem}\label{kappa}
There is a constant $C$ which depends only on the dual pair $(\G,\G')$ such that
\[
\frac{\kappa(x)}{\kappa_0}
=C\ch^{d'-r-\iota}(x) \qquad (x\in\h)\,.
\]
\end{lem}
\begin{proof}
Recall \cite[Lemma 5.7]{PrzebindaUnitary} that $\pi_{\g/\h}(x)$ is a constant multiple of $\Delta(\widehat c_-(x)) \ch^{r-\iota}(x)$,
\begin{equation}\label{pianddelta}
\pi_{\g/\h}(x)=C \Delta(\widehat c_-(x)) \ch^{r-\iota}(x)\,.
\end{equation}
For the orthogonal groups this is verified in Appendix \ref{appenB}.
It is easy to compute from \cite[Definition 4.16]{AubertPrzebinda_omega}, that
\begin{equation}\label{eq:tildec0} 
\Theta(\t c(x))^2=i^{\dim \Wv} \det\big(2^{-1}(x-1)\big)_\Wv
\qquad (x\in \sp(\Wv)\,,\ \det(x-1)\ne 0)\,.
\end{equation}
Hence there is a choice of $\t c$ so that
\begin{equation}
\label{eq:tildec1}
\Theta(\t c(x))=\left(\frac{i}{2}\right)^{\frac{1}{2}\dim \Wv} \det\big(1-x\big)_\Wv^{\frac{1}{2}} 
\qquad (x\in \g)\,.
\end{equation}
Furthermore, since the symplectic space may be realized as $\Wv=\Hom_\Bbb D(\Vv',\Vv)$, see \eqref{eq:WasHom},  we obtain that
\begin{equation}\label{eq:tildec2}
\det\big(1-x\big)_\Wv=\det(1-x)_{\Vv_\R}^{d'}\qquad (x\in \g)\,.
\end{equation}
Also, as checked in \cite[(3.11)]{PrzebindaUnipotent}, the Jacobian of $\t c_-:\g\to\G$ is a constant multiple of $\ch^{-2r}(x)$. (For reader's convenience a --slightly different-- proof is included in Appendix \ref{appenA}.)
Hence the claim follows.
\end{proof}
\begin{cor}\label{general formula for the int distr00}
For any $\phi\in \Ss(\Wv)$
\begin{equation*}
\int_{-\G^0}\check\Theta_\Pi(\t g) T(\t g)(\phi)\,dg
=C\, \kappa_0 \check\chi_\Pi(\t c(0)) \int_\h\xi_{-\mu}(\widehat {c}_-(x))\ch^{d'-r-\iota}(x)
\pi_{\g/\h}(x)\int_\Wv \chi_x(w)\phi^\G(w)\,dw\,dx\,,
\end{equation*}
where 
$C$ is a constant which depends only on the dual pair $(\G,\G')$, $\phi^\G$ is as in \eqref{phiG}, and each consecutive integral is absolutely convergent.
\end{cor}
\section{\bf An intertwining distribution in terms of orbital integrals on the symplectic space}
\label{An intertwining distribution in terms of orbital integrals on the symplectic space}

We keep the notation introduced in section \ref{section:dualpairs-supergroups}.
Let
\begin{equation}\label{classical weyl group 1.1.1}
W(\G,\h(\g))=\begin{cases}
\Sigma_l &  \text{if $\Dc=\C$}\,,\\
\Sigma_l\ltimes \{\pm 1\}^l & \text{otherwise\,.}
\end{cases}
\end{equation}
Denote the elements of $\Sigma_l$ by $\eta$ and the elements of 
$\{\pm 1\}^l$ by $\epsilon=(\epsilon_1,\epsilon_2,\dots, \epsilon_l)$, so that an arbitrary element of the group (\ref{classical weyl group 1.1.1}) is of the form $t=\epsilon\eta$, 
with $\epsilon=(1,1,\dots, 1)$, if $\Dc=\C$. This group acts on $\h(\g)$, see \eqref{h(g)}, as follows: for $t=\epsilon\eta$,
\begin{equation}\label{classical weyl group action}
t
\Big(\sum_{j=1}^ly_jJ_j\Big)=\sum_{j=1}^l \epsilon_j y_{\eta^{-1}(j)}J_j\,.
\end{equation}
As indicated by the notation, $W(\G,\h(\g))$ coincides with the Weyl group, equal to the quotient of the normalizer of $\h(\g)$ in $\G$ by the centralizer of $\h(\g)$ in $\G$.

The action of $W(\G,\h(\g))$ on $\h(\g)$ extends by duality to $i\h(\g)^*$. 
More precisely, let $e_j$ be as in \eqref{eq:ej}. If $\mu\in i\h(\g)^*$, then 
$\mu=\sum_{j=1}^l \mu_j e_j$ with all $\mu_j\in \R$. 
If $t=\epsilon\eta
\in W(\G,\h(\g))$, then 
\begin{equation}\label{dual weyl group action}
t\Big(\sum_{j=1}^l\mu_je_j\Big)=\sum_{j=1}^l \epsilon_j 
 \mu_{\eta^{-1}(j)}e_j\,.
\end{equation}
Recall the notation of Lemma \ref{general formula for the int distr} and the symbol $\delta$ from \eqref{eq:delta,beta}.
\begin{lem}\label{ximuchexplicit}
The following formulas hold for any $y=\sum_{j=1}^ly_jJ_j\in\h(\g)$,
\begin{equation}
\label{eq:ximuWeyl}
\xi_{-\mu}(\widehat{c}_-(ty))= \xi_{-t^{-1}
\mu}(\widehat{c}_-(y)) \qquad (t\in  W(\G,\mathfrak{h}(\g)))
\end{equation}
and
\begin{eqnarray}\label{ximuchexplicit1}
\xi_{-\mu}(\widehat{c}_-(y))\ch^{d'-r-\iota}(y)
&=& \prod_{j=1}^l (1+iy_j)^{\mu_j+\delta-1}
(1- iy_j)^{-\mu_j+\delta-1}\,,
\end{eqnarray}
where all the exponents are integers:
\begin{equation}\label{ximuchexplicit2}
\pm \mu_j+\delta\in \Bbb Z \qquad (1\leq j\leq l)\,.
\end{equation}
In particular, (\ref{ximuchexplicit1}) is a rational function in the variables $y_1$, $y_2$, \dots, $y_l$.
\end{lem}
\begin{prf}
By \eqref{ximuc}, 
\begin{equation*}
\xi_{-\mu}(\widehat{c}_-(y))
= \prod_{j=1}^l \left(\frac{1+iy_j}
{1-iy_j}\right)^{\mu_j}
=\prod_{j=1}^l (1+iy_j)^{\mu_j}
(1-iy_j)^{-\mu_j}\,.
\end{equation*}
Hence
(\ref{eq:ximuWeyl}) and (\ref{ximuchexplicit1}) follow from the definition of the action of 
$W(\G,\mathfrak h(\g))$, the definition of $\ch$ in \eqref{ch}, and the following 
easy-to-check formula: 
\begin{eqnarray}\label{ch explicit}
\ch(y)=\prod_{j=1}^l (1+y_j^2)^{\frac{1}{2\iota}}
&=&\prod_{j=1}^l (1+iy_j)^{\frac{1}{2\iota}}(1-iy_j)^{\frac{1}{2\iota}}.
\end{eqnarray}

Let $\lambda=\sum_{j=1}^l\lambda_j e_j$ be the highest weight of the representation $\Pi$ and  
let $\rho=\sum_{j=1}^l\rho_j e_j$ be one half times the sum of the positive roots of $\h(\g)$ in $\g_\C$. If $\mu$ is the Harish-Chandra parameter of $\Pi$, then
$\lambda+\rho=\mu=\sum_{j=1}^l\mu_je_j$; 
see Appendix \ref{appenE} for explicit values.
Hence, the statement (\ref{ximuchexplicit2}) is equivalent to
\begin{equation}\label{compact5}
\lambda_j+\rho_j+\frac{1}{2\iota}(d'-r+\iota)\in \Bbb Z\,.
\end{equation}
Indeed, if $\G=\Og_d$, then with the standard choice of the positive root system, $\rho_j=\frac{d}{2}-j$. Also, $\lambda_j\in\Bbb Z$, $\iota=1$, $r=d-1$. Hence, (\ref{compact5}) follows. Similarly, if $\G=\Ug_d$, then $\rho_j=\frac{d+1}{2}-j$, $\lambda_j+\frac{d'}{2}\in\Bbb Z$, $\iota=1$, $r=d$, which implies (\ref{compact5}). If $\G=\Sp_d$, then $\rho_j=d+1-j$, $\lambda_j\in\Bbb Z$, $\iota=\frac{1}{2}$, $r=d+\frac{1}{2}$, and (\ref{compact5}) follows. 
\end{prf}

Our next goal is to understand the integral 
$$
\pi_{\g/\h}(x)\int_\Wv \chi_x(w)\phi^\G(w)\,dw
$$
occurring in the formula for $\int_{-\G^0}\check\Theta_\Pi(\t g) T(\t g)\,dg$ in 
Lemma \ref{general formula for the int distr} and 
Corollary \ref{general formula for the int distr00}\,, in terms of orbital integrals on the symplectic space 
$\Wv$. The results depend on whether $l\leq l'$ or $l>l'$ and will be given in Lemmas 
\ref{reduction to ss-orb-int} and \ref{reduction to ss-orb-int for l>l'}. 
We first need two other lemmas. 
\begin{lem}\label{lemma:HC's formula moved} 
Fix an element $z\in\h(\g)$. Let $\z\subseteq \g$ and $\Zg\subseteq\G$ denote the centralizer of $z$. (Then $\Zg$ is a real reductive group with Lie algebra $\z$.) Denote by $\c$ the center of $\z$ and by $\pi_{\g/\z}$ the product of the positive roots for $(\g_\C,\h(\g)_\C)$ which do not vanish on $z$. Let $B(\cdot ,\cdot )$ be any non-degenerate symmetric $\G$-invariant real bilinear form on $\g$.
Then there is a constant $C_\z$ 
such that for $x\in \h(\g)$
and $x'\in \c$,
\begin{multline}
\label{HCformula}
\pi_{\g/\h(\g)}(x) \pi_{\g/\z}(x')\int_\G e^{iB(g.x,x')}\,dg\\
=C_\z \sum_{t
W(\Zg,\h(\g))\in W(\G,\h(\g))/W(\Zg,\h(\g))}\sgn_{\g/\h(\g)}(t)\pi_{\z/\h(\g)}(t
^{-1}x)e^{iB(x,t(x'))}.
\end{multline}
(Here  $\pi_{\z/\h(\g)}=1$ if $\z=\h$. Recall also the notation $g.x=gxg^{-1}$.)
\end{lem}
\begin{proof}
The proof is a straightforward modification of the argument proving Harish-Chandra's formula for the Fourier transform of a regular semisimple orbit, \cite[Theorem 2, page 104]{HC-57DifferentialOperators}.
A more general, and by now classical, result  is \cite[Proposition 34, p. 49]{OrbitesDV}. 
\end{proof}

The symplectic form $\langle\cdot,\cdot\rangle$ on $\Wv$, according to the Lie superalgebra structure introduced in \eqref{symplectic-structure}, is
\begin{equation}\label{symplectic form 0}
\langle w',w\rangle=\tr_{\Bbb D/\R}(\Sy w'w) \qquad (w',w\in\Wv)\,.
\end{equation}
Hence
\begin{equation}\label{symplectic form}
\langle x(w),w\rangle=\tr_{\Bbb D/\R}(\Sy  xw^2)\qquad (x\in \g\oplus \g'\,,\  w\in \Wv)\,.
\end{equation}
%%
%(See \cite[(2.4')]{PrzebindaLocal}.) 
Set
\begin{equation}\label{the form B}
B(x,y)=\pi\, \tr_{\Bbb D/\R}(xy) \qquad       (x,y\in\g)\,.
\end{equation}

\begin{lem}
Recall the Gaussian $\chi_x$ from \eqref{eq:chicg}. Then
\begin{equation}\label{chix and the form B}
\chi_x(w)=e^{iB(x,\tau(w))} \qquad     (x\in\g, w\in \Wv)\,.
\end{equation}
\end{lem}
\begin{proof}
Notice that, for $x\in \g\oplus\g'$ and $w\in \Wv$,
$$
\tr_{\Bbb D/\R}(\Sy  xw^2)=
\tr_{\Bbb D/\R}(xw^2|_{\V_{\overline 0}})-\tr_{\Bbb D/\R}(xw^2|_{\V_{\overline 1}})\,,
$$
where
$$
\tr_{\Bbb D/\R}(xw^2|_{\V_{\overline 0}})=
\tr_{\Bbb D/\R}(x|_{\V_{\overline 0}}
w|_{\V_{\overline 1}}w|_{\V_{\overline 0}})
=\tr_{\Bbb D/\R}(w|_{\V_{\overline 0}}
x|_{\V_{\overline 0}}
w|_{\V_{\overline 1}}))
=\tr_{\Bbb D/\R}(wxw|_{\V_{\overline 1}})
$$
and similarly 
$$
\tr_{\Bbb D/\R}(xw^2|_{\V_{\overline 1}})=\tr_{\Bbb D/\R}(wxw|_{\V_{\overline 0}})\,.
$$
Hence
$$
\langle xw,w\rangle=\tr_{\Bbb D/\R}(\Sy  xw^2)=-\tr_{\Bbb D/\R}(\Sy  wxw)=-\langle wx,w\rangle\,.
$$
Therefore
\begin{equation}\label{symplectic form0}
\langle x(w),w\rangle=2\tr_{\Bbb D/\R}(\Sy  xw^2)\qquad (x\in \g\oplus \g'\,,\  w\in \Wv)\,.
\end{equation}
Then \eqref{symplectic form} and (\ref{unnormalized moment maps}) show that
\[
\frac{\pi}{2}\langle x(w),w\rangle=B(x,\tau(w)) \qquad (x\in \g\,,\ w\in \Wv)\,,
\]
which completes the proof.
\end{proof}

The Harish-Chandra regular almost semisimple orbital integral $F(y)$, $y\in\h$,  was defined in
\cite[Definition 3.2 and Theorems 3.3 and 3.5]{McKeePasqualePrzebindaWCestimates};
see also section \ref{section:orbital integrals} above. 
In particular, 
\cite[Theorem 3.5]{McKeePasqualePrzebindaWCestimates} implies that, in the statements below, all the integrals over $\h$ involving $F(y)$ are absolutely convergent. Recall the notation $F_\phi(y)$ for $F(y)(\phi)$.
\begin{lem}\label{reduction to ss-orb-int}
Suppose $l\leq l'$. Then, with the notation of Lemma \ref{general formula for the int distr},
\[
\pi_{\g/\h}(x)\int_\Wv\chi_x(w)\phi^\G(w)\,dw = C \int_{\h\cap \tau(\Wv)}
e^{iB(x,y)} F_{\phi}(y)\,dy\,,
\]
where $C$ is a non-zero constant which depends on the dual pair $(\G,\G')$.
\end{lem}
\begin{proof}
The 
Weyl--Harish-Chandra integration formula on $\Wv$, 
see \eqref{weyl int on w 1}, \eqref{weyl int on w 2} and \eqref{products of roots in so}, 
shows that
\begin{equation}
\label{WIF-1}
\int_\Wv\chi_x(w)\phi^\G(w)\,dw =\sum_{\hs1}\int_{\tau(\hs1^+)}\pi_{\g/\h}(\tau(w))\pi_{\g'/\z'}(\tau(w))C(\hs1)
\mu_{\Oo(w), \hs1}(\chi_x\phi^\G)\,d\tau(w)\,,
\end{equation}
where 
$\hs1^+\subseteq \reg{\hs1}$ is an open fundamental domain for the action of the Weyl group $W(\Sg,\hs1)$ and
$C(\hs1)$ is a constant, determined in \cite[Lemma 2.1]{McKeePasqualePrzebindaWCestimates}.
Let us consider first the case of a semisimple orbital integral
\[
\mu_{\Oo(w), \hs1}(\chi_x\phi^\G)=\int_{\Sg/\Sg^{\hs1}}(\chi_x\phi^\G)(s.w)\,d(s\Sg^{\hs1})\,,
\]
where $\Sg^{\hs1}$ is the centralizer of $\hs1$ in $\Sg$. 
Recall the identification $y=\tau(w)=\tau'(w)$ and let us write $s=gg'$, where $g\in\G$ and $g'\in\G'$. Then
\begin{equation}
\label{chix}
\chi_x(s.w)=e^{i\frac{\pi}{2}\langle x(s.w), s.w\rangle} = e^{iB(x, \tau(s.w))}=e^{iB(x, g.\tau(w))}=e^{iB(x, g.y)}
\end{equation}
and
\begin{equation}
\label{phiGs}
\phi^\G(s.w)=\phi^\G(g'.w)\,.
\end{equation}
Since $l\leq l'$, equation \eqref{integralonS/Sh1-1} below implies that 
there is a positive constant $C_1$ such that
\[
\mu_{\Oo(w),\hs1}(\chi_x\phi^\G)=C_1\int_\G e^{iB(x,g.y)}\,dg \int_{\G'/\Zg'}\phi^\G(g'.w)\,d(g'\Zg')\,.
\]
However we know from Harish-Chandra (Lemma \ref{lemma:HC's formula moved}) that
\begin{eqnarray*}\label{Harish-Chandra's formula for the fourier transform of orbital integral}
\pi_{\g/\h}(x) \left(\int_\G e^{iB(x,g.y)}\,dg\right)\pi_{\g/\h}(y)
=C_2 \sum_{t\in W(\G,\h)}
\sgn_{\g/\h}(t)e^{iB(x,t.y)}\,.
\end{eqnarray*}
Hence, using \eqref{WIF-1} and \cite[Definition 3.2 and Lemma 3.4]{McKeePasqualePrzebindaWCestimates}, we obtain for some suitable positive constants $C_k$, 
\begin{eqnarray}\label{proof in the first case}
\pi_{\g/\h}(x)&&\hskip -1cm \int_\Wv\chi_x(w)\phi^\G(w)\,dw \\
&=& C_3\sum_{t\in W(\G,\h)}\sgn_{\g/\h}(t)\sum_{\hs1} \int_{\tau(\hs1^+)}e^{iB(x,t.y)} C(\hs1)\pi_{\g'/\z'}(y) \int_{\G'/\Zg'}\phi^\G(g'.w)\,d(g'\Zg')\,dy\nn\\
&=& C_4\sum_{t\in W(\G,\h)}\sgn_{\g/\h}(t)\int_{\bigcup_{\hs1}\tau(\hs1^+)}e^{iB(x,t.y)} F_{\phi^\G}(y)\,dy\nn\\
&=& C_4\sum_{t\in W(\G,\h)}\int_{\bigcup_{\hs1}\tau(\hs1^+)}e^{iB(x,t.y)} F_{\phi^\G}(t.y)\,dy\nn\\
&=& C_4\int_{W(\G,\h)(\bigcup_{\hs1}\tau(\hs1^+))}e^{iB(x,y)} F_{\phi^\G}(y)\,dy\nn\\
&=& C_4 \int_{\h\cap \tau(\Wv)}e^{iB(x,y)} F_{\phi^\G}(y)\,dy\,.\nn
\end{eqnarray}
Since $F_{\phi^\G}=\vol(\G) F_{\phi}=F_{\phi}$, the formula follows.

Next we consider the case $\G=\Og_{2l+1}$, $\G'=\Sp_{2l'}(\R)$, $l<l'$. Then
\begin{eqnarray*}
\mu_{\Oo(w)}(\chi_x\phi^\G)=\int_{\Sg/\Sg^{\hs1+w_0}}(\chi_x\phi^\G)(s.(w+w_0))\,d(s\Sg^{\hs1+w_0})\,,
\end{eqnarray*}
where $w_0\in \mathfrak{s}_1(\V^0)$ is a nonzero element.
Since the Cartan subspace $\hs1$ preserves the decomposition \eqref{decomposition of space for a cartan subspace}, $(w+w_0)^2=w^2+w_0^2$. Hence, 
$(s.(w+w_0))^2=s.(w^2+w_0^2)$. The element  $x\in\h$ acts by zero on $\g'$. Therefore  $x(s.(w+w_0))^2=x(s.(w+w_0))^2|_{\V_{\overline 0}}$. Since $\Sg(\V^0)=\Og_1\times\Sp_{2(l'-l)}(\R)$ we see that
$w_0^2|_{\V_{\overline 0}}=0$. Thus $xs.w_0^2|_{\V_{\overline 0}}=0$. Therefore, by \eqref{super liealgebra},
\[
\langle x(s.(w+w_0)), s.(w+w_0)\rangle=\tr(x(s.(w+w_0))^2)=\tr(x s.w^2|_{\V_{\overline 0}})
=\tr(x g.\tau(w))\,,
\]
because $s=gg'$. Hence,
\[
\chi_x(s.(w+w_0))=e^{i\frac{\pi}{2}\langle x(s.(w+w_0)), s.(w+w_0)\rangle}
= e^{iB(x, g.\tau(w))}=e^{iB(x, g.y)}
\]
and
\[
\phi^\G(s.(w+w_0))=\phi^\G(g'.(w+w_0))\,.
\]
Therefore, with $n=\tau'(w_0)$,
we obtain from \eqref{integralonS/Sh1+w0} that
\[
\mu_{\Oo(w)}(\chi_x\phi^\G)=C_1\int_\G e^{iB(x,g.y)}\,dg \int_{\G'/\Zg'{}^n}\phi^\G(g'.w)\,d(g'\Zg'{}^n)\,,
\]
where $\Zg'{}^n$ is the centralizer of $n$ in $\Zg'$.
Thus, the computation (\ref{proof in the first case}) holds again, and we are done.
\end{proof}

\begin{lem}\label{reduction to ss-orb-int for l>l'}
Suppose $l> l'$. Let $\z\subseteq\g$ and $\Zg\subseteq \G$ be the centralizers of $\tau(\hs1)$.
Then for $\phi \in \mathcal S(\Wv)$ 
\begin{multline*}
\pi_{\g/\h(\g)}(x)\int_\Wv\chi_x(w)\phi^\G(w)\,dw\\
= C \sum_{t W(\Zg,\h(\g))\in W(\G,\h(\g))/W(\Zg,\h(\g))}\sgn_{\g/\h(\g)}(t)\pi_{\z/\h(\g)}(t^{-1}.x) \int_{\tau'(\reg{\hs1})}e^{iB(x,t.y)} F_{\phi}(y)\,dy\,, 
\end{multline*}
where $C$ is a non-zero constant which depends only on the dual pair $(\G,\G')$.
\end{lem}
\begin{prf}
By the Weyl--Harish-Chandra 
integration formula with the roles of $\G$ and $\G'$ reversed, see \eqref{weyl int on w 2} and \eqref{products of roots in so}, 
\[
\int_\Wv\chi_x(w)\phi^\G(w)\,dw =C_1\int_{\tau'(\reg{\hs1})}\pi_{\g/\z}(\tau'(w))\pi_{\g'/\h'}(\tau'(w))\mu_{\Oo(w)}(\chi_x\phi^\G)\,d\tau'(w)\,,
\]
where
\[
\mu_{\Oo(w)}(\chi_x\phi^\G)=\int_{\Sg/\Sg^{\hs1}}(\chi_x\phi^\G)(s.w)\,d(s\Sg^{\hs1})\,.
\]
Recall the identification $y=\tau(w)=\tau'(w)$ and let us write $s=gg'$, where $g\in\G$ and $g'\in\G'$. Then, as in \eqref{chix} and \eqref{phiGs},
$$
\chi_x(s.w)=e^{iB(x, g.y)}
\quad\text{and}\quad
\phi^\G(s.w)=\phi^\G(g'.w)\,.
$$
Since $l> l'$, equation \eqref{integralonS/Sh1-2} implies that 
there is a constant $C_2$ 
such that
\[
\mu_{\Oo(w)}(\chi_x\phi^\G)=C_2\int_\G e^{iB(x,g.y)}\,dg \int_{\G'/\H'}\phi^\G(g'.w)\,d(g'\H')\,.
\]
By \eqref{HCformula} in Lemma \ref{lemma:HC's formula moved} and \cite[(34)]{McKeePasqualePrzebindaWCestimates}, we obtain for some constants $C_k$
\begin{eqnarray}\label{proof in the first case, l>l'}
&&\hskip -1cm \pi_{\g/\h(\g)}(x)\int_\Wv\chi_x(w)\phi^\G(w)\,dw \\
&=& C_3\sum_{t W(\Zg,\h(\g))\in W(\G,\h(\g))/W(\Zg,\h(\g))}\sgn_{\g/\h(\g)}(t)\pi_{\z/\h(\g)}(t^{-1}.x) \int_{\tau'(\reg{\hs1})}e^{iB(x,t.y)} \pi_{\g'/\h}(y)\nn\\
&&\hskip 8cm \times \int_{\G'/\H'}\phi^\G(g'.w)\,d(g'\H')\,dy\nn\\
&=& C_4\sum_{t W(\Zg,\h(\g))\in W(\G,\h(\g))/W(\Zg,\h(\g))}\sgn_{\g/\h(\g)}(t)\pi_{\z/\h(\g)}(t^{-1}.x) \int_{\tau'(\reg{\hs1})}e^{iB(x,t.y)} F_{\phi^\G}(y)\,dy\,.\nn
\end{eqnarray}
Since $ F_{\phi^\G}=\vol(\G)  F_{\phi}=F_{\phi}$, the formula follows.
\end{prf}
\begin{lem}\label{vandecorput}
Suppose $l\leq l'$.
Then there is a seminorm $q$ on $\Ss(\Wv)$ such that
\[
\Big|\int_{\h\cap \tau(\Wv)} F_\phi(y)\,e^{iB(x,y)}\,dy\Big|\leq q(\phi)\,\ch(x)^{-d'+r-\iota} \qquad (x\in \h,\, \phi\in\Ss(\Wv))\,.
\]
\end{lem}
\begin{prf}
The boundedness of the distribution-valued function $T(\wt g)$, $\t g\in \wt\G$, means that there is a seminorm $q$ on $\Ss(\g)$ such that
\[
|T(\wt g)(\phi)|\leq q(\phi) \qquad (\t g\in \wt\G\,, \phi \in \Ss(\g))\,.
\]
Hence,
\begin{equation}\label{sixth sterp}
\Big|\Theta(\t c(x))\int_\Wv\chi_x(w)\phi(w)\,dw\Big|\leq q(\phi) \qquad (x\in \g)\,.
\end{equation}
Equivalently, replacing $q(\phi)$ with a constant multiple of $q(\phi)$, and using \eqref{ch}, 
 \eqref{eq:tildec1} and \eqref{eq:tildec2}, we see that
\begin{equation}\label{seventh sterp}
\Big|\int_\Wv\chi_x(w)\phi(w)\,dw\Big|\leq q(\phi)\ch^{-d'}(x) \qquad (x\in \g)\,.
\end{equation}
Since $l\leq l'$, Lemma \ref{reduction to ss-orb-int} together with (\ref{seventh sterp}) proves that
(again up to a multiplicative constant that can be absorbed by $q(\phi)$), 
$$
\Big|\int_{\h\cap \tau(\Wv)} F_\phi(y)\,e^{iB(x,y)}\,dy\Big|\leq q(\phi)\,|\pi_{\g/\h}(x)|\ch(x)^{-d'}.
$$

Recall the constants $r$ and $\iota$ from \eqref{number r 1} and  \eqref{eq:iota}. Then, as one can verify from \eqref{product of positive roots for g}, 
\begin{equation}\label{relation of r with degree}
\max\{\deg_{y_j}\pi_{\g/\h};\ 1\leq j\leq l\}=\frac{1}{\iota} (r-1)\,,
\end{equation}
where $\deg_{y_j}\pi_{\g/\h}$ denotes the degree of $\pi_{\g/\h}(y)$ with respect to the variable $y_j$.

Also,  (\ref{relation of r with degree})  and (\ref{ch explicit}) imply that
\[
|\pi_{\g/\h}(x)|\leq C_5 \ch^{r-1}(x) \leq C_5 \ch^{r-\iota}(x) \qquad (x\in\h)\,,
\]
where $C_5$ is a constant. Thus, the claim follows.
\end{prf}

Lemmas \ref{reduction to ss-orb-int} and \ref{reduction to ss-orb-int for l>l'} allow us to 
restate Corollary \ref{general formula for the int distr00} in terms of orbital integrals on the symplectic space $\Wv$. 

\begin{cor}\label{an intermediate cor}
Suppose $l\leq l'$. Then 
for any $\phi\in \Ss(\Wv)$
\begin{eqnarray*}
&&\int_{-\G^0}\check\Theta_\Pi(\t g) T(\t g)(\phi)\,dg
=C \kappa_0\,
\check{\chi}_\Pi(\t{c}(0))
 \int_\h\xi_{-\mu}(\widehat{c_-}(x)) \ch^{d'-r-\iota}(x)
\int_{\h\cap\tau(\Wv)} e^{iB(x,y)}F_\phi(y)\,dy\,dx\,,
\end{eqnarray*}
where $C$ is a constant that depends only on the dual  pair $(\G, \G')$ and each consecutive integral is absolutely convergent.
\end{cor}
\begin{prf}
The equality is immediate from Corollary \ref{general formula for the int distr00} and Lemma \ref{reduction to ss-orb-int}.
The absolute convergence of the outer integral over $\h$ follows from Lemma \ref{vandecorput}.
\end{prf}
\begin{cor}\label{an intermediate cor, l>l'}
Suppose $l>l'$. Then for any $\phi\in \Ss(\Wv)$,
\begin{multline*}
\int_{-\G^0}\check\Theta_\Pi(\t g) T(\t g)(\phi)\,dg
=C \kappa_0\,
\check{\chi}_\Pi(\t{c}(0))
\sum_{s\in W(\G,\h(\g))}\sgn_{\g/\h(\g)}(s)\int_{\h(\g)}\xi_{-s\mu}(\widehat{c_-}(x)) \ch^{d'-r-\iota}(x) \\
\times
\pi_{\z/\h(\g)}(x)
\int_{\tau'(\reg{\hs1})} e^{iB(x,y)}F_{\phi}(y)\,dy\,dx\,,
\end{multline*}
where $C$ is a constant that depends only on the dual  pair $(\G, \G')$ and each consecutive integral is absolutely convergent.
\end{cor}
\begin{prf}  
The formula is immediate from Corollary \ref{general formula for the int distr00}, Lemma \ref{reduction to ss-orb-int for l>l'} and formula (\ref{eq:ximuWeyl}):
\begin{align*}
\frac{1}{\kappa_0}
&\int_{-\G^0}\check\Theta_\Pi(\t g) T(\t g)\,dg(\phi)\\
&= C_1 \check{\chi}_\Pi(\t{c}(0)) \int_{\h(\g)} \xi_{-\mu}(\widehat{ c_-}(x)) \ch^{d'-r-\iota}(x)
\left(\pi_{\g/\h(\g)}(x)\int_\Wv \chi_x(w)\phi^\G(w)\,dw\right)\,dx\\
&= C_2\check{\chi}_\Pi(\t{c}(0)) \int_{\h(\g)}\xi_{-\mu}(\widehat{ c_-}(x)) \ch^{d'-r-\iota}(x)\\
&\qquad\times\left(\sum_{t W(\Zg,\h(\g))\in W(\G,\h(\g))/W(\Zg,\h(\g))}\sgn_{\g/\h(\g)}(t)\pi_{\z/\h(\g)}(t^{-1}.x) \int_{\tau'(\reg{\hs1})}e^{iB(x,t.y)} F_\phi(y)\,dy\right)\,dx\\
&= \frac{C_2 \, \check{\chi}_\Pi(\t{c}(0))}{|W(\Zg,\h(\g))|}\int_{\h(\g)} \xi_{-\mu}(\widehat{ c_-}(x)) \ch^{d'-r-\iota}(x)\\
&\qquad\times\left(\sum_{t \in W(\G,\h(\g))}\sgn_{\g/\h(\g)}(t)\pi_{\z/\h(\g)}(t^{-1}.x) \int_{\tau'(\reg{\hs1})}e^{iB(x,t.y)} F_{\phi}(y)\,dy\right)\,dx\\
&= C_3\check{\chi}_\Pi(\t{c}(0)) \sum_{t \in W(\G,\h(\g))} \sgn_{\g/\h(\g)}(t) \int_{\h(\g)}\xi_{-\mu}(\widehat{ c_-}(t.x)) \ch^{d'-r-\iota}(t.x)\\
&\qquad
\times\left(\pi_{\z/\h(\g)}(x) \int_{\tau'(\reg{\hs1})}e^{iB(t.x,t.y)} F_{\phi}(y)\,dy\right)\,dx\\
&= C_3\check{\chi}_\Pi(\t{c}(0)) \sum_{t \in W(\G,\h(\g))}\sgn_{\g/\h(\g)}(t)\int_{\h(\g)}\xi_{-t^{-1}\mu}(\widehat{ c_-}(x)) \ch^{d'-r-\iota}(x)\\
&\qquad\times\left(\pi_{\z/\h(\g)}(x) \int_{\tau'(\reg{\hs1})}e^{iB(x,y)} F_{\phi}(y)\,dy\right)\,dx\,.
\end{align*}
Let $\G''$ be the isometry group of the restriction of the form $(\cdot,\cdot)$ to 
$\V_{\overline{0}}^{0,0}$
 and let 
$\h''=\sum_{j=l'+1}^l\R J_j$. Then, as in (\ref{relation of r with degree}), we check that
\begin{equation*}\label{relation of r with degree z}
\max\{\deg_{x_j}\pi_{\z/\h(\g)};\ 1\leq j\leq l\}=\max\{\deg_{x_j}\pi_{\z''/\h''};\ l'+1\leq j\leq l\}=\frac{1}{\iota} (r''-1)\,,
\end{equation*}
where $r''=
\frac{2\dim \g''_\R}{\dim {\V_{\overline{0}\; \R}^{0,0}}}$
is defined as in \eqref{number r 10}. 
Since $r-r''=d'$, we see that
\[
\ch^{d'-r-\iota}(x) |\pi_{\z/\h(\g)}(x)|\leq const\,\ch^{d'-r-\iota+r''-\iota}(x)=const\,\ch^{-2\iota}(x)\,.
\]
Furthermore, $F_\phi$ is absolutely integrable. Therefore,
the absolute convergence of the last integral over $\h(\g)$ follows from the fact that $\ch^{-2\iota}$ is absolutely integrable.  
\end{prf}

To prove Theorem \ref{main thm for l<l'} (and Theorem \ref{main thm for l>=l'}), we still need the following explicit formula for the form $B(x,y)$.
Let $\beta=\dfrac{2\pi}{\iota}$, where $\iota $ is as in (\ref{eq:iota}). Then 
\begin{equation}\label{the constant beta}
B(x,y)=-\beta\sum_{j=1}^l x_j y_j \qquad \big(x=\sum_{j=1}^l x_jJ_j\,, y=\sum_{j=1}^l y_jJ_j\in\h(\g)\big)\,.
\end{equation}
Indeed, the definition of the form $B$, (\ref{the form B}), shows that
\begin{align}\label{computation of the constant beta}
B(x,y)&=\pi\tr_{\Dc/\R}(xy)=\pi\sum_{j,k}\tr_{\Dc/\R}(J_jJ_k)x_jy_k\nn\\
&=\pi\sum_j\tr_{\Dc/\R}(-1_{\V_{\overline 0}^j})x_jy_j=-\frac{2\pi}{\iota}\sum_jx_jy_j\,.
\end{align}

\noindent {\it Proof of Theorem \ref{main thm for l<l'}}.\ 
Notice that the degree of the polynomial $Q_{a_{j},b_{j}}$ is $-a_j-b_j=2\delta-2$ and is independent of 
$\mu$ and $j$. 
Explicitly, 
\begin{equation}
\label{2deltaminus2}
2\delta-2=\frac{1}{\iota} (d'-r-\iota)\,,
\end{equation}
(where $\iota=1/2$ if $\Dc=\Ha$ and $1$ otherwise). 
Hence,
by 
\cite[Theorem 3.5]{McKeePasqualePrzebindaWCestimates}, 
the function $F_\phi$ has the required number of continuous derivatives for the formula (\ref{main thm for l<l' a}) to make sense. The operators appearing in the integrand of (\ref{main thm for l<l' a}) act on different variables and therefore commute. Also, the constants $a_j, b_j$ are integers by (\ref{ximuchexplicit2}).
Hence, equation (\ref{main thm for l<l' a}) follows from Corollary \ref{an intermediate cor}, Lemma \ref{ximuchexplicit}, 
formula \eqref{the constant beta},
and Proposition \ref{D1not smooth}. 

For the last statement about \eqref{product-theorem2}, 
we have 
\begin{equation}
\label{2deltaminus2-bis} 
d'-r-\iota=
\begin{cases}
2l'-2l &\text{if $(\G,\G')=(\Og_{2l},\Sp_{2l'}(\R))$}\,,\\
2l'-2l-1 &\text{if $(\G,\G')=(\Og_{2l+1},\Sp_{2l'}(\R))$}\,,\\
l'-l-1 &\text{if $(\G,\G')=(\Ug_{l},\Ug_{p,q}), p+q=l'$}\,,\\
l'-l-1 &\text{if  $(\G,\G')=(\Sp_{l},\Og^*_{2l'})$}\,.
\end{cases}
\end{equation}
Thus, since we assume $l\leq l'$, 
the product \eqref{product-theorem2} is a function if and only if 
$d'-r-\iota<0$, i.e. if and only if $l=l'$ and $(\G,\G')\neq (\Og_{2l},\Sp_{2l'}(\R))$.
Furthermore, \eqref{product-theorem2} contains no derivatives (but terms involving $\delta_0$ are allowed) if and only if $d'-r-\iota=0$, which corresponds to  
either $l=l'$ and $(\G,\G')= (\Og_{2l},\Sp_{2l'}(\R))$, or $l'=l+1$ and $\Dc=\C$ or 
$\Ha$.
This completes the proof.
\hfill \qed
\bigskip

Suppose now $l>l'$. Let $\h''=\sum_{j=l'+1}^{l}\R J_j$, so that
\begin{equation}\label{h' + h'' decomposition}
\h(\g)=\h\oplus \h''.
\end{equation}
Then the centralizer of $\tau(\h_{\overline{1}})$  coincides with the centralizer of $\h$ in $\g$ and is equal to $\z=\h\oplus\g''$, where $\g''$ is the Lie algebra of the group $\G''$ of the isometries of the restriction of the form $(\cdot ,\cdot )$ to $\V_{\overline 0}^0$.
Furthermore, the derived Lie algebras of $\z$ and $\g''$ coincide (i.e. $[\z,\z]=[\g'',\g'']$) and $\h''$ is a Cartan subalgebra of $\g''$.
We shall identify $\h$ and $\h'$ by means of \eqref{the identification}. This justifies writing 
$\h(\g)=\h'\oplus \h''$ when we need to emphasize the role of $\g'$. 

\begin{lem}\label{another intermediate lemma} 
Suppose $l>l'$. 
In terms of Corollary \ref{an intermediate cor, l>l'} and the decomposition (\ref{h' + h'' decomposition})
\begin{multline}\label{another intermediate lemma1}
\xi_{-s\mu}(\widehat{ c_-}(x)) \ch^{d'-r-\iota}(x) \pi_{\z/\h(\g)}(x)\\
=\left(\xi_{-s\mu}(\widehat{ c_-}(x')) \ch^{d'-r-\iota}(x')\right)\left(\xi_{-s\mu}(\widehat{ c_-}(x'')) \ch^{d'-r-\iota}(x'') \pi_{\g''/\h''}(x'')\right)\,,
\end{multline}
where $x=x'+x''\in \h(\g)$, with $x'\in \h'$ and $x''\in\h''$.
Moreover,
\begin{multline}\label{another intermediate lemma2}
\int_{\h''}\xi_{-s\mu}(\widehat{ c_-}(x'')) \ch^{d'-r-\iota}(x'') \pi_{\g''/\h''}(x'')\,dx''\\
=C\sum_{s''\in W(\G'',\h'')}
\sgn_{\g''/\h''}(s'')
\Bbb I_{\{0\}}(-(s\mu)|_{\h''}+s''\rho'')\,,
\end{multline}
where $C$ is a constant, $\rho''$ is one half times the sum of the positive roots for $(\g''_\C, \h''_\C)$ and $\Bbb I_{\{0\}}$ is the indicator function of zero.
\end{lem}
\begin{prf}
Formula (\ref{another intermediate lemma1}) is obvious, because $\widehat{ c_-}(x'+x'') =\widehat{ c_-}(x') \widehat{ c_-}(x'')$ and $\pi_{\z/\h(\g)}(x'+x'')=\pi_{\g''/\h''}(x'')$. We shall verify (\ref{another intermediate lemma2}).
Let $r''$ denote the number defined in (\ref{number r 10}) for the Lie algebra $\g''$. A straightforward computation verifies the following table:
\begin{center}
\renewcommand{\extrarowheight}{.2em}
\begin{tabular}{c||c|c|c}\label{table 3}
$\g$ & $r$  & $r''$ & $d'-r+r''$ \\[.1em]
\hline\hline
$\u_d$ & $d$ & $d-d'$ & $0$\\[.1em]
\hline
$\o_d$ & $d-1$ & $d-d'-1$ & $0$\\[.1em]
\hline
$\sp_d$ & $d+\frac{1}{2}$ & $d-d'+\frac{1}{2}$ & $0$\\[.2em]
\hline
\end{tabular}
\end{center}
\smallskip 

By \eqref{pianddelta} applied to $\G''\supseteq \H''$ and $\g''\supseteq \h''$, 
\begin{eqnarray*}
\pi_{\g''/\h''}(x'')=C_1''\Delta''(\widehat{c_-}(x'')) \ch^{r''-\iota}(x'')
\qquad 
(x''\in\h'')\,,
\end{eqnarray*}
where $\Delta''$ is the Weyl denominator for $\G''$,
\begin{eqnarray}\label{FT Theta ch 13}
\Delta''=
\kappa''_0
\sum_{s''\in W(\G'',\h'')} 
\sgn_{\g''/\h''}(s'')
\,\xi_{s''\rho''}
\end{eqnarray}
and 
\begin{equation}
\label{kappa0''}
\kappa''_0=\begin{cases}
\frac{1}{2} &\text{if $\G''=\Og_{d''}$ where $d''$ is even}\,,\\
1 &\text{otherwise}\,.
\end{cases}
\end{equation}
Hence, by \eqref{ch explicit}, the integral on the left-hand side of (\ref{another intermediate lemma2}) is a constant multiple of
\begin{multline}\label{FT Theta ch 11}
\int_{\h''}\xi_{-s\mu}(\widehat{c_-}(x''))\Delta''(\widehat{c_-}(x''))\ch^{d'-r+r''}(x'')\ch^{-2\iota}(x'')\,dx''
=2^{\dim\h''} \int_{\widehat{c_-}(\h'')}\xi_{-s\mu}(h)\Delta''(h)\,dh\,, \qquad
\end{multline}
where $\widehat{c_-}(\h'')\subseteq \widehat{\H''^0}$. 

Notice that the function
\begin{equation*}\label{FT Theta ch 12}
\widehat{\H''^0}\ni h\to \xi_{-s\mu}(h)\Delta''(h)\in\C
\end{equation*}
is constant on the fibers of the covering map 
\begin{equation}
\label{coveringH''0}
\widehat{\H''^0}\to \H''{}^0\,.
\end{equation}
Indeed, the covering \eqref{coveringH''0} is non-trivial only in two cases, namely $\G''=\Og_{2l''+1}$ and $\G''=\Ug_{l''}$ with $l''$ even; see \eqref{doublecoverofH}. In these cases, \eqref{FT Theta ch 13}
shows that this claim is true provided that the weight $-s\mu+s''\rho''$ is integral for the Cartan subgroup $\H''$ (i.e. it is equal to the derivative of a character of $\H''$). 

Suppose $\G''=\Og_{2l''+1}$. Then $\G=\Og_{2l+1}$, $\lambda_j\in \Zb$ and 
$\rho_j\in\Zb+\frac{1}{2}$. Hence, $(-s\mu)_j\in\Zb+\frac{1}{2}$. Since, 
$\rho''_j\in \Zb+\frac{1}{2}$, we see that $(-s\mu)_j+\rho''_j\in \Zb$.

Suppose now that $\G''=\Ug_{l''}$ with $l''$ even. Then $\G=\Ug_l$ and $(-s\mu)_j\in\Zb+\frac{1}{2}$. In fact, if $l'$ is even, i.e. $l=l'+l''$ is even, then $\lambda_j\in\Zb$ and $\rho_j\in\Zb+\frac{1}{2}$. If $l'$ is odd, i.e. $l=l'+l''$ is odd, then $\lambda_j\in\Zb+\frac{1}{2}$ and $\rho_j\in\Zb$. Since $\rho''_j\in \Zb+\frac{1}{2}$, in both cases, we conclude that $(-s\mu)_j+\rho''_j\in \Zb$.  

Therefore, (\ref{FT Theta ch 11}) is a constant multiple of 
\begin{eqnarray}\label{FT Theta ch 15}
\sum_{s''\in W(\G'',\h'')} 
\sgn_{\g''/\h''}(s'')
&&\hskip -1cm
\int_{\H''{}^0}\xi_{-s\mu}(h) \xi_{s''\rho''}(h)\,dh\\
&=&\left\{
\begin{array}{ll}
\vol(\H''{}^0)\,
\sgn_{\g''/\h''}(s'')
\ &\ \text{if}\ (s\mu)|_{\h''}=s''\rho'',\\
0\ &\ \text{if $(s\mu)|_{\h''}\notin W(\G'',\h'')\rho''$}\nn\,,
\end{array}
\right.\\
&=&\vol(\H''{}^0)\sum_{s''\in W(\G'',\h'')} 
\sgn_{\g''/\h''}(s'')
\Bbb I_{\{0\}}(-(s\mu)|_{\h''}+s''\rho'')\,.\nn
\end{eqnarray}
\end{prf}
\begin{cor}\label{another intermediate cor, l>l'}
Suppose $l>l'$ and keep the notation of  Lemma \ref{another intermediate lemma} .
Then 
$$
\int_{-\G^0} \check\Theta_\Pi(\t g)  T(\t g)\,dg=0
$$ 
unless there is $s\in W(\G,\h(\g))$ such that 
\begin{equation}\label{smu-h''}
(s\mu)|_{\h''}=\rho''\,.
\end{equation}
If $\G=\Og_{2l+1}$ or $\Sp_l$, then \eqref{smu-h''} is equivalent to 
\begin{equation}
\label{smu-R-H}
\mu|_{\h''}=\rho'' \quad \text{and} \quad s|_{\h''}=1\,.
\end{equation}
Suppose $\G=\Og_{2l}$ and write $\h''=\h''_0\oplus \R J_l$, where $\h''_0=\sum_{j=l'+1}^{l-1} \R J_j$.
Then \eqref{smu-h''} is equivalent to
\begin{equation}
\label{smu-R-even}
\mu|_{\h''}=\rho'', \quad s|_{\h_0''}=1, \quad \text{and} \quad s|_{\R J_l}=\pm 1\,.
\end{equation}
Finally, if $\G=\Ug_l$, then \eqref{smu-h''} holds if and only if there is
$j_0\in\{0,1,\dots,l'\}$ such that 
\begin{equation}
\label{smu-C}
\mu_{j_0+j}=\rho''_{l'+j}  \quad \text{and} \quad s(J_{j_0+j})=J_{l'+j} \qquad (1\leq j\leq l-l')\,.
\end{equation}

Suppose that \eqref{smu-h''} holds. 
Then
for any $\phi\in \Ss(\Wv)$
\begin{multline}
\label{another intermediate cor, l>l' a0}
\int_{-\G^0} \check\Theta_\Pi(\t g)  T(\t g)\,dg(\phi)
= C \, \kappa_0
\check{\chi}_\Pi(\t{c}(0)) \sum_{s\in W(\G,\h(\g)), \, (s\mu)|_{\h''}=\rho''}
\sgn_{\g/\h(\g)}(s) \\
\times  \int_{\h'}\xi_{-s\mu}(\widehat{c_-}(x)) \ch^{d'-r-\iota}(x) 
\int_{\tau'(\reg{\hs1})} e^{iB(x,y)} F_{\phi}(y)\,dy\,dx\,,
\end{multline}
where $C$ is a non-zero constant which depends 
only on the dual pair $(\G,\G')$, and each consecutive integral is absolutely convergent. 
\end{cor}
\begin{proof}
Observe that $B(x'+x'',y)=B(x',y)$ for $x'\in \h'$, $x''\in \h''$ and $y \in \tau'(\reg{\hs1})\subseteq \h'$.
We see therefore from Corollary \ref{an intermediate cor, l>l'} and Lemma \ref{another intermediate lemma} that
\begin{multline}\label{another intermediate cor, l>l'1}
\int_{-\G^0} \check\Theta_\Pi(\t g) T(\t g)\,dg(\phi)
\\
=C \,
\kappa_0\check{\chi}_\Pi(\t{c}(0))
\sum_{s\in W(\G,\h(\g))}\sum_{s''\in W(\G'',\h'')}
\sgn_{\g/\h(\g)}(s)
\sgn_{\g''/\h''}(s'')
\Bbb I_{\{0\}}(-(s\mu)|_{\h''}+s''\rho'') \\
\times \int_{\h'}\xi_{-s\mu}(\widehat{c_-}(x)) \ch^{d'-r-\iota}(x) \int_{\tau'(\reg{\hs1})} e^{iB(x,y)}F_{\phi}(y)\;dy\,dx\,.
\end{multline}
Notice that for $x\in \h'$ and $s''\in W(\G'',\h'')$, we have $s''x=x$.
Thus $\xi_{-s\mu}(\widehat{c_-}(x))=\xi_{-s''s\mu}(\widehat{c_-}(x))$ by (\ref{eq:ximuWeyl}).
Notice also that, by \eqref{h' + h'' decomposition}, $W(\G'',\h'')\subseteq W(\G,\h)$ and 
$
\sgn_{\g''/\h''}(s'')=\sgn_{\g/\h(\g)}(s'')$.
Moreover, 
$\Bbb I_{\{0\}}(-(s\mu)|_{\h''}+s''\rho'')=\Bbb I_{\{0\}}(-(s''{}^{-1}s\mu)|_{\h''}+\rho'')$. Hence, replacing $s$ by $s'' s$ in (\ref{another intermediate cor, l>l'1}), we see that this expression is equal to
\begin{multline}\label{another intermediate cor, l>l'2}
C \,\kappa_0\check{\chi}_\Pi(\t{c}(0))
\sum_{s\in W(\G,\h(\g))}\sum_{s''\in W(\G'',\h'')}
\sgn_{\g/\h(\g)}(s)
\Bbb I_{\{0\}}(-(s\mu)|_{\h''}+\rho'')\\
\times \int_{\h'}\xi_{-s\mu}(\widehat{c_-}(x)) \ch^{d'-r-\iota}(x) \int_{\tau'(\reg{\hs1})} e^{iB(x,y)}F_{\phi}(y)\,dy\,dx\,,
\end{multline}
which yields \eqref{another intermediate cor, l>l' a0}, with a new non-zero constant $C$, equal to 
$C|W(\G'',\h'')|$. Clearly (\ref{another intermediate cor, l>l'2}) is zero if there is no $s$ such that $(s\mu)|_{\h''}=\rho''$.
The absolute convergence of the integrals was checked in the proof of Corollary \ref{an intermediate cor, l>l'}.

Recall that $\h''=\sum_{j=l'+1}^l \R J_j$ and $\mu=\lambda+\rho$ where $\lambda$ is the highest weight of the genuine representation $\Pi$. 
We take a closer look at the condition $(s\mu)|_{\h''}=\rho''$.

If $\Dc=\R$ or $\Ha$, then
$\rho|_{\h''}=\rho''$. All coefficients of $\rho$ are positive and strictly decreasing by $1$ except when $\G=\Og_{2l}$, where $\rho_l=0$. Hence $s|_{\h''}$ cannot contain sign changes when $\G=\Og_{2l+1}$
or $\Sp_l$, whereas $s|_{\h''_0}$ cannot contain sign changes when $\G=\Og_{2l}$.
Using the form of the coefficients of $\lambda$, one easily sees
that \eqref{smu-h''} is equivalent to \eqref{smu-R-H} or \eqref{smu-R-even}. 

If $\G=\Ug_l$, then $\lambda=\frac{p-q}{2}+\nu$, where $\nu_1\geq \nu_2\geq \cdots \geq \nu_l$ are integers. Moreover, 
\begin{equation}
\label{rho-rho''-C}
\frac{p-q}{2}+\rho_{p+j}
=\frac{l-p-q+1}{2}-j=\rho''_{l'+j} \qquad
(1\leq j\leq l-l')\,.
\end{equation}
The Weyl group $W(\G,\h(\g))$ consists of permutations of the $J_j$'s. Hence a genuine Harish-Chandra parameter $\mu$ satisfies \eqref{smu-h''} if and only if among its coefficients $\mu_1,\dots,\mu_l$ we can find a string of $l-l'$ successive coefficients $\mu_j$ equal to $\rho''_{l'+1},\dots,\rho''_l$ and the permutation $s$ translates the corresponding string of $J_j$'s onto $J_{l'+1},\dots,J_l$.
This proves \eqref{smu-C}.
\end{proof}

In the next lemmas we study the integrals 
appearing on the right-hand side of 
\eqref{another intermediate cor, l>l' a0}. 

\begin{lem}
\label{computation-second-integral-l>l'}
For $s\in W(\G,\h(\g))$ and $y\in \tau'(\hs1)$, in the sense of distributions on $\tau'(\reg{\hs1})$,

\begin{equation}
\label{innerintegral-smu-gen}
\int_{\h'}\xi_{-s\mu}(\widehat{c_-}(x)) \ch^{d'-r-\iota}(x) e^{iB(x,y)}\,dx
=\Big(\prod_{j=1}^{l'} P_{a_{s,j},b_{s,j}}(\beta y_j)\Big) e^{-\beta\sum_{j=1}^{l'} |y_j|}\,,
\end{equation}
where $a_{s,j}$, $b_{s,j}$ and $\beta$ are as in \eqref{asj-bsj} and \eqref{eq:delta,beta}, 
and $P_{a_{s,j},b_{s,j}}$ is defined in \eqref{D0'}.
\end{lem}
\begin{proof}
This follows immediately from Lemma \ref{ximuchexplicit}, \eqref{D0''}, and Proposition \ref{D1not smooth}, since $a_{s,j}+b_{s,j}=-2\delta+2\geq 1$ for $l>l'$.
\end{proof}

Suppose that $\mu$ satisfies \eqref{smu-h''} for some $s\in W(\G,\h(\g))$. The integral 
corresponding to $s$ in \eqref{another intermediate cor, l>l' a0} vanishes when the intersection of the support of the right-hand side of \eqref{innerintegral-smu-gen} and $\tau'(\reg{\hs1})$ has an empty interior. We first study this intersection for some specific elements in $W(\G,\h(\g))$. 

If $\Dc=\R$ or $\Ha$, define $s_0=1$ as in \eqref{s0-R-H}. 
Then clearly $s_0\mu|_{\h''}=\rho''$ by \eqref{smu-R-H}.
If $\Dc=\C$, fix $j_0\in \{0,1,\dots,l'\}$ as in 
\eqref{smu-C} and define $s_{0,j_0}$ as the permutation in $W(\G,\h(\g))$
given by 
\begin{equation}
\label{s0-j_0-C}
s_{0,j_0}(J_j)=
\begin{cases}
J_{j} \qquad &(1\leq j\leq j_0)\\
J_{l'-j_0+j} \qquad &(j_0+1\leq j\leq j_0+l-l')\\
J_{j-l+l'} \qquad &(j_0+l-l'+1\leq j\leq l)\,, 
\end{cases}
\end{equation}
i.e.
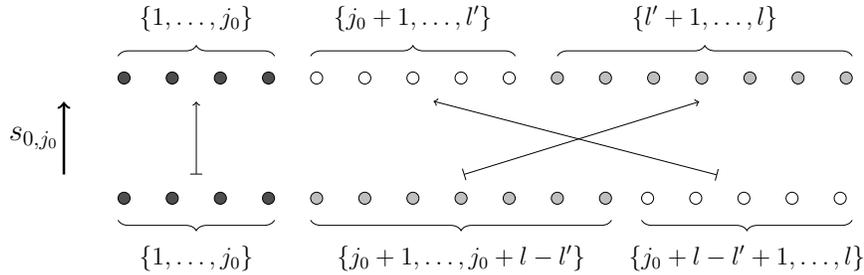
\begin{center}
\scalebox{0.8}{
\tikzset{decorate sep/.style 2 args=
{decorate,decoration={shape backgrounds,shape=circle,shape size=#1,shape sep=#2}}}

\begin{tikzpicture}
\draw[decorate sep={2mm}{8mm},fill=black!70] (0,2) -- (3,2);
\draw[decorate sep={2mm}{8mm},fill=gray!50] (7.2,2) -- (12,2);
\draw[decorate sep={2mm}{8mm}] (3.2,2) -- (6.5,2);

\draw[decorate sep={2mm}{8mm},fill=black!70] (0,0) -- (3,0);
\draw[decorate sep={2mm}{8mm},fill=gray!50] (3.2,0) -- (8.2,0);
\draw[decorate sep={2mm}{8mm}] (8.7,0) -- (12,0);

\draw [decorate,decoration={brace,amplitude=5pt,raise=2ex}]
  (-0.1,2) -- (2.5,2) node[midway,yshift=2.4em]{$\{1,\dots, j_0\}$};
\draw [decorate,decoration={brace,amplitude=5pt,raise=2ex}]
  (3.1,2) -- (6.5,2) node[midway,yshift=2.4em]{$\{j_0+1,\dots, l'\}$};
\draw [decorate,decoration={brace,amplitude=5pt,raise=2ex}]
  (7.2,2) -- (12.1,2) node[midway,yshift=2.4em]{$\{l'+1,\dots, l\}$};

\draw [decorate,decoration={brace,amplitude=5pt,mirror,raise=2ex}]
  (-0.1,0) -- (2.5,0) node[midway,yshift=-2.4em]{$\{1,\dots, j_0\}$};
\draw [decorate,decoration={brace,amplitude=5pt,mirror,raise=2ex}]
  (3.1,0) -- (8.1,0) node[midway,yshift=-2.4em]{$\{j_0+1,\dots, j_0+l-l'\}$};
\draw [decorate,decoration={brace,amplitude=5pt,mirror,raise=2ex}]
  (8.6,0) -- (12.1,0) node[midway,yshift=-2.4em]{$\{ j_0+l-l'+1,\dots, l\}$};

\node at (-1.5,1) (s0) {{\large  $s_{0,j_0}$}};
\node at (-1,0.25) (s01) {};
\node at (-1,1.75) (s02) {};
\node at (1.2,0.25) (g1) {};
\node at (1.2,1.75) (i1) {};
\node at (5.5,0.35) (g2) {};
\node at (9.7,1.65) (i2) {};
\node at (10,0.35) (g3) {};
\node at (5,1.65) (i3) {};
\draw[->,very thick] 
          (s01) edge (s02) ;
\draw[|->]    
           (g1) edge (i1)
           (g2) edge (i2)
           (g3) edge (i3) ;
\end{tikzpicture}    
}
\end{center}

\noindent Equivalently, 
\begin{equation}
\label{s0mu-Ul}
(s_{0,j_0} \mu)_j=\mu_{s_{0,j_0}^{-1}(j)}=\begin{cases}
\mu_{j} &(1\leq j\leq j_0)\\
\mu_{l-l'+j}
 &(j_0+1\leq j\leq l')\\
\mu_{j_0-l'+j} &(l'+1\leq j\leq l)\,.
\end{cases}
\end{equation}
Hence $(s_{0,j_0}\mu)|_{\h''}=\rho''$.
Notice that $s_{0,p}$ is the element $s_0$ defined in \eqref{s0-C}.

\begin{lem}
\label{lemma-support}
Let $l>l'$ and suppose that $\mu$ satisfies \eqref{smu-h''}. Let $s_0=1$, as in \eqref{s0-R-H}, if 
$\Dc=\R$ or $\Ha$, and let $s_{0,j_0}$ be as in 
\eqref{s0-j_0-C} if $\Dc=\C$.

If $\Dc=\R$ or $\Ha$, then
\begin{equation}
\label{poly-smu-R-H}
\prod_{j=1}^{l'} P_{a_{s_0,j},b_{s_0,j}}(\beta y_j)=
(2\pi)^{l'} \prod_{j=1}^{l'} P_{a_j,b_j,2}(\beta y_j) \mathbb{I}_{\R^+}(y_j) \qquad (y=\sum_{j=1}^{l'} y_j J'_j\in \h')
\end{equation}
has support equal to $\tau'(\hs1)$. 

If $\Dc=\C$, then
\begin{align}
\label{poly-smu-C}
&\prod_{j=1}^{l'}P_{a_{s_{0,j_0},j},b_{s_{0,j_0},j}}(\beta y_j)
=\!(2\pi)^{l'}  \Big(\prod_{j=1}^{j_0}  P_{a_j,b_j,2}(\beta y_j) \mathbb{I}_{\R^+}(y_j)\Big) 
\Big(\prod_{j=j_0+1}^{l'}  \!\! P_{a_{j+l-l'},b_{j+l-l'},-2}(\beta y_j) \mathbb{I}_{\R^-}(y_j)\Big) \nn\\
&\null\hskip 10.8 cm (y=\sum_{j=1}^{l'} y_j J'_j\in \h') 
\end{align}
has support equal to $\big(\sum_{j=1}^{j_0} \R^+ J'_j \big)\oplus \big( \sum_{j=j_0+1}^{l'} \R^- J'_j \big)$. This support is equal to $\tau'(\hs1)$ if $j_0=p$, whereas its intersection with $\tau'(\hs1)$ 
has empty interior if $j_0\neq p$.
\end{lem}
\begin{proof}
Let $\Dc=\R$ or $\Ha$. By \eqref{eq:iota}, \eqref{number r 1}, \eqref{eq:delta,beta} and Appendix \ref{appenE} and since $\mu|_{\h''}=\rho''=\rho|_{\h''}$, we see that 
\begin{equation*}
\mu_1>\dots>\mu_{l'}>\mu_{l'+1}=\rho''_{l'+1}=-\delta\,,
\end{equation*}
These inequalities are equivalent to 
\begin{equation}
\label{mu-leq-delta-R-H-bis}
a_1=-\mu_1-\delta+1<a_2=-\mu_2-\delta+1 < \dots < a_{l'}=-\mu_{l'}-\delta+1\leq 0
\end{equation}
because the $\mu_j$'s and $\delta$ 
are either all in $\Zb$ or all in $\Zb+\frac{1}{2}$.
Hence $P_{a_j,b_j,-2}=0$ for all $1\leq j\leq l'$ by \eqref{eq:Pabminus2}. 
Since $a_j+b_j=-2\delta+2>2$, we see that $b_j> 2-a_j\geq 1$.  Therefore, the polynomial $P_{a_j,b_j,2}$ is nonzero for all $1\leq j\leq l'$. 
Hence the function on the right-hand side of \eqref{poly-smu-R-H}
has support equal to $\sum_{j=1}^{l'}\R^+J'_j=\tau'(\hs1)$. 
\smallskip

Let now $\Dc=\C$. By \eqref{smu-C}, \eqref{rho-rho''-C}, \eqref{eq:delta,beta} and \eqref{2deltaminus2}, 
\begin{eqnarray*}
&&\mu_1>\mu_2> \cdots > \mu_{j_0}>\mu_{j_0+1}=\rho''_{l'+1}=\frac{l-l'-1}{2}=-\delta (>0)\,, \\
&&(0>) \delta=-\frac{l-l'-1}{2}=\rho''_l=\mu_{j_0+l-l'}>\mu_{j_0+l-l'+1}>\cdots >\mu_l\,.
\end{eqnarray*}
Since the $\mu_j$'s and $\delta$ are either all in $\Zb$ or all in $\Zb+\frac{1}{2}$, these inequalities are equivalent to 
\begin{eqnarray}
\label{list-signs-mu-C}
&&a_1=-\mu_1-\delta+1<a_2=-\mu_2-\delta+1< \cdots <a_{j_0}=-\mu_{j_0}-\delta+1\leq 0 \nn\\
&&0 \geq b_{j_0+l-l'+1}=\mu_{j_0+l-l'+1}-\delta+1 > \cdots >b_l=\mu_l-\delta+1\,.
\end{eqnarray}
Hence, 
\begin{alignat*}{2}
P_{a_j,b_j,-2}&=0 \quad &\text{i.e.} \quad &P_{a_j,b_j}(y_j)=2\pi P_{a_j,b_j,2}(y_j) \mathbb{I}_{\R^+}(y_j) \qquad (1\leq j\leq j_0)\,,\\
P_{a_j,b_j,2}&=0 \quad &\text{i.e.} \quad &P_{a_j,b_j}(y_j)=2\pi P_{a_j,b_j,-2}(y_j) \mathbb{I}_{\R^-}(y_j) \qquad (j_0+l-l'+1\leq j\leq l)\,.
\end{alignat*}
The polynomials appearing in these expressions of $P_{a_j,b_j}$ are nonzero because
$a_j+b_j=-2\delta+2>0$ for all $j$. 
By \eqref{explicit tau and tau' on cartan subspace} and the convention on the symbols $\delta_j$'s for the dual pair $(\Ug_l,\Ug_{p,q})$ with $l>l'=p+q$,  the claims on the support of  the right-hand side of \eqref{poly-smu-C} follow. 
\end{proof}

Let $\Dc=\C$. Suppose that there is $s\in W(\G,\h(\g))$ such that $(s\mu)|_{\h''}=\rho''$ and that the string of coefficients of $\mu$ equal to those of $\rho''$, see  \eqref{smu-C}, starts at $j_0+1$, where $j_0\in \{0,1,\dots,l'\}$. Then $s=s_{0,j_0}$ satisfies $(s\mu)|_{\h''}=\rho''$.
Lemma \ref{lemma-support} shows that if $j_0\neq p$ then the intersection of the support of $\prod_{j=1}^{l'}P_{a_{s_{0,j_0},j},b_{s_{0,j_0},j}}$ with $\tau'(\hs1)$ has empty interior. We now prove that if $j_0\neq p$ the same holds for the support of
$\prod_{j=1}^{l'}P_{a_{s,j},b_{s,j}}$ for every $s\in W(\G,\h(\g))$ such that $(s\mu)|_{\h''}=\rho''$.

\begin{lem}
\label{lemma:j0-not-p}
Let $\Dc=\C$. Suppose that $\mu$ and $s\in W(\G,\h(\g))$ satisfy \eqref{smu-C} for $j_0\in \{0,1,\dots,l'\}$. If $j_0\neq p$, then the intersection of the support of $\prod_{j=1}^{l'}P_{a_{s,j},b_{s,j}}$ 
with $\tau'(\hs1)$ has empty interior. 
\end{lem}
\begin{proof}
Since 
$$
s_{0,j_0}(J_{j_0+j})=J_{l'+j}\,, \qquad s(J_{j_0+j})=J_{l'+j} \qquad (1\leq j\leq l-l')\,,
$$
the composition $s^{-1}s_{0,j_0}$ fixes the elements of $\{J_{j_0+1},\dots,J_{j_0+l-l'}\}$ and permutes those of 
$\{J_{1},\dots,J_{j_0}\} \cup \{J_{j_0+l-l'+1},\dots,J_{l}\}$.
Then $s^{-1}=(s^{-1}s_{0,j_0})s_{0,j_0}^{-1}$ maps %(string to string) 
$\{J_{l'+1},\dots,J_{l}\}$ onto 
$\{J_{j_0+1},\dots,J_{j_0+l-l'}\}$ and hence $\{J_{1},\dots,J_{l'}\}$ bijectively onto 
$\{J_{1},\dots,J_{j_0}\} \cup \{J_{j_0+l-l'+1},\dots,J_{l}\}$. Therefore 
$\{ (s\mu)_j=\mu_{s^{-1}(j)}; 1\leq j\leq l'\}$ is a permutation of $\{\mu_j; 1\leq j\leq j_0\} \cup 
\{\mu_j; j_0+l-l'+1\leq j \leq l\}$.  
By \eqref{list-signs-mu-C}, there are $j_0$ negative $a_j$ and $l'-j_0$ negative $b_j$ for 
$1\leq j\leq l'$. The same is then true for the $a_{s,j}$ and the $b_{s,j}$. The support of 
$\prod_{j=1}^{l'}P_{a_{s,j},b_{s,j}}$ is therefore a Cartesian product (in some order) of $j_0$ copies of $\R^+$ 
and $l'-j_0$ copies of $\R^-$. Its intersection with $\tau'(\hs1)$ has therefore empty interior if 
$j_0\neq p$.
\end{proof}

When the intersection of the support of $\prod_{j=1}^{l'}P_{a_{s,j},b_{s,j}}$ and $\tau'(\hs1)$ has empty interior, the integral on the right-hand side of \eqref{another intermediate cor, l>l' a0} that 
corresponds to $s$ vanishes. Lemma \ref{lemma:j0-not-p} shows that every such integral is zero when $j_0\neq p$. This yields the following corollary. 

\begin{cor}
\label{cor:j0-not-p}
Suppose that $\Pi$ is a genuine representation of $\wt{\Ug}_l$ with Harish-Chandra parameter 
$\mu$ satisfying \eqref{smu-C} for $j_0\in \{0,1,\dots,l'\}$. If $j_0\neq p$ then 
$$
f_{\Pi\otimes \Pi'}=\int_{\Ug_l} \check{\Theta}_\Pi (\wt{g}) T(\wt{g}) \; dg=0\,.
$$
Thus, if $\Pi$ is a genuine representation of $\wt{\Ug}_l$ which 
occurs in Howe's correspondence, then 
its highest weight must be of the form $\lambda=
\sum_{j=1}^l 
\big(\frac{p-q}{2} +\nu_j\big) e_j$ 
where 
$$
\nu_1\geq \nu_2\geq \dots \geq\nu_p\geq \nu_{p+1}=\dots =\nu_{l-q}=0\geq \nu_{l-q+1} \geq \dots \geq \nu_l\,.
$$
\end{cor}
\begin{prf}
Only the last statement requires proof. 
We know from Lemma \ref{lemma:j0-not-p} that $j_0=p$. Hence the first line of 
\eqref{list-signs-mu-C} looks as follows:
\[
\mu_1+\delta-1>\mu_2+\delta-1> \dots >\mu_p+\delta-1\geq 0\,.
\]
Since
\[
\mu_j+\delta-1=\lambda_j+\rho_j+\delta-1=\lambda_j-\frac{p-q}{2}+p-j \qquad (1\leq j\leq p)\,,
\]
%\[
%\mu_p+\delta-1=\lambda_p+\rho_p+\delta-1=\lambda_p-\frac{p-q}{2}\,,
%\]
we see that
\[
\nu_j=\lambda_j-\frac{p-q}{2} \qquad (1\leq j\leq p)\,,
\]
satisfies 
\[
\nu_1\geq \nu_2\geq \dots \geq \nu_p\geq 0\,.
\]
By a similar analysis of the second line of \eqref{list-signs-mu-C}, the claim follows.
\end{prf}
In the proof of Theorem \ref{main thm for l>=l'} we will see that the condition on the highest weight of 
$\Pi$ is also sufficient for the nonvanishing of the intertwining distributions. 

Because of Corollary \ref{cor:j0-not-p}, we can restrict ourselves to the case $j_0=p$ when $\G=\Ug_l$. In this case, to simplify notation, we will write $s_0$ instead of $s_{0,p}$.
Hence
\begin{equation}
\label{s0-redefined}
s_0=1 \quad (\text{if $\Dc=\R$ or $\Ha$}) \quad \text{and} \quad s_0=s_{0,p} \quad 
(\text{if $\Dc=\C$})\,.
\end{equation}
Observe that this notation allows us to write
\begin{equation}
\label{poly-smu}
\prod_{j=1}^{l'} P_{a_{s_0,j},b_{s_0,j}}(\beta y_j)=
(2\pi)^{l'} \prod_{j=1}^{l'} P_{a_{s_0,j},b_{s_0,j},2\delta_j}(\beta y_j) \mathbb{I}_{\delta_j \R^+}(y_j)\,,
\end{equation}
which unifies \eqref{poly-smu-R-H} and \eqref{poly-smu-C}.

Suppose that $s\in W(\G,\h(\g))$ satisfies \eqref{smu-h''} and $j_0=p$ if $\Dc=\C$. Then 
\begin{equation}
\label{ss0inv}
ss_0^{-1}|_{\h''}=1 \qquad \text{and} \qquad ss_0^{-1}(\h)=\h\,.
\end{equation}
The condition $ss_0^{-1}(\h)=\h$ and the identification \eqref{the identification}, 
allow us to consider $ss_0^{-1}$ as isomorphisms of $\h'$. In the 
following lemma we prove that such a $s$ contributes to the right-hand side of \eqref{another intermediate cor, l>l' a0} if and only if 
$ss_0^{-1}\in W(\G',\h')$. Moreover, in this case, the contribution from $s$ agree with that of $s_0$.

\begin{lem}\label{lemma:contribution}
Let $l>l'$ and let $\mu$ and $s\in W(\G,\h(\g))$ satisfy \eqref{smu-h''} with $j_0=p$ if $\Dc=\C$. 
The integral 
\begin{equation}
\label{double-integral-smu}
\int_{\h'}\xi_{-s\mu}(\widehat{c_-}(x)) \ch^{d'-r-\iota}(x) \int_{\tau'(\reg{\hs1})} e^{iB(x,y)}F_{\phi}(y)\,dy\,dx
\end{equation}
is zero:
\begin{itemize}
\item[(a)] 
if $ss_0^{-1}|_{\h}$ acts by some sign changes, when $\Dc=\R$ or $\Ha$,
\item[(b)]  
if $ss_0^{-1}|_{\h}$ does not stabilize $\{J_1,\dots,J_p\}$ (and $\{J_{p+1},\dots,J_{l'}\}$), 
when $\Dc=\C$.
\end{itemize}

Equivalently, by identifying $\h$ and $\h'$ via \eqref{the identification}, 
the integral \eqref{double-integral-smu} is zero unless $ss_0^{-1}\in W(\G',\h')$. 
Moreover, \eqref{another intermediate cor, l>l' a0} becomes:
for any $\phi\in \Ss(\Wv)$
\begin{multline}
\label{another intermediate cor, l>l' a0-bis}
\int_{-\G^0} \check\Theta_\Pi(\t g)  T(\t g)\,dg(\phi)
= C \,\kappa_0
\check{\chi}_\Pi(\t{c}(0)) \int_{\tau'(\reg{\hs1})} 
\Big(\prod_{j=1}^{l'} P_{a_{s_0,j},b_{s_0,j}}(\beta y_j)\Big) 
e^{-\beta\sum_{j=1}^{l'} |y_j|}
F_{\phi}(y)\,dy\,,
\end{multline}
where $C$ is a non-zero constant which depends on the dual pair $(\G,\G')$.
\end{lem}
\begin{proof}
Let $\Dc=\R$ or $\Ha$. Suppose that $ss_0^{-1}(J_j)=-J_j$ for some $j\in \{1,\dots,l'\}$. 
Then $(s\mu)_j=-(s_0\mu)_j$. Thus $P_{a_{s,j},b_{s,j}}$ is supported in $\R^-$, and the support 
of \eqref{innerintegral-smu-gen} has a lower dimensional intersection with $\tau'(\hs1)$.

The case $\Dc=\C$ is similar: if $ss_0^{-1}(J_i)=J_j$ where $1\leq i\leq p<j\leq l'$, then $(s\mu)_j=(s_0\mu)_i$, which interchanges the $i$-th and $j$-th indices $a$ and $b$
of $s\mu$ and $s_0\mu$. The support of \eqref{innerintegral-smu-gen} has therefore a lower dimensional intersection with $\tau'(\hs1)$.

By the above and by identifying  $\h$ and $\h'$ via \eqref{the identification}, we can restrict the sum on the right-hand side of \eqref{another intermediate cor, l>l' a0} to the set of $s\in W(\G,\h(\g))$ such that $ss_0^{-1}|_\h\in W(\G',\h')$ and $ss_0^{-1}|_{\h''}=1$. Therefore, the sum can be  
parametrized by 
$W(\G',\h')$.  
By \eqref{innerintegral-smu-gen} and since $\sgn_{\g/\h(\g)}(ss_0^{-1})=\sgn_{\g'/\h'}(ss_0^{-1})$, we obtain 
that $\int_{-\G^0} \check\Theta_\Pi(\t g)  T(\t g)\,dg(\phi)$ is 
$\kappa_0 \check{\chi}_\Pi(\t{c}(0))$ times a constant multiple of 
\begin{multline*}
\sum_{s'\in W(\G',\h')}
\sgn_{\g'/\h'}(s') 
\int_{\h'}\xi_{-s's_0\mu}(\widehat{c_-}(x)) \ch^{d'-r-\iota}(x) \int_{\tau'(\reg{\hs1})} e^{iB(x,y)} F_{\phi}(y)\,dy\,dx\\
= \sum_{s'\in W(\G',\h')} \sgn_{\g'/\h'}(s') 
\int_{\tau'(\reg{\hs1})} 
\Big(\prod_{j=1}^{l'} P_{a_{s's_0,j},b_{s's_0,j}}(\beta y_j)\Big) 
e^{-\beta\sum_{j=1}^{l'} |y_j|}
F_{\phi}(y)\,dy\,.
\end{multline*}
Observe that  
$$
\prod_{j=1}^{l'} P_{a_{s's_0,j},b_{s's_0,j}}(\beta y_j)=
\prod_{j=1}^{l'} P_{a_{s_0,j},b_{s_0,j}}(\beta ({s'}^{-1}y)_j)
$$
because $s'\in 
W(\G',\h')$ permutes the indices $1\leq j\leq l'$.
Recall also that $F_\phi(y)$ transforms as the sign representation with respect to the action of $W(\G',\mathfrak{h}')$.
Formula \eqref{another intermediate cor, l>l' a0-bis} therefore follows. The new non-zero constant $C$ is the one appearing in 
\eqref{another intermediate cor, l>l' a0} times $
|W(\G',\h')|
$ times ${\rm \sgn}_{\g/\h(\g)}(s_0)$, which is equal to $1$ if $\Dc=\R$ or $\Ha$ and $(-1)^{q(l-l')}$ if $\Dc=\C$. 
\end{proof}

\noindent {\it Proof of Theorem \ref{main thm for l>=l'}}.
It remains to show that if the highest weight $\lambda$ of $\Pi$ satisfies the conditions (a) or (b), then the integral \eqref{integral-Thm3}, i.e. \eqref{another intermediate cor, l>l' a0-bis}, is nonzero.

By \eqref{poly-smu}, the function $\prod_{j=1}^{l'} P_{a_{s_0,j},b_{s_0,j}}(\beta y_j)$
has support equal to $\tau'(\hs1)$ and we can rewrite the right-hand side of 
\eqref{another intermediate cor, l>l' a0-bis} as a constant multiple of
\begin{equation}
\label{to-make-W-skew}
\kappa_0 \check{\chi}_\Pi(\t{c}(0))  \int_{\tau'(\reg{\hs1})} 
\Big(\prod_{j=1}^{l'} P_{a_{s_0,j},b_{s_0,j},2\delta_j}(\beta y_j)\Big) 
e^{-\beta\sum_{j=1}^{l'} |y_j|}
F_{\phi}(y)\,dy\,.
\end{equation}
By the $W(\G',\h')$-skew-invariance of $F_\phi$, we can replace the term
$$
\Big(\prod_{j=1}^{l'} P_{a_{s_0,j},b_{s_0,j},2\delta_j}(\beta y_j)\Big) 
e^{-\beta\sum_{j=1}^{l'} |y_j|}
$$
in the integral \eqref{to-make-W-skew}
by its $W(\G',\h')$-skew-invariant component 
\begin{equation}
\label{W-skew-invariant-component}
\Big(\frac{1}{|W(\G',\h')|} \;  \sum_{s'\in W(\G',\h')} \sgn_{\g'/\h'}(s') 
\prod_{j=1}^{l'} 
P_{a_{s_0,j},b_{s_0,j},2\delta_j} (\beta (s'y)_j) \Big)
e^{-\beta \sum_{j=1}^{l'} |y_j|}\,.
\end{equation}
Here we have used that $\sum_{j=1}^{l'} |(s'y)_j|=\sum_{j=1}^{l'} |y_j|$. Notice that 
$$\prod_{j=1}^{l'} 
P_{a_{s_0,j},b_{s_0,j},2\delta_j} (\beta (s'y)_j) 
=\prod_{j=1}^{l'} P_{a_{{s'}^{-1}s_0,j},b_{{s'}^{-1}s_0,j},2\delta_j} (\beta y_j)
$$ 
because  $W(\G',\mathfrak{h}')$ only permutes the $y$-coordinates for which the $\delta_j$'s have equal sign. Moreover, \eqref{W-skew-invariant-component} is non-zero because $P_{a_{{s'}^{-1}s_0,j},b_{{s'}^{-1}s_0,j},2\delta_j} (\beta y_j)$ is not $W(\G',\mathfrak{h}')$-invariant 
when $W(\G',\mathfrak{h}')\neq 1$. 
Indeed, the condition $\mu_1 > \mu_2 > \cdots > \mu_{l'}$ implies $b_1>b_2>\cdots>b_{l'}$
and $a_1<a_2<\cdots<a_{l'}$. If $W(\G',\mathfrak{h}')\neq 1$, then there are at least two indices 
$j\neq j'$ such that $\delta_j=\delta_{j'}$ and the corresponding factors in 
\eqref{W-skew-invariant-component} have different degrees. (If $b\geq 1$ then the degree of $P_{a,b,2}$ is $b-1$ and if $a\geq 1$ then that of $P_{a,b,-2}$ is $a-1$.)

By \eqref{W-skew-invariant-component}, the integral \eqref{to-make-W-skew} is a constant 
multiple of 
%%%
\begin{equation}
\label{integral-with-Phi-l>l'}
\kappa_0 \check{\chi}_\Pi(\t{c}(0))  \int_{\tau'(\reg{\hs1})} 
\Phi(y)\pi_{\g/\z}(y)
F_{\phi}(y)\,dy\,,
\end{equation}
%%%
where 
\begin{multline}
\label{Phi-l>l'-reminder}
\Phi(y)=\frac{
\sum_{s'\in W(\G',\mathfrak{h}')} \sgn_{\g'/\h'}(s') 
\prod_{j=1}^{l'} P_{a_{s_0,j},b_{s_0,j},2\delta_j} (\beta (s'y)_j)
}{\pi_{\g/\z}(y)} \; e^{-\beta \sum_{j=1}^{l'} |y_j|}\\
(w\in \reg{\hs1}, y=\tau'(w))\,.
\end{multline}

By \eqref{product of positive roots for g - bis} and \eqref{product of positive roots for g'/z'}, we see that there is a non-zero constant $C_\z$, depending of $(\G,\G')$,
such that
\begin{equation}
\label{pigz-pig'h'}
\pi_{\g/\z}(y)=C_\z \pi_{\g'/\h'}(y) \det(y)_{\Vv'}^\gamma \qquad (y=\tau(w)=\tau'(w),\, w\in \hs1)\,,
\end{equation}
where
\begin{equation}
\label{gamma}
\gamma=\begin{cases}
l-l' &\text{if $\Dc=\C$}\\
l-l'+\frac{1}{2} &\text{if $\Dc=\Ha$}\\
l-l'-\frac{1}{2} &\text{if $\Dc=\R$ and $\g=\mathfrak{so}_{2l}$}\\
l-l' &\text{if $\Dc=\R$ and $\g=\mathfrak{so}_{2l+1}$}\\
\end{cases}
\end{equation}
and $\det(g')_{\Vv'}$ denotes the determinant of $g'$ as an element of 
$\G'\subseteq \GL_\Dc(\Vv')$.
(See the remark after \eqref{E7bis} in Appendix \ref{appenD} for the case $\Dc=\Ha$.)

Recall from Remark \ref{rem:Weyl-invariant l>l'} that $W(\G',\h')=W(\K',\h')$, where $\K'$ is maximal compact in $\G'$. Split $\pi_{\g'/\h'}$ as a product of the compact and the noncompact positive roots:
$$
\pi_{\g'/\h'}(y)=\pi_{\k'/\h'}(y)\pi^{\text{nc}}_{\g'/\h'}(y)\,.
$$
Explicitly, 
$$
\pi_{\k'/\h'}(\sum_{j=1}^{l'} y_jJ'_j)=
\begin{cases}
\prod_{1\leq j<k\leq l'}i(- y_j+ y_k) & \text{if $\Dc=\R,\Ha$}\,,\\[.2em]
\prod_{1\leq j<k\leq p}i(-y_j+y_k) \prod_{1\leq j<k\leq q}i(-y_{p+j}+y_{p+k})& 
\text{if $\Dc=\C$}\,.
\end{cases}
$$
The polynomial in 
parenthesis in
\eqref{W-skew-invariant-component}
is $W(\G',\h')$-skew-invariant. Hence it is divisible by $\pi_{\k'/\h'}(y)$ and 
the fraction
\begin{equation}
\label{polynomial part}
\frac{\sum_{s'\in W(\G',\h')} \sgn_{\g'/\h'}(s') \prod_{j=1}^{l'} P_{a_{s_0,j},b_{s_0,j},2\delta_j}(\beta (s'y)_j)}{\pi_{\k'/\h'}(y)} 
\qquad (y\in \h')
\end{equation}
is a $W(\G',\h')$-invariant polynomial. 
Therefore $\Phi$ is a $W(\G',\h')$-invariant real-valued nonzero continuous function on 
$\tau'(\reg{\hs1})$.
Thus Proposition \ref{cor:sufficiently many Fphi} proves the equality \eqref{integral on W for l>=l'}
and shows that the integral 
\eqref{integral-with-Phi-l>l'}
does not vanish for suitably chosen $\phi\in C_c^\infty(\Wv)^\G$.
\qed

\begin{rem}
\label{rem:exponential}
Let us consider the term $e^{-\beta \sum_{j=1}^{l'} |y_j|}$ appearing in \eqref{Phi-l>l'-reminder}. 
Notice that for $w=\sum_{j=1}^{l'}w_j u_j\in \hs1$,
\[
\sum_{j=1}^{l'} |y_j|=\sum_{j=1}^{l'} |J_j'{}^*(\tau'(w))|=\sum_{j=1}^{l'} w_j^2=\sum_{j=1}^{l'} \delta_j J_j'{}^*(\tau'(w))
=\left(\sum_{j=1}^{l'} \delta_j J_j'{}^*\right)\circ\tau'(w)\,.
\]
This is a quadratic polynomial on $\hs1$, invariant under the Weyl group $W(\Sg,\hs1)$. Such a polynomial has no $\G\G'$-invariant extension to $\Wv$, unless $\G'$ is compact. 
Indeed, 
suppose $P$ is a real-valued $\G\G'$-invariant polynomial on $\Wv$ such that
\[
P(w)=\left(\sum_{j=1}^{l'} \delta_j J_j'{}^*\right)\circ\tau'(w) \qquad (w\in \hs1)\,.
\]
Then $P$ extends uniquely to a complex-valued $\G_\C\G'_\C$-invariant polynomial on the complexification $\Wv_\C$ of $\Wv$. Hence, by the Classical Invariant Theory, \cite[Theorems 1A and 1B]{HoweRemarks}
there is a $\G'_\C$-invariant polynomial $Q$ on $\g'_\C$ such that $P=Q\circ\tau'$. Hence,
\[
Q(\tau'(w))=P(w)=\left(\sum_{j=1}^{l'} \delta_j J_j'{}^*\right)\circ\tau'(w) \qquad (w\in \hs1)\,.
\]
Since $\tau'(\hs1)$ spans 
$\h'$, we see that the restriction of $Q$ to $\h'$ is
\[
Q|_{\h'}=\sum_{j=1}^{l'} \delta_j J_j'{}^*\in \h'_\C\,.
\]
Since $Q$ is $\G'_\C$-invariant, the restriction $Q|_{\h'}$ has to be invariant under the corresponding Weyl group. There are no linear invariants if $\G'=\Sp_{2l'}(\R)$ or $\Og^*_{2l'}$. Therefore $\G'=\Ug_{p,q}$, $p+q=l'$. But in this case the invariance means that all the $\delta_j$ are equal. Hence $\G'=\Ug_{l'}$ is compact.
In the case $\G'=\Ug_{l'}$, the sum of squares coincides with $\langle J(w), w\rangle$ for a positive complex structure $J$ on $\Wv$ which commutes with $\G$ and $\G'$ and therefore 
\begin{equation}\label{GaussianNonGaussian}
e^{-\beta\sum_{j=1}^{l'} |\delta_j J_j'{}^*(\tau'(w))|}
\end{equation}
extends to a Gaussian on $\Wv$. If $\G'$ is not compact then \eqref{GaussianNonGaussian} extends to a $\G\G'$-invariant function on $\Wv$, which is bounded but is not a Gaussian. 
\end{rem}

%%%%%%%%%%%%%%%%%%%%%%%%%%%%
%%
\section{\bf The special case $(\Og_{2l},\Sp_{2l'}(\R))$ with $l\leq l'$}
\label{The special case even}
Here we consider the case $(\G, \G')=(\Og_{2l},\Sp_{2l'}(\R))$ and suppose that the character $\Theta_\Pi$ is not supported in the preimage of the connected identity component $\wt{\G^0}$. 
This is equivalent to $\lambda_l=0$, where $\lambda$ is the highest weight of $\Pi$. 
The case $l>l'$ was considered in Theorem \ref{main theorem O2l for l>l'}. 
Since the 
dual pair $(\Og_2,\Sp_{2l'}(\R))$ was treated in section \ref{section:O2}, we will  
suppose in the sequel that 
$2\leq l\leq l'$.
  Recall the element $s\in\G$, \eqref{s}, with centralizer in $\h$ equal to $\h_s=\sum_{j=1}^{l-1} \R J_j$, and the spaces
\[
\V_{\overline 0, s}=\V_{\overline 0}^1\oplus \V_{\overline 0}^2\oplus  \cdots\oplus \V_{\overline 0}^{l-1}\oplus \R v_{2l}\,,
\ \ \  \V_s=\V_{\overline 0, s}\oplus \V_{\overline 1}\,.
\]
The corresponding dual pair is $(\G_s, \G'_s)=(\Og_{2l-1}, \Sp_{2l'}(\R))$ acting on the symplectic space $\Wv_s=\Hom(\V_{\overline 1},\V_{\overline 0, s})$.

The ordered basis 
$
v_1, v_2, ..., v_{2l-2}, v_{2l-1}, v_{2l}
$
of $\V_{\overline 0}$, leads to the identifications
\[
\End(\V_{\overline 0})=\M_{2l,2l}(\R)\,,\qquad \End(\V_{\overline 0, s})=\M_{2l-1,2l-1}(\R)\,.
\]
In these terms, the Cartan subgroup $\H\subseteq \G$ consists of the block diagonal matrices
\[
\begin{pmatrix}
r(\theta_1) &  &  & 0\\ 
 & \ddots & &\\
&  & r(\theta_{l-1}) & \\ 
0 &  & & r(\theta_{l})
\end{pmatrix}\,,
\]
with diagonal blocks
\[
r(\theta)=\begin{pmatrix}
\cos(\theta) &   \sin(\theta) \\ 
-\sin(\theta) & \cos(\theta)
\end{pmatrix}\,,\qquad (\theta\in\R)\,.
\]
Set
\begin{equation}\label{bullet}
h_\bullet=\begin{pmatrix}
r(\theta_1) &  &   0\\ 
 & \ddots & \\
0&  & r(\theta_{l-1}) 
\end{pmatrix}
\end{equation}
and let $\H_\bullet$ denote the group of all matrices \eqref{bullet}.
Then the centralizer $\H^s\subseteq \H$ of $s$ consists of the matrices
\[
\left(
\begin{array}{c|c}
h_\bullet & 0 \\ 
\hline
0 & 
\begin{matrix}
\epsilon & 0\\
0 & \epsilon
\end{matrix}\!\!
\end{array}
\right)\,,\qquad (\epsilon=\pm 1)\,.
\]
The connected component of the identity $(\H^s)^0\subseteq \H^s$ is the set of these matrices with $\epsilon=1$. 
The group $\G_s$ and its connected identity component $\G_s^0$ contain Cartan subgroups $\H_s^0\subseteq \G_s^0$ and $\H_s\subseteq \G_s$ consisting of matrices
\[
\Omat{1}
\quad \text{and}\quad 
\Omat{\epsilon}
\,,\qquad (\epsilon=\pm 1)\,,
\]
respectively. 

\begin{lem}\label{conjugacy classes in Gs}
Every element of the connected component $\G^0s$ is $\G$-conjugate to an element of $(\H^s)^0 s$.
\end{lem}
\begin{prf}
Fix an element $g\in \G$. As shown in 
\cite[page 114]{Curtis_book}, 
$g$ preserves a subspace of $\Vv$ of dimension $1$ or $2$. Hence $\Vv$ decomposes into a direct sum of $g$-irreducible subspaces of dimension $1$ or $2$, and the claim follows.
\end{prf}

Let $\xi_\nu$ denote the character of $\H_\bullet$ 
whose derivative at the identity is $\nu\in i\h_s^*$. 
In particular, for $h_\bullet$ as in \eqref{bullet}, 
\[
\xi_{e_j}(h_\bullet)=e^{-i\theta_j} \qquad (1\leq j\leq l-1)\,.
\]
(The negative sign in the exponent is due to fact that $e_j=-iJ_j^*$.)

The elements $e_j\pm e_k$ ($1\leq j<k\leq l-1$) and 
$2e_j$ ($1\leq j\leq l-1$) form a system of type $C_{l-1}$ which is dual to that of 
$((\g_s)_\C,(\h_s)_\C)$.
The corresponding $\rho$-function and the Weyl denominator 
are respectively
\begin{equation}\label{rhosd}
\rho^{\tC}_s=(l-1)e_1+(l-2)e_2 + \dots + e_{l-1}\,
\end{equation}
and
\begin{multline}\label{DeltasD}
\Delta^{\tC}_s
\Big(\!\Omat{1}\!\Big)=
\xi_{\rho^{\tC}_s}(h_\bullet) 
\prod_{1\leq j<k\leq l-1}(1-\xi_{e_k-e_j}(h_\bullet))(1-\xi_{-e_j-e_k}(h_\bullet)) \cdot \prod_{j=1}^{l-1}(1-\xi_{-2 e_j}(h_\bullet))\\ \qquad (h_\bullet\in \H_\bullet)\,.
\end{multline}

Observe that the Weyl group of the root system of type $C_{l-1}$ coincides with $W(\G_s^0, \h_s)$. It consists of all permutations and sign changes of the $e_1$, \dots, $e_{l-1}$. 
It acts on $\H_s^0=\Big\{ \Omat{1}; h_\bullet \in \H_\bullet\Big\}$ 
and hence on $\H_\bullet$.

The following two lemmas follow respectively from \cite[Theorems 2.5 and 2.6]{Wendt}.

%%\cite[Theorem 2.6]{Wendt}
\begin{lem}\label{Wendt1}
For any continuous $\G$-invariant function $f:\G^0s\to\C$,
\[
\int_{\G^0s} f(g)\,dg=
\frac{1}{|W(\G^0_s,\h_s)|}
\int_{\H_\bullet} f\Big(\Obigmat{h_\bullet}{1}{1}s\Big)
\left|\Delta^{\textup{\tC}}_s\Big(\!\Omat{1}\!\Big)\right|^2\,dh_\bullet\,,
\]
where $s=\Obigmat{1_\bullet}{1}{-1}$, see \eqref{s}. 
\end{lem}

Notice that the coverings
\[
\wt{\G^0 s}\to \G^0 s\,,\qquad \wt{\G^0}\to \G^0
\]
split (see Appendix \ref{appenD}). Hence we may choose continuous sections
\begin{equation}\label{Wendt21sections}
(\H^s)^0 s\ni hs\to \widetilde{hs}\in \wt{(\H^s)^0 s}\qquad \text{and}\qquad (\H^s)^0\ni h\to \t h\in \wt{(\H^s)^0}\,.
\end{equation}
\begin{lem}\label{Wendt2}
Consider the map 
\begin{equation*}
(\H^s)^0\ni h\to \widetilde{hs}\in \widetilde{(\H^s)^0 s}
\end{equation*}
obtained by composing the multiplication by $s$ and the fixed continuous section.
Then
\begin{equation}\label{Wendt22}
\Theta_\Pi\Big( 
\stackrel{\resizebox{13mm}{1.4mm}{$\sim$}}{\Obigmat{h_\bullet}{1}{1} s}\Big)
= 
D_\Pi
\Theta_{\Pi_s}
\Big(\!\Omat{1}\!\Big)
\qquad (h_\bullet\in\H_\bullet)\,,
\end{equation}
where
\begin{equation}\label{Wendt23}
\Theta_{\Pi_s}
\Big(\!\Omat{1}\!\Big)
= \frac{\sum_{t\in W(\G^0_s,\h_s)} \sgn_{\g_s/\h_s}(t)\xi_{t(\lambda+\rho^{\textup{\tC}}_s)}(h_\bullet)}{
\Delta^{\textup{\tC}}_s
\Big(\!\Omat{1}\!\Big)}
\qquad (h_\bullet\in\H_\bullet)\,,
\end{equation}
$\lambda$ is the highest weight of $\Pi$ 
(recall that $\lambda_l=0$), the sign character $\sgn_{\g_s/\h_s}(t)$ is defined by
\[
\Delta^{\textup{\tC}}_s
\Big(t\Omat{1}\!\Big)=
\sgn_{\g_s/\h_s}(t)\Delta^{\textup{\tC}}_s
\Big(\!\Omat{1}\!\Big)\qquad (t\in W(\G^0_s,\h_s))\,,
\]
and  
\begin{equation}\label{Wendt24} 
D_\Pi=\pm 1\,.
\end{equation}
\end{lem}
%%%%%%%%%%%%%%%%%%%%%%%%%%%%ùù
%%
\begin{lem}\label{3}
For $\phi\in \Ss(\Wv)$,
\begin{align*}
\int_{\G^0s} &\check\Theta_\Pi(\t g) T(\t g)(\phi)\,dg\\
&=\frac{1}{|W(\G_s^0,\h_s)|} \int_{\H_\bullet} 
\check\Theta_\Pi 
\Big(\wt{\Obigmat{h_\bullet}{1}{-1}}\Big)
\left|
\Delta^{\textup{\tC}}_s\Big(\!\Omat{1}\!\Big)
\right|^2 
T_s\wt{\Omat{-1}}
\left(\phi^\G|_{\Wv_s}\right)\,dh_\bullet\,.
%%\\
\end{align*}
%%%
\end{lem}
\begin{proof}
Clearly, the integral on the left-hand side does not change if we replace $\phi$ by $\phi^\G$. Hence we may assume that $\phi=\phi^\G$. By Lemma \ref{Wendt1}, the left-hand side multiplied by 
$|W(\G_s^0,\h_s)|$ is equal to
\begin{equation}\label{5}
\int_{\H_\bullet}
\check\Theta_\Pi\Big( 
\stackrel{\resizebox{13mm}{1.4mm}{$\sim$}}{\Obigmat{h_\bullet}{1}{1} s}\Big)
\left|\Delta^{\textup{\tC}}_s\Big(\!\Omat{1}\!\Big)\right|^2
T\Big( 
\stackrel{\resizebox{13mm}{1.4mm}{$\sim$}}{\Obigmat{h_\bullet}{1}{1} s}\Big)
(\phi)\,dh_\bullet\,.
\end{equation}
Apply Lemma \ref{lemma:WandT} to the decomposition $\Wv=\Wv_s\oplus \Wv_s^\perp$.
For $h\in (\H^s)^0$, 
$$
hs=\Obigmat{h_\bullet}{1}{1}\Obigmat{1}{1}{-1}=\Obigmat{h_\bullet}{1}{-1}\,.
$$
So 
$$
hs|_{\Wv_s}=\Omat{-1} \quad \text{and} \quad hs|_{\Wv_s^\perp}=1|_{\Wv_s^\perp}\,.
$$
Hence $(hs-1)|_{\Wv_s}$ maps onto $\Wv_s$ and $(hs-1)|_{\Wv_s^\perp}=0$. This
shows that  the restriction of $\mu_\Wv$ to $(hs-1)\Wv$
is $\mu_{\Wv_s}\otimes \delta_0$, where $\delta_0$ is the Dirac delta on 
$\Wv_s^\perp$. 
Therefore, for an appropriate choice of the lift of the element $\Omat{-1}$ on the right-hand side, 
\[
T\Big(\stackrel{\resizebox{13mm}{1.4mm}{$\sim$}}{\Obigmat{h_\bullet}{1}{1} s}\Big)
(\phi)=T_s\wt{\Omat{-1}}(\phi|_{\Wv_s})\,.
\]
Thus, \eqref{5} is equal to
\begin{equation}\label{6}
\int_{\H_\bullet}
\check\Theta_\Pi\Big( 
\stackrel{\resizebox{13mm}{1.4mm}{$\sim$}}{\Obigmat{h_\bullet}{1}{1} s}\Big)
\left|\Delta^{\textup{\tC}}_s\Big(\!\Omat{1}\!\Big)\right|^2
T_s\wt{\Omat{-1}}(\phi|_{\Wv_s})\,dh_\bullet\,.
\end{equation}
The lemma follows from \eqref{6}. 
\end{proof}

\begin{lem}\label{10}
Let $\mu^{\textup{\tC}}=\lambda+\rho_s^{\textup{\tC}}$. Then,
for $\phi\in \Ss(\Wv)$,
\begin{align*}
\int_{\G^0s} &\check\Theta_\Pi(\t g) T(\t g)(\phi)\,dg\\
&=D_\Pi \int_{\H_\bullet} 
\xi_{-\mu^{\textup{\tC}}}(h_\bullet)
\Delta^{\textup{\tC}}_s\Big(\!\Omat{1}\!\Big)
T_s\wt{\Omat{-1}}
\left(\phi^\G|_{\Wv_s}\right)\,dh_\bullet\,,
\end{align*}
where $\xi_{-\mu^{\textup{\tC}}}(h_\bullet)$ makes sense because $\lambda_l=0$. 
\end{lem}
\begin{proof}
This follows from Lemma \ref{3}. Indeed, 
notice that
\[
\check\Theta_\Pi \wt{\Obigmat{h_\bullet}{1}{-1}}
=\Theta_\Pi \wt{\Obigmat{h_\bullet^{-1}}{1}{-1}}
=\Theta_\Pi \Big( 
\stackrel{\resizebox{13mm}{1.4mm}{$\sim$}}{\Obigmat{h_\bullet^{-1}}{1}{1} s}\Big)\,.
\]
Hence \eqref{Wendt22} and \eqref{Wendt23} show that
\begin{align*}
\check\Theta_\Pi \wt{\Obigmat{h_\bullet}{1}{-1}}
&=D_\Pi
\Theta_{\Pi_s}{\Obigmat{h_\bullet^{-1}}{1}{1}}\\
&=D_\Pi
\frac{\sum_{t\in W(\G_s^0, \h_s)}
\sgn_{\g_s/\h_s}(t) \xi_{t^{-1}\mu^{\tC}} (h_{\bullet}^{-1})}{
\Delta^{\textup{\tC}}_s\Big(\!\Omatinv{1}\!\Big)}\,.
\end{align*}
Furthermore,
\[
\Delta^{\textup{\tC}}_s\Big(\!\Omatinv{1}\!\Big)
=\overline{\Delta^{\textup{\tC}}_s\Big(\!\Omat{1}\!\Big)} 
\]
and for $t\in W(\G_s^0, \h_s)$,
\[
\Delta^{\textup{\tC}}_s\Big(t\Omat{1}\!\Big)
=\sgn_{\g_s/\h_s}(t) 
\Delta^{\textup{\tC}}_s\Big(\!\Omat{1}\!\Big)\,.
\]
Therefore
\[
\check\Theta_\Pi \wt{\Obigmat{h_\bullet}{1}{-1}} 
\left|\Delta^{\textup{\tC}}_s\Big(\!\Omat{1}\!\Big)\right|^2
=D_\Pi
\sum_{t\in W(\G_s^0, \h_s)} \xi_{t^{-1}\mu^{\tC}} (h_{\bullet}^{-1})
\Delta^{\textup{\tC}}_s\Big(t\Omat{1}\!\Big)\,.
\]
Notice that
\[
\xi_{t^{-1}\mu^{\tC}} (h_{\bullet}^{-1})=\xi_{-\mu^{\tC}} (th_{\bullet})
\]
and since $\phi^\G$ is $\G$-invariant,
\[
T_s\wt{\Omat{-1}}
\left(\phi^\G|_{\Wv_s}\right)
=T_s\left(
\stackrel{\resizebox{-20mm}{1.4mm}{$\backsim$}}{t\Omat{-1} }\right)
\left(\phi^\G|_{\Wv_s}\right)=
T_s\wt{\Omatplus{th_\bullet}{-1}}
\left(\phi^\G|_{\Wv_s}\right)
\,.
\]
Therefore
\begin{align*}
\int_{\H_\bullet} \check\Theta_\Pi &\wt{\Obigmat{h_\bullet}{1}{-1}} 
\left|\Delta^{\textup{\tC}}_s\Big(\!\Omat{1}\!\Big)\right|^2
T_s\wt{\Omat{-1}}
\left(\phi^\G|_{\Wv_s}\right)\,dh_\bullet\\
&= D_\Pi
\int_{\H_\bullet} 
\sum_{t\in W(\G_s^0, \h_s)} \xi_{-\mu^{\tC}} (th_{\bullet})
\Delta^{\textup{\tC}}_s\Big(t\Omat{1}\!\Big)
T_s\wt{\Omatplus{th_\bullet}{-1}}
\left(\phi^\G|_{\Wv_s}\right) \,dh_\bullet\\
&=
|W(\G_s^0, \h_s)| D_\Pi\int_{\H_\bullet} 
\xi_{-\mu^{\tC}} (h_{\bullet})
\Delta^{\textup{\tC}}_s\Big(\Omat{1}\!\Big)
T_s\wt{\Omat{-1}}
\left(\phi^\G|_{\Wv_s}\right) \,dh_\bullet\,.
\end{align*}
\end{proof}
Consider the 
Cayley transform $c_\bullet:\h_\bullet\to \H_\bullet$ and
the (modified) Cayley transform $c_\odot:\h_s\to (\H^s)^0$ defined by
\begin{align}\label{CayleyBullet}
c_\bullet\begin{pmatrix}
x_1J_1& &0\\
 & \ddots & \\
0 & & x_{l-1}J_{l-1}
\end{pmatrix} &= \begin{pmatrix}
c(x_1J_1)& &0\\
 & \ddots & \\
0 & & c(x_{l-1}J_{l-1})
\end{pmatrix}\\
c_\odot
\Obigdiag{x_1J_1}{x_{l-1}J_{l-1}}{0}
&=
\Obigdiag{c(x_1J_1)}{c(x_{l-1}J_{l-1})}{1}\,, \quad \text{i.e. $c_\odot=c_\bullet \times\exp$} \,.
\end{align}
Notice that $c_\odot$ differs from the usual Cayley transform 
$c_s$ on $\h_s$, defined at the beginning of section \ref{Intertwining distributions as an integral over g},
 for which $c_s(\diag(x_1J_1,\dots,x_{l-1}J_{l-1},0))=\diag(c(x_1J_1), \dots, c(x_{l-1}J_{l-1}), -1)$.
Let $j_{\h_s}$ denote the Jacobian of the map $c_\odot$. 
Set
\[
\pi_{\g_s/\h_s}^{\tC}
\Obigdiag{x_1J_1}{x_{l-1}J_{l-1}}{0}
 =
 \prod_{1\leq j<k\leq l-1}
 (-x_j^2+x_k^2)\cdot \prod_{j=1}^{l-1} (-2ix_j)\,.
\]
\begin{lem}\label{partial formulas}
There are constants $A$ and $D$ such that for $x=\sum_{j=1}^{l-1} x_jJ_j\in \h_s$, 
\begin{align}
\Delta_s^{\textup{\tC}}(c_\odot(x))
&=A\,\pi_{\g_s/\h_s}^{\textup{\tC}}(x)\prod_{j=1}^{l-1}(1+x_j^2)^{-l+1}\,,\label{partial formulas 1}\\
\Theta_s\left(
\wt{\Omatplus{c_\bullet(x)}{-1}}
\right)
&= \left(\frac{i}{2}\right)^{(2l-1)l'} 
2^{l'} \prod_{j=1}^{l-1}(1+x_j^2)^{l'}\,,
\label{partial formulas 2}\\
j_{\h_s}(x)&=\prod_{j=1}^{l-1} 2 (1+x_j^2)^{-1}\label{partial formulas 3}
\end{align}
and hence
\begin{equation}\label{partial formulas 4}
\Delta_s^{\textup{\tC}}(c_\odot(x))
\Theta_s\left(
\wt{\Omatplus{c_\bullet(x)}{-1}}
\right)j_{\h_s}(x)
=D\,\pi_{\g_s/\h_s}^{\textup{\tC}}(x) \prod_{j=1}^{l-1}(1+x_j^2)^{l'-l}\,.
\end{equation}
\end{lem}
\begin{proof}
Part \eqref{partial formulas  1} may be verified via the argument used in Appendix \ref{appenB}, but easier -- without the square roots. Formula \eqref{partial formulas  2}  follows from \eqref{eq:tildec1}, and \eqref{partial formulas  3} from Appendix \ref{appenA} applied to the group $\SO_2$.
\end{proof}
\begin{lem}\label{10cayley}
With the notation of Lemmas \ref{10} and \ref{partial formulas},
\begin{multline*}
\int_{\G^0s} \check\Theta_\Pi(\t g) T(\t g)(\phi)\,dg
= C D_\Pi \check{\chi}_\Pi(\t{c}(0))
\int_{\h_\bullet} 
\xi_{-\mu^{\textup{\tC}}}(c_-(x)) \prod_{j=1}^{l-1}(1+x_j^2)^{l'-l} \\
\times \pi_{\g_s/\h_s}
\Big(\!\Omatplus{x}{0}\!\Big)
\int_{\Wv_s}
\chi_{\Omatplus{x}{0}}(w)\phi^\G(w)\,dw\,dx\,,
\end{multline*}
where $\check{\chi}_\Pi$ is the central character of $\Pi$ and $D_\Pi=\pm 1$.
\end{lem}
\begin{proof}
We start with the formula of Lemma \ref{10}, use the equality
\[
\xi_{-\mu^{\tC}}(h_\bullet)=\xi_{-\mu^{\tC}}((-1)_\bullet)\xi_{-\mu^{\tC}}(-h_\bullet)\,,
\] 
apply the change of variables, $h_\bullet=c_\bullet(x)$ and use the formula \eqref{partial formulas 4},
noticing that $\pi_{\g_s/\h_s}^{\textup{\tC}}$ is a constant multiple of $\pi_{\g_s/\h_s}$. 
 Here $c_\bullet:\h_\bullet\to \H_\bullet$, so $c_\bullet(0)=(-1)_\bullet$.

It remains to prove that $\xi_{-\mu^{\tC}}((-1)_\bullet)$ is a constant multiple of the central character of $\Pi$ evaluated at $\t{c}(0)$. For this, let $v\neq 0$ be a highest 
vector of $\Pi$. 
For now, let us denote by $\xi^{\wt{\H}}_{-\lambda}$ and $\xi^{\H}_{-\lambda}$
the characters defined by $\lambda$ on $\wt\H$ and $\H$, respectively.
 Then 
$\xi_{-\lambda}^{\wt{\H}}(\t{c}(0))=\xi_{-\lambda}^{\H}(c(0))$ 
because $\lambda$ is integral; see Appendix \ref{appenE}. 
Hence 
$
\xi_{-\lambda}^{\H}(c(0))v
=\Pi(\t{c}(0))v=\check{\chi}_\Pi(\t{c}(0))v$. This implies 
that 
$\check{\chi}_\Pi(\t{c}(0))=
\xi_{-\lambda}^{\H}(c(0))
$. 
Since $\lambda_l=0$, 
$$
\xi_{-\lambda}^{\H}(c(0))=\xi_{-\lambda}^{\H}(-I_{2l})=
\xi_{-\lambda}^{\H}\Omatplus{(-1)_\bullet}{I_2}=\xi_{-\lambda}((-1)_\bullet)
=\xi_{\rho_s^{\rm C}}((-1)_\bullet)\xi_{-\mu^{\rm C}}((-1)_\bullet)\,,
$$
where $\xi_{\rho_s^{\rm C}}((-1)_\bullet)=\pm 1$. 
\end{proof}
Recall from \eqref{the form B} the symmetric bilinear form
\[
B(x_\bullet,y_\bullet)\qquad (x_\bullet,y_\bullet\in\h_\bullet)\,.
\]
\begin{cor}\label{10cayleycor}
There is a constant $C$ depending only on the dual pair and a constant 
$D_\Pi=\pm 1$ distinguishing the representations $\Pi$ and $\Pi \otimes \det$, 
such that
\begin{align*}
\int_{\G^0s} &\check\Theta_\Pi(\t g) T(\t g)(\phi)\,dg\\
&= C D_\Pi 
\check{\chi}_\Pi(\t{c}(0)) 
\int_{\h_\bullet} \int_{\h_\bullet}
\prod_{j=1}^{l-1}(1+ix_j)^{\mu^{\textup{\tC}}_j+l'-l} (1 -ix_j)^{-\mu^{\textup{\tC}}_j+l'-l}  
e^{iB(x_\bullet,y_\bullet)}F_{\phi^\G|_{\Wv_s}}
\left(\begin{array}{c|c}
y_\bullet & 0\\
\hline
0 & 0
\end{array}
\right)
\,dy_\bullet\,dx_\bullet\,.
\end{align*}
\end{cor}
\begin{proof}
By Lemma \ref{reduction to ss-orb-int},
\[
\pi_{\g_s/\h_s}(x) \int_{\Wv_s} \chi_{x}(w)\phi^\G(w)\,dw\\
=C\int_{\h_\bullet}e^{iB(x_\bullet,y_\bullet)}F_{\phi^\G|_{\Wv_s}}
\left(\begin{array}{c|c}
y_\bullet & 0\\
\hline
0 & 0
\end{array}
\right)
\,dy_\bullet\,.
\]
By the proof of Lemma \ref{ximuchexplicit},
\[
\xi_{-\mu^{\tC}}(c_-(x))=\prod_{j=1}^{l-1}(1+ix_j)^{\mu_j^{\tC}}(1-ix_j)^{-\mu_j^{\tC}}\,.
\]
Hence the formula follows from Lemma \ref{10cayley}.
\end{proof}
\noindent{\it Proof of Theorem \ref{main thm for l<l', special}}.
To prove \eqref{main thm for l<l' a, special}, we proceed as in the proof of Theorem \ref{main thm for l<l'}, using Corollary \ref{10cayleycor}. 
%%
%%%%%%%%%%%%%%%%%%%%%%%%%%%%%%%%%
\section{\bf The special case $(\Og_{2l+1},\Sp_{2l'}(\R))$ with $1\leq l\leq l'$}\label{The special case, odd 1}
Recall the decomposition \eqref{2l+1 smaller than 2l' special}.
As in the previous section, we denote the objects corresponding to $\Wv_s$ by the subscript $s$, for instance $\Theta_s$ and $T_s$. 
Similarly, we denote the objects corresponding to $\Wv_s^\perp$ 
by the substrict $\perp$, for instance $\Theta_\perp$ and $T_\perp$. . 

If $\H$ is our Cartan subgroup of $\G$, then the elements of connected identity component $\H^0$ 
are of the form $h=\Omatplus{h_\bullet}{1}$ with  $h_\bullet$ in the Cartan subgroup $\H_s$  of
$\G_s$; see \eqref{bullet}.
%and the Lie algebras $\g$ and $\g_s$ share the same Cartan subalgebra $\h=\h_s$.%
Since any element 
$h\in\H^0$ acts trivially on $\Wv_s^\perp$, we see that
$
(h-1)\Wv=(h_\bullet-1)\Wv_s
$.
Hence,
\[
\mu_{(h-1)\Wv}=\mu_{(h_\bullet-1)\Wv_s}\otimes \delta_0\,,
\]
where $\delta_0$ is the Dirac delta on $\Wv_s^\perp$.  

\begin{lem}
There is a choice of coverings $\wt{\H^0}\to \H^0$ and $\wt{\H_\bullet}\to \H_\bullet$ such that 
the map $\wt{\H_\bullet} \in \wt{h_\bullet} \to \wt{\Omatplus{h_\bullet}{1}}\in \wt{\H^0}$ is a Lie group 
isomorphism and 
\begin{equation}\label{Tsodd, special odd 1}
\Theta(\t h)=\Theta_s(\wt{h_\bullet})\quad\text{and} \quad T(\t h)= T_s(\wt{h_\bullet}) \otimes \delta_0 \qquad (h\in \H^0)\,. 
\end{equation}
\end{lem}
\begin{proof}
We apply Lemma \ref{lemma:WandT} to the decomposition $\Wv=\Wv_s\oplus \Wv_s^\perp$.
Then $h|_{\Wv_s}=h_\bullet$ and $h|_{\Wv_s^\perp}=1$.
Choose $\t{1}$ such that $\t{1}=1_{\wt{\Sp}(\Wv_s^\perp)}$ (the identity of the metaplectic group). Hence 
$T_{\perp}(\t{1})=\delta_0$ and, by Lemma \ref{lemma:WandT},
$$
T(\t h)= \frac{\chi_+(\t{h})}{\chi_{s,+}(\wt{h_\bullet})\chi_{\perp,+}(\t{1})} \; 
T_s(\wt{h_\bullet}) \otimes \delta_0\,,
$$
where $\chi_+$, $\chi_{s,+}$ and $\chi_{\perp,+}$ are defined according to \eqref{chi+onSp} 
for $\Wv$, $\Wv_s$ and $\Wv_s^\perp$, respectively.

We now show that $\chi_{+}(\t{h})=\chi_{s,+}(\wt{h_\bullet})$ and that $\chi_{\perp,+}(\t{1})=1$, 
which will complete the proof of the second equality in \eqref{Tsodd, special odd 1}. 

We choose complete polarizations 
$$
\Wv_s=\X_1 \oplus \Y_1 \quad\text{and} \quad \Wv_s^\perp=\X_2 \oplus \Y_2
$$
preserved by $\G=\Og_{2l+1}$. Then 
$$
\Wv=\X\oplus \Y \qquad (\X=\X_1 \oplus \X_2, \, \Y=\Y_1  \oplus \Y_2)
$$
is a complete polarization preserved by $\G$. The double covers can be realized as
\begin{align*} 
\wt{\G}&=\{(g,\zeta) \in \G\times \C^\times; (\det g)_\X=(\det g)^{l'}=\zeta^2\}\,,\\
\wt{\G|_{\Wv_s}}&=\{(g,\zeta) \in \G|_{\Wv_s}\times \C^\times; (\det g)_{\X_1}=(\det g)^{l'}=\zeta^2\}\,.
\end{align*}
(See Appendix \ref{appenD}.)
Furthermore, by \cite[Proposition 4.28]{AubertPrzebinda_omega}, 
\begin{equation}
\label{chi-plus-forO2l+1}
\frac{\Theta(\t g)}{|\Theta(\t g)|}=\frac{\det^{-1/2}(\t g)}{|\det^{-1/2}(\t{g})|} \qquad (\t{g}\in \wt{\G})\,.
\end{equation}
Since for $h\in \H^0$
$$
(\det h)_\X=(\det h|_{\X_1})_{\X_1}\,,
$$
we see that we may choose the cover $\wt{\H_\bullet}$ adjusted to $\wt{\H^0}$ so that 
\begin{equation}
\label{equality-chi-plus-s}
\chi_+(\t h)=\chi_{s,+}(\wt{h_\bullet}) \qquad (h=\Omat{1})\,.
\end{equation}
As recalled on page \pageref{notation-Od-irreps}, for any 
$\t g$ in the metaplectic group such that $g$ preserves the decomposition 
$\Wv_s^\perp=\X_2 \oplus \Y_2$, the restriction of the Weil representation acts by
$$
\omega(\t g)f(x)=\det(\t g)^{-1/2}f(g^{-1}x) \qquad (x\in \X_2)\,.
$$
Applying this equality to $\t{1}\in \wt{\G}|_{\Wv_s^\perp}$, we see that $\det(\t 1)^{-1/2}=1$.
Thus \eqref{chi-plus-forO2l+1} implies that $\chi_{\perp,+}(\t{1})=1$. This proves
the second equality in \eqref{Tsodd, special odd 1}. 

To prove that $\Theta(\t h)=\Theta_s(\wt{h_\bullet})$, observe first that $\Theta^2(1)=1$ 
by \cite[Definition 4.16]{AubertPrzebinda_omega}. Therefore $|\Theta_\perp(\t 1)|=1$. 
As shown in the proof of Lemma \ref{lemma:WandT}, this implies that $|\Theta(\t h)|=|\Theta_s(\wt{h_\bullet})|$. So the claim follows from \eqref{equality-chi-plus-s}. 
\end{proof}

\noindent{\it Proof of Theorem \ref{main thm for l<l', special odd 1}}. 
As in \eqref{CayleyBullet}, consider the Cayley transform $c_\bullet:\h_s\to \H_s$ and 
the modified Cayley transform 
$c_\odot:\h_s\to \H^0$, defined by
\begin{equation}
\label{Cayley-hs}
c_\odot(\diag(x_1J_1, \dots, x_lJ_l,0))=
\diag(v_1,\dots,v_l,1), \quad v_j=\frac{
-ix_j+1}{-ix_j-1}
 \qquad (x_j\in\R, 1\leq j\leq l)\,,
\end{equation}
i.e. $c_\odot=c_\bullet \times \exp$.
See Appendix \ref{appenB} for the above realization of $\H^0$.
Notice also that $W(\G,\h)=W(\G_s,\h_s)$.

By \eqref{Tsodd, special odd 1} and since 
$c_\bullet(\h_s)$
is dense in $\H_s$,
\begin{align}
\label{main thm for l<l' a, special odd 11}
&\int_{\G^0}\check\Theta_\Pi(\t g) T(\t g)(\phi)\,dg\\
&=\frac{1}{|W(\G_s,\h_s)|}\int_{\H_s}\check\Theta_\Pi\Big(
\widetilde{\Omat{1}}\Big)
\Delta\Big(\widehat{\Omat{1}}\Big)\overline{\Delta\Big(\widehat{\Omat{1}}\Big)} 
T_s(\t h_\bullet)(\phi^\G|_{\Wv_s})\,dh_\bullet \nn\\  
&=\frac{4^l}{|W(\G_s,\h_s)|}\int_{\h_s}
\check\Theta_\Pi\Big(\widetilde{\Omatplus{c_\bullet(x)}{1}}\Big)
\overline{\Delta\Big(\widehat{\Omatplus{c_\bullet(x)}{1}}\Big)}
\Delta\Big(\widehat{\Omatplus{c_\bullet(x)}{1}}\Big)
\Theta_s(\widetilde{c_\bullet}(x)) \nn\\
&\quad\times
\int_{\Wv_s}\chi_x(w)\left(\phi^\G|_{\Wv_s}\right)(w)\,dw \cdot \ch^{-2}(x) \; dx
\,,
\end{align}
where the jacobian of the map $c_\odot:\h_s\to \H^0$
is computed using Appendix \ref{appenA} for $\G=\SO_2$.
As shown in Appendix \ref{appenA}, the Weyl group of $(\Spin_{2l+1}, 
\widehat{\H^0})$ is isomorphic to the Weyl group of $(\SOg_{2l+1}, \H^0)$
 and the covering
$\widehat{\H^0}\to\H^0$ intertwines the action of these groups. As before, we 
denote both Weyl groups by 
$W(\G,\h)$. For every $t\in W(\G,\h)$ and $x\in \h$, we have $t c_\odot (x)=c_\odot(tx)$.
Indeed, a permutation acts on $c_\odot (x)$
by permuting the coordinates of $x$, and 
a sign change $\varepsilon=\pm 1$ acts on each coordinate by 
$$
\varepsilon: v=\frac{-ix+1}{-ix-1} \to  v^\varepsilon
=\frac{-i\varepsilon x+1}{-i \varepsilon x-1} \qquad 
(x\in \R)\,.
$$
Therefore, 
$$
t\, \widehat{c_\odot}(x)=\widehat{c_\odot(tx)} \qquad (x\in \h, \, 
x_j\neq 0, \, 1\leq j\leq l)\,.
$$
Consequently, if $\mu$ is the Harish-Chandra parameter of $\Pi$, then 
\begin{equation}
\label{xi_mu c_bullet}
\xi_{-t\mu}(\widehat{c_\odot(x)})=\xi_{-\mu}(t\,\widehat{c_\odot}(x))=\xi_{-\mu}(\widehat{c_\odot(tx)}) \qquad (t\in W(\G,\h), \, x\in \h, \, 
x_j\neq 0, \, 1\leq j\leq l)\,.
\end{equation}

For $x$ as in \eqref{xi_mu c_bullet},
we now proceed as in Lemma \ref{general formula for the int distr}:

\begin{align*}
\check\Theta_\Pi\Big(\widetilde{\Omatplus{c_\bullet(x)}{1}}\Big)
\overline{\Delta\Big(\widehat{\Omatplus{c_\bullet(x)}{1}}\Big)}&=
\Theta_\Pi\Big(\widetilde{\Omatplus{c_\bullet(x)}{1}}^{-1}\Big)
\Delta\Big(\widehat{\Omatplus{c_\bullet(x)}{1}}^{-1}\Big)\\
&=\sum_{t\in W(\G,\h)}\sgn_{\g/\h}(t)\xi_{-t\mu}\Big(\widehat{\Omatplus{c_\bullet(x)}{1}}\Big)\\
&=\sum_{t\in W(\G,\h)}\sgn_{\g/\h}(t)\xi_{-\mu}\Big(\widehat{t\Omatplus{c_\bullet(x)}{1}}\Big)\\
&=\sum_{t\in W(\G_s,\h_s)}\sgn_{\g/\h}(t)\xi_{-\mu}(\widehat{c_\bullet}(tx))\,,
\end{align*}
and 
$$
\Theta_s(\widetilde{c_\bullet}(x)) \left(
\int_{\Wv_s}\chi_x(w)\left(\phi^\G|_{\Wv_s}\right)(w)\,dw \right) \ch^{-2}(x) 
$$
is a $W(\G_s,\h_s)$-invariant function of $x\in \h_s$. 
Hence \eqref{main thm for l<l' a, special odd 11} is a constant multiple of
\begin{align}
\label{intermediate for O2l+1, l smaller}
&\frac{4^l}{|W(\G_s,\h_s)|} 
\sum_{t\in W(\G_s,\h_s)} \sgn_{\g/\h}(t) \int_{\h_s} \xi_{-\mu}(\widehat{c_\bullet}(tx))
\Delta\Big(\widehat{\Omatplus{c_\bullet(x)}{1}}\Big)
\Theta_s(\wt{c_\bullet}(x)) \nn\\
&\qquad\qquad\times \left(\int_{\Wv_s}\chi_x(w)\left(\phi^\G|_{\Wv_s}\right)(w)\,dw\right) \ch^{-2}(x)  \; dx \nn\\
&\qquad=4^l \int_{\h_s} \xi_{-\mu}(\widehat{c_\bullet}(x))
\Delta\Big(\widehat{\Omatplus{c_\bullet(x)}{1}}\Big)
\Theta_s(\wt{c_\bullet}(x))\frac{1}{\pi_{\g_s/\h_s}(x)} \ch^{-2}(x) \nn\\
&\qquad\qquad\times\left(\pi_{\g_s/\h_s}(x)\int_{\Wv_s}\chi_x(w)\left(\phi^\G|_{\Wv_s}\right)(w)\,dw \right) dx\,.
\end{align}

Appendix \ref{appenB}, \eqref{eq:tildec1} and \eqref{ch explicit}
show that there is a constant $C_1$ such that 
\begin{multline}
\label{Deltaandpi-c}
\Delta\Big(\widehat{\Omatplus{c_\bullet(x)}{1}}\Big)
\Theta_s(\wt{c_\bullet}(x))
\frac{1}{\pi_{\g_s/\h_s}(x)} \, \ch^{-2}(x)
%&=C_1\ch^{-2l+1}(x)\,\ch^{2l'}(x)\,\ch^{-2}(x) \nn\\
=C_1\ch^{2l'-2l-1}(x) \,
\prod_{j=1}^l {\rm sgn}(x_j)\\
\qquad (x\in \h,\, x=\sum_{j=1}^l x_j J_j,\, x_j\neq 0)
\,. 
\end{multline}
By Lemma \ref{reduction to ss-orb-int}, there is a constant $C_2$ such that
\begin{equation} \label{t(c(x))(phi)-orbital-integral-restrictions}
\pi_{\g_s/\h_s}(x)\int_{\Wv_s}\chi_{x}(w)\left(\phi^\G|_{\Wv_s}\right)(w)\,dw=C_2\int_{\h_s}
e^{iB(x,y)} F_{\phi^\G|_{\Wv_s}}(y)\,dy\,. 
\end{equation}
Notice that $\mu_j+\frac{1}{2}$ is a positive integer for $1\leq j\leq l$. By 
Lemma \ref{ximuchexplicit} and \eqref{sqrtc}, 
\begin{align} 
\label{ximucbullet}
\xi_{-\mu}(\widehat{c_\bullet}(x))&=
\prod_{j=1}^l \Big( \frac{ix_j+1}{ix_j-1} \Big)^{\mu_j+\frac{1}{2}}  
\sqrt{\frac{ix_j-1}{ix_j+1} } \nn\\
&=
\prod_{j=1}^l  \frac{(ix_j+1)^{\mu_j+\frac{1}{2}}}{(ix_j-1)^{\mu_j+\frac{1}{2}}}\;
 \frac{\sqrt{ix_j-1}}{\sqrt{ix_j+1}}  \nn  \\
&=\prod_{j=1}^l  \frac{(ix_j+1)^{\mu_j+\frac{1}{2}}}{(-1)^{\mu_j+\frac{1}{2}} (1-ix_j)^{\mu_j+\frac{1}{2}}}  
 \frac{\sqrt{1-ix_j}}{\sqrt{ix_j+1}} \, i{\rm sgn}(x_j)  \nn  \\
&=i^l (-1)^{|\mu|+\frac{l}{2}}
\prod_{j=1}^l 
 (1+ix_j)^{\mu_j} (1-ix_j)^{-\mu_j} \prod_{j=1}^l {\rm sgn}(x_j)\,, 
\end{align}
where $|\mu|=\sum_{j=1}^l 
\mu_j$. Since $\delta=\frac{1}{2}(2l'-2l +1)$, see  
\eqref{eq:delta,beta}, we get from \eqref{ch explicit}
\begin{equation}
\label{ximuch} 
\xi_{-\mu}(\widehat{c_\bullet}(x))\ch^{2l'-2l-1}(x)
=i^l (-1)^{|\mu|+\frac{l}{2}}
\prod_{j=1}^l 
 (1+ix_j)^{-a_j} (1-ix_j)^{-b_j} \prod_{j=1}^l {\rm sgn}(x_j)\,,
\end{equation}
where $a_j$ and $b_j$ are as in \eqref{eq:ajbj}.

The above implies that 
\eqref{intermediate for O2l+1, l smaller} is equal to a constant multiple of 
\begin{align}
\label{integral-to-be-computed-O2l+1}
\int_{\h_s}  \xi_{-\mu}&
(\widehat{c_\bullet}(x))\ch^{2l'-2l-1}(x)
e^{iB(x,y)} F_{\phi^\G|_{\Wv_s}}(y)\,dy\,dx \nn\\
&=i^l (-1)^{|\mu|+\frac{l}{2}} \int_{\h_s} \prod_{j=1}^l 
(1+ix_j)^{-a_j}
(1-ix_j)^{-b_j}  \int_{\h_s}
e^{iB(x,y)} F_{\phi^\G|_{\Wv_s}}(y)\,dy\,dx\,. 
\end{align}
Since $\tau(\Wv_s)\cap \h_s=\h_s$ for $\Dc=\R$, we are in the situation considered by Theorem \ref{main thm for l<l'}, see also Corollary \ref{an intermediate cor}.  Hence the same computation as in 
Theorem \ref{main thm for l<l'} 
shows that \eqref{integral-to-be-computed-O2l+1} is equal to 
$i^l (-1)^{|\mu|+\frac{l}{2}}$ times
\begin{multline}
\label{thm2-O2l+1-1}
\int_{\h_s} \prod_{j=1}^l 
\Big(P_{a_j,b_j}(\beta y_j)e^{-\beta|y_j|}+\beta^{-1}Q_{a_j,b_j}(-\beta^{-1} \partial_{y_j})\delta_0(y_j)\Big)
F_{\phi^\G|_{\Wv_s}}(y)\, dy=\\
\int_{\h_s} \prod_{j=1}^l 
\Big(p_j(y_j)+q_{j}(-\partial_{y_j})\delta_0(y_j)\Big)
F_{\phi^\G|_{\Wv_s}}(y)\, dy\,.
\end{multline}
Recall from Appendix \ref{appenE} that the highest weights of $\Pi$ are integers $\lambda_1\geq \lambda_2\geq \dots \geq \lambda_l\geq 0$ and that $\rho=\sum_{j=1}^l (l+\frac{1}{2}-j) e_j$. 
Hence 
\begin{equation}
\label{minus one to length lambda}
(-1)^{|\mu|+\frac{l}{2}}=(-1)^{\frac{l(l+1)}{2}} (-1)^{|\lambda|}\,.
\end{equation}

We now look at $F_{\phi^\G|_{\Wv_s}}$ when $l=l'$.
By \eqref{integralonS/Sh1-1}, there is a constant $C_1>0$ such that 
\begin{multline}
\label{orbital integral Gs}
\int_{\Sg/\Sg^{\hs1}} \phi(s.w) \; d(s\Sg^{\hs1})= 
C_1 \int_{\G}\int_{\G'/\Zg'} \phi(gg'.w) \; dg\,d(g'\Zg') \qquad
 (\phi\in \mathcal{S}(\Wv),\, w\in \reg{\hs1})\,.
\end{multline}
Because of the embedding $\G_s\subseteq \G$ and the normalization $\vol(\G_s)=1$,  
\begin{align*}
\int_{\G_s}\!\int_{\G'/\Zg'} &\phi^\G(g_sg'.w) \; dg_s\,d(g'\Zg')\\
&=\int_{\G_s}\!\int_{\G'/\Zg'} \!\int_{\G} \phi((gg_s)g'.w)) \; dg\,dg_s\,d(g'\Zg')\\
&=\int_{\G}\!\int_{\G'/\Zg'} \phi(gg'.w) \; dg\,d(g'\Zg')
\qquad (\phi\in \mathcal{S}(\Wv),\, w\in \reg{\hs1})\,.
\end{align*}
Hence, for arbitrary $\phi\in\Ss(\Wv)$, 
\begin{equation}
\label{orbital-integrals-G-Gs-O2l+1}
\mu_{\mathcal{O}(w), \hs1}(\phi^\G|_{\Wv_s})=\mu_{\mathcal{O}(w), \hs1}(\phi)
\qquad (w\in \reg{\hs1})\,.
\end{equation}
Since $\pi_{\g'/\h'}(y)=\pi_{\g_s'/\h'_s}(y)$ by \eqref{product of positive roots for g - bis}, 
we conclude that there is constant $C_2$ such that 
\begin{equation}
\label{equal orbital integrals}
F_{\phi^\G|_{\Wv_s}}= C_2 F_{\phi}=C_2 F_{\phi^\G} \qquad (\phi\in \Ss(\Wv))\,.
\end{equation} 
This finishes the proof of Theorem \ref{main thm for l<l', special odd 1}.
 \hfill\qed

\begin{rem}
When $l< l'$ the Weyl--Harish-Chandra orbital integrals involve almost semisimple elements, see
\eqref{muwforoodd}, 
and the $F_{\phi^\G|_{\Wv_s}}$ is not necessarily proportional to $F_{\phi}$
as a function of $\phi\in \mathcal{S}(\Wv)$. Indeed, let $w_0\in\ss1(\V^0)$, as in \eqref{muwforoodd}. Then by \eqref{integralonS/Sh1+w0}, there is a constant $C_3>0$ such that 
\begin{multline}
\label{orbital integral Gs}
\int_{\Sg/\Sg^{\hs1+w_0}} \phi(s.(w+w_0))) \; d(s\Sg^{\hs1+w_0})= 
C_3 \int_{\G}\int_{\G'/{\Zg'\,}^n} \phi(gg'.(w+w_0)) \; dg\,d(g'{\Zg'\,}^n) \\
 (\phi\in \mathcal{S}(\Wv),\, w\in \reg{\hs1})\,,
\end{multline}
where ${\Zg'}^n$ is the centralizer of $n=\tau'(w_0)$ in $\G'$. Because of the embedding $\G_s\subseteq \G$ and the normalization $\vol(\G_s)=1$,  
\begin{align*}
\int_{\G_s}\!\int_{\G'/{\Zg'}^n} &\phi^\G(g_sg'.(w+w_0)) \; dg_s\,d(g'{\Zg'\,}^n)\\
&=\int_{\G_s}\!\int_{\G'/{\Zg'}^n} \!\int_{\G} \phi((gg_s)g'.(w+w_0)) \; dg\,dg_s\,d(g'{\Zg'\,}^n)\\
&=\int_{\G}\!\int_{\G'/{\Zg'}^n} \phi(gg'.(w+w_0)) \; dg\,d(g'{\Zg'\,}^n)
\qquad (\phi\in \mathcal{S}(\Wv),\, w\in \reg{\hs1})\,.
\end{align*}
However it may happen that 
$$
\int_{\G_s}\!\int_{\G'/{\Zg'}^n} \phi^\G(g_sg'.(w+w_0)) \; dg_s\,d(g'{\Zg'\,}^n)
\ne \int_{\G_s}\!\int_{\G'/{\Zg'}^n} \phi^\G(g_sg'.w) \; dg_s\,d(g'{\Zg'\,}^n)\,.
$$
Hence, by \eqref{orbital integral Gs}, there is generally no positive constant $C_4$ such that, for arbitrary $\phi\in\Ss(\Wv)$,
\begin{equation}
\label{orbital-integrals-G-Gs-O2l+1}
\mu_{\mathcal{O}(w), \hs1}(\phi^\G|_{\Wv_s})=C_4\mu_{\mathcal{O}(w), \hs1}(\phi)
\qquad (w\in \reg{\hs1})\,.
\end{equation}
\end{rem}

%%%%%%%%%%%%%%%%%%%%%
\section{\bf Proof of Theorem \ref{thm:det}}
\label{section: proof of thm det}

Before proving Theorem \ref{thm:det}, let us remark that we will not need to distinguish between the  cases $l>l'$ and $l\leq l'$. We will be working with a Cartan subgroup of $\G$, which we shall denote by $\H$ and not by $\H(\g)$ as previously done when $l>l'$. This is justified because the Cartan subspaces of $\Wv$, which led to the decomposition $\h(\g)=\h\oplus\h''$, play no role here. 
On the other hand, we will need to distinguish between the even and odd orthogonal groups. 

Consider first the case $\G=\Og_{2l}$ with $l>1$. Retain the notation introduced at the beginning of section \ref{The special case even} and let 
$\rho_s^{\tC}$ be as in \eqref{rhos_O2l}. 
Then the functions $\xi_{\rho_s}$ and 
$\Delta_s$ 
for $\G_s=\Og_{2l-1}$ are defined on the double cover $\widehat{\H_s^0}$ of $\H_s^0$ introduced in 
section \ref{Intertwining distributions as an integral over h}:
\[
\Delta_s(\widehat{h})
=
\xi_{\rho_s}(\widehat{h}) 
\prod_{1\leq j<k\leq l-1}(1-\xi_{e_k-e_j}(h))(1-\xi_{-e_j-e_k}(h)) \cdot \prod_{j=1}^{l-1}(1-\xi_{- e_j}(h))  \qquad (h\in \H_s^0)\,.
\]  
Nevertheless, $|\Delta_s(\widehat{h})|$ is well defined as a function on $\H_s^0$ itself, and can be considered as a function on $\H_\bullet\subseteq \Og_{2(l-1)}$ by setting $|\Delta_s(h_\bullet)|=\left|\Delta_s\Big(\!\Omat{1}\!\Big)\right|$. 
Observe that for $\nu=\sum_{j=1}^{l-1} \nu_j e_j\in i\h_s^*$ with $\nu_j\in\Ze$ for $1\leq j\leq l$, 
\begin{equation}
\label{xinu}
\xi_\nu(h_\bullet)=\xi_\nu\Big(\Omat{1}\Big)
=\prod_{j=1}^{l-1} e^{-i\nu_j\theta_j} \qquad 
(h_\bullet=\exp\Big(\sum_{j=1}^{l-1} \theta_j J_j\Big) \in \H_\bullet)\,.
\end{equation}
Hence 
\begin{equation}
\label{xinu-minus}
\xi_\nu(-h_\bullet)=(-1)^{|\nu|} \xi_\nu(h_\bullet) \qquad \text{where $|\nu|=\sum_{j=1}^{l-1} \nu_j$}\,.
\end{equation}
Since
$1=|\xi_{\rho_s}(h_\bullet)|=|\xi_{\rho^{\tC}_s}(h_\bullet) |$
and 
\[
(1-\xi_{-2 e_j}(h_\bullet))=(1-\xi_{- e_j}(h_\bullet))(1+\xi_{- e_j}(h_\bullet))\,,
\]
we see that
\[
\left|\Delta^{\tC}_s\Big(\!\Omat{1}\!\Big)\right|
=\Big|\Delta_s\Big(\Omat{1}\Big)\Big|\cdot\prod_{j=1}^{l-1}|1+\xi_{- e_j}(h_\bullet)|\,.
\]
Furthermore, by \eqref{xinu},
\[
\prod_{j=1}^{l-1}|1+\xi_{- e_j}(h_\bullet)|^2=\prod_{j=1}^{l-1}(1+\xi_{e_j}(h_\bullet))(1+\xi_{- e_j}(h_\bullet))=\det(1+h_\bullet)\,.
\]
Thus
\begin{align}
\label{deltaCdelta}
\left|\Delta^{\tC}_s\Big(\!\Omat{1}\!\Big)\right|^2
&=\Big|\Delta_s\Omat{1}\Big|^2  \det(1+h_\bullet) \nn\\
&=\frac{1}{2}\,\Big|\Delta_s\Omat{1}\Big|^2 \det\Big(1+\Omat{1}\Big)\,.
\end{align}
Finally, by \eqref{xinu-minus},
\begin{equation}
\label{deltaCminus}
\Delta^{\tC}_s(
\Big(-\!\Omat{1}\!\Big)
=(-1)^{l(l-1)/2} \Delta^{\tC}_s
\Big(\!\Omat{1}\!\Big)\,.
\end{equation}
By Lemma \ref{3}, \eqref{deltaCminus} and \eqref{deltaCdelta}, for $\phi\in \Ss(\Wv)$,
\begin{align*}
\int_{\G^0s} &\check\Theta_\Pi(\t g) T(\t g)(\phi)\,dg\\
&=\frac{1}{|W(\G_s^0,\h_s)|} \int_{\H_\bullet} 
\check\Theta_\Pi 
\wt{\Obigmat{h_\bullet}{1}{-1}}
\left|\Delta^{\tC}_s\Big(\!\Omat{1}\!\Big)\right|^2
T_s\wt{\Omat{-1}}
\left(\phi^\G|_{\Wv_s}\right)\,dh_\bullet\\
&=\frac{1}{|W(\G_s^0,\h_s)|} \int_{\H_\bullet} 
\check\Theta_\Pi 
\wt{\Obigmat{-h_\bullet}{1}{-1}}
\left|\Delta^{\tC}_s\Big(\!\Omat{1}\!\Big)\right|^2
T_s\wt{\left(
\begin{array}{c|c}
-h_\bullet & 0   \\ 
\hline
0 &   -1
\end{array}
\right)}
\left(\phi^\G|_{\Wv_s}\right)\,dh_\bullet\\
&=\frac{1}{2|W(\G_s^0,\h_s)|} \int_{\H_\bullet} 
\check\Theta_\Pi \Big(\stackrel{\resizebox{-17mm}{1.4mm}{$\backsim$}}
{\iota_s\Big(\!-\!\Omat{1}\Big)}\Big)
\Big|\Delta_s\!\Omat{1}\Big|^2 \det\Big(1+\Omat{1}\Big)\\
&\qquad \qquad \times 
T_s\Big(\stackrel{\resizebox{-14mm}{1.4mm}{$\backsim$}}{\!-\!\Omat{1}}\Big)
\left(\phi^\G|_{\Wv_s}\right)\,dh_\bullet\,,
\end{align*}
where $\iota_s:-\G_s^0\to \G$ is the embedding given, in terms of matrices, by 
\[
\left(
\begin{array}{c|c}
a & b\\
\hline
c & d
\end{array}\right)
\to 
\left(
\begin{array}{c|c|c}
a & 0 & b\\
\hline
0 & 1 & 0\\
\hline
c & 0 & d
\end{array}\right), \qquad \text{with $a\in \M_{2l-2,2l-2} (\R)$, \;
$d\in \R$}\,.
\]
Now, Weyl's integration formula on $\G_s^0$ yields
$$
\int_{\G^0s} \check\Theta_\Pi(\t g) T(\t g)(\phi)\,dg=
\frac{1}{2} \int_{\G_s^0} 
\check\Theta_\Pi(\iota_s (-g))\det(1+g)T_s\big(\wt{-g}\big)
\left(\phi^\G|_{\Wv_s}\right)\,dg\,.
$$
Making the change of variables $g\to -g$ on the right-hand side, we get \eqref{main thm for l<l' a, special odd 1}.

Let now $\G=\Og_{2l+1}$ with $l\geq 1$. The Cartan subgroup $\H$ of $\G$ is described in Appendix \ref{appenB}. In particular,  $\H^0=\{(u_1,u_2,\dots,u_l,1); u_j\in \SO_2, 1\leq j\leq l\}$.

Suppose first that $1\leq l\leq l'$. On page \pageref{2l+1 smaller than 2l' special-the page},  we introduced $\G_s\subseteq \G$ as the subgroup acting trivially on the 1-dimensional subspace $\V^0_{\overline{0}}$ of $\V_{\overline{0}}$. Considering $\G_s$ as a group of isomorphisms of $\V^1_{\overline{0}}\oplus\cdots \oplus \V^l_{\overline{0}}$ identifies the Cartan subgroup $\H_s$ of $\G_s$ with 
\begin{equation}
\label{identification-Hs-O_2l+1}
\{h_\bullet=(u_1,u_2,\dots,u_l); u_j\in \SO_2, 1\leq j\leq l\}\,.
\end{equation}
The identification of $\H_s$ with  \eqref{identification-Hs-O_2l+1} applies when $l>l'$ as well. Indeed, in this case $\G_s\subseteq \G$ was defined on page \pageref{O_2l+1, l>l', notation} as the subgroup acting trivially on the 1-dimensional subspace $\V^{0,0}_{\overline{0}}$ of $\V_{\overline{0}}$. The identification therefore holds when we consider $\G_s$ as a group of isomorphisms of $(\V^{0,0}_{\overline{0}})^\perp\oplus \V^1_{\overline{0}}\oplus\cdots \oplus \V^l_{\overline{0}}$.

Recall from \eqref{doublecoverofH} the double covering $\widehat{\H^0} \ni \widehat{h} \to h \in \H^0$ of $\H^0$ on which the functions $\xi_{\rho}$ and $\Delta$ are well-defined.
It is easy to check that
\begin{equation}\label{DeltaDet}
\left|\Delta(\widehat{\Omat{1}})\right|^2=
|\Delta_s(h_\bullet)|^2\det(1-h_\bullet)\qquad (h_\bullet\in \H_s^0)\,,
\end{equation}
where 
$$
\Delta_s(h_\bullet)= 
\xi_{\rho_s}(h_\bullet)
\prod_{1\leq j< k\leq l} (1-\xi_{-e_j+e_k}(h_\bullet))(1-\xi_{-e_j-e_k}(h_\bullet))\,.
$$
(The product is empty if $l=1$. In this case, $\Delta_s(h_\bullet)=1$ for all $h_\bullet$.)
Recall from \eqref{Tsodd, special odd 1} (or \eqref{Tsodd, special odd 2}) that 
$T(\wt{\Omat{1}})=T_s(\t h_\bullet) \otimes \delta_0$ for 
$h_\bullet \in  \H_s^0$, 
where $\delta_0$ is the Dirac delta on $\Wv_s^\perp$. 

Hence, by Weyl's integration formula and \eqref{DeltaDet}, for $\phi\in\Ss(\Wv)$, 
\begin{align*}
\int_{\G^0} \check\Theta_\Pi(\t g) &T(\t g)(\phi)\,dg
=\frac{1}{|W(\G^0,\h)|}
\int_{\H^0}\check\Theta_\Pi(\t{h})
|\Delta(\widehat{h})|^2
T(\t{h})(\phi^\G)\,dh \\  
&=\frac{1}{2|W(\G^0_s,\h_s)|}\int_{\H_s^0}
\check\Theta_\Pi\Big(\wt{\Omat{1}}\Big)\det(1-h_\bullet)|\Delta_s(h_\bullet)|^2 T_s(\t h_\bullet)(\phi^\G|_{\Wv_s})\,dh_\bullet\\  
&=
\frac{1}{2}\,
\int_{\G_s^0}\check\Theta_\Pi(\t g)\det(1-g) T_s(\t g)(
\phi^\G|_{\Wv_s}
)\,dg\,.
\end{align*}
This proves \eqref{main thm for l<l' a, special odd 1} for $\G=\Og_{2l+1}$.

%%%%%%%%%%%%%%%%%%%%%%%
\section{\bf A different look at the pair $(\Og_{2l+1},\Sp_{2l'}(\R))$ with $l>l'$}
\label{The special case, odd 2}

Recall the 
decompositions  
$\h(\g)=\h\oplus \h''$ from \eqref{h' + h'' decomposition}
and
$\Wv=\Wv_s\oplus\Wv_s^\perp$ from \eqref{main thm for l>=l', special odd 2 -1}. 
Recall also that we often identify $\h$ and $\h'$ via \eqref{the identification}.
As before, we denote the objects corresponding to $\Wv_s$ by the subscript $s$: 
$\g_s$, $\G_s$, $\Theta_s$, and $T_s$. 
In particular, $\h_s=\h(\g)$, see \eqref{h(g)}, and $\H_s^0=\H(\g)^0$.
Since any element $h\in\H(\g)^0$ acts trivially on $\Wv_s^\perp$, we see that
\[
(h-1)\Wv=(h-1)\Wv_s\,.
\]
Hence, as in \eqref{Tsodd, special odd 1}, 
\begin{equation}\label{Tsodd, special odd 2}
\Theta(\t h)=\Theta_s(\t h)\quad\text{and} \quad T(\t h)=T_s(\t h) \otimes \delta_0 \qquad (h\in \H(\g)^0)\,,
\end{equation}
where $\delta_0$ is the Dirac delta on $\Wv_s^\perp$. 

We consider the (modified) Cayley transform 
$
c_\odot:\h(\g)\to \H(\g)^0$ defined as in \eqref{Cayley-hs}.
Notice that 
$$
c_\odot(x'+x'')=c(x')
c_\odot(x'') \qquad (x'\in \h=\h', x''\in \h'')\,,
$$
where $c:\h'\to \H'$ is the usual Cayley transform.

Let $\z_s$ denote the centralizer of $\h$ in $\g_s$. Then $\z_s=\h\oplus \g''_s$, where 
$\g_s''$ is the Lie algebra of the group $\G''_s$ of isometries of the restriction of the form $(\cdot,\cdot)$ to the 
$2(l-l')$-dimensional real vector space $(\V_{\overline{0}}^{0,0})^\perp$. Then $\h''$ is a Cartan subalgebra of $\g''_s$. 
The following lemma is a variation of Lemma \ref{another intermediate lemma}  in the present situation. 

\begin{lem}
\label{another intermediate lemma for O2l+1-bis} 
Suppose $l>l'$ and let $\mu$ be the Harish-Chandra parameter of a genuine irreducible 
representation of $\wt{\Og}_{2l+1}$. 
In terms of the decomposition (\ref{h' + h'' decomposition})
\begin{multline}\label{another intermediate lemma1-bis}
\xi_{-s\mu}(\widehat{c}_\odot(x))
 \ch^{2l'-2l-1}(x) \pi_{\z_s/\h(\g)}(x)\\
=\left(\xi_{-s\mu}(\widehat{c}(x')) \ch^{2l'-2l-1}(x')\right)
\left(\xi_{-s\mu}(\widehat{c}_\odot(x'')) \ch^{2l'-2l-1}(x'') \pi_{\g_s''/\h''}(x'')\right)\,,
\end{multline}
where $x=x'+x''\in \h(\g)$, with $x'\in \h$ and $x''\in\h''$.
Moreover,
\begin{multline}\label{another intermediate lemma2-bis}
\int_{\h''}\xi_{-s\mu}(\widehat{c}_\odot(x'')) \ch^{2l'-2l-1}(x'') \pi_{\g_s''/\h''}(x'')\,dx''\\
=C\sum_{s''\in W(\G'',\h'')}
\sgn_{\g''/\h''}(s'')
\Bbb I_{\{0\}}(-(s\mu)|_{\h''}+s''\rho'')\,,
\end{multline}
where $C$ is a constant, $\rho''$ is one half times the sum of the positive roots for $(\g''_\C, \h''_\C)$ and $\Bbb I_{\{0\}}$ is the indicator function of zero.
\end{lem}
\begin{proof}
Formula (\ref{another intermediate lemma1-bis}) is obvious, because $\pi_{\z_s/\h(\g)}(x'+x'')=\pi_{\g''_s/\h''}(x'')$. We shall verify (\ref{another intermediate lemma2-bis}).
By \eqref{Delta-O2l+1-c} applied to $\g''\supseteq \h''$, 
$$
\pi_{\g''_s/\h''}(x'')=C_1''\Delta''(\widehat{c}_\odot(x'')) \ch^{2(l-l')-1}(x'')
\qquad 
(x''\in\h'')\,,
$$
where $\Delta''$ is the Weyl denominator for $\G''$, see \eqref{FT Theta ch 13}.
Hence, the integral (\ref{another intermediate lemma2-bis}) is a constant multiple of
$$
\int_{\h''}\xi_{-s\mu}(\widehat{c}_\odot(x''))\Delta''(\widehat{c}(x''))\ch^{-2}(x'')\,dx''
=2^{\dim\h''} \int_{\widehat{c}_\odot(\h'')}\xi_{-s\mu}(h)\Delta''(h)\,dh\,, \qquad
$$
where 
$
\widehat{c}_\odot(\h'')
\subseteq \widehat{\H''^0}$. 
We therefore 
obtain the right-hand side of \eqref{another intermediate lemma2-bis} as in the proof of 
Lemma \ref{another intermediate lemma}. 
\end{proof}

\medskip

\noindent{\it Proof of Theorem \ref{main thm for l>=l', special odd 2}}.
Similar computations as those done in section \ref{The special case, odd 1} together with
\eqref{Tsodd, special odd 2} 
and $\h(\g)=\h_s$
imply that the 
left-hand 
side of \eqref{main thm for l<l' a, special odd 1-bis} is a constant multiple of
\begin{multline}\label{main thm for l<l' a, special odd 2}  
\frac{1}{|W(\G^0,\h(\g))|}\int_{\h(\g)}\left(
\Theta_\Pi(
\widehat{c}_\odot(x)^{-1})\Delta(\widehat{c}_\odot(x)^{-1})
\right)\left(\frac{\Delta(
\widehat{c}_\odot(x)
)}{\pi_{\g_s/\h(\g)}(x)} 
\Theta_s(
\t{c}_\odot(x)
)\right)\\
\times \pi_{\g_s/\h(\g)}(x)\int_{\Wv_s}\chi_x(w)\left(\phi^\G|_{\Wv_s}\right)(w)\,dw\ch^{-2}(x)\,dx\,,
\end{multline}
where $c_\odot(\h_s)$ is a dense subset of $\H(\g)^0$.
Lemma \ref{reduction to ss-orb-int for l>l'} shows that there is a constant $C_1$ such that
\begin{multline*}
\pi_{\g_s/\h(\g)}(x)\int_{\Wv_s}\chi_{x}(w)\,\phi^\G|_{\Wv_s}(w)\,dw\\
=C_1 \int_{\tau'(\reg{\hs1})} \sum_{t W(\Zg_s,\h(\g))\in W(\G_s,\h(\g))/W(\Zg_s,\h(\g))}\sgn_{\g_s/\h(\g)}(t)\pi_{\z_s/\h(\g)}(t^{-1}.x) e^{iB(x,t.y)} F_{\phi^\G|_{\Wv_s}}(y)\,dy\,, 
\end{multline*}
where $\z_s\subseteq \g_s$ is the centralizer of $\h=\h'$.
By \eqref{Deltaandpi-c}, for a suitable constant $C_1$, 
for all $x=\sum_{j=1}^l x_j J_j\in \h$ with $x_j\neq 0$ for $1\leq j\leq l$,
$$
\frac{\Delta(\widehat{c}_\odot(x))}{\pi_{\g_s/\h(\g)}(x)} 
\Theta_s(\widetilde{c}_\odot(x)) \ch^{-2}(x)=C_1 \ch^{2l'-2l-1}(x) 
\left(\prod_{j=1}^l {\rm sgn}(x_j) \right)
\,.
$$
Hence \eqref{main thm for l<l' a, special odd 2} is equal to a constant multiple of 
\begin{multline*}
\sum_{u\in W(\G, \h(\g))}\sgn_{\g/\h(\g)}(u) \int_{\h(\g)}\int_{\tau'(\reg{\hs1})} 
\xi_{-u.\mu}(\widehat{c}_\odot(x))\ch^{2l'-2l-1}(x) 
\left(\prod_{j=1}^l {\rm sgn}(x_j) \right)
\\
\times \, \sum_{t W(\Zg_s,\h(\g))\in W(\G_s,\h(\g))/W(\Zg_s,\h(\g))}\sgn_{\g_s/\h(\g)}(t)\pi_{\z_s/\h(\g)}(t^{-1}x) e^{iB(x,ty)} F_{\phi^\G|_{\Wv_s}}(y)\,dy\,dx\,.
\end{multline*}
Notice that for $t\in W(\G,\h(\g))=W(\G_s,\h(\g))$ and $x\in \h(\g)$,
\begin{equation}
\prod_{j=1}^l {\rm sgn}(tx_j)=\frac{
{\rm sgn}_{\g/\h(\g)}(t)}{
{\rm sgn}_{\g_s/\h(\g)}(t)}\; \prod_{j=1}^l {\rm sgn}(x_j)\,.
\end{equation}
Interchanging the sums, changing the variable of integration $x$ to $tx$ and using that $\ch(tx)=\ch(x)$ and $B(tx,ty)=B(x,y)$, we see that \eqref{main thm for l<l' a, special odd 2} is a constant multiple of
\begin{multline*}
\sum_{t W(\Zg_s,\h(\g))\in W(\G_s,\h(\g))/W(\Zg_s,\h(\g))}\sum_{u\in W(\G, \h(\g))} \sgn_{\g/\h(\g)}(u) \sgn_{\g/\h(\g)}(t) \\
\times \int_{\h(\g)}\int_{\tau'(\reg{\hs1})} 
\xi_{-\mu}(\widehat{c}_\odot(u^{-1}tx)) \ch^{2l'-2l-1}(x) 
\left(\prod_{j=1}^l {\rm sgn}(x_j) \right)
\pi_{\z_s/\h(\g)}(x)
e^{iB(x,y)} F_{\phi^\G|_{\Wv_s}}(y)\,dy\,dx\,.
\end{multline*}
Now, replace $u\in W(\G, \h(\g))$ with $tu$, where $t\in W(\G_s,\h(\g))=W(\G, \h(\g))$. 
Hence, 
\eqref{main thm for l<l' a, special odd 2} is a constant multiple of
\begin{multline}\label{TheAbove}
\sum_{u\in W(\G, \h(\g))} \sgn_{\g/\h(\g)}(u) \int_{\h(\g)}\int_{\tau'(\reg{\hs1})} 
\xi_{-\mu}(\widehat{c}_\odot(u^{-1}x))\ch^{2l'-2l-1}(x)
\left(\prod_{j=1}^l {\rm sgn}(x_j) \right)
\\
\times \, \pi_{\z_s/\h(\g)}(x) e^{iB(x,y)} F_{\phi^\G|_{\Wv_s}}(y)\,dy\,dx\,.
\end{multline}
Lemma \ref{another intermediate lemma for O2l+1-bis}, together with the identification 
\eqref{the identification} of $\h$ and $\h'$,
implies that this last expression is  a constant multiple of
\begin{align}
\label{intermediate-integral-O2l+1}
& \sum_{u\in W(\G, \h(\g))}\sgn_{\g/\h(\g)}(u) 
\Big(\sum_{u''\in W(\G'',\h'')} \sgn_{\g''/\h''}(u'')\mathbb{I}_{\{0\}}(-(u\mu)|_{\h''}+u''\rho'')\Big) \nn\\
&\qquad \times
\int_{\h'}\int_{\tau'(\reg{\hs1})} 
\xi_{-u\mu}(\widehat{c}_\odot(x)) \ch^{2l'-2l-1}(x)
\left(\prod_{j=1}^l {\rm sgn}(x_j) \right)
 e^{iB(x,y)}  F_{\phi^\G|_{\Wv_s}}(y)\,dy\,dx \nn\\
&=\sum_{\substack{u\in W(\G, \h(\g))\\(u\mu)|_{\h''}=\rho'' }} \!\! \sgn_{\g/\h(\g)}(u) \int_{\h'}
\xi_{-u\mu}(\widehat{c}_\odot(x))
\,\ch^{2l'-2l-1}(x)
\left(\prod_{j=1}^l {\rm sgn}(x_j) \right)
 \\
& \qquad \times
\int_{\tau'(\reg{\hs1})}  e^{iB(x,y)}  F_{\phi^\G|_{\Wv_s}}(y)\,dy\,dx\,. 
\end{align}
As in 
\eqref{ximuch},
 for $u\in W(\G,\h(\g))$ and $x\in \h'$,
\begin{equation}
\label{xiumucbullet}
\xi_{-u\mu}(\widehat{c_\odot}(x))\ch^{2l'-2l-1}(x)
\left(\prod_{j=1}^l {\rm sgn}(x_j) \right)
=i^l
(-1)^{|u\mu|+\frac{l}{2}} \prod_{j=1}^l (1+ix_j)^{-a_{u,j}} (1-ix_j)^{-b_{u,j}}
\,,
\end{equation}
where $|u\mu|=\sum_{j=1}^j (u\mu)_j$ and $a_{u,j}$, $b_{u,j}$ are as in \eqref{asj-bsj}.
Hence, computations as in the proof of Lemma \ref{computation-second-integral-l>l'} lead to 
the following equality, which holds in the sense of distributions on $\tau'(\reg{\hs1})$
for every $u\in W(\G,\h(\g))$:
\begin{multline}
\label{innerintegral-umu-gen}
\int_{\h'}\xi_{-u\mu}(\widehat{c_\bullet}(x)) \ch^{2l'-2l-1}(x) 
\left(\prod_{j=1}^l {\rm sgn}(x_j) \right)
e^{iB(x,y)}\,dx\\
=i^l (-1)^{|u\mu|+\frac{l}{2}} \Big(\prod_{j=1}^{l'} P_{a_{u,j},b_{u,j}}(2\pi y_j)\Big) e^{-2\pi\sum_{j=1}^{l'} |y_j|}\,,
\end{multline}
where $P_{a_{u,j},b_{u,j}}$ is defined in \eqref{D0'}.

The sum on the right-hand side of \eqref{intermediate-integral-O2l+1} is over the elements
$u\in W(\G,\h(\g))$ for which $(u\mu)|_{\h''}=\rho''$. By Corollary \ref{another intermediate cor, l>l'},
this has two consequences. The first is that this sum is $0$ unless $\mu$ satisfies 
$\mu|_{\h''}=\rho''$. As seen in the proof of  Theorem \ref{main thm for l>=l'}, this means that the highest weight $\lambda=\mu-\rho$ of $\Pi$ satisfies
condition (a) of that theorem. The second consequence is that for the $\mu$
satisfying $\mu|_{\h''}=\rho''$, an element $u\in W(\G,\h(\g))$ can give a nonzero contribution to the sum in \eqref{intermediate-integral-O2l+1}  only if $u|_{\h''}=1$. The latter condition holds for instance if $u=1$.

Suppose in the following that $\mu$ satisfies $\mu|_{\h''}=\rho''$. Consider first the case $u=1$.
By Lemma \ref{lemma-support}, 
\begin{equation}
\label{poly-smu-R-bis}
\prod_{j=1}^{l'} P_{a_{j},b_{j}}(2\pi y_j)=
(2\pi)^{l'} \prod_{j=1}^{l'} P_{a_j,b_j,2}(2\pi y_j) \mathbb{I}_{\R^+}(y_j) \qquad (y=\sum_{j=1}^{l'} y_j J'_j\in \h')
\end{equation}
has support equal to $\tau'(\hs1)$. Because of \eqref{innerintegral-umu-gen}, we can proceed as in Lemma \ref{lemma:contribution} to show that if $u\in W(\G,\h(\g))$ satisfies $(u\mu)|_{\h''}=\rho''$ and changes the sign of some coordinates (i.e. $y_j\to -y_j$ for some $j$), then the corresponding integral on the right-hand side of \eqref{intermediate-integral-O2l+1} is zero. Recalling that $(u\mu)|_{\h''}=\rho''$ implies $u|_{\h''}=1$, we see that 
all terms in this sum vanish but those corresponding $u\in W(\G',\h')\subseteq W(\G,\h(\g))$. 
The sum is hence over $u\in W(\G',\h')$ and 
formula \eqref{intermediate-integral-O2l+1} becomes a constant multiple of 
$$
\sum_{u\in W(\G', \h')}  \sgn_{\g'/\h'}(u)  (-1)^{|u\mu|+\frac{l}{2}} \int_{\tau'(\reg{\hs1})}   \left(\prod_{j=1}^{l'} P_{a_{u,j},b_{u,j},2}(2\pi y_j)\right)
e^{-2\pi \sum_{j=1}^{l'} |y_j|}  F_{\phi^\G|_{\Wv_s}}(y)\,dy\,.
$$
If $u\in W(\G',\h')$ then
$|u\mu|=|\mu|$. 
Recall from \eqref{minus one to length lambda} that 
$(-1)^{|\mu|+\frac{l}{2}}=(-1)^{\frac{l(l+1)}{2}} (-1)^{|\lambda|}$.

By the $W(\G',\h')$-skew invariance of $F_{\phi^\G|_{\Wv_s}}(y)$ the above integral is therefore a constant multiple of 
\begin{equation}
|W(\G', \h')| (-1)^{|\lambda|} \int_{\tau'(\reg{\hs1})}  \left(\prod_{j=1}^{l'} P_{a_j,b_j,2}(2\pi y_j)\right)
e^{-2\pi \sum_{j=1}^{l'} |y_j|}  F_{\phi^\G|_{\Wv_s}}(y)\,dy\,.
\end{equation}
It remains to show that, as a function of $\phi$, $F_{\phi^\G|_{\Wv_s}}$ is a constant multiple of here
$F_{\phi^\G}=F_{\phi}$. This follows from the same argument used for \eqref{equal orbital integrals} 
in the case $l=l'$, using \eqref{product of positive roots for g - bis} and \eqref{integralonS/Sh1-2} instead of \eqref{integralonS/Sh1-1}. (Notice that since $\G$ is compact, the integral on $\G/\Zg$ is 
${\rm vol}(\Zg)^{-1}$ times the same integral over $\G$.)
This concludes the proof of \eqref{main thm for l<l' a, special odd 1-bis}.

\hfill\qed
\smallskip
%%%%%%%%%%%%%%%%%%%%%%%%%%%%%%%%%%%%%%%%%%
%%
\begin{rem}
The factor $(-1)^{|\lambda|}$ appearing on the right-hand side of \eqref{main thm for l<l' a, special odd 1-bis} in Theorem \ref{main thm for l>=l', special odd 2} turns out to be a constant multiple of 
$\check{\chi}_\Pi(\t{c}(0))$, the value at $\t{c}(0)$ of the central character of $\Pi$, as in Theorems \ref{main theorem O2l for l>l'} and \ref{main thm for l>=l'}. However, we do not have a
proof of this fact independent of the known classification of the representations occurring in Howe's correspondence for the dual pair $(\G,\G')=(\Og_{2l+1},\Sp_{2l'}(\R))$, see e.g. \cite[Appendix (A.4)]{PrzebindaInfinitesimal}. Assume the classification. 
If $l>l'$, given $\lambda$, there is a unique representation $\Pi$ of $\wt{\G}$ occurring in the correspondence with highest weight $\lambda$. 
We see from \cite[(A.4.2.1)]{PrzebindaInfinitesimal} that the highest weight $\lambda'$ of the corresponding representation $\Pi'$ of $\wt{\G'}$ is of the form $\lambda'=\eta+\lambda''$, where 
$\lambda''$ is integral and $|\lambda''|=|\lambda|$.

Let $v\neq 0$ be a highest weight vector of $\Pi'$ and let $\wt{c'}:\g'\to \wt{\G'}$ be the lift of the Cayley transform satisfying $\wt{c'}(0)=\wt{c}(0)$ (Recall that $c'(0)=-1=c(0)$ is in the center of the symplectic group and hence in $\G\cap \G'$.) 
Then 
$$
\check{\chi}_{\Pi'}(\wt{c'}(0))v=\Pi'(\wt{c'}(0))v=\xi_{\lambda'}(\wt{c'}(0))v\,,
$$
which implies that $\check{\chi}_{\Pi'}(\wt{c'}(0))=\xi_{\lambda'}(\wt{c'}(0))$. 
Since $\lambda''$ has integral coordinates 
$$
\xi_{\lambda''}(\wt{c'}(0))=\xi_{\lambda''}(c(0))=(-1)^{|\lambda''|}=(-1)^{|\lambda|}\,.
$$
Hence
$$
\xi_{\lambda'}(\wt{c'}(0))=\xi_{\eta}(\wt{c'}(0))\xi_{\lambda''}(\wt{c'}(0))
=\xi_{\eta}(\wt{c'}(0))(-1)^{|\lambda|}\,.
$$
Since $\Pi$ and $\Pi'$ agree on the center of the symplectic group,
$\xi_{\lambda}(\wt{c}(0))=\xi_{\lambda'}(\wt{c'}(0))$, yielding 
$$
\xi_{\lambda}(\wt{c}(0))=\xi_{\eta}(\wt{c}(0))(-1)^{|\lambda|}\,,$$
where $\xi_{\eta}(\wt{c}(0))$ is a constant independent of the representation $\Pi$.
\end{rem}

%%%%%%%%%%%%%
\section{\bf Proof of Corollary \ref{HW-l <=l'}}
\label{section: proof corollary on UU l<=p+q}

We will distinguish two cases:
\begin{enumerate}
\item[(a)] $0\leq p<l=p+q$,
\item[(b)] $0\leq p<l<p+q$.
\end{enumerate}
In both cases, we shall prove that if 
\begin{equation}
\label{integral-not-0-UU}
\int_{\G}\check{\Theta}_{\Pi}(\wt{g})T(\wt{g})\, dg\neq 0\,,
\end{equation}
then 
$\lambda_{p+1}\leq\frac{p-q}{2}$ and $\lambda_{l-q}\geq\frac{p-q}{2}$. Here the second condition is empty if $l\leq q$.

Consider first case (a). Then $a_j+b_j=-2\delta+2=1$ for all $1\leq j\leq l$. So $Q_{a_j,b_j}=0$ 
for all $1\leq j\leq l$, and hence, in the notation of \eqref{main thm for l<l' a}, 
\begin{equation}
\label{integrand-UU-l=p+q}
\prod_{j=1}^l 
\big( p_j(y_j)+q_j(-\partial_{y_j})\delta_0(y_j)\big)
 F_\phi(y)=
\Big(\prod_{j=1}^{l} P_{a_j,b_j}(\beta y_j)\Big) e^{-\sum_{j=1}^l |y_j|} F_\phi(y)\,.
\end{equation}
Moreover, by Lemma \ref{lem:Pabpolynomial}, for every $1\leq j\leq l$, at most one between 
$P_{a_j,b_j,2}$ and $P_{a_j,b_j,-2}$ can be nonzero. By \cite[Lemma 3.5]{McKeePasqualePrzebindaWCestimates} and because $l>p=l-q>0$, 
\begin{align}
\label{htauW-UU-b}
\h\cap \tau(\Wv)&=W(\G,\h)\big\{ y=\sum_{j=1}^l y_j J_j: y_1,\dots, y_{\max(l-q,0)}\geq 0 \geq y_{p+1}, \dots, y_l\big\} \\[-2mm]
&=\big\{ y=\sum_{j=1}^l y_j J_j: \text{$p$ coordinates $y_j$ are $\geq 0$ and 
$q$ coordinates $y_j$ are $\leq 0$}\big\}\,. \nn
\end{align}
If \eqref{integral-not-0-UU} holds, then $P_{a_j,b_j,2}\neq 0$ for $p$ coordinates $y_j$ and 
$P_{a_j,b_j,-2}\neq 0$ for $q$ coordinates $y_j$. The first condition is equivalent to 
 $b_j\geq 1$ for $p$ values of $j$.
The second condition is equivalent to $a_j\geq 1$, equivalently, $b_j\leq 0$ for $q(=l-p)$ values of $j$. Since the $b_j$'s are strictly decreasing, we conclude that if \eqref{integral-not-0-UU} holds, then 
$$
b_1> \cdots > b_{p} >0 \geq b_{p+1}>\cdots >b_l\,.
$$
But, for $1\leq j\leq l$,
\[
b_j=\lambda_j+\rho_j-\delta+1=\lambda_j+\frac{l}{2}-j+1\,.
\]
Hence 
$b_p>0$ is equivalent to $\lambda_p\geq \dfrac{p-q}{2}$, and
$b_{p+1}\leq 0$ is equivalent to $\lambda_{p+1} \leq \dfrac{p-q}{2}$\,.
%%
%
%Hence $b_p>0$ is equivalent to  
%\[
%\lambda_p\geq \frac{p-q}{2}
%\]
%and $b_{p+1}\leq 0$ is equivalent to
%\[
%\frac{p-q}{2}\geq \lambda_{p+1}\,.
%\]
This proves the claim in the case (a).

Let us now come to case (b). Then $Q_{a_j,b_j}\neq 0$ for all $1\leq j\leq l$ because $a_j+b_j=-2\delta+2<1$.
Recall the integral \eqref{main thm for l<l' a}:
\[
\int_{\h\cap\tau(\Wv)}
\left(\prod_{j=1}^l  \left(p_j(y_j) +q_j(-\partial_{y_j})\delta_0(y_j)\right)\right)\cdot F_\phi(y)\,dy\,.
\]
For $\gamma\subseteq \{1,2,\dots, l\}$, let $|\gamma|$ denote its cardinality and set
$\gamma^c=\{1,2,\dots, l\}\setminus \gamma$. Clearly,
\begin{equation}\label{product_sum}
\prod_{j=1}^l  \left(p_j(y_j) +q_j(-\partial_{y_j})\delta_0(y_j)\right)
=\sum_{\gamma\subseteq \{1,2,\dots, l\}} 
\Big( \prod_{j\in \gamma^c} p_j(y_j) \Big)
\Big( \prod_{j\in \gamma} q_j(-\partial(J_j))\delta_0(y_j)\Big)\,.
\end{equation}
For $s\in W(\G,\h)$ let
\begin{equation}
\label{Ys}
Y_s=\big\{ y=\sum_{j=1}^l y_j J_j: y_{s(1)},\dots, y_{s(\max(l-q,0))}\geq 0 \geq y_{s(p+1)}, \dots, y_{s(l)}\big\}\,.
\end{equation}
By \eqref{htauW-UU-b}, 
$
\h\cap \tau(\Wv)=\bigcup_{s\in W(\G,\h)} Y_s\,. 
$
Notice that $Y_s=Y_{s'}$ if the permutations $s$ and $s'$ differ at most on the set $\{\max(l-q,0)+1, \dots, p\}$. Hence one may choose a subset $W_0(\G,\h)\subseteq W(\G,\h)$ such that the union 
$$
\h\cap \tau(\Wv)=\bigcup_{s\in W_0(\G,\h)} Y_s
$$
is disjoint.
Hence the integral in \eqref{main thm for l<l' a} is a sum of the integrals over these $Y_s$'s.
We consider each of them separately. Let then $s\in W_0(\G,\h)$ be fixed. 
We see from \eqref{product_sum} that the integral over $Y_s$ is equal to 
\begin{equation}
\label{integralYs-gamma-1}
\sum_{\gamma\subseteq \{1,2,\dots, l\}} 
\int_{Y_s} 
\Big( \prod_{j\in \gamma^c} p_j(y_j) \Big)
\Big( \prod_{j\in \gamma} q_j(-\partial(J_j))\delta_0(y_j)\Big)
F_\phi(y)\, dy\,,
\end{equation}
where empty products are equal to $1$.

As in case (a), by Lemma \ref{lem:Pabpolynomial}, for every $1\leq j\leq l$, at most one between $P_{a_j,b_j,2}$ and $P_{a_j,b_j,-2}$ can be nonzero.
By \eqref{Ys}, if the integral \eqref{integralYs-gamma-1} is nonzero then 
\begin{align*}
&\ j\in \{s(1), \dots, s(l-q)\}\cap \gamma^c \text{ implies}  \quad
\text{ $P_{a_j,b_j,2}\neq 0$, i.e. 
$b_j\geq 1$ \quad (\text{for the $l>q$ case}) },
\\
&\ j\in \{s(p+1), \dots, s(l)\}\cap \gamma^c \text{ implies} \quad\text{ $P_{a_j,b_j,-2}\neq 0$, i.e. 
$a_j\geq 1$}\,.
\end{align*}
%%%
For $\Gamma\in \{\gamma^c,\gamma\}$, define
$$
Y_{s,\Gamma}=\Big\{ y_\Gamma=\sum_{j\in \Gamma} y_j J_j: 
\left\{\text{
\begin{tabular}{l}
$y_j\geq 0$ for all $j\in \{s(1),\dots,s(l-q)\}\cap \Gamma$\, \\
$y_j\leq 0$ for all $j\in \{s(p+1),\dots,s(l)\}\cap \Gamma$
\end{tabular}
}
\right.
\Big\}\,,
$$
where the first line of conditions has to be omitted when $l\leq q$.
Then $Y_{s}=Y_{s,\gamma^c}\times Y_{s,\gamma}$ and
\eqref{integralYs-gamma-1} becomes
\begin{align}
\label{integralYs-gamma}
\sum_{\gamma\subseteq \{1,2,\dots, l\}}  &
\int_{Y_{s, \gamma^c}} 
\Big( \prod_{j\in \{s(1),\dots,s(l-q)\} \cap \gamma^c} p_j(y_j)\mathbb{I}_{\R^+}(y_j) \Big) 
\Big( \prod_{j\in \{s(\max(l-q,0)+1),\dots,s(p)\} \cap \gamma^c} p_j(y_j) \Big) \nn \\
&\times\Big( \prod_{j\in  \{s(p+1),\dots,s(l)\}\cap \gamma^c} p_j(y_j)\mathbb{I}_{\R^-}(y_j) \Big) \nn\\
&\times\left(\int_{Y_{s, \gamma}} 
\Big( \prod_{j\in \gamma} q_j(-\partial(J_j))\delta_0(y_j)\Big)
F_\phi(y)\; dy_\gamma\right)\, dy_{\gamma^c}\nn\\
=
\sum_{\gamma\subseteq \{1,2,\dots, l\}}  &
\int_{Y_{s, \gamma^c}} 
\Big( \prod_{j\in \{s(1),\dots,s(l-q)\} 
\cap \gamma^c} p_j(y_j)\mathbb{I}_{\R^+}(y_j) \Big) 
\Big( \prod_{j\in \{s(\max(l-q,0)+1),\dots,s(p)\} \cap \gamma^c} p_j(y_j) \Big) \nn \\
&\times\Big( \prod_{j\in  \{s(p+1),\dots,s(l)\}\cap \gamma^c} p_j(y_j)\mathbb{I}_{\R^-}(y_j) \Big) \nn\\
&\times\left( \prod_{j\in \gamma} q_j(\partial(J_j)) F_\phi(y)|_{y_j=0, j\in \gamma}\right)\, dy_{\gamma^c}\,,
\end{align}
where the first products are empty unless $l>q$ and
empty products are equal to $1$.

Suppose that $l>q$ and there is $j_\gamma\in \{s(1),\dots,s(l-q)\} \cap \gamma$. Then every $y=\sum_{j=1}^l y_j J_j$ with $y_j\geq 0$ for $j\in \{s(1),\dots,s(l-q)\} \cap \gamma^c$, $y_j\leq 0$ for $j\in \{s(p+1),\dots,s(l)\}\cap \gamma^c$ and $y_j=0$ for $j\in \gamma$ belongs to
$$
\Big\{ y=\sum_{j=1}^l y_j J_j: 
\left\{\text{
\begin{tabular}{l}
$y_j\geq 0$ for all $j\in \{s(1),\dots,s(l-q)\}\setminus \{j_\gamma\}$, \\
$y_j\leq 0$ for all $j\in \{s(p+1),\dots,s(l)\}$,\\
 $y_{j_\gamma}=0$
\end{tabular}
}
\right.
\Big\}\subseteq \partial(\h\cap\tau(\Wv))\,,
$$
where $\partial(\h\cap\tau(\Wv))$ denotes the boundary of $\h\cap\tau(\Wv)$.
For all $1\leq j\leq l$,
$$
\deg Q_{a_j,b_j}=-a_j-b_j=2\delta-2= p+q-l-1\,.
$$
Hence, the term
$\Big( \prod_{j\in \gamma} q_j(\partial(J_j)) F_\phi(y) \Big)|_{y_j=0, j\in \gamma}$ is zero on 
$\partial(\h\cap\tau(\Wv))$ by \cite[Theorem 3.5]{McKeePasqualePrzebindaWCestimates}.
Choosing $j=j_\gamma$, we see that the integral corresponding to $\gamma$ in \eqref{integralYs-gamma} vanishes. Similarly (and not only in the case $l>q$), the integral corresponding to $\gamma$ vanishes if there is $j_\gamma\in \{s(p+1),\dots,s(l)\} \cap \gamma$. 
The sum in \eqref{integralYs-gamma} therefore reduces to a sum over the $\gamma$ having no intersection with $\{s(1), \dots, s(\max(l-q,0))\}\cup\{s(p+1), \dots, s(l)\}$. For these $\gamma$'s, 
\begin{align*}
&\{s(1),\dots,s(\max(l-q,0))\} \cap \gamma^c=\{s(1),\dots,s(\max(l-q,0))\}\,,\\
&\{s(p+1),\dots,s(l)\} \cap \gamma^c=\{s(p+1),\dots,s(l)\}\,. 
\end{align*}
Hence,
\begin{align*}
b_{s(j)}&\geq 1 \;\text{for $1\leq j\leq l-q$, if $l>q$}\,,\\
a_{s(j)}&\geq 1 \;\text{for $p+1\leq j\leq l$}\,.
\end{align*}
In particular, there are at least $\max(l-q,0)$ elements $b_j\geq 1$. 
So $b_{l-q}\geq 1$ if $l>q$.
Similarly, there are at least $l-p$ elements 
$a_j\geq 1$. So $a_{p+1}\geq 1$.
As in the case (a), we conclude that if the integral over $Y_s$ corresponding to this $\gamma$ is 
not zero, then $\lambda_{l-q}\geq \frac{p-q}{2}$ (when $l>q$ holds) 
and $\lambda_{p+1}\leq \frac{p-q}{2}$.

This applies to all $\gamma$ and all $s$. Hence, if \eqref{integral-not-0-UU} is satisfied, then 
$\lambda_{l-q}\geq \frac{p-q}{2}$  (when $l>q$ holds)  and $\lambda_{p+1}\leq \frac{p-q}{2}$.
This concludes the proof of Corollary \ref{HW-l <=l'}.

\section{\bf Proof of Corollary \ref{HW-H-l <=l'}}
\label{section: proof corollary HW-H-l <=l'}

Before entering into the proof of Corollary \ref{HW-H-l <=l'}, let us consider 
the dual pair $(\G,\G')=(\Sp_l, \Og^*_{2l'})$ with arbitrary $l\leq l'$. 
Let $\Pi$ be an irreducible genuine representation of 
$\wt{\G}$. We want to prove that the intertwining distribution
corresponding to $\Pi$ is nonzero. For this, it suffices to show that the integral on the right-hand side of \eqref{main thm for l<l' a} is nonzero for suitable functions $\phi\in \Ss(\Wv)$.
The explicit expression of that integral depends on the values of the parameters $a_j$ and $b_j$ 
constructed from the Harish-Chandra parameter $\mu_1>\mu_2 >\dots>\mu_l$ of $\Pi$.

The parameters of the pair $(\Sp_l, \Og^*_{2l'})$ are $d=l$, $d'=l'$, $\iota=1/2$ and hence $\delta=l'-l$.  
Notice that $-a_j-b_j=2\delta-2=2(l'-l-1)$ does not depend on $j$. 
No $q_j$-term occurs in \eqref{main thm for l<l' a} if and only if $-a_j-b_j<0$, i.e. if and only if 
$l=l'$. Every $q_j$-term is a constant multiple of a delta distribution if and only if $-a_j-b_j=0$, 
i.e. if and only if $l+1=l'$. 
In all other cases, the $q_j$-terms are distributions and not measures.

As recalled in Appendix \ref{appenE}, the highest weights of $\Pi$ are integers $\lambda_j$
satisfying $\lambda_1\geq \lambda_2 \geq \cdots \geq \lambda_l\geq 0$ and the $\rho$-function 
for $(\g,\h)$ is $\rho=\sum_{j=1}^l (l+1-j) e_j$. 
Hence $a_j=-\mu_j-\delta+1\leq 0$ , i.e. $P_{a_j,b_j,-2}=0$, for all $1\leq j\leq l$.  
On the other hand, the sign of 
$$
b_j=\mu_j-\delta+1=\lambda_j+(l+1-j)-l'+l+1 \qquad (1\leq j\leq l)
$$
might depend on $j$.
Recall that $b_1>b_2>\dots >b_l$. All the $b_j$ are positive provided so is $b_l$, and 
$b_l=\lambda_l+2+l-l'>0$ if and only if $\lambda_l\geq l'-l-1$. 
In this case, $P_{a_j,b_j,2}\neq 0$ (and hence $p_j\neq 0$) for all $1\leq j\leq l$.
Notice that the condition $\lambda_l\geq l'-l-1$ is automatically satisfied when $l'-l-1\leq 0$, 
that is $l'\in \{l, l+1\}$.

%
%On the other hand, $b_j\leq 0$ for all $1\leq j\leq l$ provided $b_1\leq 0$, i.e. $\lambda_1\leq l'-2l-1$. In this case, $P_{a_j,b_j,2}=0$ (and hence $p_j=0$) for all $1\leq j\leq l$.

\medskip
\noindent \textit{Proof of Corollary \ref{HW-H-l <=l'}.} \;
The discussion preceeding this proof shows that if $\lambda_l\geq l'-l-1$ then, for $1\leq j\leq l$,
\begin{equation}
\label{pj-for-H}
p_j(y_j)=2\pi P_{a_j,b_j,2}(y_j) \mathbb{I}_{\R^+}(y_j) e^{-2\pi |y_j|}  \qquad (y_j\in \R)\,,
\end{equation}
where $P_{a_j,b_j,2}$ is a nonzero polynomial of degree $b_j-1 (\geq 0)$. 
Let $W_0(\G,\h)$ denote the subgroup of $W(\G,\h)$ acting as permutations on the variables $y_j$
of $y=\sum_{j=1}^l y_j J_j \in \h$. Then 
\begin{equation}
\label{polynomial-for-H}
\pi_{\g/\h} (y) \sum_{t\in W_0(\G,\h)} {\rm sgn}_{\g/\h}(t) \prod_{j=1}^l P_{a_j,b_j,2}((ty)_j)   \qquad (y\in \h)\,.
\end{equation}
is a $W_0(\G,\h)$-invariant real-valued polynomial on $\h$. It is nonzero because 
$\deg(P_{a_1,b_1,2})>\deg(P_{a_2,b_2,2})>\dots >\deg(P_{a_l,b_l,2})$. 
Let $\Ug$ be an open, nonempty, $W(\G,\h)$-invariant set with compact closure $\overline{\Ug} \subseteq\reg{\h}$. Observe that $\Ug\cap \tau(\reg{\hs1})$ is nonempty, open, $W_0(\G,\h)$-invariant and with compact closure contained in $\tau(\reg{\hs1})$. We choose such a $\Ug$ so that the polynomial 
\eqref{polynomial-for-H} has constant sign on $\Ug\cap \tau(\reg{\hs1})$.

By Lemma \ref{lem:supportF}, we can choose a nonzero function $\phi\in C^\infty_c(\Wv)^\G$ such that $\phi\geq 0$ and 
$\supp F_\phi\subseteq \Ug$. It follows, in particular, that $F_\phi$, as well as all its partial derivatives, vanishes along the root hyperplanes $y_j=0$, where $1\leq j\leq l$. For such a $\phi$, the right-hand side of \eqref{main thm for l<l' a} reduces to a constant multiple of 
\begin{equation}
\label{reduced integral H}
%\check{\chi}_\Pi(\t{c}(0)) 
\int_{\h\cap\tau(\Wv)}
\left(\prod_{j=1}^l p_j(y_j)\right)\cdot F_\phi(y)\,dy\,.
\end{equation}
By \eqref{pj-for-H}, we can replace the domain of integration $\h\cap\tau(\Wv)$ with 
$\tau(\reg{\hs1})$. Choose a smooth  $W(\G,\h)$-invariant function $\alpha_\phi$ on $\h$ which is equal to 1 on $\Ug$ and has compact support contained in $\reg{\h}$.
Then $\frac{\alpha_\phi}{\pi_{\g/\h}}$ is a smooth $W(\G,\h)$-skew-invariant function on $\h$. 
Set
$$
\Phi(y)=\frac{(2\pi)^l}{|W_0(\G,\h)|} \,\frac{\alpha_\phi(y)}{\pi_{\g/\h}(y)}  \left(\sum_{t\in W_0(\G,\h)} {\rm sgn}_{\g/\h}(t) \prod_{j=1}^l P_{a_j,b_j,2}((ty)_j)\right) e^{-2\pi\sum_{j=1}^l |y_j|}   \qquad (y\in \h)\,.
$$
This is a nonzero smooth $W_0(\G,\h)$-invariant function on $\h$. 
Since $\pi_{\g/\h}(y) F_\phi(y)$ is $W_0(\G,\h)$-invariant, the integral in \eqref{reduced integral H} can be written as
\begin{equation}
\label{final-nonzero-integral-H}
\int_{\tau(\reg{\hs1})} \Phi(y) \pi_{\g/\h}(y) F_\phi(y) \, dy\,.
\end{equation}
By \eqref{products of roots in so}, \eqref{Ch1} and \eqref{eq:originalf-def}, 
$$
\pi_{\g/\h}(y) F_\phi(y)= C |\pi_{\mathfrak{s}_{\overline{0}}/\hs1^2}(w^2)| \int_{\Sg/\Sg^{\hs1}} \phi(s.w) \, d(s\Sg^{\hs1}) \qquad (y=\tau(w)=\tau'(w))\,.
$$
Like $F_\phi$, it is supported in 
$\Ug$ and is a nonzero constant multiple of a function of constant sign. Moreover, by \eqref{polynomial-for-H}, $\Phi$ is nonzero and with constant sign 
in $\Ug\cap \tau(\reg{\hs1})$. Thus \eqref{final-nonzero-integral-H}, and hence the intertwining distribution evaluated at $\phi$, is nonzero.

\begin{rem}
Suppose that $l\leq l'$. Among all dual pairs with one member compact, $(\Sp_l, \Og^*_{2l'})$ is the easiest for computing the intertwining distributions, both because $\G=\Sp_l$ is connected and because there is only one conjugacy class of Cartan subspaces in $\Wv$. Still, establishing if the integral giving the intertwining distribution is nonzero is problematic also in this case as soon as there are nonconstant polynomials $Q_{a_j,b_j}$. The reason is that, at present, we do not have sufficient information on the derivatives of the Cauchy--Harish-Chandra integrals. For the orbital integrals for the adjoint action of a Lie group on it Lie algebra, the relevant information is contained in Harish-Chandra's work; see e.g. \cite[Theorem 9, p. 37]{VaradarajanHarmonic}.
\end{rem}

\section{\bf A sketch of a computation of the wave front set of $\Pi'$}
\label{section:wavefront}
\begin{cor}\label{noproofWF}
For any representation $\Pi\otimes\Pi'$ which occurs in the restriction of the Weil representation to the dual pair $(\wt\G, \wt\G')$, 
\[
WF(\Pi')=\tau'(\tau^{-1}(0))\,.
\]
\end{cor}
Here $WF(\Pi')$ stands for the wave front of the character $\Theta_{\Pi'}$ at the identity
and $0=WF(\Pi)$ since $\Pi$ is finite dimensional.

The complete proof is rather lengthy but unlike the one provided in \cite[Theorem 6.11]{PrzebindaUnipotent}, it is independent of \cite{VoganGelfand}. We sketch the main steps below. The details may be found in \cite{McKeePasqualePrzebindaWC_WF}.

The variety $\tau^{-1}(0)\subseteq \Wv$ is the closure of a single $\G\G'$-orbit $\Oo$; see e.g. \cite[Lemma 2.16]{PrzebindaUnipotent}. There is a positive $\G\G'$-invariant measure $\mu_{\Oo}$ on this orbit which defines a homogeneous distribution. We denote its degree by $\deg \mu_\Oo$. 

Recall that if $\Vv$ is a $n$-dimensional real vector space, $t>0$ and $M_tv=tv$ for $v\in \Vv$, 
then the pullback of $u\in \Ss'(V)$ by $M_t$ is $M_t^*u\in \Ss'(\Vv)$, defined by
$$
(M_t^*u)(\phi)=t^{-n} u\big(\phi\circ M_{t^{-1}}\big) \qquad (\phi\in \Ss(\Vv))\,.
$$
In particular, for $\Vv=\Wv$
\[
M_t^*\mu_\Oo=t^{\deg \mu_\Oo}\mu_\Oo\,.
\]

Define $\tau'_*:\Ss'(\Wv)\to \Ss'(\g')$ by 
$
\tau'_*(u)(\psi)=u(\psi\circ\tau')\,.
$
Then, for $t>0$,
\begin{equation}
\label{Mtau'}
t^{2\dim \g'} M^*_{t^2} \circ\tau'_*=t^{\dim\Wv}  \tau'_* \circ M^*_t\,.
\end{equation}

A rather lengthy but straightforward computation based on Theorems \ref{main thm for l<l'}, \ref{main thm for l>=l'} and \ref{main thm for l<l', special},
shows that
\begin{equation}\label{noproofWF1}
t^{\deg \mu_\Oo}M_{t^{-1}}^*f_{\Pi\otimes\Pi'} \underset{t\to 0}\to C\,\mu_{\Oo}\,,
\end{equation}
as tempered distributions on $\Wv$, where $C$ is a non-zero constant. 

Let $\fo$ indicate a Fourier transform on $\Ss'(\g')$. Then, for $t>0$,
\begin{equation}
M_t^*\circ\fo =t^{-\dim \g'} \fo \circ M_{t^{-1}}\,.
\end{equation}
Hence,
in the topology of $\Ss'(\g')$,
\begin{equation}\label{noproofWF2}
t^{2\deg \mu_{\Oo'}}M_{t^{2}}^*\fo\tau'_*( f_{\Pi\otimes\Pi'})\underset{t\to 0+}{\to}C\fo\mu_{\Oo'}\,,
\end{equation}
where $C\ne 0$ and $\Oo'=\tau'(\Oo)$.

There is an easy to verify inclusion $WF(\Pi')\subseteq\overline{\Oo'}$, \cite[(6.14)]{PrzebindaUnipotent} and a formula for the character $\Theta_{\Pi'}$ in terms of $\mathcal F(\tau'_*( f_{\Pi\otimes\Pi'}))$, 
\begin{equation}\label{noproofWF3}
\frac{1}{\sigma}\cdot \t c_-^*\Theta_{\Pi'}=\widehat{\tau'_*( f_{\Pi\otimes\Pi'})}\,, 
\end{equation}
where $\sigma$ is a smooth function, \cite[Theorem 6.7]{PrzebindaUnipotent}.
By combining this with the following elementary lemma, one completes the argument.
\begin{lem}\label{wave front set 1}
Suppose $f, u\in\Ss'(\R^n)$ and $u$ is homogeneous of degree $d\in \C$. Suppose
\begin{equation}\label{wave front set 1.1}
t^dM_{t^{-1}}^*f (\psi)\underset{t\to 0+}{\to}u (\psi) \qquad (\psi\in \Ss(\R^n))\,.
\end{equation}
Then
\begin{equation}\label{wave front set 1.2}
WF_0(\hat f)\supseteq \supp u\,,
\end{equation}
where the subscript $0$ indicates the wave front set at zero and
\[
f(x)=\int_{\R^n}\hat f(y)e^{2\pi ix\cdot y}\,dy\,.
\] 
\end{lem}
%%

%%%%%%%%%%%%%%%%%%%%%%%%%%%%%%

%%%%%%%%%%%%%%%%%%
%%%%%     APPENDICES
%%%%%%%%%%%%%%%%%%
%%
%%
\appendix

\renewcommand{\thethh}{\Alph{section}.\fontindex{thh}}
\renewcommand{\theequation}{\Alph{section}.\fontindex{equation}}
%%%%%%%%%%%%%%%%%%%%%%%%%

\section{\bf Products of positive roots}
\label{section:positive root products}
\setcounter{thh}{0}
\setcounter{equation}{0}

Keep the notation introduced in section \ref{section:dualpairs-supergroups}. Recall, in particular, that 
$\sum_{j=1}^{l''} y_jJ_j \in \hs1^2|_{\V_{\overline 0}}$ and $\sum_{j=1}^{l''}y_jJ_j'\in \hs1^2|_{\V_{\overline 1}}$
%\begin{equation*}
%\sum_{j=1}^{l''} y_jJ_j \in \hs1^2|_{\V_{\overline 0}} \qquad \text{and} \qquad  \sum_{j=1}^{l''}y_jJ_j'\in \hs1^2|_{\V_{\overline 1}}\,
%\end{equation*}
are identified via \eqref{the identification}. Here $l''=\min(l,l')$.
\medskip

Suppose $l\leq l'$. We can choose the system of the positive roots of $\h$ in $\g_\C$ so that their product is given by the formula
\begin{align}
\label{product of positive roots for g}
\pi_{\g/\h}&(\sum_{j=1}^l y_jJ_j)\\
&=\left\{
\begin{array}{lll}
\prod_{1\leq j<k\leq l}i(- y_j+ y_k) & \text{if} & \Dc=\C\,,\\[.2em]
\prod_{1\leq j<k\leq l}(-y_j^2+y_k^2)\cdot 
\prod_{j=1}^l (-2iy_j)
 & \text{if} & \Dc=\Ha\,,\\[.2em]
\prod_{1\leq j<k\leq l}(-y_j^2+y_k^2) & \text{if} & \Dc=\R\ \text{and}\ \g=\mathfrak s\mathfrak o_{2l}\,,\\[.2em]
\prod_{1\leq j<k\leq l}(-y_j^2+y_k^2)\cdot 
\prod_{j=1}^l (-iy_j)
& \text{if} & \Dc=\R\ \text{and}\ \g=\mathfrak s\mathfrak o_{2l+1}\,.\\
\end{array} 
\right. \nn
\end{align}
Let $\z'\subseteq \g'$ be the centralizer of $\h$. We may choose the order of roots of $\h$ in $\g'_\C/\z'_\C$ so that the product of all of them is equal to
\begin{align}\label{product of positive roots for g'/z'}
&\pi_{\g'/\z'}(\sum_{j=1}^{l} y_jJ'_j)\\
&\;=\left\{
\begin{array}{lll}
\prod_{1\leq j<k\leq l}i(- y_j+ y_k)\cdot \prod_{j=1}^l (-iy_j)^{d'-d} & \text{if} & \Dc=\C\,,\\[.2em]
\prod_{1\leq j<k\leq l}(-y_j^2+y_k^2)\cdot \prod_{j=1}^l (-y_j^2)^{d'-d} & \text{if} & \Dc=\Ha\,,\\[.2em]
\prod_{1\leq j<k\leq l}(-y_j^2+y_k^2)\cdot \prod_{j=1}^l (-2iy_j) \cdot \prod_{j=1}^l (-iy_j)^{d'-d}& \text{if} & \Dc=\R\ \text{and}\ \g=\mathfrak s\mathfrak o_{2l}\,,\\[.2em]
\prod_{1\leq j<k\leq l}(-y_j^2+y_k^2)\cdot \prod_{j=1}^l (-2iy_j) \cdot \prod_{j=1}^l (-iy_j)^{d'-d+1}& \text{if} & \Dc=\R\ \text{and}\ \g=\mathfrak s\mathfrak o_{2l+1}\,.\\
\end{array}
\right.\nn
\end{align}
\medskip 

Suppose $l> l'$. We can choose the system of the positive roots of $\h'$ in $\g'_\C$ so that their product is given by the formula
\begin{equation}\label{product of positive roots for g - bis}
\pi_{\g'/\h'}(\sum_{j=1}^{l'} y_jJ'_j)=\left\{
\begin{array}{lll}
\prod_{1\leq j<k\leq l'}i(- y_j+ y_k) & \text{if} & \Dc=\C\,,\\[.2em]
\prod_{1\leq j<k\leq l'}(-y_j^2+y_k^2) & \text{if} & \Dc=\Ha\,,\\[.2em]
\prod_{1\leq j<k\leq l'}(-y_j^2+y_k^2)\cdot \prod_{j=1}^{l'} 
(-2iy_j)
 & \text{if} & \Dc=\R\,.
\end{array}
\right.
\end{equation}
Moreover, let $\z\subseteq \g$ be the centralizer of $\h$. We may choose the positive roots of $\h$ in $\g_\C/\z_\C$ so that their product is equal to
\begin{align}
\label{product of positive roots for g'/z'}
\pi_{\g/\z}&(\sum_{j=1}^{l'} y_jJ_j)\\
&=\left\{
\begin{array}{lll}
\prod_{1\leq j<k\leq l'}i(- y_j+ y_k)\cdot \prod_{j=1}^{l'} (-iy_j)^{d-d'} & \text{if} & \Dc=\C\,,\\[.2em]
\prod_{1\leq j<k\leq l'}(-y_j^2+y_k^2)\cdot \prod_{j=1}^{l'}(- 2iy_j) \cdot \prod_{j=1}^{l'} (-y_j^2)^{d-d'} & \text{if} & \Dc=\Ha\,,\\[.2em]
\prod_{1\leq j<k\leq l'}(-y_j^2+y_k^2)\cdot \prod_{j=1}^{l'}
(-iy_j)^{d-d'}
& \text{if} & \Dc=\R\ \text{and}\ \g=\mathfrak s\mathfrak o_{2l}\,,\\[.2em]
\prod_{1\leq j<k\leq l'}(-y_j^2+y_k^2)\cdot \prod_{j=1}^{l'}
(-iy_j) 
\cdot \prod_{j=1}^{l'}
(-iy_j)^{d-d'-1}
& \text{if} &\Dc=\R\ \text{and}\ \g=\mathfrak s\mathfrak o_{2l+1}\,.\\
\end{array}
\right.\nn
\end{align}
%%

%%%%%%%%%%%%%%%%%%
%%%%%%%%%%%%%%%%%%
\section{\bf The Jacobian of the Cayley transform}
\label{appenA}
\setcounter{thh}{0}
\setcounter{equation}{0}
Here we determine the Jacobian of the modified Cayley transform $c_-:\g\to\G$. A straightforward computation shows that for a fixed $x\in\g$, 
\[
c_-(x+y)c_-(x)^{-1}-1=(1-x-y)^{-1}2 y (1+x)^{-1}\qquad (y\in \g)\,.
\]
Hence the derivative (tangent map) is given by
\begin{equation}\label{derivativeofc}
c_-'(x)y=(1-x)^{-1}2y (1-x)^{-1}\qquad (y\in \g)\,.
\end{equation}
Recall that $\G$ is the isometry group of a hermitian form $(\cdot,\cdot)$ on $\Vv$. Hence we have the adjoint 
\[
\End_\Bbb D(\Vv)\ni g\to g^*\in \End_\Bbb D(\Vv)
\]
defined by
\[
(gu,v)=(u,g^*v)\qquad (u,v\in\Vv)\,.
\]
Let us view the Lie algebra $\g$ as a real vector space and consider the map
\[
\gamma:\GL_\Bbb D(\Vv)\to \GL(\g)\,,\ \ \ \gamma(g)(y)=g y g^*\,.
\]
Then $\det \circ \gamma :\GL_\Bbb D(\Vv)\to\R^\times$ is a group homomorphism. 
Hence there is a number $s\in\R$ such that
\[
\det(\gamma(g))=\left(\det(g)_{\Vv_\R}\right)^s \qquad (g\in \GL_\Bbb D(\Vv))\,,
\]
where the subscript $\R$ indicates that we are viewing $\Vv$ as a vector space over $\R$.
On the other hand, for a fixed number $a\in\R^\times$,
\[
\det(\gamma(aI_\Vv))=a^{2\dim \g}\quad \text{and}\quad 
\det(aI_\Vv)_{\Vv_\R}=a^{\dim \Vv_\R}\,.
\]
Hence,
\[
\det(\gamma(g))=\left(\det(g)_{\Vv_\R}\right)^{\frac{2\dim \g}{\dim \Vv_\R}} \qquad (g\in \GL_\Bbb D(\Vv))\,.
\]
If $x\in \g$, then $1\pm x\in \GL_\Bbb D(\Vv)$ and 
$$
(1\pm x)^*= 1\mp x \quad \text{and}\quad \big((1\pm x)^{-1}\big)^*= (1\mp x)^{-1}\,.
$$
Hence 
\begin{equation*}
c_-'(x)y=
2(1-x)^{-1}y (1+x)^{-1}c_-(x)=
2\big(\gamma((1- x)^{-1})y\big)c_-(x) 
\qquad (y\in \g)\,.
\end{equation*}
Notice that $|\det(c_-(x))|=1$ because $c(\g)\subseteq \G$. 
Therefore 
\begin{equation}\label{Jacobianofc}
|\det(c_-'(x))|= 2^{\dim \g}\det(1-x)_{\Vv_\R}^{-\frac{2\dim \g}{\dim \Vv_\R}}=2^{\dim \g} \ch(x)^{-2r}
 \qquad (x\in\g)\,,
\end{equation}
where $\ch$ and $r$ are as in \eqref{ch} and \eqref{number r 10}, respectively.
\section{\bf The Weyl denominator lifted by the Cayley transform}
\label{appenB}
\setcounter{thh}{0}
\setcounter{equation}{0}
%%%
%%
Consider the orthogonal matrix group 
\[
\G=\Og_{2l+1}=\{g\in \GL_{2l+1}(\R);\ gg^t=I\}\,.
\]
The spin group is a connected two-fold cover
\[
\Spin_{2l+1}\to\SOg_{2l+1}
\]
of the special orthogonal group. 
We identify 
\begin{equation}
\label{identification-C-matrices}
a+ib=
\begin{pmatrix}
a & -b\\
b & a
\end{pmatrix} 
\qquad (a,b\in\R)\,.
\end{equation}
Then 
\[
\SO_2(\R)=\{u\in\C;\ |u|=1\}\,.
\]
Fix the diagonal Cartan subgroup
\[
\H=\{\diag(u_1, u_2, \dots, u_l,\pm 1);\; u_j\in\SO_2(\R)\,,\; 1\leq j\leq l\}\subseteq \Og_{2l+1}\,.
\] 
Then the connected identity component of $\H$ is
\[
\H^0=\{\diag(u_1, u_2,\dots, u_l,1);\; u_j\in\SO_2(\R);\; 1\leq j\leq l\}\,.
\] 
Denote by $\widehat{\H^0}\subseteq \Spin_{2l+1}$ the preimage of $\H^0$. The Weyl group of $(\Spin_{2l+1}, 
\widehat{\H^0})$ is isomorphic to the Weyl group of $(\SOg_{2l+1}, \H^0)$ and the covering
\[
\widehat{\H^0}\to\H^0
\]
intertwines the action of these groups. As explained in \cite[Lemma 6.3.4 and Theorem 6.3.5]{GoodmanWallach}, one may realize $\widehat{\H^0}$ as the quotient
\[
\widehat{\H^0}=\left(\SOg_2\right)^l/\K\,,
\]
where $\K$ consists of all elements $(z_1, z_2,\dots, z_l)\in \left(\SOg_2\right)^l$ such that each $z_j=\pm 1$ and $z_1z_2\cdots z_l=1$. The Weyl group is generated by the inverses $z_j\to z_j^{-1}$ and permutations of the coordinates. It acts on the Lie algebra $\h$ via the permutations and all sign changes.
The covering map is realized as
\[
\widehat{\H^0}\ni (z_1, z_2, \dots, z_l)\K\to \diag (z_1^2, z_2^2, \dots, z_l^2, 1)\in \H^0\,.
\]
%%%%%%%%%%%%%%

Let $a\in \R$ and define $\theta_a$ by $a=\tan\big(\frac{\theta_a}{2}\big)$. 
Then 
$$
c_-(-ia)=\frac{1-ia}{1+ia}=e^{-i\theta_a}\,.
$$
Set $J=\begin{pmatrix}
0 & 1 \\ -1 & 0
\end{pmatrix}$.
Under the identification \eqref{identification-C-matrices},
$J$ is identified with $-i$. Hence,
\begin{equation} 
\label{{complex-coords-c_}}
c_-(aJ)= (I+aJ)(I-aJ)^{-1}=\exp(\theta_a J)\,
\end{equation}
Therefore the range of the Cayley transform
\[
c_-(\h)=\{\diag(u_1, u_2, ..., u_l, 1);\ u_j\ne -1\ \text{for all}\ j\}
\]
is stable under the action of the Weyl group and $c_-$ intertwines the action of the Weyl group on the Lie algebra and on the group. 
Pick the following branch of the complex square root,
\[
\sqrt{r e^{i\theta}}=\sqrt{r} e^{i\frac{\theta}{2}} \qquad (r>0, -\pi<\theta<\pi)
\]
and set
\[
\sigma:c_-(\h)\ni \diag(u_1, u_2, \dots, u_l,1)\to \diag(\sqrt{u_1}, \sqrt{u_2}, \dots, \sqrt{u_l})\K\in \widehat{\H^0}\,.
\]
This is a section of the covering map which intertwines the Weyl group actions. 
Define
\begin{equation}\label{hatc-}
\widehat c_-(x)=\sigma(c_-(x))\qquad (x\in\h)\,.
\end{equation}
Then $\widehat c_- $ also intertwines the Weyl group actions. Explicitly,
\[
\widehat c_-(\diag(x_1J_1, x_2J_2,\dots, x_lJ_l,0))=\diag(\sqrt{u_1}, \sqrt{u_2}, \dots, \sqrt{u_l})\K\,,
\]
where 
\[
u_j=\frac{1-ix_j}{1+ix_j}\,.
\]
In these terms, the usual choice of the positive roots $e_j\pm e_k$, with $1\leq j<k\leq l$, and $e_j$, with $1\leq j\leq l$ together with \eqref{ximuc} gives
$$
\xi_{e_j} (\diag(u_1, u_2, \dots, u_l,1))=u_j\,.
$$
Hence,
\begin{eqnarray*}
\xi_{-e_j+e_k} (\diag(u_1, u_2, \dots, u_l,1))&=&u_j^{-1} u_k\,,\\
\xi_{-e_j-e_k} (\diag(u_1, u_2, \dots, u_l,1))&=&u_j^{-1}u_k^{-1}\,,\\
\xi_{-e_j} (\diag(u_1, u_2, \dots, u_l,1))&=&u_j^{-1}\,,\\
\xi_\rho (\diag(\sqrt{u_1}, \sqrt{u_2}, \dots, \sqrt{u_l},1)K)&=& u_1^{l-1}u_2^{l-2} \cdots u_{l-1}\xi\,,
\end{eqnarray*}
where
\[
\xi=\sqrt{u_1}\sqrt{u_2} \dots\sqrt{u_l}\,.
\]
We now verify the following formula
\begin{equation}\label{Delta-O2l+1-cminus}
\Delta(\widehat c_-(x))=C_1 \pi_{\g/\h}(x) \ch^{-2l+1}(x) \qquad (x\in \h)\,,
\end{equation}
where where $C_1=2^{l^2}$.
It is easy to check that
\begin{multline}
\label{sqrtcminus}
%\sqrt{u_j}=
\sqrt{\frac{1+z_j}{1-z_j}}=\frac{\sqrt{1+z_j}}{\sqrt{1-z_j}}\,,\quad
1+z_j=\sqrt{1+z_j}\sqrt{1+z_j}\,, \quad \sqrt{1+x_j^2}=\sqrt{1+z_j}\sqrt{1-z_j}\\
(z_j=-ix_j,\,  x_j\in \R)\,.
\end{multline}
We shall use the polynomial identity
\begin{equation}
\label{polyidentity}
\prod_{1\leq j<k\leq l} a_j b_k= \Big( \prod_{j=1}^{l} a_j^{l-j}\Big) \Big(\prod_{k=1}^{l} b_k^{k-1} \Big) 
\end{equation}
when either $b_j=1$ or $b_j=a_j$ for all $1\leq j\leq l$. 
By \eqref{eq:Weyl-den} and \eqref{polyidentity},
\begin{align*}
\Delta(\widehat c_-(x))&=\xi \Big( \prod_{j=1}^{l-1} u_j^{l-j}\Big)  \prod_{1\leq j<k\leq l} 
(1-u_j^{-1}u_k^{-1})(1-u_j^{-1}u_k)  \prod_{j=1}^{l} (1-u_j^{-1})\\
&=\xi   \prod_{1\leq j<k\leq l} 
(u_j-u_k^{-1})(1-u_j^{-1}u_k)  \prod_{j=1}^{l} (1-u_j^{-1})\,.
\end{align*}
By \eqref{sqrtcminus},
\begin{align*}
u_j-u_k^{-1}&=\frac{1+z_j}{1-z_j}-\frac{1-z_k}{1+z_k}=
\frac{2(z_j+z_k)}{(1-z_j)(1+z_k)}\,,\\
1-u_j^{-1}u_k&=1-\frac{1-z_j}{1+z_j}\frac{1+z_k}{1-z_k}=
\frac{2(z_j-z_k)}{(1+z_j)(1-z_k)}\,,
\\
1-u_j^{-1}&=1-\frac{1-z_j}{1+z_j}=\frac{2z_j}{1+z_j}=\frac{2z_j}{\sqrt{1+z_j}\sqrt{1+z_j}}\,.
\end{align*}

Since $\xi=\prod_{j=1}^l \sqrt{u_j}$, we obtain by \eqref{sqrtcminus}, \eqref{polyidentity} and \eqref{product of positive roots for g}, 
\begin{align*}
\Delta(\widehat c_-(x))&=
2^{l^2} 
\prod_{1\leq j<k\leq l}  \frac{1}{(1-z_j^2)(1-z_k^2)}
\prod_{j=1}^l \frac{1}{\sqrt{1+z_j}\sqrt{1-z_j}} 
\prod_{1\leq j<k\leq l} (z_j+z_k) (z_j-z_k) \prod_{j=1}^l z_j\\ 
&=2^{l^2} \Big(\prod_{j=1}^l \frac{1}{(1-z_j^2)^{l-1}} \prod_{j=1}^l \frac{1}{\sqrt{1+x_j^2}} \Big) \pi_{\g/\h}(x) \qquad (x\in \h)\,,
\end{align*}
which gives \eqref{Delta-O2l+1-cminus}.

Recall from \eqref{Cayley-hs} that if $x=\diag(x_1J_1, x_2J_2, \dots, x_lJ_l, 0)\in \h=\h_s$, 
then $c_\odot(x)=\diag(v_1,v_2,\dots, v_l,1)$ has coordinates 
\[
v_j=c(x_jJ_j)=-c_-(x_jJ_j)=-u_j\,,\qquad (1\leq j\leq l)
\]
with $|v_j|=1$ and $v_j\ne 1$ for all $j$. 
The identification \eqref{identification-C-matrices} implies the identification
\[
v_j=-u_j=\frac{z_j+1}{z_j-1}\,,\qquad (
z_j=-ix_j, \,
1\leq j\leq l)\,.
\]

On the subset where $v_j\ne \pm 1$ for all $j$ define
\begin{equation}\label{hatcbullet}
\widehat c_\odot(x)=\sigma( c_\odot(x)) \qquad (x\in\h, x_j\ne 0, 1\leq j\leq l)\,.
\end{equation}
We now prove the following equality:
\begin{equation}\label{Delta-O2l+1-c}
\Delta(\widehat 
c_\odot(x))=C_2 \left(\prod_{j=1}^l 
{\rm sgn}(x_j)\right) 
 \pi_{\g_s/\h_s}(x) \ch^{-2l+1}(x) 
\qquad (x\in\h, x_j\ne 0, 1\leq j\leq l)\,,
\end{equation}
where $C_2=
(2i)^{l^2} 
$ and ${\rm sgn}(x_j)=x_j/|x_j|
$.
(Notice that $\Delta(\widehat c_\odot(x))$ is singular at $x_j=0$ because so is the fixed section $\sigma$, which depends on our choice of $\sqrt{\cdot}$.)
It is easy to check that 
\begin{multline}
\label{sqrtc} 
\sqrt{\frac{z_j+1}{z_j-1}}=\frac{\sqrt{z_j+1}}{\sqrt{z_j-1}} 
\,,\qquad
z_j+1=\sqrt{z_j+1}\sqrt{z_j+1}\,, 
\qquad
\sqrt{z_j-1}=
-i\,
{\rm sgn}
(x_j) \sqrt{1-z_j}\,,
\\
 -i\,
{\rm sgn}
(x_j)
\sqrt{1+x_j^2}=\sqrt{z_j+1}\sqrt{z_j-1}
\qquad\qquad (z_j=
-ix_j,
\, x_j\in \R\setminus \{0\})\,.
\end{multline}

As before,
\[
\Delta(\widehat{c}_\odot(x))=
\xi 
\prod_{1\leq j<k\leq l} 
(v_j-v_k^{-1})(1-v_j^{-1}v_k)  \prod_{j=1}^{l} (1-v_j^{-1})\,,
\]
where, by \eqref{sqrtc},
\begin{align*}
v_j-v_k^{-1}&=\frac{z_j+1}{z_j-1}-\frac{z_k-1}{z_k+1}=
\frac{2(z_j+z_k)}{(z_j-1)(z_k+1)}\,,\\
1-v_j^{-1}v_k&=1-\frac{z_j-1}{z_j+1}\frac{z_k+1}{z_k-1}=
\frac{2(-z_j+z_k)}{(z_j+1)(z_k-1)}\,,\\
1-v_j^{-1}&=1-\frac{z_j-1}{z_j+1}=
\frac{2}{z_j+1}
=\frac{2}{\sqrt{z_j+1}
\sqrt{z_j+1}}\,.
\end{align*}
Since $\xi=\prod_{j=1}^l \sqrt{v_j}$, we obtain by \eqref{sqrtcminus}, \eqref{polyidentity} and \eqref{product of positive roots for g},
\begin{align*}
\Delta(\widehat{c}_\odot(x))&=
2^{l^2} 
\Big(\prod_{j=1}^l  \sqrt{\frac{z_j+1}{z_j-1}} \Big)\Big(\prod_{1\leq j<k\leq l}  \frac{1}{(z_j^2-1)(z_k^2-1)}
 \Big)\\
&\qquad \times
\Big(\prod_{j=1}^l \frac{1}{\sqrt{z_j+1}\sqrt{z_j+1}} \Big)
\Big(\prod_{1\leq j<k\leq l} (z_j+z_k) (-z_j+z_k)\Big)\\
&=2^{l^2} 
\Big(
\prod_{j=1}^l \frac{\sqrt{z_j+1}}{\sqrt{z_j-1}} \frac{1}{\sqrt{z_j+1}\sqrt{z_j+1}} 
\Big) 
\Big(\prod_{j=1}^l \frac{1}{(1-z_j^2)^{l-1}} \Big)\\
&\qquad \times
(-1)^{l(l-1)/2}
\Big(\prod_{1\leq j<k\leq l} (z_j+z_k) (z_j-z_k)\Big)\\
&=i^{l(l-1)}2^{l^2} 
 \Big(\prod_{j=1}^l \frac{1}{\sqrt{z_j-1}\sqrt{z_j+1}} \Big)
\Big(\prod_{j=1}^l \frac{1}{(1-z_j^2)^{l-1}} \Big)\pi_{\g_s/\h_s}(x)\\
&=
(2i)^{l^2} 
\Big(\prod_{j=1}^l \frac{{\rm sgn}(x_j)}{\sqrt{1+x_j^2}}  \Big)
\Big( \prod_{j=1}^l
\frac{1}{(1+x_j^2)^{l-1}}  \Big) \pi_{\g_s/\h_s}(x) \qquad (x\in \h\setminus \{0\})\,,
\end{align*}
which gives \eqref{Delta-O2l+1-c}.
%%
%%
%%%

\section{\bf The special functions $P_{a,b}$ and $Q_{a,b}$}
\label{appenC}
\setcounter{thh}{0}
\setcounter{equation}{0}
For two integers $a$ and $b$ define the following functions in the real variable $\xi$,
\begin{eqnarray}\label{eq:Pab2}
P_{a,b,2}(\xi)&=&\left\{
\begin{array}{ll}
\sum_{k=0}^{b-1}\frac{a(a+1)\cdots(a+k-1)}{k!(b-1-k)!}2^{-a-k}\xi^{b-1-k}&\text{if}\ b\geq 1\\
0&\text{if}\ b\leq 0,
\end{array}
\right.\\
\label{eq:Pabminus2}
P_{a,b,-2}(\xi)&=&\left\{
\begin{array}{ll}
(-1)^{a+b-1}\sum_{k=0}^{a-1}\frac{b(b+1)\cdots(b+k-1)}{k!(a-1-k)!}(-2)^{-b-k}\xi^{a-1-k}&\text{if}\ a\geq 1\\
0&\text{if}\ a\leq 0,
\end{array}
\right.
\end{eqnarray}
where $a(a+1)\cdots(a+k-1)=1$ if $k=0$. 
Notice that 
\begin{equation}\label{D0}
P_{a,b,-2}(\xi)=P_{b,a,2}(-\xi) \qquad (\xi\in\R,\ a,b\in\Zb)\,.
\end{equation}
Set
\begin{eqnarray}\label{D0'}
P_{a,b}(\xi)&=&2\pi(P_{a,b,2}(\xi)\Bbb I_{\R^+}(\xi)+P_{a,b,-2}(\xi)\Bbb I_{\R^-}(\xi))\\
&=&2\pi(P_{a,b,2}(\xi)\Bbb I_{\R^+}(\xi)+P_{b,a,2}(-\xi)\Bbb I_{\R^+}(-\xi))\,,\nn
\end{eqnarray}
where $\Bbb I_S$ denotes the indicator function of the set $S$.
Also, let
\begin{eqnarray}\label{D0''}
Q_{a,b}(iy)=2\pi \left\{
\begin{array}{ll}
0 & \text{if}\ a+b\geq 1\,,\\
\sum_{k=b}^{-a} \frac{a(a+1)\cdots(a+k-1)}{k!}2^{-a-k}(1-iy)^{k-b} & \text{if}\ -a>b-1\geq 0\,,\\
\sum_{k=a}^{-b} \frac{b(b+1)\cdots(b+k-1)}{k!}
2^{-b-k}(1+iy)^{k-a} & \text{if}\ -b>a-1\geq 0\,,\\
(1+iy)^{-a}(1-iy)^{-b} & \text{if}\ a\leq 0\ \text{and}\ b\leq 0\,.
\end{array}
\right.
\end{eqnarray}
Observe also that 
\begin{equation}
\label{absymmetryPabQab}
P_{b,a}(\xi)=P_{a,b}(-\xi) \qquad \text{and}\qquad Q_{b,a}(iy)=Q_{a,b}(-iy)\,.
\end{equation}
The following elementary fact will be crucial at several points.
\begin{lem}
\label{lem:Pabpolynomial}
Suppose that $a+b\leq 1$. Then at most one between $P_{a,b,2}$ and $P_{a,b,-2}$ can be non-zero. Hence $P_{a,b}$ is either 0 or the restriction of a polynomial to a half line. 
\end{lem}
%%%  proof: $P_{a,b,2}\neq 0$ iff $b\geq 1$. Hence $a\leq 1-b\leq 0$. So $P_{a,b,-2}=0$
%%
\begin{rem}
\label{Pab-Laguerre}
Let $\Gamma$ denote the gamma function.
If $k$ is a nonnegative integer, then 
$$
a(a+1)\cdots(a+k-1)=\frac{\Gamma(a+k)}{\Gamma(a)}\,,
$$
which is often shortened by the Pochhammer symbol $(a)_k$. Another useful formula is
$$
a(a+1)\cdots(a+k-1)=(-1)^k (-a)(-a-1)\cdots(-a-k+1)=
                                (-1)^k\frac{\Gamma(-a+1)}{\Gamma(-a+1-k)}\,.
$$
In this notation, for an integer $b\geq 1$ and $h=0,1,\dots, b-1$,
$$
(b-1-h)!=\frac{(b-1)!}{(-b+1)_h} \qquad\text{and} \qquad 
\Gamma(-a-b+2+h)=\Gamma(-a-b+2) \; (-a-b+2)_h
\,.
$$
Hence
\begin{eqnarray*} 
\label{Pab2-confluent}
P_{a,b,2}(\xi)&=&\sum_{k=0}^{b-1} (-1)^k \frac{\Gamma(-a+1)}{\Gamma(-a+1-k)}\
\frac{1}{k!(b-1-k)!} 2^{-a-k} \xi^{b-1-k}\\
&=& \Gamma(-a+1) \sum_{h=0}^{b-1} (-1)^{b-1-h} \frac{1}{\Gamma(-a-b+2+h)} \, 
\frac{1}{(b-1-h)! h!}  2^{-a-b+1+h} \xi^{h}\\
&=&(-1)^{b-1} 2^{-a-b+1} \frac{\Gamma(-a+1)}{\Gamma(-a-b+2) \, (b-1)!} 
\sum_{h=0}^{b-1} \frac{(-b+1)_h}{(-a-b+2)_h h!} (2\xi)^{h}\\
&=&(-1)^{b-1} 2^{-a-b+1} \frac{\Gamma(-a+1)}{\Gamma(-a-b+2) \, (b-1)!} 
\, {}_1 F_1 \big(-b+1;-a-b+2;2\xi\big)\\
&=&(-1)^{b-1} 2^{-a-b+1} L^{-a-b+1}_{b-1}(2\xi)\,,
\end{eqnarray*}
where ${}_1 F_1$ is the confluent hypergeometric function and $L^\alpha_n(x)$ 
is a Laguerre polynomial. See \cite[6.9(36), \S 10.12]{Erdelyi}.
\end{rem}
\begin{pro}\label{D1}
For any $a,\ b\in\Bbb Z$, the formula
\begin{equation}\label{D1.1}
\int_\R (1+iy)^{-a}(1-iy)^{-b}\phi(y)\,dy \qquad (\phi\in \Ss(\R))
\end{equation}
defines a tempered distribution on $\R$. The restriction of the Fourier transform of this distribution to $\R\setminus \{0\}$ is a function given by 
\begin{eqnarray}\label{D1.2}
\int_\R
(1+iy)^{-a}(1-iy)^{-b} e^{-iy\xi}\,dy
=P_{a,b}(\xi)e^{-|\xi|}.
\end{eqnarray}
The right-hand side of (\ref{D1.2}) is an absolutely integrable function on the real line and thus defines a tempered distribution on $\R$. 
Furthermore,
\begin{eqnarray}\label{D1.3}
(1+iy)^{-a}(1-iy)^{-b}=\frac{1}{2\pi}\int_\R P_{a,b}(\xi) e^{-|\xi|} e^{iy\xi}\,dy+\frac{1}{2\pi}Q_{a,b}(iy)
\end{eqnarray}
and hence,
\begin{eqnarray}\label{D1.4}
\int_\R(1+iy)^{-a}(1-iy)^{-b}e^{-iy\xi}\,dy=P_{a,b}(\xi)e^{-|\xi|}+ Q_{a,b}(-\frac{d}{d\xi})\delta_0(\xi)\,.
\end{eqnarray}
\end{pro}
\begin{prf}
Since, $|1\pm iy|=\sqrt{1+y^2}$, (\ref{D1.1}) is clear. The integral (\ref{D1.2}) is equal to
\begin{multline}\label{D1.30}
\frac{1}{i}\int_{i\R}
(1+z)^{-a}(1-z)^{-b} e^{-z\xi}\,dz\\
=2\pi (-\Bbb I_{\R^+}(\xi)\,{\rm res}_{z=1}(1+z)^{-a}(1-z)^{-b} e^{-z\xi}
+\Bbb I_{\R^-}(\xi)\,{\rm res}_{z=-1}(1+z)^{-a}(1-z)^{-b} e^{-z\xi})\,.
\end{multline}
The computation of the two residues is straightforward and (\ref{D1.2}) follows.

Since
$$
\int_0^\infty e^{-\xi} e^{i\xi y}\,d\xi=(1-iy)^{-1}\,,
$$
we have 
\begin{equation}\label{d1}
\int_0^\infty \xi^m e^{-\xi} e^{i\xi y}\,d\xi
=\left(\frac{d}{d(iy)}\right)^{m}(1-iy)^{-1}=m! (1-iy)^{-m-1} \qquad (m=0, 1, 2, \dots)\,.
\end{equation}
Thus, if $b\geq 1$, then 
\begin{eqnarray*}
&&\int_0^\infty P_{a,b,2}(\xi) e^{-\xi} e^{i\xi y}\,d\xi\\
&&\qquad =\sum_{k=0}^{b-1} \frac{a(a+1)\cdots(a+k-1)}{k!}2^{-a-k}(1-iy)^{-b+k}\\
&&\qquad =(1-iy)^{-b}2^{-a}\sum_{k=0}^{b-1} \frac{(-a)(-a-1)\cdots(-a-k+1)}{k!}
\left(-\frac{1}{2}(1-iy)\right)^{k}\,.
\end{eqnarray*}
Also, if $a\leq 0$, then
\begin{eqnarray*}
2^a(1+iy)^{-a}&=&\left(1-\frac{1}{2}(1-iy)\right)^{-a}
=\sum_{k=0}^{-a}
\left(
\begin{array}{r}
-a\\
k
\end{array}
\right) 
\left(-\frac{1}{2}(1-iy)\right)^{k}\\
&=&\sum_{k=0}^{-a} \frac{(-a)(-a-1)\cdots(-a-k+1)}{k!}
\left(-\frac{1}{2}(1-iy)\right)^{k}.
\end{eqnarray*}
Hence, 
\begin{eqnarray}\label{d2}
&&\int_0^\infty P_{a,b,2}(\xi) e^{-\xi} e^{i\xi y}\,d\xi-
(1+iy)^{-a}(1-iy)^{-b}\\
&&\qquad=(1-iy)^{-b}2^{-a}\left(\sum_{k=0}^{b-1} \frac{(-a)(-a-1)\cdots(-a-k+1)}{k!}
\left(-\frac{1}{2}(1-iy)\right)^{k}\right.\nn\\
&&\hskip 3.6cm \left. -\sum_{k=0}^{-a} \frac{(-a)(-a-1)\cdots(-a-k+1)}{k!}
\left(-\frac{1}{2}(1-iy)\right)^{k}\right)\,.\nn
\end{eqnarray}
Recall that $P_{a,b,-2}=0$ if $a\leq 0$. Hence, (\ref{D1.2}) shows that (\ref{d2}) is the inverse Fourier transform of a distribution supported at $\{0\}$, hence a polynomial.

Suppose $-a<b-1$. Then (\ref{d2}) is equal to
$$
2^{-a}\sum_{k=-a+1}^{b-1} \frac{(-a)(-a-1)\cdots(-a-k+1)}{k!}\left(-\frac{1}{2}\right)^{k}(1-iy)^{k-b}\,,
$$
which is zero because $(-a)(-a-1)\cdots(-a-k+1)=0$ for $k\geq -a+1$.
If $-a=b-1$, then (\ref{d2}) is obviously zero. 

Suppose $-a>b-1$. Then (\ref{d2}) is equal to
\begin{equation}\label{d3}
-2^{-a}\sum_{k=b}^{-a} \frac{(-a)(-a-1)\cdots(-a-k+1)}{k!}\left(-\frac{1}{2}\right)^{k}(1-iy)^{k-b}\,.
\end{equation}
As in (\ref{d1}) we have
\begin{equation*}\label{d4}
\int_{-\infty}^0 \xi^m e^{\xi} e^{i\xi y}\,d\xi
=\left(\frac{d}{d(iy)}\right)^{m}(1+iy)^{-1}=(-1)^m m! (1+iy)^{-m-1} \qquad (m=0, 1, 2, \dots)\,.
\end{equation*}
Suppose $a\geq 1$. Then 
\begin{eqnarray*}
&&\int_{-\infty}^0 P_{a,b,-2}(\xi) e^{\xi} e^{i\xi y}\,d\xi\\
&&\qquad =(-1)^{a+b-1}\sum_{k=0}^{a-1} \frac{b(b+1)\cdots(b+k-1)}{k!}(-2)^{-b-k}
(-1)^{a-1+k}(1+iy)^{-a+k}\\
&&\qquad =(1+iy)^{-a}2^{-b}\sum_{k=0}^{a-1} \frac{(-b)(-b-1)\cdots(-b-k+1)}{k!}
\left(-\frac{1}{2}(1+iy)\right)^{k}\,.
\end{eqnarray*}
Also, if $b\leq 0$, then
\begin{eqnarray*}
&&2^{b}(1-iy)^{-b}=
\sum_{k=0}^{-b} \frac{(-b)(-b-1)\cdots(-b-k+1)}{k!}
\left(-\frac{1}{2}(1+iy)\right)^{k}\,.
\end{eqnarray*}
Hence, 
\begin{eqnarray}\label{d5}
&&\int_{-\infty}^0 P_{a,b,-2}(\xi) e^{\xi} e^{i\xi y}\,d\xi-
(1+iy)^{-a}(1-iy)^{-b}\\
&&\qquad=(1+iy)^{-a}2^{-b}\left(\sum_{k=0}^{a-1} \frac{(-b)(-b-1)\cdots(-b-k+1)}{k!}
\left(-\frac{1}{2}(1+iy)\right)^{k}\right.\nn\\
&&\hskip 3.6cm \left. -\sum_{k=0}^{-b} \frac{(-b)(-b-1)\cdots(-b-k+1)}{k!}
\left(-\frac{1}{2}(1+iy)\right)^{k}\right)\,.\nn
\end{eqnarray}
As before, we show that (\ref{d5}) is zero if $-b\leq a-1$. If $-b>a-1$, then (\ref{d5}) is equal to
$$
-2^{-b}\sum_{k=a}^{-b} \frac{(-b)(-b-1)\cdots(-b-k+1)}{k!}
\left(-\frac{1}{2}\right)^{k}(1+iy)^{k-a}\,.
$$
If $a\geq 1$ and $b\geq 1$, then our computations show that
\begin{eqnarray}\label{d6}
&&\int_0^{\infty} P_{a,b,2}(\xi) e^{-\xi} e^{i\xi y}\,d\xi
+\int_{-\infty}^0 P_{a,b,-2}(\xi) e^{\xi} e^{i\xi y}\,d\xi-
(1+iy)^{-a}(1-iy)^{-b}
\end{eqnarray}
is a polynomial which tends to zero if $y$ goes to infinity. Thus (\ref{d6}) is equal zero. This completes the proof of (\ref{D1.3}). The statement (\ref{D1.4}) is a direct consequence of (\ref{D1.3}).
\end{prf}
The test functions which occur in Proposition \ref{D1} need not be in the Schwartz space. In fact the test functions we shall use in our applications are not necessarily smooth. Therefore we shall need a more precise version of the formula (\ref{D1.4}). This requires a definition and two well-known lemmas. 

Following Harish-Chandra denote by $\Ss(\R^\times)$ the space of the smooth complex valued functions defined on $\R^\times$ whose all derivatives are rapidly decreasing at infinity and have limits at zero from both sides. For $\psi\in\Ss(\R^\times)$ let
\[
\psi(0+)=\underset{x\to 0+}{\lim}\ \psi(\xi)\,,\quad 
\psi(0-)
=\underset{x\to 0-}{\lim}\ \psi(\xi),\quad \langle \psi\rangle_0=\psi(0+)-\psi(0-)\,.
\]
In particular the condition $\langle\psi\rangle_0=0$ means that $\psi$ extends to a continuous function on $\R$.
\begin{lem}\label{D1.4lem1}
Let $c=0, 1, 2, \dots$ and let $\psi\in \Ss(\R^\times)$. Suppose
\begin{equation}\label{D1.4lem1.1}
\langle \psi\rangle_0=\dots=\langle \psi^{(c-1)}\rangle_0=0\,.
\end{equation}
(The condition (\ref{D1.4lem1.1}) is empty if $c=0$.) Then
\begin{equation}\label{D1.4lem1.2}
\left | \int_{\R^\times}e^{-iy\xi}\psi(\xi)\,d\xi \right | \leq \min\{1,|y|^{-c-1}\}(|\langle \psi^{(c)}\rangle_0|+\parallel \psi^{(c+1)}\parallel_1+
\parallel \psi\parallel_1)
\end{equation}
\end{lem}
\begin{prf}
Integration by parts shows that for $z\in\C^\times$
\begin{eqnarray*}
&&\int_{\R^+}e^{-z\xi}\psi(\xi)\,d\xi=z^{-1}\psi(0+)+\dots+z^{-c-1}\psi^{(c)}(0+) +z^{-c-1}\int_{\R^+}e^{-z\xi}\psi^{(c+1)}(\xi)\,d\xi\,,\\
&&\int_{\R^-}e^{-z\xi}\psi(\xi)\,d\xi=-z^{-1}\psi(0-)-\dots-z^{-c-1}\psi^{(c)}(0-) +z^{-c-1}\int_{\R^-}e^{-z\xi}\psi^{(c+1)}(\xi)\,d\xi\,.
\end{eqnarray*}
Hence,
\begin{eqnarray*}
&&\int_{\R^\times}e^{-z\xi}\psi(\xi)\,d\xi\\
&=&z^{-1}\langle \psi\rangle_0+\dots+z^{-c}\langle \psi^{(c-1)}\rangle_0 +z^{-c-1}\langle \psi^{(c)}\rangle_0 +z^{-c-1}\int_{\R^\times}e^{-z\xi}\psi^{(c+1)}(\xi)\,d\xi
\end{eqnarray*}
and (\ref{D1.4lem1.2}) follows.
\end{prf}
\begin{lem}\label{D1.4lem2}
Under the assumptions of Lemma \ref{D1.4lem1}, with $1\leq c$,
\[
\int_\R\int_{\R^\times}(iy)^k e^{-iy\xi}\psi(\xi)\,d\xi\,dy=2\pi \psi^{(k)}(0) \qquad (0\leq k\leq c-1)\,,
\]
where each consecutive integral is absolutely convergent.
\end{lem}
\begin{prf}
Since
\[
\int_{\R} |y|^{c-1} \min\{1, |y|^{-c-1}\}\,dy<\infty\,,
\]
the absolute convergence follows from Lemma \ref{D1.4lem1}. Since the Fourier transform of $\psi$ is absolutely integrable and since $\psi$ is continuous at zero, Fourier inversion formula \cite[(7.1.4)]{Hormander} shows that
\begin{equation}\label{D1.4lem2.1}
\int_\R\int_{\R^\times} e^{-iy\xi}\psi(\xi)\,d\xi\,dy=2\pi \psi(0)\,.
\end{equation}
Also, for $0<k$,
\begin{eqnarray*}
&&\int_{\R^\times}(iy)^k e^{-iy\xi}\psi(\xi)\,d\xi=\int_{\R^\times}(-\partial_\xi)\left((iy)^{k-1} e^{-iy\xi}\right)\psi(\xi)\,d\xi\\
&=&\int_{\R^+}(-\partial_\xi)\left((iy)^{k-1} e^{-iy\xi}\right)\psi(\xi)\,d\xi
+\int_{\R^-}(-\partial_\xi)\left((iy)^{k-1} e^{-iy\xi}\right)\psi(\xi)\,d\xi\\
&=&(iy)^{k-1}\psi(0+)+\int_{\R^+}(iy)^{k-1} e^{-iy\xi}\psi'(\xi)\,d\xi\\
&-&(iy)^{k-1}\psi(0-)+\int_{\R^-}(iy)^{k-1} e^{-iy\xi}\psi'(\xi)\,d\xi\\
&=&(iy)^{k-1}\langle \psi\rangle_0+\int_{\R^\times}(iy)^{k-1} e^{-iy\xi}\psi'(\xi)\,d\xi\,.
\end{eqnarray*}
Hence, by induction on $k$ and by our assumption
\begin{eqnarray*}
\int_{\R^\times}(iy)^k e^{-iy\xi}\psi(\xi)\,d\xi&=&(iy)^{k-1}\langle \psi\rangle_0+(iy)^{k-2}\langle \psi'\rangle_0+\dots+\langle \psi^{(k-1)}\rangle_0\\
&+&\int_{\R^\times} e^{-iy\xi}\psi^{(k)}(\xi)\,d\xi\\
&=&\int_{\R^\times}e^{-iy\xi}\psi^{(k)}(\xi)\,d\xi\,.
\end{eqnarray*}
Therefore our lemma follows from (\ref{D1.4lem2.1}).
\end{prf}
The following proposition is an immediate consequence of Lemmas \ref{D1.4lem1}, \ref{D1.4lem2}, and the formula (\ref{D1.3}).
\begin{pro}\label{D1not smooth}
Fix two integers $a,\ b\in\Bbb Z$ and a function $\psi\in \Ss(\R^\times)$. Let $c=-a-b$. If $c\geq 0$ assume that
\begin{equation}\label{D1.4lem1.1not smooth}
\langle \psi\rangle_0=\dots=\langle \psi^{(c)}\rangle_0=0\,.
\end{equation}
Then
\begin{eqnarray}\label{D1.4not smooth}
&&\int_\R\int_{\R^\times}(1+iy)^{-a}(1-iy)^{-b}e^{-iy\xi}\psi(\xi)\,d\xi\,dy\\
&=&\int_{\R^\times}P_{a,b}(\xi)e^{-|\xi|}\psi(\xi)\,d\xi+ Q_{a,b}(\partial_\xi)\psi(\xi)|_{\xi=0}\nn\\
%&=&\int_{\R}\left(P_{a,b}(\xi)e^{-|\xi|}\psi(\xi)+ \delta_0(\xi)Q_{a,b}(\partial_\xi) \psi(\xi)\right)\,d\xi\nn\\
&=&\int_{\R}\left(P_{a,b}(\xi)e^{-|\xi|}+ Q_{a,b}(-\partial_\xi)\delta_0(\xi)\right)\psi(\xi)\,d\xi\,,\nn
\end{eqnarray}
where $\delta_0$ denotes the Dirac delta at $0$.\\
(Recall that $Q_{a,b}=0$ if $c<0$ and $Q_{a,b}$ is a polynomial of degree if $c$, if $c\geq 0$.)
\end{pro}
Let $\Ss(\R^+)$ be the space of the smooth complex valued functions whose all derivatives are rapidly decreasing at infinity and have limits at zero. 
Then $\Ss(\R^+)$ may be viewed as the subspace of the functions in $\Ss(\R^\times)$ which are zero on $\R^-$. Similarly we define $\Ss(\R^-)$.
The following propositions are direct consequences of Proposition \ref{D1not smooth}. We sketch  independent proofs below.
\begin{pro}\label{propD2}
There is a seminorm $p$ on the space $\Ss(\R^+)$ such that 
\begin{equation}\label{propD2.0}
\left| \int_{\R^+} e^{-z\xi}\psi(\xi)\,d\xi\right| \leq \min\{1,|z|^{-1}\} p(\psi) \qquad (\psi\in\Ss(\R^+),\ Re\,z\geq 0)\,,
\end{equation}
and similarly for $\Ss(\R^-)$.

Fix integers $a,b\in\Zb$ with $a+b\geq 1$. Then for any function $\psi\in\Ss(\R^+)$,
\begin{equation}\label{propD2.1}
\int_\R(1+iy)^{-a}(1-iy)^{-b}\int_{\R^+} e^{-iy\xi}\psi(\xi)\,d\xi\,dy
={2\pi} \int_{\R^+}P_{a,b,2}(\xi)e^{-\xi}\psi(\xi)\,d\xi\,,
\end{equation}
and any function $\psi\in \Ss(\R^-)$,
\begin{equation}\label{propD2.2}
\int_\R(1+iy)^{-a}(1-iy)^{-b}\int_{\R^-} e^{-iy\xi}\psi(\xi)\,d\xi\,dy
={2\pi} \int_{\R^-}P_{a,b,-2}(\xi)e^{\xi}\psi(\xi)\,d\xi\,,
\end{equation}
where each consecutive integral is absolutely convergent.
\end{pro}
\begin{prf}
Clearly
\[
\left| \int_{\R^+} e^{-z\xi}\psi(\xi)\,d\xi\right| \leq \int_{\R^+} e^{-\Re\,z\xi}|\psi(\xi)|\,d\xi\leq \parallel \psi\parallel_1\,.
\]
Integration by parts shows that for $z\ne 0$,
\[
 \int_{\R^+} e^{-z\xi}\psi(\xi)\,d\xi=z^{-1}\psi(0)+z^{-1} \int_{\R^+} e^{-z\xi}\psi'(\xi)\,d\xi\,.
\]
Hence (\ref{propD2.0}) follows with $p(\psi)=|\psi(0)|+\parallel\psi\parallel_1+\parallel\psi'\parallel_1$.

Let $a,b\in\Zb$ be such that $a+b\geq 1$. Then the function
\[
(1+z)^{-a}(1-z)^{-b}\int_{\R^+} e^{-z\xi}\psi(\xi)\,d\xi
\]
is continuous on $\Re z\geq 0$ and meromorphic on $\Re z>0$ and (\ref{propD2.0}) shows that it is dominated by
$|z|^{-2}$.
Therefore Cauchy's Theorem implies that the left-hand side of (\ref{propD2.1}) is equal to
\begin{equation*}\label{propD2.h}
-2\pi\, {\rm res}_{z=1}\left((1+z)^{-a}(1-z)^{-b}\int_{\R^+} e^{-z\xi}\psi(\xi)\,d\xi\right).
\end{equation*}
The computation of this residue is straightforward.  This verifies (\ref{propD2.1}).
The proof of (\ref{propD2.2}) is entirely analogous.
\end{prf}
\section{\bf The covering $\wt \G \to \G$}
\label{appenD}
\setcounter{equation}{0}
\setcounter{thh}{0}
In this appendix we recall some results about the splitting of the restrictions $\wt \L\to \L$ of the metaplectic covering
\begin{equation}\label{E1}
1 \to \{\pm 1\} \to \wt\Sp(\Wv)\to \Sp(\Wv) \to 1\,
\end{equation}
to a subgroup $\L$ of the compact member $\G$ of a dual pair $(\G, \G')$ as in \eqref{classificationGG'}. This is well known, but we could not find a reference sketching the proofs of the results we are using in this paper. We are therefore providing a short and complete argument.

If $\K$ is a maximal compact subgroup of $\Sp(\Wv)$, then $\wt \K$ is a maximal compact subgroup of $\wt \Sp(\Wv)$. The group $\wt \Sp(\Wv)$ is connected, noncompact, semisimple and with finite center $\wt \Zg$. (Since $\wt \Sp(\Wv)$ is a double cover of 
$\Sp(\Wv)$, only the connectedness needs to be commented. It follows from the fact that the covering \eqref{E1} does not split; see e.g. \cite[Proposition 4.20]{AubertPrzebinda_omega}
or the original proof \cite[p. 199]{WeilWeil}).
The maximal compact subgroup $\wt \K$ is therefore connected; see e.g. \cite[Chapter VI, Theorem 1.1]{HelgasonDifferential}. Hence the covering 
\begin{equation}\label{E2}
\wt\K\to\K
\end{equation}
does not split.

As is well known, $\K$ is isomorphic to a compact unitary group. In fact, if $\Wv=\R^{2n}$ and 
\begin{equation}
J_{2n}=\begin{pmatrix}
0 & I_n \\ -I_n & 0
\end{pmatrix}\,,
\end{equation}
then 
\begin{equation}\label{E3}
\Sp_{2n}(\R)^{J_{2n}}=\Big\{ \begin{pmatrix} a & -b\\ b & a\end{pmatrix}; \; a,b\in \GL_n(\R),\; ab^t=
ba^t
,\; aa^t+bb^t=I_n\Big\}
\end{equation}
is a maximal compact subgroup of $\Sp_{2n}(\R)$ and 
\begin{equation}\label{E3bis}
\Sp_{2n}(\R)^{J_{2n}}\ni \begin{pmatrix} a & -b\\ b & a\end{pmatrix} \to a+ib\in \Ug_n
\end{equation}
is a Lie group isomorphism. Any two maximal compact subgroups of $\Sp(\Wv)$ are conjugate by an inner automorphism. Let $\K \to \Sp_{2n}(\R)^{J_{2n}}$ be the corresponding isomorphism.  Composition with \eqref{E3bis} fixes then an isomorphism $\phi:\K \to \Ug_n$.
Set
\begin{equation}\label{E4}
{\wt\K}^\phi=\{(u,\zeta)\in\K\times\C^\times;\ \det(\phi(u))=\zeta^2\}
\end{equation}
Recall the bijection between equivalence classes of $n$-fold path-connected coverings and
the conjugacy classes of index-$n$ subgroups of the fundamental group  (see e.g. \cite[Theorem 1.38]{Hatcher}). Then, up to an isomorphism of coverings, $\Ug_n$ has only one connected double cover. 
Hence \eqref{E2} is isomorphic to
\begin{equation}\label{E5}
{\wt\K}^\phi\ni (u,\zeta)\to u\in \K\,.
\end{equation}
Let $\L\subseteq\K$ be any subgroup and 
\begin{equation}\label{E7}
\wt\L\to\L
\end{equation}
the restriction of the covering \eqref{E2} to $\L$. Let ${\wt \L}^\phi$ be the preimage of $\L$ in ${\wt \K}^\phi$. Then \eqref{E7} splits if and only if 
\begin{equation}\label{E7bis}
{\wt\L}^\phi\to\L
\end{equation}
splits, i.e. 
there is a group homomorphism $\L\ni g \to \zeta(g)\in \Ug_1\subset \C^\times$ such that 
$\zeta(g)^2=\det(\phi(g))$ for all $g \in \L$. For instance, if $\L$ is a connected subgroup of 
$\K$ such that
\begin{equation}\label{E6}
\L \subseteq \{u\in\K;\ \det(\phi(u))=1\}\,,
\end{equation}
then \eqref{E7} splits. 

To fix $\phi$,
let $(\Vv,(\cdot,\cdot))$ and $(\Vv',(\cdot,\cdot)')$ be the defining spaces of $\G$ and $\G'$, respectively, with $\dim_\Dc \Vv=d$ and $\dim_\Dc \Vv'=d'$.
Realize $\Wv$ as $\Vv\otimes_\Dc \Vv'$, considered as a real symplectic space, with 
symplectic form 
$\langle \cdot,\cdot\rangle=\tr_{\Dc/\R} \big((\cdot,\cdot)\otimes(\cdot,\cdot)'\big)$, 
where $\tr_{\Dc/\R}$ denotes the reduced trace; see \cite[\S 5]{howetheta} and \cite[p. 169]{Weil_Basic}. 
Then the group $\G$ is viewed as a subgroup of $\Sp(\Wv)$ via the identification $\G\ni g\to g\otimes 1 \in \Sp(\Wv)$. \footnote{Following the notation at the beginning of Section \ref{section:dualpairs-supergroups}, one should identify $g$ and $(g^{-1})^t\otimes 1$.}
Similarly, $\G'$ is viewed as a subgroup of $\Sp(\Wv)$ via the identification $\G'\ni g' \to 
1\otimes g'\in \Sp(\Wv)$.
Recall that $n$-by-$n$-matrices over $\C$ can be identified with $2n$-by-$2n$ matrices over $\R$ under the isomorphism 
$$
\alpha: M \to \begin{pmatrix}
\Re M & -\Im M\\ \Im M & \Re M
\end{pmatrix}\,.
$$ 
Moreover, $n$-by-$n$-matrices over $\Ha$ can be identified with $2n$-by-$2n$ matrices over $\C$ under the isomorphism 
$$
\beta: M \to \begin{pmatrix}
z_1(M) & -\overline{z_2(M)}\\ z_2(M)& \overline{z_1(M)}
\end{pmatrix}\,.
$$ 
Here, for $v\in \Ha$, we write $v=z_1(v)+jz_2(v)$ with $z_1(v), z_2(v)\in \C$, and we similarly define $z_1(M)$ and $z_2(M)$ if $M$ is a matrix over $\Ha$.

Since $\G$ is compact, there is a compatible positive complex structure $J$ on $\Wv$ such that the maximal compact subgroup $\K=\Sp(\Wv)^J$ of $\Sp(\Wv)$ contains $\G$. Moreover, since $\G$ commutes with $J$, there is $J'\in \G'$ such that $J=1\otimes J'$. 
Set $I_{p,q}=\begin{pmatrix} I_p &0\\0 & 
-I_q
\end{pmatrix}$. Then, the explicit expressions of $J'$ with respect to the standard basis of $\Vv\simeq \Dc^d$ and of $J$ with respect to the standard basis of $\Wv\simeq \R^{2n}$ are given as follows:
\medskip

%%%
\begin{center}
\extrarowheight=.2em
\begin{tabular}{|c|c|c|c|}
\hline
$(\G,\G')$ & $J'$ & $n$ & $J$\\[.2em]
\hline
\hline
$(\Og_d,\Sp_{2m}(\R))$ & $J_{2m}$ & $md$ & $J_{2md}$ \\[.2em]
\hline
$(\Ug_d,\Ug_{p,q})$ & $-iI_{p,q}$ & $d(p+q)$ & $\begin{pmatrix} 0 & I_{dp,dq} \\ -I_{dp,dq} & 0 \end{pmatrix}$\\[.2em]
\hline
$(\Sp_d,\Og^*_{2m})$ & $-jI_{m}$ & $2md$ & $\begin{pmatrix} J_{2pm} & 0 \\ 0 & J_{2pm} \end{pmatrix}$\\
\hline
\end{tabular}
\end{center}
\medskip

\noindent
Notice that in the $(\Ug_d,\Ug_{p,q})$-case we have $SJS^{-1}=J_{2d(p+q)}$ for 
$S=\begin{pmatrix}
I_{d(p+q)} & 0 \\
0 &  I_{dp,dq}
\end{pmatrix}$; in the $(\Sp_d,\Og^*_{2m})$-case, $TJT^{-1}=J_{4pm}$ for 
$T=\begin{pmatrix} I & 0 & 0 & 0\\ 0 & 0 & I & 0 \\ 0 & I & 0 & 0 \\ 0 & 0 & 0 & I
\end{pmatrix}$.
Hence, in all cases we can embed $\G$ in \eqref{E3} from the identification $g \to g\otimes 1 \in \Sp(\Wv)^J$ followed by the isomorphism of $\Sp(\Wv)^J$ and $\Sp_{2n}(\R)^{J_{2n}}$ corresponding to the conjugations by $S$ or $T$, and then apply \eqref{E3bis}.
We obtain:
%%up to conjugation
\begin{equation}
\label{E7bis}
\det(\phi(g))=\begin{cases}
\det(g)_\Vv^m &\text{if $(\G,\G')=(\Og_d,\Sp_{2m}(\R))$}\\
\det(g)_\Vv^{p-q} &\text{if $(\G,\G')=(\Ug_d,\Ug_{p,q})$}\\
1 & \text{if $(\G,\G')=(\Sp_d,\Og^*_{2m})$} 
\end{cases}\,,
\end{equation}
where $\det(g)_\Vv $ denotes the determinant of $g$ as an element of 
$\G\subseteq\GL_\Dc(\Vv)$. (The determinant of an $n$-by-$n$ matrix over $\Ha$ can be reduced to a determinant of a $2n$-by-$2n$ matrix over $\C$ via the isomorphism $\beta$. For elements of $\Sp(d)$, this notion of determinant coincides with other possible notions of quaternionic determinants; see \cite{Aslaksen96} for additional information.)

%It follows from \eqref{E6} and \eqref{E7bis} that the covering $\wt \Sp_d \to \Sp_d$ always splits and $\wt \Og_d \to \Og_d$ splits if $m$ is even. More generally, we have the following proposition. 

\begin{pro}\label{pro:det-covering}
The covering $\wt \G \to \G$ splits if and only if $\det(\phi(g))$ is a square. This happens for all pairs $(\G,\G')$ different from $(\Og_d,\Sp_{2m}(\R))$ with $m$ odd and $(\Ug_d,\Ug_{p,q})$ with $p+q$ odd. 
In these two non-splitting cases, the covering 
$\wt \G \to \G$ is isomorphic to the $\det^{1/2}$-covering
\begin{equation}\label{E8}
\sqrt{\G} \ni (g,\zeta) \to g\in \G
\end{equation}
where 
\begin{equation}
\label{E9}
\sqrt{\G}=\{(g,\zeta)\in \G\times \C^\times; \zeta^2=\det(g)_\Vv\}\,.
\end{equation}
\end{pro}
\begin{proof}
By \eqref{E7bis} there is a group homomorphism $\G\ni g \to \zeta(g)\in \Ug_1 \subseteq \C^\times$ so that $\zeta(g)^2=\det(\phi(g))$ for all pairs $(\G,\G')$ except at most the two cases listed in the statement of the Proposition.  

Suppose that $\G'=\Sp_{2m}(\R)$, and let $\zeta:\Og_d\to \Ug_1$ be a continuous group homomorphism so that $\zeta(g)^2=\det(g)_\Vv^m=(\pm 1)^m$. Then  $\zeta(\Og_d)\subseteq \{\pm 1,\pm i\}$ and it is a subgroup with at most two elements. So $\zeta(\Og_d)\subseteq \{\pm 1\}$. On the other hand, if $g \in \Og_d \setminus \SO_d$, 
then $\det(g)_\G=-1$. Thus $\zeta(g)^2\neq \det(g)_\Vv^m$ if $m$ is odd.

Suppose now that $\G'=\Ug_{p,q}$, and let $\zeta:\Ug_d\to \Ug_1$ be a continuous group homomorphism so that $\zeta(g)^2=\det(g)_\Vv^{p-q}$. Restriction to $\Ug_1 \equiv \{\diag(h, 1\dots, 1); h \in \Ug_1\} \subseteq \Ug_d$ yields a continuous group homomorphism $h \in \Ug_1 \to \zeta(h) \in  \Ug_1$. Thus, there is  $k\in \Zb$ so that $\zeta(h)=h^k$ for all $h \in \Ug_1$. 
So $h^{2k}=\zeta(h)^2=\det(\diag(h, 1,\dots, 1))^{p-q}$ implies that $p+q$ must be even. 

For the last statement, consider for $k \in\Zb$ the covering $\Mg_k=\{(g,\zeta)\in \G\times \C^\times;\zeta^2=\det(g)_\Vv^{2k+1}\}$ of $\G$. 
Then $(g,\zeta)\to (g,\zeta^{\frac{1}{2k+1}})$ is a covering isomorphism between $\Mg_k$ and $\Mg_0$.   
\end{proof}

\begin{rem}
\label{rem:Fock}
Keep the notation of \eqref{E4} and 
let $\alpha:\widetilde{\K}^\phi\to \widetilde{\K}$ be the isomorphism lifting 
$\phi^{-1}:\Ug_n \to \K$. Then, by \cite[Proposition 4.39]{FollandPhase} or \cite[(1.4.17)]{PrzebindaExample}, the map 
$$(u,\zeta)\to \zeta^{-1} \omega(\alpha(u,\zeta))$$ 
is independent of $\zeta$.
\end{rem}

%%%%%%%%%%%%%%%%
\section{\bf On the nonoccurrence of the determinant character of $\Og_d$ in Howe's correspondence}
\label{det and omega}
\setcounter{thh}{0}
\setcounter{equation}{0}

Consider the reductive dual pair $(\Og_d,\Sp_{2n}(\R))$ where $d>n$.
Let $\M_{d,n}(\R)$ denote the space of $d\times n$ matrices with real coefficients and
consider the Schr\"odinger model for the Weil representation $\omega$, with space of 
smooth vectors $\mathcal{S}=\mathcal{S}(\M_{d,n}(\R))$. Moreover, let $\chi_+$ be the character of 
$\wt{\Og}_d$ defined in \eqref{chi+-on-O}. 
As recalled on page \pageref{notation-Od-irreps}, the 
representation $\omega\otimes\chi_+^{-1}$ descends to a representation $\omega_0$ of $\Og_d$
given by 
\begin{equation}
\label{omega0-appen}
\omega_0(g)f(x)=f(g^{-1}x) \qquad (g\in \Og_d, \, f\in \mathcal{S}, \, x\in \M_{d,n}(\R))\,.
\end{equation}
In this appendix, we prove that, under the assumption that $d>n$, the determinant character $\det$ does not occur in $\omega_0$. 
This property is a consequence of \cite[(C.43) Corollary]{PrzebindaExample} (which considers the more general case of the pseudo-orthogonal groups $\Og_{p,q}$, where $p+q=d>n$). However, the proof in \cite{PrzebindaExample} uses part of the classification of the $\K$-types of representations occurring in Howe's correspondence, determined by \cite{KashiwaraVergne}. The proof below, which follows  the $p$-adic case in \cite[p. 399]{Rallis-Howeduality}, is classification-free.

\begin{pro}
\label{non-occurence-det}
If $d>n$, then $\det$ does not occur in $\omega_0$. In other words: if $d>n$, then there is no character $\sigma$ of $\wt{\Og}_d$ occurring in Howe's correspondence such that $\sigma\otimes \chi_+^{-1}$ descends to the determinant character $\det$ of $\Og_d$. 
\end{pro}
\begin{proof}
We argue by contradiction. Suppose $f_0\in \mathcal{S}$ is a non-zero function satisfying 
$$
f_0(g^{-1}x)=\det(g) f_0(x) \qquad (g\in \Og_d, \, x\in \M_{d,n}(\R))\,.
$$
Define
$
\Zg=\{x\in \M_{d,n}(\R): \text{$x$ has maximal rank $n$}\}\,.
$ 
Then $\Zg$ is $\Og_d$-invariant and, by the density of $\Zg$ in $\M_{d,n}(\R)$, $f_0|_{\Zg}\neq 0$. 
Decompose $\Zg$  as a union of $\Og_d$-orbits $\mathcal{O}$. Then there is an $\Og_d$-orbit 
$\mathcal{O}$ such that $f_0|_{\mathcal{O}}\neq 0$. 
Set $\varphi=f_0|_{\mathcal{O}}$. Then 
\begin{equation}
\label{varphi-det}
\varphi(g^{-1}x)=\det(g) \varphi(x) \qquad (g\in \Og_d, \, x\in \mathcal{O})\,.
\end{equation}
Since $\mathcal{O}\subseteq \Zg$, the centralizer of any element in $\mathcal{O}$ is isomorphic to $\Og_{d-n}$. Hence 
$\mathcal{O}=\Og_d/\Og_{d-n}$ and $\varphi\in \Ind_{\Og_{d-n}}^{\Og_d}(1)$. 
By \eqref{varphi-det}, $\det$ occurs in $\Ind_{\Og_{d-n}}^{\Og_d}(1)$. Frobenius' reciprocity then
implies that the character $\det|_{\Og_{d-n}}$ contains $1$, i.e. $\det|_{\Og_{d-n}}=1$. 
This is clearly impossibile, and we have reached a contradiction. Thus
$\det$ cannot occur in $\omega_0$.
\end{proof}

%%%%%%%%%%%%%%%%%
\section{\bf Tensor product decomposition of the embedding $T$ over complementary invariant symplectic subspaces of $\Wv$}
\label{WandT}
\setcounter{thh}{0}
\setcounter{equation}{0}

We keep the notation introduced in section \ref{section:preliminaries}. Let
\begin{equation}
\label{chi+onSp}
\chi_+(\wt g)=\frac{\Theta(\wt g)}{|\Theta(\wt g)|} \qquad (g\in \Sp(\Wv))
\end{equation}
(Recall that $\chi_+$ is not a character on $\wt{\Sp}(\Wv)$, since $\wt{\Sp}(\Wv)$ does not have any nontrivial character. However, $\chi_+$ becomes a character when restricted to specific subgroups of 
$\wt{\Sp}(\Wv)$, such as $\wt{\Og}_d$; see \eqref{chi+-on-O}.)
By definition, see \eqref{Tt}, 
\begin{equation}
\label{chi+invT}
\chi_+^{-1}(\wt{g})T(\wt{g})=|\Theta(\wt{g})|\chi_{c(g)}\mu_{(g-1)\Wv}  \qquad (g\in \Sp(\Wv))
\end{equation}
descends to a distribution on $\Sp(\Wv)$. 

Let $\Wv=\Wv_1\oplus \Wv_2$ be an orthogonal decomposition of $\Wv$, and endow each subspace 
$\Wv_j$ (where $j=1,2$) of the symplectic form $\inner{\cdot}{\cdot}_j=\inner{\cdot}{\cdot}|_{\Wv_j\times \Wv_j}$. Suppose that $g\in \Sp(\Wv)$ preserves $\Wv_1$ and $\Wv_2$. Let $g_1$ and $g_2$ 
respectively denote the restrictions $g|_{\Wv_1}$ and $g|_{\Wv_2}$ of $g$ to 
these subspaces. 
Suppose we have chosen a complete polarization $\Wv=\X\oplus\Y$ of $\Wv$ such that 
$\X=\X_1\oplus \X_2$ and $\Y=\Y_1\oplus \Y_2$, where $\Wv_1=\X_1\oplus \Y_1$ and 
$\Wv_2=\X_2\oplus \Y_2$ are complete polarizations. Similarly, suppose that the compatible positive complex structures $J$, $J_1$, $J_2$ on $\Wv$, $\Wv_1$, $\Wv_2$, respectively, satisfy $J=J_1\times J_2$. Then $J(\X)=\Y$ if and only if $J(\X_1)=\Y_1$ and $J(\X_2)=\Y_2$, which we assume. 

Write $T_\Wv$, $T_{\Wv_1}$ and $T_{\Wv_2}$ for the distributions corresponding to $\wt{\Sp}(\Wv)$, $\wt{\Sp}(\Wv_1)$, $\wt{\Sp}(\Wv_2)$, respectively. Similar notation will apply to other symbols occurring in the computations below. For the tensor product of tempered distributions, we refer to 
\cite[Corollary of Theorem 51.6, especially (51.7)]{Treves}.

\begin{lem}
\label{lemma:WandT}
In the above notations, 
$$
|\Theta_\Wv(\wt{g})|\chi_{c(g)}\mu_{(g-1)\Wv}=|\Theta_{\Wv_1}(\wt{g_1})|\chi_{c(g_1)}\mu_{(g_1-1)\Wv_1}\otimes |\Theta_{\Wv_2}(\wt{g_2})|\chi_{c(g_2)}\mu_{(g_2-1)\Wv_2}\,.
$$
Consequently, independently of the choice of the preimages $\t g$, $\t g_1$ and $\t g_2$   of $g$, $g_1$ and $g_2$ in 
$\wt{\Sp}(\Wv)$, $\wt{\Sp}(\Wv_1)$, $\wt{\Sp}(\Wv_2)$, respectively,
$$
\chi_{\Wv,+}^{-1}(\wt{g})T_{\Wv}(\wt{g})=\chi_{\Wv_1,+}^{-1}(\wt{g_1})T_{\Wv_1}(\wt{g_1})\otimes
\chi_{\Wv_2,+}^{-1}(\wt{g_2})T_{\Wv_2}(\wt{g_2})
\,.
$$
Hence, if the elements $\t g$, $\t g_1$ and $\t g_2$ respectively are chosen so that
\[
\chi_{\Wv,+}^{-1}(\wt{g})=\chi_{\Wv_1,+}^{-1}(\wt{g_1})\chi_{\Wv_2,+}^{-1}(\wt{g_2})\,,
\]
then
$$
T_{\Wv}(\wt{g})=T_{\Wv_1}(\wt{g_1})\otimes T_{\Wv_2}(\wt{g_2})
\,.
$$
\end{lem}
\begin{proof}
Since $\Wv=\Wv_1\oplus \Wv_2$ and $g_1=g|_{\Wv_1}$, $g_2=g|_{\Wv_2}$, we have 
$(g-1)\Wv=(g_1-1)\Wv_1\oplus (g_2-1)\Wv_2$. Recall from \cite[Definitions 4.16, 4.18 and 4.23]{AubertPrzebinda_omega} that 
$$
\Theta(\wt{g})^2=\Theta^2(g) \qquad (g\in \Sp(\Wv))\,.
$$
Thus 
$|\Theta_\Vv(\wt{g})|^2=|\Theta^2_\Vv(g)|$ for $\Vv\in \{\Wv,\Wv_1,\Wv_2\}$. 
It follows that $|\Theta_\Wv(\wt{g})|=|\Theta_{\Wv_1}(\wt{g_1})||\Theta_{\Wv_2}(\wt{g_2})|$, and this independently of the choice of the preimages of $g$, $g_1$ and $g_2$ in 
$\wt{\Sp}(\Wv)$, $\wt{\Sp}(\Wv_1)$, $\wt{\Sp}(\Wv_2)$, respectively. Since the decomposition
$\Wv=\Wv_1\oplus \Wv_2$ is orthogonal, 
$$
\inner{c(g)w}{w}=\inner{c(g_1)w_1}{w_1}_1+\inner{c(g_2)w_2}{w_2}_2
\qquad (w_j\in (g_j-1)\Wv_j, \, j=1,2,\, w=w_1+w_2)\,,
$$
where $c$ denotes the Cayley transform. 
Therefore $\chi_{c(g)}=\chi_{c(g_1)}\otimes\chi_{c(g_2)}$ on $\Wv=\Wv_1\oplus \Wv_2$.
Finally, the normalization of measures on subspaces of $\Wv$ fixed at the beginning of section \ref{section:preliminaries} is such that $\mu_{(g-1)\Wv}=\mu_{(g_1-1)\Wv_1}\otimes \mu_{(g_2-1)\Wv_2}$.
\end{proof}

%%%%%%%%%%%%%%%%%
\section{\bf Highest weights of irreducible genuine representations of $\wt\G$}
\label{appenE}
\setcounter{thh}{0}
\setcounter{equation}{0}
\setlength{\abovedisplayskip}{0pt plus 3pt}
\setlength{\belowdisplayskip}{7pt plus 3pt minus 4pt}

In this appendix we collect the roots and weights for the irreducible genuine representations of $\wt\G$, where $\G$ is a compact member of a reductive dual pair $(\G,\G')$.
Let $\h$ be a fixed Cartan subalgebra of the Lie algebra $\g$ of $\G$.
We denote by $\Delta^+$ a choice of positive roots for $(\g_\C,\h_\C)$ and by $\rho$ the one-half of their sum. Each genuine irreducible representation of $\wt \G$  has highest weight $\lambda= \sum_{j=1}^l \lambda_j e_j$ listed below. 
\smallskip

\underline{$(\G,\G')=(\Ug_l,\Ug_{p,q}), \; l\geq 1, \, q\geq p\geq 0, \, p+q\geq 1$:}
\medskip

If $l=1$, then $\h_\C=\g_\C$.
If $l\geq 2$, then:
\begin{eqnarray*}
\addtolength{\jot}{0\jot}
&&\Delta^+=\{e_j-e_k; \; 1\leq j<k\leq l\} \; \text{(type $A_{l-1}$)}\,, \quad
\rho=\sum_{j=1}^l \Big(\frac{l+1}{2}-j\Big)e_j\,,\\[-1em]
\label{genuine-hw-Ul}
&&\lambda_j=\frac{p-q}{2}+\nu_j, \quad \nu_j\in \Zb, \quad \nu_1 \geq \nu_2 \geq \cdots \geq \nu_l\,.
\end{eqnarray*}

\underline{$(\G,\G')=(\Og_{2l+1},\Sp_{2l'}(\R)), \; l\geq 0, \, l'\geq 1$:}
\medskip

If $l=0$, then $\g=0$. If $l\geq 1$, then:
\begin{eqnarray*}
&&\Delta^+=\{e_j \pm e_k; \; 1\leq j<k\leq l\} \cup \{e_j; \; 1\leq j\leq l\}  \; \text{(type $B_{l}$)}\,, \quad \rho=\sum_{j=1}^l \Big(l+\frac{1}{2}-j\Big)e_j\,,\\[-1em]
\label{genuine-hw-O2l+1}
&&\lambda_j \in \Zb, \quad \lambda_1\geq \lambda_2 \geq \cdots \geq \lambda_l\geq 0\,.\\[.1em]
&&\text{
There are two irreducible genuine representations of highest weight $\lambda$.
}
\end{eqnarray*}

\underline{$(\G,\G')=(\Sp_{l},\Og^*_{2l'}), \; l\geq 1, \, l'\geq 1$ (for $l'=1$ this is a degenerate pair)}:
\begin{eqnarray*}
&&\Delta^+=\{e_j \pm e_k; \; 1\leq j<k\leq l\} \cup \{2e_j; \; 1\leq j\leq l\}  \; \text{(type $C_{l}$)}\,,\quad\rho=\sum_{j=1}^l (l+1-j) e_j\,,\\[-1em]
\label{genuine-hw-Spl}
&&\lambda_j \in \Zb, \quad \lambda_1\geq \lambda_2 \geq \cdots \geq \lambda_l\geq 0\,.
\end{eqnarray*}

\underline{$(\G,\G')=(\Og_{2l},\Sp_{2l'}(\R)), \; l\geq 1, \, l'\geq 1$:}
\medskip

If $l=1$, then $\h_\C=\g_\C$. If $l\geq 2$, then:
\begin{eqnarray*}
&&\Delta^+=\{e_j \pm e_k; \; 1\leq j<k\leq l\}  \; \text{(type $D_{l}$)}\,,\quad\rho=\sum_{j=1}^l (l-j)e_j\,,\\[-1em]
\label{genuine-hw-O2l}
&&\lambda_j \in \Zb, \quad \lambda_1\geq \lambda_2 \geq \cdots \geq |\lambda_l|\,.
\end{eqnarray*}
\quad\parbox{15.5truecm}{
The weights $(\lambda_1,\dots,\lambda_{l-1},\pm \lambda_l)$ yield the same representation of $\Og_{2l}$ if $\lambda_l\neq 0$. \\

If $\lambda_l=0$, there are two irreducible genuine representations of highest weight $\lambda$. }

%%

%%%%%%%%%%%%%%%%%%%%%%
%%%%%%%%%%%%%%%%%%%%%%
\section{\bf Integration on the quotient space $\Sg/\Sg^\hs1$}
\setcounter{thh}{0}
\setcounter{equation}{0}

We retain the notation of sections \ref{section:dualpairs-supergroups} and \ref{section:orbital integrals}. The purpose of this appendix is to prove the following lemma. 

\begin{lem}
\label{lem:orbitalintegrals-groupintegrals}
Suppose first that $\G\neq \Og_{2l+1}$ with $l<l'$. 
Then there are positive constants $C_1$ and $C_2$ such that for all $\phi\in C_c(\Wv)$ and $w\in {\hs1}^{\rm reg}$
\begin{align}
\label{integralonS/Sh1-1}
\int_{\Sg/\Sg^{\hs1}} \phi(s.w) \; d(s\Sg^{\hs1})&= 
C_1 \int_{\G}\int_{\G'/\Zg'} \phi((g,g').w) \; dg\,d(g'\Zg') \qquad \text{if $l\leq l'$}\\
\label{integralonS/Sh1-2}
\int_{\Sg/\Sg^{\hs1}} \phi(s.w) \; d(s\Sg^{\hs1})&= 
C_2 \int_{\G/\Zg}\int_{\G'} \phi((g,g').w) \; d(g\Zg)\,dg' \qquad \text{if $l>l'$}\,.
\end{align}
Now, let $\G=\Og_{2l+1}$ with $l<l'$ and let $w_0\in \mathfrak{s}_1(\V^0)$ be a nonzero element. 
Then there is a positive constant $C_3$ such that for all $\phi\in C_c(\Wv)$ and 
$w\in {\hs1}^{\rm reg}$
\begin{equation}
\label{integralonS/Sh1+w0}
\int_{\Sg/\Sg^{\hs1+w_0}} \phi(s.(w +w_0)) \; d(s\Sg^{\hs1+w_0})= 
C_3 \int_{\G}\int_{\G'/{\Zg'\,}^n} \phi((g,g'). (w +w_0)) \; dg\,d(g'{\Zg'\,}^n)\,,
\end{equation}
where ${\Zg'\,}^n$ is the centralizer in $\Zg'$ of $n=\tau'(w_0)$.
\end{lem}
%%%

Before proving Lemma \ref{lem:orbitalintegrals-groupintegrals}, let us consider the special case of the dual pair $(\G,\G')=(\Og_{1},\Sp_{2n}(\R))$, which is not included in this lemma but will be needed in its proof. In the notation of section \ref{section:dualpairs-supergroups}, $\V=\V_{\overline{0}}\oplus\V_{\overline{1}}$, where 
$\dim \V_{\overline{0}}=1$ and $\dim \V_{\overline{1}}=2n$. 
We have the identifications
$$
\Sg=\G\times \G'=\Og(\V_{\overline{0}}) \times \Sp(\V_{\overline{1}})\,, \qquad \Wv=\Hom(\V_{\overline{1}},\V_{\overline{0}})\,.
$$
Let $0\neq w_0\in \Wv$.
We shall describe ${\rm Stab}_{\G'}(w_0)$, the stabilizer of $w_0$ in $\G'=\Sp(\V_{\overline{1}})$, as well as $\big(\Og(\V_{\overline{0}}) \times \Sp(\V_{\overline{1}})\big)^{w_0^2}$ and $\big(\Og(\V_{\overline{0}}) \times \Sp(\V_{\overline{1}})\big)^{w_0}$.

Since $\dim \Ker w_0=\dim \Wv-1$, we see that $\dim (\Ker w_0)^\perp=1$. Let $\Xv=(\Ker w_0)^\perp$. Since $\dim \Xv=1$, this is an isotropic subspace of $\Wv$. Furthermore $\Ker w_0=\Xv^\perp$.
Let $\Yv\subseteq \Wv$ be a subspace  of dimension 1 such that $\Wv=\Ker w_0 \oplus \Yv$. Set $\Uv=(\Xv+ \Yv)^\perp$. Then the restriction of the symplectic form of $\Wv$ to $\Uv$ is non-degenerate and 
\begin{equation}
\label{decompositionV1}
\V_{\overline{1}}=\Xv\oplus \Uv\oplus \Yv\,.
\end{equation}
Let $P_\Yv\subseteq \G'$ be the parabolic subgroup preserving $\Yv$. Then we have an isomorphism 
$$
P_\Yv=\GL_1(\Yv)\times \Sp(\Uv)\times \N\,,
$$
where $\N$ is the uniponent radical, isomorphic to a Heisenberg group. We see from \eqref{decompositionV1} that 
\begin{equation}
\label{StabG'w}
{\rm Stab}_{\G'}(w_0)=\{1\} \times \Sp(\Uv)\times \N\,.
\end{equation}

If $w_1,w_2\in \mathfrak{s}_{\overline{1}}(\V)$ are non-zero and such that $w_1^2=w_2^2$, 
then $w_2=\pm w_1$. Equivalently, let $\tau':\Wv\to \g'=\sp(\Wv)$ denote the unnormalized 
moment map. Then $\tau'(w_1)=\tau'(w_2)$ implies $w_2=\pm w_1$, because $\Og_1$ acts transitively on the fibers of $\tau'$. Equivalently, if one thinks of $\Wv$ as $\M_{1,2n}(\R)$ and 
setting $w^*=Jw^t$ for $J=\begin{pmatrix} 0 & 1_{n}\\ -1_{n} & 0 \end{pmatrix}$, one has that
$w_1^*w_1=w_2^*w_2$. This is equivalent to  $w_1^tw_1=w_2^t w_2$, which implies $w_2=\pm w_1$.

Now, one readily checks that $g'\in \Sp(\V_{\overline{1}})^{w_0^2}$ if and only if $g'\tau'(w_0){g'}^{-1}=\tau'(w_0)$. Since, 
for $g'\in \Sp(\V_{\overline{1}})$,
$$
g'\tau'(w_0){g'}^{-1}=g'w_0^*w_0{g'}^{-1}=(w_0{g'}^{-1})^*(w_0{g'}^{-1})=\tau'(w_0{g'}^{-1})\,,
$$
this is equivalent to $\tau'(w_0{g'}^{-1})=\tau'(w_0)$, i.e. $w_0{g'}^{-1}=\pm w_0$. In turn, this means that 
$\pm g'\in {\rm Stab}_{\G'}(w_0)$. Thus
\begin{equation}
\label{G'w2}
{\Sp(\V_{\overline{1}})\,}^{w_0^2}=\{\pm 1\} \times \Sp(\Uv)\times \N\,.
\end{equation}
It follows that 
\begin{align}
\label{Sw2-O1}
\big(\Og(\V_{\overline{0}})\times \Sp(\V_{\overline{1}})\big)^{w_0^2}&=\{\pm 1\} \times \big(\{\pm 1\} \times \Sp(\Uv)\times \N\big)\\
\noalign{\noindent and}
\label{Sw-O1}
\big(\Og(\V_{\overline{0}}) \times \Sp(\V_{\overline{1}})\big)^{w_0}&=\{(\varepsilon;\varepsilon,m,n);\varepsilon=\pm 1, \, m\in \Sp(\Uv), \, n\in \N\}\,.
\end{align}
Notice that they do not depend on the choice of $0\neq w_0\in \Wv$. Moreover, 
$$
\big(\Og(\V_{\overline{0}})\times \Sp(\V_{\overline{1}})\big)^{w_0^2}/\big(\Og(\V_{\overline{0}}) \times \Sp(\V_{\overline{1}})\big)^{w_0}=(\{\pm 1\} \times \{\pm 1\})/\{\pm (1,1)\}
$$
is a group isomorphic to $\Og_1$.

\medskip

\noindent \textit{Proof of Lemma \ref{lem:orbitalintegrals-groupintegrals}.\;}
We now prove \eqref{integralonS/Sh1-1}, excluding for the moment the pair 
$(\G,\G')=(\Og_{2l+1},\Sp_{2l}(\R))$.

If $l\leq l'$, then $\h'=\h\oplus \h''$. Write $\z'=\h\oplus \z''$ and, for the corresponding groups, $\Zg'=\H\times \Zg''$.  Then $\Sg^{\hs1^2}=\H\times \Zg'$.
 
Let $\Delta:\H \to \G\times \G'$ be defined by $\Delta(h)=(h,(h,1_{l'-l}))$, where 
$1_r$ denotes the identity matrix of size $r$.
Then $\Sg^{\hs1}=\Delta(\H)(\{1_{l}\}\times (\{1_{l}\}\times \Zg''))$. Set
$$
\Lg=\Sg^{\hs1^2}/\Sg^{\hs1}=(\H\times \H\times \Zg'')/\Sg^{\hs1}=(\H\times \H \times \{1_{l'-l}\}) /\Delta(\H)\,.
$$
Then $\Lg$ is a compact abelian group because so is $\H$. It acts on $\Sg/\Sg^\hs1$ by 
$$
(g,g')\Sg^{\hs1}\cdot (h_1,h_2,1_{l'-l})\Delta(\H)=(gh_1,g'(h_2,1_{l'-l}))\Sg^{\hs1}\,.
$$
The action is proper and free. Hence the quotient space 
$(\Sg/\Sg^\hs1)/\Lg$, i.e. the space of orbits for this action, has a unique structure of smooth manifold such that the canonical projection $\Sg/\Sg^{\hs1} \to (\Sg/\Sg^{\hs1})/\Lg$ is a principal fiber bundle with structure group $\Lg$. 
Since we have fixed a Haar measure on $\H$, we also have Haar measures on 
$\H\times \H \times \{1_{l'-l}\}$ and $\Delta(\H)$. This fixes a quotient measure on
$\Lg=(\H\times \H \times \{1_{l'-l}\}) /\Delta(\H)$. Recall the notation $d(sS^{\hs1})$ for the quotient measure of $\Sg/\Sg^{\hs1}$. 
Then there is a unique 
measure $ds^\bullet$ on $(\Sg/\Sg^{\hs1})/\Lg$ such that for all $\Phi\in C_c(\Sg/\Sg^{\hs1})$
\begin{align*}
&\int_{\Sg/\Sg^{\hs1}} \Phi(s\Sg^{\hs1}) \; d(s\Sg^{\hs1})\\
&=
\int_{(\Sg/\Sg^{\hs1})/\Lg} \left( \int_{(\H\times \H \times \{1_{l'-l}\}) /\Delta(\H)} 
\Phi\big((g,g')(h_1,h_2,1_{l'-l})\Sg^{\hs1}\big)d((h_1,h_2,1_{l'-l})\Delta(\H)) \right) d(g,g')^\bullet\\
&=\frac{1}{\vol(\Delta(\H))} \int_{(\Sg/\Sg^{\hs1})/\Lg} \left( \int_{\H\times \H} 
\Phi\big((g,g')(h_1,h_2,1_{l'-l})\Sg^{\hs1}\big)d(h_1,h_2) \right) d(g,g')^\bullet\,;
\end{align*}
see e.g. \cite[\S 3.13, p. 183]{Duistermaat-Kolk}. 
As a set, 
\begin{align}
\label{SmodSh1modL}
(\Sg/\Sg^{\hs1})/\Lg&=\Big((\G\times \G')/\Sg^{\hs1}\Big)/\Big((\H\times \H \times \Zg'')/\Sg^{\hs1}\Big) \nn\\
&=(\G\times \G')/(\H\times \H \times \Zg'') \nn\\
&=(\G\times \G')/(\H\times\Zg')=\G/\H\times \G'/\Zg'\,,
\end{align}
where the second equality holds under the identification $(g,g')\Sg^{\hs1} L=(g,g')(\H\times\H\times\Zg'')$.
Since the measure $d(sS^{\hs1})$ on $\Sg/\Sg^{\hs1}$ is invariant with respect to the action of $\Sg$ by left-translation and this action commutes with the right-action of $\Lg$ on $\Sg/\Sg^{\hs1}$, the measure $ds^\bullet$ is left $\Sg$-invariant. By the above identification, $(\G\times \G')/(\H\times\Zg')$ is endowed with an $\Sg$-invariant measure, which must be a positive multiple of the quotient measure of those of $\G\times \G'$ and $\H\times\Zg'$. Thus 
$ds^\bullet$ is a positive multiple of the product measure of the quotient measures of 
$\G/\H$ and $\G'/\Zg'$. 
In conclusion, there is a positive constant $C$ such that for every 
$\Phi\in C_c(\Sg/\Sg^{\hs1})$
\begin{align}
\label{integral on S/Shh1 and on G/HxG'/Z'}
&\int_{\Sg/\Sg^{\hs1}} \Phi(s\Sg^{\hs1}) \; d(s\Sg^{\hs1}) \nn\\
&=C \int_{\G/\H \times \G'/\Zg'} \left( \int_{\H\times \H} 
\Phi\big((g,g')(h_1,h_2,1_{l'-l})\Sg^{\hs1}\big)d(h_1,h_2) \right) d(g\H)\, d(g'\Zg')\nn\,.
\end{align}
Suppose that  $\Phi(s)=\phi(s.w)$, where $\phi\in C_c(\Wv)$ and $w\in {\hs1}^{\rm reg}$.
Hence $\phi(s\Sg^{\hs1}.w)=\phi(s.w)$.
Observe that 
$$
(g,g')(h_1,h_2,1_{l'-l}).w=gh_1w(h_2^{-1},1_{l'-l})g'^{-1}=gh_1h_2^{-1}wg'^{-1}
=(gh_1h_2^{-1},g').w\,.
$$
Hence
\begin{align*}
 \int_{\H\times \H} \phi\big((g,g')(h_1,h_2,1_{l'-l}).w\big) \, d(h_1,h_2)&=
 \int_{\H}\int_{\H} \phi((gh_1,g').w)\; dh_1\, dh_2\\
 &=\vol(\H) \int_{\H} \phi((gh_1,g').w)\; dh_1
\end{align*}
and  
\begin{align*}
\int_{\G/\H \times \G'/\Zg'}&\int_{\H\times \H} \phi\big((g,g')(h_1,h_2,1_{l'-l}).w\big)\, d(h_1,h_2)\, d(g\H)\,d(g'\Zg')\\
&=\vol(\H) \int_{\G/\H}\int_{\G'/\Zg'} \Big(\int_{\H} \phi((gh_1,g').w) dh_1\Big) d(g\H)\,d(g'\Zg')\\
&=\vol(\H) \int_{\G}\int_{\G'/\Zg'} \phi((g,g').w) \; dg\,d(g'\Zg')\,.
\end{align*}
In conclusion, there is a positive constant $C$ such that for all $\phi\in C_c(\Wv)$ and $w\in {\hs1}^{\rm reg}$
\begin{equation}
\int_{\Sg/\Sg^{\hs1}} \phi(s.w) \; d(s\Sg^{\hs1}) = C \int_{\G}\int_{\G'/\Zg'} \phi((g,g').w) \; dg\, d(g'\Zg')\,.
\end{equation}

Let us now consider the dual pair $(\G,\G')=(\Og_{2l+1},\Sp_{2l'}(\R))$ with $1\leq l\leq l'$. 
We keep the notation introduced on page \pageref{2l+1 smaller than 2l' special-the page}. 
In particular, $\V^0=\V^0_{\overline 0}\oplus\V^0_{\overline 1}$ where $\dim \V_{\overline 0}^{0}=1$
and $\dim \V_{\overline 1}^{0}=2(l'-l)$. Each $h\in \H^0$ fixes $\V^0_{\overline 0}$ and hence every 
$h\in \H$ is of the form $h=(h_\bullet,\varepsilon)$ where $h_\bullet \in \Og(\V^1_{\overline 0}\oplus \cdots \V^l_{\overline 0})\simeq \Og_{2l}$ and $\varepsilon \in \Og(\V^0_{\overline 0})$. 
The elements $h_\bullet$ form a Cartan subgroup $\H_\bullet$ of $\Og(\V^1_{\overline 0}\oplus \cdots \V^l_{\overline 0})$. At the group level, the decomposition $\h'=\h\oplus\h''$ arising from the identification \eqref{the identification} corresponds to a decomposition $\H'=\H_\bullet\times \H''$
of the Cartan subgroup $\H'$ of $\G'$. 

If $l=l'$, then $\h''=0$ and the equality $\z'=\h'=\h$ corresponds, at the group level, to 
$\Zg'=\H'=\H_\bullet$. Hence $\Sg^{\hs1^2}=\H\times \Zg'=\H\times \H_\bullet\cong \H_\bullet\times \H_\bullet \times \Og(\V^0_{\overline 0})$ and $\Sg^{\hs1}=\{(h_\bullet,\varepsilon,h_\bullet); h_\bullet \in \H_\bullet\} \cong \Delta(\H_\bullet)\times \Og(\V^0_{\overline 0})$, where 
$\Delta(\H_\bullet)=\{(h,h); h\in \H_\bullet\}$. Thus $\Lg=\Sg^{\hs1^2}/\Sg^{\hs1}\cong (\H_\bullet   \times \H_\bullet)/\Delta(\H_\bullet)$ is a compact abelian group and, as a set, 
$$
(\Sg/\Sg^{\hs1})/\Lg=\big((\G\times\G')/\Sg^{\hs1}\big)/ \big((\H\times \Zg')/\Sg^{\hs1}\big)=
\G/\H\times \G'/\Zg'\,,
$$
as in \eqref{SmodSh1modL}. Hence \eqref{integralonS/Sh1-1} follows as in the general case $l\leq l'$.

Let us now consider the dual pair $(\G,\G')=(\Og_{2l+1},\Sp_{2l'})$ with $1\leq l<l'$. 
Let $0\neq w_0\in \mathfrak{s}_1(\V^0)=\Hom(\V_{\overline 1}^0,
\V_{\overline 0}^0)$. We shall describe $\Sg^{(\hs1+w_0)^2}$ and its subgroup $\Sg^{\hs1+w_0}$.

Since $\hs1$ preserves the decomposition \eqref{decomposition of space for a cartan subspace}, we see that $(\hs1+w_0)^2=\hs1^2+w_0^2$ and hence
\begin{align}
\nn
\Sg^{(\hs1+w_0)^2}=\Sg^{\hs1^2+w_0^2}=\big(\Sg^{\hs1^2}\big)^{w_0^2}&=\H_\bullet \times \Og(\V_{\overline 0}^{0}) \times\H_\bullet \times \Sp(\V_{\overline 1}^{0})^n,\\
\label{S(h1+w)2-O2l+1}
&\simeq\H_\bullet \times \H_\bullet \times \big(\Og(\V_{\overline 0}^{0}) \times \Sp(\V_{\overline 1}^{0})\big)^{w_0^2},
\end{align}
where $\Og(\V_{\overline 0}^{0})=\{\pm 1\}$ and $\Sp(\V_{\overline 1}^{0})^n$ is the centralizer of $n=\tau'(w_0)$ in the symplectic group $\Sp(\V_{\overline 1}^{0})$.
Notice that we can also write 
\begin{equation}
\label{S(h1+w)2-O2l+1-bis}
\Sg^{(\hs1+w_0)^2}=\H \times {\Zg'\,}^n\,,
\end{equation}
where ${\Zg'\,}^n$ is the centralizer of $n$ in $\Zg'$.
In the identification \eqref{S(h1+w)2-O2l+1}, 
\begin{align}
\nn
\Sg^{\hs1+w_0}&=\big\{(h,h,s); h\in \H_\bullet, s\in \big(\Og(\V_{\overline 0}^{0}) \times \Sp(\V_{\overline 1}^{0})\big)^{w_0}\big\}\\
\label{S(h1+w)-O2l+1}
&=\Delta(\H_\bullet)\times \big(\Og(\V_{\overline 0}^{0}) \times \Sp(\V_{\overline 1}^{0})\big)^{w_0}\,.
\end{align}
The groups $\big(\Og(\V_{\overline 0}^{0}) \times \Sp(\V_{\overline 1}^{0})\big)^{w_0^2}$ and 
$\big(\Og(\V_{\overline 0}^{0}) \times \Sp(\V_{\overline 1}^{0})\big)^{w_0}$ are computed as in 
\eqref{Sw2-O1} and \eqref{Sw-O1}, respectively, with $\V$ replaced by $\V^0$.
Then
\begin{align*}
\Lg=\Sg^{(\hs1+w_0)^2}/\Sg^{\hs1+w_0}&\simeq (\H_\bullet \times \H_\bullet)/\Delta(\H_\bullet) \times \big(\Og(\V_{\overline 0}^{0}) \times \Sp(\V_{\overline 1}^{0})\big)^{w_0^2}/
\big(\Og(\V_{\overline 0}^{0}) \times \Sp(\V_{\overline 1}^{0})\big)^{w_0}\\
&\cong (\H_\bullet \times \H_\bullet)/\Delta(\H_\bullet) \times \{\pm 1\}\,,
\end{align*}
which is a compact abelian group. 
By \eqref{S(h1+w)2-O2l+1-bis}, we therefore obtain that, as a set,
\begin{align*}
(\Sg/\Sg^{\hs1+w_0})/\Lg=(\G\times\G')/(\H \times {\Zg'\,}^n)=
\G/\H\times \G'/{\Zg'\,}^n\,,
\end{align*}
and \eqref{integralonS/Sh1+w0} follows as in the general case $l\leq l'$.

The proof of \eqref{integralonS/Sh1-2} is similar to that of \eqref{integralonS/Sh1-1} and left to reader. 
\null\hfill\qed

%%%%%%%%%%%%%%%%%%%%%%
\biblio

\newcommand{\etalchar}[1]{$^{#1}$}
\begin{thebibliography}{ABP{\etalchar{+}}07}

\bibitem[AB95]{AdamsBarbaschcomplex}
J.~Adams and D.~Barbasch.
\newblock {Reductive dual pairs correspondence for complex groups}.
\newblock {\em {J. Funct. Anal.}}, {132}:{1--42}, {1995}.

\bibitem[ABP{\etalchar{+}}07]{AdamsBarbaschPaulTrapaVogan}
J.~Adams, D.~Barbasch, A.~Paul, P.~Trapa, and D.~A. Vogan.
\newblock {Unitary Shimura corresondence for split real groups}.
\newblock {\em {J. Amer. Math. Soc.}}, {20}:{701--751}, {2007}.

\bibitem[ABR09]{Ricci_2009}
F.~Astengo, B.~Di Blasio, and F.~Ricci.
\newblock {Gelfand pairs on the Heisenberg group and Schwartz functions}.
\newblock {\em {J. Funct. Anal.}}, {256}:{1565--1587}, {2009}.

\bibitem[Ada83]{Adamsdiscrete}
J.~Adams.
\newblock {Discrete spectrum of the dual pair $(O(p,q), Sp(2m,\Bbb R)$}.
\newblock {\em {Invent. Math.}}, {74}:{449--475}, {1983}.

\bibitem[Ada87]{Adamshighestweight}
J.~Adams.
\newblock {Unitary highest weight modules}.
\newblock {\em {Adv. Math.}}, {63}:{113--137}, {1987}.

\bibitem[Ada98]{Adamslift}
J.~Adams.
\newblock {Lifting of characters on orthogonal and metaplectic groups}.
\newblock {\em {Duke Math. J.}}, {92}(1):{129--178}, {1998}.

\bibitem[AP14]{AubertPrzebinda_omega}
A.-M. Aubert and T.~Przebinda.
\newblock {A reverse engineering approach to the Weil Representation}.
\newblock {\em {Central Eur. J. Math.}}, {12}:{1500--1585}, {2014}.

\bibitem[Asl96]{Aslaksen96}
Helmer Aslaksen.
\newblock Quaternionic determinants.
\newblock {\em Math. Intelligencer}, 18(3):57--65, 1996.

\bibitem[BP14]{BerPrzeCHC_inv_eig}
F.~Bernon and T.~Przebinda.
\newblock {The Cauchy Harish-Chandra integral and the invariant
  eigendistributions}.
\newblock {\em {Internat. Math. Res. Notices}}, {14}:{3818--3862}, {2014}.

\bibitem[Cur84]{Curtis_book}
Morton~L. Curtis.
\newblock {\em Matrix groups}.
\newblock Universitext. Springer-Verlag, New York, second edition, 1984.

\bibitem[DK00]{Duistermaat-Kolk}
J.~J. Duistermaat and J.~A.~C. Kolk.
\newblock {\em Lie groups}.
\newblock Universitext. Springer-Verlag, Berlin, 2000.

\bibitem[DKP97]{DaszKrasPrzebindaComplex}
A.~Daszkiewicz, W.~Kra\'skiewicz, and T.~Przebinda.
\newblock {Nilpotent Orbits and Complex Dual Pairs}.
\newblock {\em {J. Algebra}}, {190}:{518--539}, {1997}.

\bibitem[DKP05]{DaszKrasPrzebindaK-S2}
A.~Daszkiewicz, W.~Kra\'skiewicz, and T.~Przebinda.
\newblock {Dual Pairs and Kostant-Sekiguchi Correspondence. II. Classification
  of Nilpotent Elements}.
\newblock {\em {Central Eur. J. Math.}}, {3}:{430--464}, {2005}.

\bibitem[DM99]{DeligneMorgan99}
Pierre Deligne and John~W. Morgan.
\newblock Notes on supersymmetry (following {J}oseph {B}ernstein).
\newblock In {\em Quantum fields and strings: a course for mathematicians,
  {V}ol. 1, 2 ({P}rinceton, {NJ}, 1996/1997)}, pages 41--97. Amer. Math. Soc.,
  Providence, RI, 1999.

\bibitem[DP96]{DaszPrzebindaInv}
A.~Daszkiewicz and T.~Przebinda.
\newblock {The oscillator character formula, for isometry groups of split forms
  in deep stable range}.
\newblock {\em {Invent. Math.}}, {123}({2}):{349--376}, {1996}.

\bibitem[DV90]{OrbitesDV}
M.~Duflo and M.~Vergne.
\newblock Orbites coadjointes et cohomologie \'equivariante.
\newblock In {\em The orbit method in representation theory (Copenhagen,
  1988)}, volume~82 of {\em Progr. Math.}, pages 11--60. Birkh\"auser Boston,
  Boston, MA, 1990.

\bibitem[EHW83]{EnrightHoweWallach}
T.~J. Enright, R.~Howe, and N.~R. Wallach.
\newblock {A classification of unitary highest weight modules}.
\newblock {\em {Proceedings of Utah Conference, 1982}}, pages {97--143},
  {1983}.

\bibitem[Erd53]{Erdelyi}
A.~Erdelyi.
\newblock {\em {Higher Transcendental Functions}}.
\newblock McGraw-Hill, {1953}.

\bibitem[Fol89]{FollandPhase}
Gerald~B. Folland.
\newblock {\em Harmonic analysis in phase space}, volume 122 of {\em Annals of
  Mathematics Studies}.
\newblock Princeton University Press, Princeton, NJ, 1989.

\bibitem[GW98]{GoodmanWallach}
R.~Goodman and N.~Wallach.
\newblock {\em {Representations and Invariants of the Classical Groups}}.
\newblock {Cambridge University Press}, {1998}.

\bibitem[GW09]{GoodmanWallach2009}
R.~Goodman and N.~Wallach.
\newblock {\em {Symmetry, Representations and Invariants}}.
\newblock {Springer}, {2009}.

\bibitem[{Har}55]{HC1955c}
{Harish-Chandra}.
\newblock {Representations of Semisimple Lie Groups IV}.
\newblock {\em {Amer. J. Math.}}, {77}:{743--777}, {1955}.

\bibitem[{Har}56]{HC-56}
{Harish-Chandra}.
\newblock {The Characters of Semisimple Lie Groups}.
\newblock {\em {Trans. Amer. Math. Soc}}, {83}:{98--163}, {1956}.

\bibitem[{Har}57]{HC-57DifferentialOperators}
{Harish-Chandra}.
\newblock {Differential operators on a semisimple Lie algebra}.
\newblock {\em {Amer. J. Math.}}, {79}:{87--120}, {1957}.

\bibitem[Hat02]{Hatcher}
Allen Hatcher.
\newblock {\em Algebraic topology}.
\newblock Cambridge University Press, Cambridge, 2002.

\bibitem[He03]{HeHongyu}
H.~He.
\newblock {Unitary Representations and Theta Correspondence for Type I
  Classical Groups}.
\newblock {\em {J. Funct. Anal.}}, {199}:{92--121}, {2003}.

\bibitem[Hel78]{HelgasonDifferential}
S.~Helgason.
\newblock {\em {Differential Geometry, Lie Groups, and Symmetric Spaces}}.
\newblock {Academic Press}, {1978}.

\bibitem[HLS11]{HarrisLiSun}
M.~Harris, J.-S Li, and B.~Sun.
\newblock Theta correspondences for close unitary groups.
\newblock In {\em Arithmetic geometry and automorphic forms}, volume~19 of {\em
  Adv. Lect. Math. (ALM)}, pages 265--307. Int. Press, Somerville, MA, 2011.

\bibitem[H{\"o}r83]{Hormander}
L.~H{\"o}rmander.
\newblock {\em {The Analysis of Linear Partial Differential Operators I}}.
\newblock {Springer Verlag}, {1983}.

\bibitem[How79]{howetheta}
R.~Howe.
\newblock {$\theta $}-series and invariant theory.
\newblock In {\em Automorphic forms, representations and $L$-functions (Proc.
  Sympos. Pure Math., Oregon State Univ., Corvallis, Ore., 1977), Part 1},
  Proc. Sympos. Pure Math., XXXIII, pages 275--285. Amer. Math. Soc.,
  Providence, R.I., 1979.

\bibitem[How89a]{HoweRemarks}
R.~Howe.
\newblock {Remarks on Classical Invariant Theory}.
\newblock {\em {Trans. Amer. Math. Soc.}}, {313}:{539--570}, {1989}.

\bibitem[How89b]{HoweTrans}
R.~Howe.
\newblock {Transcending Classical Invariant Theory}.
\newblock {\em {J. Amer. Math. Soc. 2}}, {2}:{535--552}, {1989}.

\bibitem[HT92]{HoweTan}
R.~Howe and E.~Tan.
\newblock {\em {Nonabelian harmonic analysis. Applications of ${\rm
  SL}(2,R)$}}.
\newblock {Springer Verlag}, {1992}.

\bibitem[Kna86]{knappLie2}
A.~Knapp.
\newblock {\em Representation Theory of Semisimple groups, an overview based on
  examples}.
\newblock Princeton Mathematical Series. Princeton University Press, Princeton,
  New Jersey, 1986.

\bibitem[Kob15]{TKobayashiProgram}
Toshiyuki Kobayashi.
\newblock A program for branching problems in the representation theory of real
  reductive groups.
\newblock In {\em Representations of reductive groups}, volume 312 of {\em
  Progr. Math.}, pages 277--322. Birkh\"{a}user/Springer, Cham, 2015.

\bibitem[KP16a]{Kobayashi-Pevzner-DSBO-I}
Toshiyuki Kobayashi and Michael Pevzner.
\newblock Differential symmetry breaking operators: {I}. {G}eneral theory and
  {F}-method.
\newblock {\em Selecta Math. (N.S.)}, 22(2):801--845, 2016.

\bibitem[KP16b]{Kobayashi-Pevzner-DSBO-II}
Toshiyuki Kobayashi and Michael Pevzner.
\newblock Differential symmetry breaking operators: {II}. {R}ankin-{C}ohen
  operators for symmetric pairs.
\newblock {\em Selecta Math. (N.S.)}, 22(2):847--911, 2016.

\bibitem[KS15]{Kobayashi_Speh_AMS}
Toshiyuki Kobayashi and Birgit Speh.
\newblock Symmetry breaking for representations of rank one orthogonal groups.
\newblock {\em Mem. Amer. Math. Soc.}, 238(1126), 2015.

\bibitem[KV78]{KashiwaraVergne}
M.~Kashiwara and M.~Vergne.
\newblock On the segal-shale-weil representation and harmonic polynomials.
\newblock {\em Invent. Math.}, 44:1--47, 1978.

\bibitem[Li89]{Jian-ShuLiSingular}
Jian-Shu Li.
\newblock {Singular unitary representations of classical groups}.
\newblock {\em {Invent. Math.}}, {97}({2}):{237--255}, {1989}.

\bibitem[Lit06]{Littlewood}
Dudley~E. Littlewood.
\newblock {\em The theory of group characters and matrix representations of
  groups}.
\newblock AMS Chelsea Publishing, Providence, RI, 2006.
\newblock Reprint of the second (1950) edition.

\bibitem[LM15]{LockMaassocvar}
H.~Y. Loke and J.J. Ma.
\newblock {Invariants and $K$-spectrum of local theta lifts}.
\newblock {\em {Compos. Math.}}, {151}:{179--206}, {2015}.

\bibitem[LP22]{LokePrzebinda_stable}
Hung~Yean Loke and Tomasz Przebinda.
\newblock {The character correspondence in the stable range over a $p$-adic
  field}.
\newblock {preprint, arXiv:2207.07298}, {2022}.

\bibitem[LP24]{LokePrzebinda_chc_padic_def}
Hung~Yean Loke and Tomasz Przebinda.
\newblock A {C}auchy {H}arish-{C}handra integral for a dual pair over a p-adic
  field, the definition and a conjecture.
\newblock {\em J. Funct. Anal.}, 287(7):Paper No. 110540, 66, 2024.

\bibitem[LPTZ03]{Jian-ShuLi-Cheng-boZhu}
Jian-Shu Li, Annegret Paul, Eng-Chye Tan, and Chen-Bo Zhu.
\newblock {The explicit duality correspondence of (Sp$(p,q)$,
  O$\null^\ast(2n))$}.
\newblock {\em J. Funct. Anal.}, 200(1):71--100, 2003.

\bibitem[Mer17]{AllanMerino-Thesis}
Allan Merino.
\newblock Caract{\`e}res de repr{\'e}sentations unitaires de plus haut poids
  via la correspondance de {H}owe et la formule de {R}ossmann-{D}uflo-{V}ergne,
  2017.
\newblock Th{\`e}se (Ph.D.)--Universit{\'e} de Lorraine.

\bibitem[Mer20]{Merino2019characters}
Allan Merino.
\newblock Characters of some unitary highest weight representations via the
  theta correspondence.
\newblock {\em J. Funct. Anal.}, 279(8):108698, 70, 2020.

\bibitem[{Moe}89]{Moeglinarchi}
{Moeglin C.}
\newblock {Correspondance de Howe pour les paires duales r\'eductives duales:
  quelques calculs dans la cas archim\'dien}.
\newblock {\em {J. Funct. Anal.}}, {85}:{1--85}, {1989}.

\bibitem[{Moe}98]{Moeglinarchiwave}
{Moeglin, C}.
\newblock {Correspondance de Howe et front d'onde}.
\newblock {\em {Adv. in Math.}}, {133}:{224--285}, {1998}.

\bibitem[MPP15]{McKeePasqualePrzebindaSuper}
M.~McKee, A.~Pasquale, and T.~Przebinda.
\newblock {Semisimple orbital integrals on the symplectic space for a real
  reductive dual pair}.
\newblock {\em {J. Funct. Anal.}}, {268}:{275--335}, {2015}.

\bibitem[MPP20]{McKeePasqualePrzebindaWCestimates}
M.~McKee, A.~Pasquale, and T.~Przebinda.
\newblock {Estimates of derivatives of elliptic orbital integral in a
  symplectic space}.
\newblock {\em {J. of Lie Theory}}, 30(2):489--512, 2020.

\bibitem[MPP23]{McKeePasqualePrzebindaWCSymmetryBreakingUU}
M.~McKee, A.~Pasquale, and T.~Przebinda.
\newblock {Symmetry breaking operators for the dual pair $(\Ug_l,\Ug_{l'})$}.
\newblock {to appear in Indag. Math. arxiv:2312.05546}, {2023}.

\bibitem[MPP24]{McKeePasqualePrzebindaWC_WF}
M.~McKee, A.~Pasquale, and T.~Przebinda.
\newblock {The wave front set correspondence for dual pairs with one member
  compact}.
\newblock {\em {Acta Math. Sinica, English Series.}}, {40}:823--869, {2024}.

\bibitem[MSZ17]{MaSunZhu_2017}
Jia-Jun Ma, Binyong Sun, and Chengbo Zhu.
\newblock {Unipotent representations of real classical groups}.
\newblock {arxiv:1712.0552v1}, 2017.

\bibitem[Pan10]{PanReal}
Shu-Yen Pan.
\newblock {Orbit correspondence for real reductive dual pairs}.
\newblock {\em {Pacific J. Math.}}, {248}({2}):{403--427}, {2010}.

\bibitem[Pau98]{AnnegretunitaryI}
A.~Paul.
\newblock {Howe correpondence for real unitary groups I}.
\newblock {\em {J. Funct. Anal.}}, {159}:{384--431}, {1998}.

\bibitem[Pau00]{AnnegretunitaryII}
A.~Paul.
\newblock {Howe correspondence for real unitary groups II}.
\newblock {\em {Proc. Amer. Math. Soc.}}, {128}:{3129--3136}, {2000}.

\bibitem[Pau05]{Annegretorthosymplectic}
A.~Paul.
\newblock {On the Howe correspondence for symplectic-orthogonal dual pairs}.
\newblock {\em {J. Funct. Anal.}}, {228}:{270--310}, {2005}.

\bibitem[PP08]{ProtsakPrzebindaStable}
V.~Protsak and T.~Przebinda.
\newblock {On the occurrence of admissible representations in the real Howe
  correspondence in stable range}.
\newblock {\em {Manuscripta Math.}}, {126}:{135--141}, {2008}.

\bibitem[Prz89]{PrzebindaExample}
T.~Przebinda.
\newblock {The oscillator duality correspondence for the pair $(\Og(2,2),
  \Sp(4,\Bbb R))$}.
\newblock {\em {Memoirs of the AMS}}, {403}, {1989}.

\bibitem[Prz91]{PrzebindaUnipotent}
T.~Przebinda.
\newblock {Characters, dual pairs, and unipotent representations}.
\newblock {\em { J. Funct. Anal.}}, {98}({1}):{59--96}, {1991}.

\bibitem[Prz93]{PrzebindaUnitary}
T.~Przebinda.
\newblock {Characters, dual pairs, and unitary representations}.
\newblock {\em {Duke Math. J.}}, {69}({3}):{547--592}, {1993}.

\bibitem[Prz96]{PrzebindaInfinitesimal}
T.~Przebinda.
\newblock {The duality correspondence of infinitesimal characters}.
\newblock {\em { Coll. Math}}, {70}:{93--102}, {1996}.

\bibitem[Prz00]{PrzebindaCauchy}
T.~Przebinda.
\newblock {A Cauchy Harish-Chandra Integral, for a real reductive dual pair}.
\newblock {\em {Invent. Math.}}, {141}({2}):{299--363}, {2000}.

\bibitem[Prz06]{PrzebindaLocal}
T.~Przebinda.
\newblock {Local Geometry of Orbits for an Ordinary Classical Lie Supergroup}.
\newblock {\em {Central Eur. J. Math.}}, {4}:{449--506}, {2006}.

\bibitem[Prz18]{PrzebindaStableUnitary}
T.~Przebinda.
\newblock {The character and the wave front set correspondence in the stable
  range}.
\newblock {\em {J. Funct. Anal.}}, {274}:{1284--1305}, {2018}.

\bibitem[Ral84]{Rallis-Howeduality}
S.~Rallis.
\newblock On the {H}owe duality conjecture.
\newblock {\em Compositio Math.}, 51(3):333--399, 1984.

\bibitem[Ren98]{RenardLift}
D.~Renard.
\newblock {Transfer of orbital integrals between $\Mp(2n,\R)$ and
  $\SOg(n,n+1)$}.
\newblock {\em {Duke Math J.}}, {95}:{125--450}, {1998}.

\bibitem[Tr{\`e}67]{Treves}
F.~Tr{\`e}ves.
\newblock {\em {Topological Vector Spaces, Ditributions and Kernels}}.
\newblock {Academic Press}, {1967}.

\bibitem[Var89]{VaradarajanHarmonic}
V.~S. Varadarajan.
\newblock {\em An introduction to harmonic analysis semisimple groups}.
\newblock Cambridge University press, 1989.

\bibitem[Vog78]{VoganGelfand}
D.~Vogan.
\newblock {Gelfand-Kirillov dimension for Harish-Chandra modules}.
\newblock {\em Invent. Math.}, 48:75--98, 1978.

\bibitem[Wal84]{Wallachholomorphic}
N.~Wallach.
\newblock The asymptotic behavior of holomorphic representations.
\newblock {\em M\'em. Soc. Math. France (N.S.)}, 15:291--305, 1984.

\bibitem[Wei64]{WeilWeil}
Andr\'{e} Weil.
\newblock Sur certains groupes d'op\'{e}rateurs unitaires.
\newblock {\em Acta Math.}, 111:143--211, 1964.

\bibitem[Wei73]{Weil_Basic}
A.~Weil.
\newblock {\em Basic Number Theory}.
\newblock Springer-Verlag, 1973.
\newblock Classics in Mathematics.

\bibitem[Wen01]{Wendt}
R.~Wendt.
\newblock {Weyl's Character Formula for Non-connected Lie Groups and Orbital
  Theory for Twisted Affine Lie Algebras}.
\newblock {\em {J. Funct. Anal.}}, {180}:{31--65}, {2001}.

\end{thebibliography}
\end{document}